\documentclass[reqno]{amsart}
\usepackage{amssymb}
\usepackage{amsmath}
\usepackage{amsthm}
\usepackage[usenames]{color}

\usepackage{a4wide}

\newtheorem{theorem}{Theorem}[section]

\newtheorem{lemma}[theorem]{Lemma}
\newtheorem{proposition}[theorem]{Proposition}
\newtheorem{remark}[theorem]{Remark}

\numberwithin{equation}{section}
\usepackage[colorlinks]{hyperref}
\usepackage{marginnote}

\usepackage{color}
\newcommand{\Red}{\textcolor{red}}

\usepackage[justification=centering]{caption}
\usepackage{tikz}
\usetikzlibrary{positioning}

\begin{document}
\title[ KdV limit for VPL system]{KdV limit for the Vlasov-Poisson-Landau system}

\author[R.-J. Duan]{Renjun Duan}
\address[R.-J. Duan]{Department of Mathematics, The Chinese University of Hong Kong, Shatin, Hong Kong,
	People's Republic of China}
\email{rjduan@math.cuhk.edu.hk}

\author[D.-C. Yang]{Dongcheng Yang}
\address[D.-C. Yang]{Department of Mathematics, The Chinese University of Hong Kong, Shatin, Hong Kong,
	People's Republic of China}
\email{dcyang@math.cuhk.edu.hk}

\author[H.-J. Yu]{Hongjun Yu}
\address[H.-J. Yu]{School of Mathematical Sciences, South China Normal University, Guangzhou 510631, People's Republic of China}
\email{yuhj2002@sina.com}

%\date{\today}

%\def\baselinestretch{1}

\begin{abstract}
Two fundamental models in plasma physics are given by the Vlasov-Poisson-Landau system
and the compressible Euler-Poisson system which both capture the complex dynamics of plasmas under the self-consistent electric field  
interactions at the kinetic and fluid levels, respectively. Although there have been extensive studies on the long wave limit of the Euler-Poisson system  towards Korteweg-de Vries equations,  few results on this topic are known for the Vlasov-Poisson-Landau system due to the  complexity of the system and its underlying multiscale feature. In this article, we derive and justify the Korteweg-de Vries equations from the Vlasov-Poisson-Landau system modelling the motion of ions under the Maxwell-Boltzmann relation. Specifically, under the Gardner-Morikawa transformation 
$$
(t,x,v)\to (\delta^{\frac{3}{2}}t,\delta^{\frac{1}{2}}(x-\sqrt{\frac{8}{3}}t),v)
$$
with $ \varepsilon^{\frac{2}{3}}\leq \delta\leq \varepsilon^{\frac{2}{5}}$ and $\varepsilon>0$ being the Knudsen
number, we construct smooth solutions of the rescaled Vlasov-Poisson-Landau system over an arbitrary finite time interval  that can converge uniformly to  smooth solutions of the Korteweg-de Vries equations as $\delta\to 0$. Moreover, the explicit rate of convergence in $\delta$  is also obtained.
The proof is obtained by an appropriately chosen scaling and the intricate weighted energy method through the micro-macro decomposition around local Maxwellians.
\end{abstract}



\maketitle

\setcounter{tocdepth}{1}
\tableofcontents
\thispagestyle{empty}

\thispagestyle{empty}

\section{Introduction}\label{sec.1}
% \vskip 0.2cm

\subsection{Formulation of the problem}\label{sec.1.1}
Two fundamental models in plasma physics are given by the Vlasov-Poisson-Landau (VPL in short) system
and the  Euler-Poisson system which both capture the complex dynamics of plasmas under the self-consistent electric field  
interactions at the kinetic and fluid levels, respectively, cf. \cite{Chen-F,Krall-Trivelpiece}. 
It is well known that the  Euler-Poisson system can be derived from the VPL system 
via the formal Hilbert expansion. Moreover, many celebrated nonlinear dispersive equations, such as  Korteweg-de Vries (KdV in short),  
Zakharov-Kuznetsov, Kadomstev-Petviashvili, nonlinear Schr\"{o}dinger, etc, can be formally derived from the Euler-Poisson system 
under various asymptotic limits. A natural and important question is whether or not those dispersive equations can 
be derived from the VPL system. In this paper, we derive rigorously the KdV equation from the VPL system as the first step for this subject.

We consider the spatially one-dimensional VPL system  modelling the motion of ions under a Boltzmann relation, which reads
\begin{equation}
	\label{1.1}
	\left\{
	\begin{array}{rl}
		\partial_{\bar{t}}\overline{F}+\bar{v}_{1}\partial_{\bar{x}}\overline{F}
		-\partial_{\bar{x}}\phi\partial_{\bar{v}_{1}}\overline{F}&=\frac{1}{\varepsilon}Q(\overline{F},\overline{F}),
		\\
		-\partial^{2}_{\bar{x}}\overline{\phi}&=\rho-\rho_\mathbf{e}, \quad \rho=\int_{\mathbb{R}^{3}}\overline{F}\,d\bar{v}.
	\end{array} \right.
\end{equation}
In the above expression, the parameter $\varepsilon>0$ is small and reciprocal to logarithm of the Debye shielding length, and it plays the same role as the Knudsen number in the Boltzmann theory, cf.~\cite{Cercignani, Sone}. The unknowns are  $\overline{F}=\overline{F}(\bar{t},\bar{x},\bar{v})\geq0$
standing for the density distribution function for ions with velocity $\bar{v}=(\bar{v}_{1},\bar{v}_{2},\bar{v}_{3})\in\mathbb{R}^{3}$ at position
$\bar{x}\in\mathbb{R}$ and time $\bar{t}\geq0$. 
The self-consistent electric potential $\overline{\phi}=\overline{\phi}(\bar{t},\bar{x})$  is coupled with the distribution function $\overline{F}$ through the Poisson equation.
The density of electrons $\rho_\mathbf{e}$ are described by the classical Boltzmann relation, cf. \cite{Bastdos-Golse2018,Duan-Yang-Yu-VPL,Guo-Pu}
\begin{equation}
\label{1.2}	
\rho_\mathbf{e}=e^{\overline{\phi}}.
\end{equation}
We remark that one can refer to \cite{Flynn} for a rigorous derivations from a two-species system via the  massless electron limit.

The collision between particles is given by the Landau operator:
\begin{equation}
\label{1.3}
Q(F_{1},F_{2})(\bar{v})=\nabla_{\bar{v}}\cdot\int_{\mathbb{R}^{3}}\Phi(\bar{v}-\bar{v}_{*})\left\{F_{1}(\bar{v}_{*})\nabla_{\bar{v}}F_{2}(\bar{v})
-F_{2}(\bar{v})\nabla_{\bar{v}_{*}}F_{1}(\bar{v}_{*})\right\}\,d\bar{v}_{*},
\end{equation}
Here the famous Landau (Fokker-Planck) kernel $\Phi(\xi)=[\Phi_{ij}(\xi)]$ is given by
\begin{equation*}
\Phi_{ij}(\xi)=\frac{1}{|\xi|}\big(\delta_{ij}-\frac{\xi_i\xi_j}{|\xi|^{2}}\big),\quad 1\leq i,j\leq 3,
\end{equation*}
with $\delta_{ij}$ being the Kronecker delta.

At the formal level, the VPL system \eqref{1.1} tends to the following non-isentropic compressible Euler-Poisson system \eqref{1.4} 
as $\varepsilon\to$,
 which is an important "two-fluid" model for a plasma:
\begin{equation}
\left\{
\begin{array}{rl}
		\label{1.4}
		&\partial_{\bar{t}}\overline{R}+\partial_{\bar{x}}(\overline{R}\overline{U})=0,
		\\
		&\partial_{\bar{t}}(\overline{R}\overline{U})+\partial_{\bar{x}}(\overline{R}\overline{U}^2)+\partial_{\bar{x}}\overline{P}
		+\overline{R}\partial_{\bar{x}}\overline{\Pi}=0,
		\\
		&\partial_{\bar{t}}[\overline{R}(\overline{\varTheta}+\frac{|\overline{R}|^{2}}{2})]+\partial_{\bar{x}}[\overline{R}\overline{U}
		(\overline{\varTheta}+\frac{|\overline{U}|^{2}}{2})
		+\overline{P}\overline{U}]+\overline{R}\overline{U}\partial_{\bar{x}}\overline{\Pi}=0,
		\\
		&-\partial^2_{\bar{x}}\overline{\Pi}=\overline{R}-e^{\overline{\Pi}},
	\end{array} \right.
\end{equation}
where $\overline{R}=\overline{R}(\bar{t},\bar{x})>0$, $\overline{U}=\overline{U}(\bar{t},\bar{x})$, $\overline{\varTheta}
=\overline{\varTheta}(\bar{t},\bar{x})>0$ and $\overline{\Pi}=\overline{\Pi}(\bar{t},\bar{x})$ 
are the density, velocity, temperature of the ions and the electric potential, respectively. The pressure $\overline{P}$ is given by the ideal gas law as $\overline{P}=K\overline{R}\overline{\varTheta}$
with the gas constant $K=\frac{2}{3}$ taken for convenience. 
The rigorous mathematical justification of the hydrodynamic limit from the Vlasov-Poisson-Boltzmann system
to the isentropic compressible Euler-Poisson system was done in \cite{Guo-Jang} by Guo and Jang, see also
some more profound works on this topic \cite{Duan-Yang-Yu-VMB,Duan-Yang-Yu-VPL,Lei-1}

Recall that the problem on the {\it long wave limit} of the Euler-Poisson system can be traced back to \cite{Korteweg,Su-1,Washimi} for formal derivations;
see recent progress \cite{Lannes,Guo-Pu} for rigorous proofs.
Following the idea in \cite{Guo-Pu}, we use the classical Gardner-Morikawa transformation \cite{Su-1}
\begin{equation}
\label{1.5}
x=\delta^{1/2}(\bar{x}-A\bar{t}), \quad \quad t=\delta^{3/2}\bar{t},
\end{equation}
and let $(R,U,\varTheta,\Pi)(t,x)=(\overline{R},\overline{U},\overline{\varTheta},\overline{\Pi})(\bar{t},\bar{x})$ to rewrite
the system \eqref{1.4} as
\begin{equation}
	\left\{
	\begin{array}{rl}
		\label{1.6}
		&\delta\partial_{t}R-A\partial_{x}R+\partial_{x}(RU)=0,
		\\
		&\delta R\partial_{t}U-A R\partial_{x}U+RU\partial_{x}U+\frac{2}{3}R\partial_{x}\varTheta+\frac{2}{3}\varTheta\partial_{x}R+R\partial_{x}\Pi=0,
		\\
		&\delta\partial_{t}\varTheta-A\partial_{x}\varTheta+\frac{2}{3}\varTheta\partial_{x}U
		+U\partial_{x}\varTheta=0,
		\\
		&\delta\partial^2_{x}\Pi=e^{\Pi}-R,
	\end{array} \right.
\end{equation}
where $A>0$ is a constant to be determined and $\delta>0$ is the amplitude of the initial disturbance.
Assume $(R,U,\varTheta,\Pi)(t,x)$ has the following formal expansion:
\begin{equation}
	\label{1.7}
	\left\{
	\begin{array}{rl}
		&R=1+\delta\rho_1+\delta^2\rho_2+\delta^3\rho_3+\cdot\cdot\cdot
		\\
		&U=\delta u^{(1)}_1+\delta^2 u^{(2)}_1+\delta^3 u^{(3)}_1+\cdot\cdot\cdot
		\\
		&\varTheta=\frac{3}{2}+\frac{3}{2}\delta\theta_1+\frac{3}{2}\delta^2\theta_2+\frac{3}{2}\delta^3\theta_3+\cdot\cdot\cdot
		\\
		&\Pi=\delta\phi_1+\delta^2 \phi_2+\delta^3 \phi_3+\cdot\cdot\cdot.
	\end{array} \right.	
\end{equation}
Then, substituting \eqref{1.7} into \eqref{1.6}, one gets a power series of $\delta$, whose coefficients depend on
$(\rho_k,u^{(k)}_1,\theta_k,\phi_k)$ for $k=1,2,\cdot\cdot\cdot$. This leads us to obtain
\begin{equation}
\label{1.8}
		u^{(1)}_{1}=A\rho_{1},\quad
		\theta_{1}=\frac{2}{3}\rho_{1},
		\quad \phi_1=\rho_1,\quad A=\sqrt{\frac{8}{3}},
\end{equation}
where $\rho_1$ satisfies the KdV equation:
\begin{equation}
\label{1.9}	
2\sqrt{\frac{8}{3}}\partial_{t}\rho_{1}+\frac{58}{9}\rho_1\partial_{x}\rho_{1}+\partial^3_{x}\rho_{1}=0.
\end{equation}
For a detailed formal derivation, see Section \ref{sub2.2} later on. Moreover, we also get the following relations:
\begin{equation}
	\left\{
	\begin{array}{rl}
		\label{1.10}
		&u^{(2)}_{1}=\sqrt{\frac{8}{3}}\rho_{2}-\int^{x} [\partial_{t}\rho_{1}+\partial_{x}(\rho_1u^{(1)}_1)](t,\xi)\,d\xi, 
		\\
		&\theta_{2}=\frac{2}{3}\rho_{2}-\frac{1}{9}\rho^2_{1},
		\\
		&\phi_2=\rho_2+\partial^2_{x}\rho_{1}-\frac{1}{2}\rho^2_1,
	\end{array} \right.
\end{equation}
where  $\rho_2$ satisfies the linearized inhomogeneous KdV equation:
\begin{equation}
	\label{1.11}
	2\sqrt{\frac{8}{3}}\partial_{t}\rho_{2}+\frac{58}{9}\partial_{x}(\rho_{1}\rho_2)+\partial^3_{x}\rho_{2}=g(\rho_1),
\end{equation}
with the function $g(\rho_1)$ depending only on $\rho_1$.
Similar relationship also holds for $(\rho_k,u^{(k)}_1,\theta_k,\phi_k)$ for $k=3,4,\cdot\cdot\cdot$.
Interested readers may refer to \cite{Guo-Pu} for a detailed formal derivation.

The well-posedness theory on the Cauchy problem of the KdV equation has been well-studied in \cite{Kenig,Colliander} and the references therein.
From \cite{Kenig}, we have the following proposition on the existence result of the  KdV equation \eqref{1.8}-\eqref{1.9}.
\begin{proposition}\label{prop.1.1}
	Let $s_1\geq2$ be a sufficiently large integer. Then there exists $\tau_1>0$ such that the Cauchy problem on the KdV equation \eqref{1.9} and 
	\eqref{1.8} with initial data $(\rho_1,u^{(1)}_1,\theta_1,\phi_1)(0,x)\in H^{s_1}(\mathbb{R})$ satisfying \eqref{1.8},  admits a unique
	smooth  solution
	\begin{equation}
		\label{1.12}
		(\rho_1,u_1^{(1)},\theta_1,\phi_1)(t,x)\in L^{\infty}(-\tau_1,\tau_1;H^{s_1}(\mathbb{R})).
	\end{equation}
 Furthermore, by using the conservation laws
	of the KdV equation, we can extend the solution to any time interval $[-\tau,\tau]$.
\end{proposition}
Based on Proposition \ref{prop.1.1}, the existence result of the linearized inhomogeneous KdV equation \eqref{1.11} can be summarized in the following,
whose proofs is similar to the one  in \cite[Theorem 1.2]{Guo-Pu}.
\begin{proposition}\label{prop.2.1}
	Let $s_1>3$ as in \eqref{1.12} and $s_2\leq s_1-3$ be a sufficiently large integer. 
	Then for any $\tau>0$, the Cauchy problem \eqref{1.11} and \eqref{1.10} with initial data $(\rho_2,u^{(2)}_1,\theta_2,\phi_2)(0,x)\in H^{s_2}(\mathbb{R})$
satisfying \eqref{1.10}, admits a unique smooth solution
	\begin{equation}
		\label{1.13}
		(\rho_2,u_1^{(2)},\theta_2,\phi_2)(t,x)\in L^{\infty}(-\tau,\tau;H^{s_2}(\mathbb{R})).
	\end{equation}
\end{proposition}
In order to derive the KdV equation from the VPL system \eqref{1.1}, we have to propose a long wave
scaling for the VPL system. Motivated by \eqref{1.5}, we introduce the scaled transformation
\begin{equation}
\label{1.14}
	x=\delta^{1/2}(\bar{x}-A\bar{t}), \quad \quad t=\delta^{3/2}\bar{t},\quad \quad v=\bar{v}.
\end{equation}
Using \eqref{1.14} and letting $F(t,x,v)=\overline{F}(\bar{t},\bar{x},\bar{v})$ and $\phi(t,x)=\bar{\phi}(\bar{t},\bar{x})$ in \eqref{1.1},
we obtain the rescaled VPL system:
\begin{equation}
	\label{1.15}
	\left\{
	\begin{array}{rl}
		\delta\partial_{t}F-A\partial_{x}F+v_{1}\partial_{x}F-\partial_{x}\phi\partial_{v_{1}}F&=\frac{1}{\delta^{1/2}\varepsilon}Q(F,F),
		\\
		-\delta\partial^{2}_{x}\phi&=\rho-e^{\phi}, \quad \rho=\int_{\mathbb{R}^{3}}F\,dv.
	\end{array} \right.
\end{equation}
Note that the Euler-Poisson system \eqref{1.6} can be formally derived from the rescaled  VPL system \eqref{1.15}
by taking the fluid limit $\varepsilon\to 0$, while the KdV equation \eqref{1.9} can be formally derived from the system \eqref{1.6}
by taking the long wave limit $\delta\to 0$. Our goal in this work is to  perform simultaneously the fluid limit and  long wave limit:
\begin{equation}
	\label{1.16}
\varepsilon\to 0, \quad  \mbox{and}\quad \delta\to 0.
\end{equation}
The relations between the different systems can be described in Figure \ref{fig.asp} below:
\begin{figure}[h]
	\begin{tikzpicture}	
		\node (A) at (0,0) [draw, thick, fill=white, fill opacity=0.7,
		minimum width=2cm,minimum height=0.8cm] {$\mbox{VPL system}$};
		
		\draw[->] (1.1,0) -- (3.2,0); 
		\node[above] at (2.1,0){$\varepsilon \to 0$};
		
		%\draw[->] (0,-0.4) -- (0,-2); 
		%\node[right] at (0,-1.2){$\varepsilon \to 0$};
		
		\node (B) at (5,0) [draw, thick, fill=white, fill opacity=0.7,
		minimum width=2cm,minimum height=0.8cm] {$\mbox{Euler-Poisson system}$};
		\draw[->] (5,-0.4) -- (5,-2); 
		\node[right] at (5,-1.2){$\delta \to 0$};
		
		%\node (C) at (0,-2.4) [draw, thick, fill=white, fill opacity=0.7,
		%minimum width=2cm,minimum height=0.8cm] {$\mbox{Euler-Poisson}$};
		%\draw[->] (1.15,-2.5) -- (2.75,-2.5); 
		%\node[above] at (1.8,-2.5){$\lambda \to 0$};
		
		\node (D) at (5,-2.4) [draw, thick, fill=white, fill opacity=0.7,
		minimum width=2cm,minimum height=0.8cm] {$\mbox{KdV equation}$};
		%\draw[->] (2.4,-3) -- (3.3,-3); 
		%\node[above] at (2.85,-3){$\varepsilon \rightarrow 0$};
		
		\coordinate (M) at (1.1,-0.4);
		\coordinate (N) at (3.8,-2.0);
		\draw [->] (M) -- (N) {};
		\node[above] at (1.4,-1.9){$ \varepsilon \to 0, \delta\to 0$};
	\end{tikzpicture}
	\caption{Dispersive limits of the VPL system}
	\label{fig.asp}	
\end{figure}
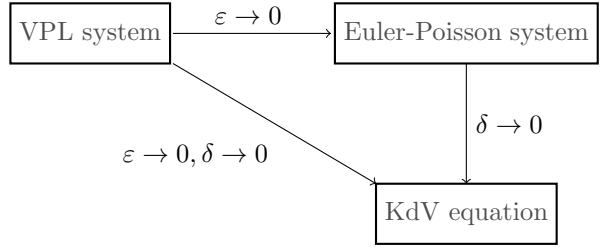

\subsection{Macro-micro decomposition}\label{sec.1.2}
For the above purpose, we present the macro-micro decomposition of the solution for the VPL system  \eqref{1.15}
with respect to the local Maxwellian, that was initiated by Liu-Yu \cite{Liu-Yu} and developed by Liu-Yang-Yu \cite{Liu-Yang-Yu}
for the Boltzmann equation.

First of all, it is well known that the Landau collision operator  admits five collision invariants:
\begin{equation*}
\psi_{0}(v)=1, \quad \psi_{i}(v)=v_{i},~(i=1,2,3),\quad \psi_{4}(v)=\frac{1}{2}|v|^{2},
\end{equation*}
namely, it holds that
\begin{equation*}
	\int_{\mathbb{R}^{3}}\psi_{i}(v)Q(F,F)\,dv=0,\quad \mbox{for $i=0,1,2,3,4$.}
\end{equation*}
For a given solution $F(t,x,v)$ of the VPL system \eqref{1.15}, we introduce five macroscopic (fluid) quantities:
the mass density $\rho(t,x)>0$, momentum $\rho(t,x)u(t,x)$, and
energy density $\mathcal{E}(t,x)+\frac 12|u(t,x)|^2$, given as
\begin{equation}
	\label{1.17}
	\left\{
	\begin{array}{rl}
		\rho(t,x)&\equiv\int_{\mathbb{R}^{3}}\psi_{0}(v)F(t,x,v)\,dv,
		\\
		\rho(t,x) u_{i}(t,x)&\equiv\int_{\mathbb{R}^{3}}\psi_{i}(v)F(t,x,v)\,dv, \quad \mbox{for $i=1,2,3$,}
		\\
		\rho(t,x)(\mathcal{E}(t,x)+\frac{1}{2}|u(t,x)|^{2})&\equiv\int_{\mathbb{R}^{3}}\psi_{4}(v)F(t,x,v)\,dv,
	\end{array} \right.
\end{equation}
and the corresponding local Maxwellian:
\begin{equation}
	\label{1.18}
	M\equiv M_{[\rho,u,\theta]}(t,x,v):=\frac{\rho(t,x)}{\sqrt{(2\pi K\theta(t,x))^{3}}}\exp\big\{-\frac{|v-u(t,x)|^{2}}{2K\theta(t,x)}\big\}.
\end{equation}
Here $u(t,x)=(u_{1},u_{2},u_{3})(t,x)$ is the bulk velocity, $\mathcal{E}(t,x)>0$ is the internal energy which is related to the temperature $\theta(t,x)$ by
$\mathcal{E}=\frac{3}{2}K\theta=\theta$ with $K=\frac{2}{3}$.

Letting the $L^{2}$ inner product in  $v\in\mathbb{R}^{3}$  be defined as
$\langle h,g\rangle =\int_{\mathbb{R}^{3}}h(v)g(v)\,d v$.
Then  the subspace spanned by the collision invariants has the following set of orthogonal basis:
\begin{equation}
	\label{1.19}
	\left\{
	\begin{array}{rl}
		&\chi_{0}(v)=\frac{1}{\sqrt{\rho}}M,
		\quad \chi_{i}(v)=\frac{v_{i}-u_{i}}{\sqrt{K\rho\theta}}M, \quad i=1,2,3,
		\\
		&\chi_{4}(v)=\frac{1}{\sqrt{6\rho}}(\frac{|v-u|^{2}}{K\theta}-3)M,
		\\
		&\langle \chi_{i},\frac{\chi_{j}}{M}\rangle=\delta_{ij},
		\quad \mbox{for ~~$i,j=0,1,2,3,4$}.
	\end{array} \right.
\end{equation}
The macroscopic projection  $P_{0}$ and  the microscopic projection $P_{1}$ can be defined as
\begin{equation}
	\label{1.20}
	P_{0}h\equiv\sum_{i=0}^{4}\langle h,\frac{\chi_{i}}{M}\rangle\chi_{i},\quad P_{1}h\equiv h-P_{0}h,
\end{equation}
respectively, where  $P_{0}$ and  $P_{1}$  are  orthogonal projections satisfying
$$
P_{0}P_{0}=P_{0},\quad
P_{1}P_{1}=P_{1},\quad
P_{1}P_{0}=P_{0}P_{1}=0.
$$
A function $h(v)$ is called microscopic or non-fluid if
\begin{equation}
	\label{1.21}
	\langle h(v),\psi_{i}(v)\rangle=0, \quad \mbox{for $i=0,1,2,3,4$}.
\end{equation}
Using the notations above, the solution $F(t,x,v)$ of the VPL system \eqref{1.15} can be decomposed as
\begin{equation}
	\label{1.22}
	F=M+G, \quad P_{0}F=M, \quad P_{1}F=G,
\end{equation}
where $M=M(t,x,v)$ defined in \eqref{1.18} and $G=G(t,x,v)$
represent the macroscopic  (fluid) and microscopic  (non-fluid) component of the solution, respectively.

Then the first equation of \eqref{1.15} becomes
\begin{equation}
	\label{1.23}
	\delta\partial_t(M+G)-A\partial_x(M+G)+v_{1}\partial_x(M+G)-\partial_x\phi\partial_{v_{1}}(M+G)
	=\frac{1}{\delta^{1/2}\varepsilon}L_{M}G+\frac{1}{\delta^{1/2}\varepsilon}Q(G,G),
\end{equation}
where $L_{M}$ is the linearized  collision operator with respect to the local Maxwellian $M$, given by
\begin{equation}
	\label{1.24}
	L_{M}h:=Q(h,M)+Q(M,h),
\end{equation}
and its null space $\mathcal{N}_M$ is spanned by the macroscopic variables $\chi_{i}(v)~(i=0,1,2,3,4)$.

\subsection{Notation, Weight, norm and energy functional}\label{sec.1.3}
Before we state the main theorem of this paper, we introduce some notations.
We shall use $\langle \cdot , \cdot \rangle$  to denote the standard $L^{2}$ inner product in $\mathbb{R}_{v}^{3}$
with its corresponding $L^{2}$ norm $|\cdot|_2$. We denote $( \cdot , \cdot )$ the $L^{2}$ inner product in
$\mathbb{R}_{x}$ or $\mathbb{R}_{x}\times \mathbb{R}_{v}^{3}$  with its corresponding $L^{2}$ norm $\|\cdot\|$.
Let $\alpha$ and $\beta$ be nonnegative integer and a multi-indices
$\beta=[\beta_1,\beta_2,\beta_3]$, respectively. Denote 
$$
\partial_{\beta}^{\alpha}=\partial_{x}^{\alpha}
\partial_{v_{1}}^{\beta_{1}}\partial_{v_{2}}^{\beta_{2}}\partial_{v_{3}}^{\beta_{3}}.
$$
If each component of $\beta$ is not greater than the corresponding one  of
$\overline{\beta}$, we use the standard notation
$\beta\leq\overline{\beta}$. And $\beta<\overline{\beta}$ means that
$\beta\leq\overline{\beta}$ and $|\beta|<|\overline{\beta}|$.
$C^{\bar\beta}_{\beta}$ is the usual  binomial coefficient.
Throughout the paper, generic positive constants are denoted by either c or C varying from line to line, and they are independent of
the parameters $\delta$ and $\varepsilon$.
Motivated by \cite{Strain-Guo} and \cite{Duan,Duan-Yang-Zhao-M3}, we define a time-velocity weight function
\begin{equation}
	\label{1.25}
	w(\alpha,\beta)(t,v):=\langle v\rangle^{2(l-|\alpha|-|\beta|)}e^{\frac{q_1}{(1+t)^{q_2}}\frac{\langle v\rangle^2}{2}},\quad  l\geq |\alpha|+|\beta|,
\end{equation}
where $\langle v\rangle=\sqrt{1+|v|^{2}}$ and  $0\leq q_1<1$ with $q_2>0$ will be  chosen in the proof later,
see also Theorem \ref{thm1.1}.

Throughout the paper, we fix a normalized global Maxwellian with the fluid constant state $(1,0,3/2)$
\begin{equation}
	\label{1.26}
	\mu=M_{[1,0,\frac{3}{2}]}(v):=(2\pi)^{-\frac{3}{2}}\exp\big(-\frac{|v|^{2}}{2}\big)
\end{equation}
as a reference equilibrium state. With \eqref{1.26},
the Landau collision frequency is given by
\begin{equation}
	\label{1.27}
	\sigma_{ij}(v):=\Phi_{ij}\ast \mu=\int_{{\mathbb R}^3}\Phi_{ij}(v-v_{\ast})\mu(v_{\ast})\,dv_{\ast}.
\end{equation}
Here $[\sigma_{ij}(v)]_{1\leq i,j\leq 3}$ is a positive-definite self-adjoint matrix.
We denote the weighted $L^2$ norms as
$$
|\partial^\alpha_\beta
g|^2_{w}\equiv|w(\alpha,\beta)\partial^\alpha_\beta
g|^2_{2}:=\int_{{\mathbb
		R}^3}w^{2}(\alpha,\beta)|\partial^\alpha_\beta g(x,v)|^2\,dv,
$$
and
$$
\|\partial^\alpha_\beta
g\|^2_{w}\equiv\|w(\alpha,\beta)\partial^\alpha_\beta
g\|^2_{2}:=\int_{\mathbb R}\int_{{\mathbb
		R}^3}w^{2}(\alpha,\beta)|\partial^\alpha_\beta g(x,v)|^2\,dv\,dx.
$$
By the linearization of the nonlinear Landau operator around $\mu$, we define the weighted dissipation norms as  
$$
|g|^2_{\sigma,w}:=\sum_{i,j=1}^3\int_{{\mathbb
		R}^3}w^{2}[\sigma_{ij}\partial_ig\partial_jg+\sigma_{ij}\frac{v_i}{2}\frac{v_j}{2}g^2]\,dv,\ \mbox{and}\ \ \|g\|_{\sigma,w}:=\big\||g|_{\sigma,w}\big\|.
$$
Also $|g|_{\sigma}=|g|_{\sigma,1}$ and $\|g\|_{\sigma}=\|g\|_{\sigma,1}$.
Furthermore, it holds from \cite[Lemma 5]{Strain-Guo} that
\begin{equation}
	\label{1.28}
	|g|_{\sigma,w}\approx |w\langle v\rangle^{-\frac{1}{2}}g|_2+\Big|w\langle v\rangle^{-\frac{3}{2}}\nabla_vg\cdot\frac{v}{|v|}\Big|_2+\Big|w\langle v\rangle^{-\frac{1}{2}}\nabla_vg\times\frac{v}{|v|}\Big|_2.
\end{equation}

Next we introduce the appropriate perturbations. 
In view of $(\rho_1,u_1^{(1)},\theta_1,\phi_1)(t,x)$ and $(\rho_2,u_1^{(2)},\theta_2,\phi_2)(t,x)$
constructed in Proposition \ref{prop.1.1} and Proposition \ref{prop.2.1} respectively, we denote
\begin{equation}
	\left\{
	\begin{array}{rl}
		\label{1.29}
		&\bar{\rho}=1+\delta\rho_1+\delta^2\rho_2,
		\\
		&\bar{u}_1=\delta u^{(1)}_1+\delta^2 u^{(2)}_1,\quad \bar{u}_2=\bar{u}_3=0,
		\\
		&\bar{\theta}=\frac{3}{2}+\frac{3}{2}\delta\theta_1+\frac{3}{2}\delta^2\theta_2,
		\\
		&\bar{\phi}=\delta\phi_1+\delta^2 \phi_2.
	\end{array} \right.
\end{equation}
In terms of  \eqref{1.29}, we define the following local Maxwellian 
\begin{equation}
	\label{1.30}
	\overline{M}:=M_{[\bar{\rho},\bar{u},\bar{\theta}](t,x)}(v)
	=\frac{\bar{\rho}(t,x)}{(2\pi K\bar{\theta}(t,x))^{3/2}}\exp\big(-\frac{|v-\bar{u}(t,x)|^{2}}{2K\bar{\theta}(t,x)}\big).
\end{equation}
With \eqref{1.29}, \eqref{1.26}, \eqref{1.17} and \eqref{1.22} in hand, 
we define the macroscopic perturbation $(\widetilde{\rho},\widetilde{u},\widetilde{\theta},\widetilde{\phi})(t,x)$ and the microscopic perturbation $f(t,x,v)$ by
\begin{equation}
	\label{1.31}
	\left\{
	\begin{array}{rl}
		&(\widetilde{\rho},\widetilde{u},\widetilde{\theta},\widetilde{\phi})(t,x)
		=({\rho}-\bar{\rho},u-\bar{u},\theta-\overline{\theta},\phi-\overline{\phi})(t,x),
		\\
		&f(t,x,v)=\frac{G(t,x,v)-\overline{G}(t,x,v)}{\sqrt{\mu}}.
	\end{array} \right.
\end{equation}
Here the correction term $\overline{G}(t,x,v)$ is defined as
\begin{equation}
	\label{1.32}
	\overline{G}=\delta^{1/2}\varepsilon L_{M}^{-1}\{ P_{1}v_{1}M(\frac{|v-u|^{2}
		\partial_x\bar{\theta}}{2K\theta^{2}}+\frac{(v-u)\cdot\partial_x\bar{u}}{K\theta})\},
\end{equation}
with $L_{M}^{-1}$ being the inverse operator of $L_{M}$.

In order to construct the long-time existence of smooth 
solutions for the VPL system \eqref{1.15} near a local Maxwellian $\overline{M}$ defined in \eqref{1.30},
a key point is to  establish uniform energy estimates on  $(\widetilde{\rho},\widetilde{u},\widetilde{\theta},\widetilde{\phi})(t,x)$
and $f(t,x,v)$ in the high order Sobolev space. To this end, for any small constant $q_1\in(0,1)$ and given constant $q_2>0$ in \eqref{1.25}, 
we define the following instant energy functional $\mathcal{E}_{2,l,q_1}(t)$ as
\begin{eqnarray}
	\label{1.33}
	\mathcal{E}_{2,l,q_1}(t):=&&\mathcal{E}_{2}(t)+\sum_{|\alpha|\leq 1}\|\partial^{\alpha}f(t)\|_{w}^{2}
	+\varepsilon^2\sum_{|\alpha|=2}\|\partial^{\alpha}f(t)\|_{w}^{2}+\sum_{|\alpha|+|\beta|\leq 2,|\beta|\geq1}\|\partial^{\alpha}_{\beta}f(t)\|_{w}^{2},
\end{eqnarray}
where $\mathcal{E}_{2}(t)$ is the instant energy without the weight, given by
\begin{eqnarray}
	\label{1.34}
	\mathcal{E}_{2}(t):=&&\sum_{|\alpha|\leq1}\{\|\partial^{\alpha}(\widetilde{\rho},\widetilde{u},\widetilde{\theta})(t)\|^{2}
	+\|\partial^{\alpha}f(t)\|^{2}+\|\partial^{\alpha}\widetilde{\phi}(t)\|^{2}
	+\delta\|\partial^{\alpha}\partial_x\widetilde{\phi}(t)\|^{2}\}
	\notag\\
	&&\quad+\frac{\varepsilon^2}{\delta}\sum_{|\alpha|=2}\{\|\partial^{\alpha}(\widetilde{\rho},\widetilde{u},\widetilde{\theta})(t)\|^{2}
	+\|\partial^{\alpha}f(t)\|^{2}+
	\|\partial^{\alpha}\widetilde{\phi}(t)\|^{2}+\delta\|\partial^{\alpha}\partial_x\widetilde{\phi}(t)\|^{2}\}.
\end{eqnarray}
Correspondingly, the energy dissipation functional is given by
\begin{equation}
	\label{1.35}
	\mathcal{D}_{2,l,q_1}(t):=\mathcal{D}_{2}(t)+\frac{1}{\delta^{3/2}\varepsilon}\{\sum_{|\alpha|\leq 1}\|\partial^{\alpha}f(t)\|_{\sigma,w}^{2}
	+\varepsilon^2\sum_{|\alpha|=2}
	\|\partial^{\alpha}f(t)\|_{\sigma,w}^{2}
	+\sum_{|\alpha|+|\beta|\leq 2,|\beta|\geq1}\|\partial^{\alpha}_{\beta}f(t)\|_{\sigma,w}^{2}\}, 
\end{equation}
where $	\mathcal{D}_{2}(t)$ is the dissipation rate without the weight, given by
\begin{eqnarray}
	\label{1.36}
	\mathcal{D}_{2}(t):&&=\frac{\varepsilon}{\delta^{1/2}}
	\sum_{1\leq|\alpha|\leq 2}\{\|\partial^{\alpha}(\widetilde{\rho},\widetilde{u},\widetilde{\theta})(t)\|^{2}
	+\|\partial^{\alpha}\widetilde{\phi}(t)\|^{2}
	+\delta\|\partial^{\alpha}\partial_x\widetilde{\phi}(t)\|^{2}\}
	\notag\\
	&&\hspace{1cm}+\frac{1}{\delta^{3/2}\varepsilon}\{\sum_{|\alpha|\leq1}\|\partial^\alpha f(t)\|^2_{\sigma}
	+\frac{\varepsilon^2}{\delta}\sum_{|\alpha|=2}\|\partial^\alpha f(t)\|^2_{\sigma}\}.
\end{eqnarray}
Here we restrict the highest order derivatives in all functionals to the second
order for simplicity. Much higher orders can be considered so that smooth
solutions can be obtained.

For the technical reason, we consider the fluctuation amplitude $\delta$ is a function
of $\varepsilon$ satisfying
\begin{equation}
\label{1.37}
\varepsilon^{\frac{2}{3}}	\leq \delta\leq\widetilde{C}^{-1}\varepsilon^{\frac{2}{5}},
\end{equation}
where $\widetilde{C}>1$ is a given constant deponding on $C_0$ as in \eqref{3.1}.
Throughout the paper, $\varepsilon$ is small enough.

\subsection{Main result}\label{sec.1.4}
With the above preparation, we are now ready to state our main theorem
on the long wave limit from the rescaled  VPL system \eqref{1.15} to the
KdV equation \eqref{1.9}.
\begin{theorem}
\label{thm1.1}
Let  $\tau>0$ be any given time and $(\rho_1,u_1^{(1)},\theta_1,\phi_1)(t,x)$ be a smooth solution for the KdV equation \eqref{1.8}-\eqref{1.9}
given  by Proposition \ref{prop.1.1}. Construct the local Maxwellian 
\begin{equation}
\label{1.38}
M^{\delta}\equiv M_{[1+\delta\rho_1,\delta \hat{u},\frac{3}{2}+\frac{3}{2}\delta\theta_1]}(t,x,v):=\frac{1+\delta\rho_1(t,x)}{\sqrt{(2\pi [1+\delta\theta_1(t,x)])^{3}}}\exp\big\{-\frac{|v-\delta \hat{u}(t,x)|^{2}}{2[1+\delta\theta_1(t,x)]}\big\},
\end{equation}
with $\hat{u}=(u_1^{(1)},0,0)$. Assume that the initial data $F_{0}(x,v)\geq0$ satisfies 
\begin{equation}
	\label{1.39}
	\mathcal{E}_{2,l,q_1}(t)\mid_{t=0}\,\leq\, \delta^4,
\end{equation}
where in \eqref{1.25} we have chosen $q_1\in(0,1)$ small enough and some given $q_2>0$.
Then there exist small constants $\varepsilon_0>0$
and $\delta_0>0$ such that for any $\varepsilon\in(0,\varepsilon_0)$ and any $\delta\in(0,\delta_0)$ satisfying the restriction \eqref{1.37},
the VPL system \eqref{1.15}  admits a unique smooth solution $(F(t,x,v),\phi(t,x))$ for all $t\in[0,\tau]$. Moreover, it holds that
$F(t,x,v)\geq 0$ and
\begin{equation}
	\label{1.40}
	\sup_{0\leq t\leq\tau}\mathcal{E}_{2,l,q_1}(t)
	\leq C\delta^4.
\end{equation}	
In particular, there exists a constant $C>0$ independent of $\delta$ such that
\begin{equation}
	\label{1.41}
\sup_{0\leq t\leq\tau}\{\|\frac{F(t,x,v)-M^{\delta}(t,x,v)}{\sqrt{\mu}}\|_{L^2_xL^2_v}+\|\frac{F(t,x,v)-M^{\delta}(t,x,v)}{\sqrt{\mu}}\|_{L^\infty_xL^2_v}\}\leq C\delta^2,
\end{equation}
and 
\begin{equation}
	\label{1.42}
\sup_{0\leq t\leq\tau}\{\|\phi(t,x)-\delta\phi_{1}(t,x)\|+\|\phi(t,x)-\delta\phi_{1}(t,x)\|_{L_{x}^{\infty}}\}\leq C\delta^2.
\end{equation}
\end{theorem}
Several remarks are in order.

\begin{remark}\label{rmk1.4}
Theorem \ref{thm1.1}  shows that the solutions $F(t,x,v)$ of the VPL system converge to a local Maxwellian
whose fluid quantities are determined by the smooth solution of the KdV equation. This provides the rigorous justification regarding 
hydrodynamic limit of dispersive equations from the kinetic equations under the effect of collisions. 
The new ideas and methods introduced here can be applied to deal with the problem on the dispersive limit of other kinetic models for collisional plasma.
\end{remark}

\begin{remark}\label{rmk1.5}
	We point out that the set of non-negative smooth initial data $F_0$ satisfying \eqref{1.39} is nonempty. 
	In fact, the initial value  can be specifically chosen as
	\begin{equation}
		\label{1.43}
		F_0(x,v)=M_{[\bar{\rho},\bar{u},\bar{\theta}]}(0,x,v), 
		%\quad \phi_0(x)=\bar{\phi}(0,x),
	\end{equation}
	with $\phi_0(x)=\bar{\phi}(0,x)$, which automatically satisfies \eqref{1.39}. Here $(\bar{\rho},\bar{u},\bar{\theta},\bar{\phi})(0,x)$ is defined in \eqref{1.29}.
	Indeed, in view of \eqref{1.43} and the decomposition $F(t,x,v)=M_{[\rho,u,\theta]}(t,x,v)+G(t,x,v)$,
	we have $M_{[\rho,u,\theta]}(0,x,v)=M_{[\bar{\rho},\bar{u},\bar{\theta}]}(0,x,v)$ and $G(0,x,v)=0$.
	This implies that $(\rho,u,\theta)(0,x)=(\bar{\rho},\bar{u},\bar{\theta})(0,x)$ and $\overline{G}(0,x,v)+\sqrt{\mu}f(0,x,v)=0$.
	Then it holds that
	%\begin{equation*}
	$(\widetilde{\rho},\widetilde{u},\widetilde{\theta},\widetilde{\phi})(0,x)
	=({\rho}-\bar{\rho},u-\bar{u},\theta-\bar{\theta},\phi-\bar{\phi})(0,x)=0$
	%\quad \mbox{and} \quad	
	%\end{equation*}
	and 
	$$
	f(0,x,v)=-\frac{\overline{G}(0,x,v)}{\sqrt{\mu}}.
	$$
By these facts, \eqref{1.33}, \eqref{1.34}, \eqref{1.37} and Lemma \ref{lem3.2}, we get
	\begin{eqnarray*}
		\mathcal{E}_{2,l,q_1}(0)=&&
		\sum_{|\alpha|\leq1}(\|\frac{\partial^{\alpha}\overline{G}(0)}{\sqrt{\mu}}\|^{2}+\|\frac{\partial^{\alpha}\overline{G}(0)}{\sqrt{\mu}}\|_{w}^{2})
		+\frac{\varepsilon^{2}}{\delta}\sum_{|\alpha|=2}\|\frac{\partial^{\alpha}\overline{G}(0)}{\sqrt{\mu}}\|^{2}
		\\
		&&
		+\varepsilon^{2}\sum_{|\alpha|=2}\|\frac{\partial^{\alpha}\overline{G}(0)}{\sqrt{\mu}}\|_{w}^{2}
		+\sum_{|\alpha|+|\beta|\leq 2,|\beta|\geq1}\|\partial^{\alpha}_{\beta}(\frac{\overline{G}(0)}{\sqrt{\mu}})\|_{w}^{2}
		\\
		\leq&&C\delta^{3}(\varepsilon+\delta^2)^2\leq \frac{1}{2}\delta^4.
	\end{eqnarray*}
\end{remark}

\subsection{Relevant literature}\label{sec.1.5}
In the following we review some related works on the problems considered in this paper.
There have been extensive mathematical studies on the fluid dynamic limits of the kinetic equations over the years.
In the context of the Boltzmann equation, we may refer readers to two books \cite{Golse,Saint-Raymond} 
for the detailed presentation of limits of the Boltzmann equation to either the 
compressible Euler and Navier-Stokes systems or the incompressible Euler and Navier-Stokes systems.
It has essentially three different mathematical approaches to rigorously justify these formal limits.
For the approach based on the Hilbert or Chapman-Enskog expansions, we refer to \cite{Caflisch,Grad,Guo-Huang-Wang,Guo-Jang-Jiang-2010,Hang-Wang2-Yang,Jiang-Luo,Liu-2014,XinZeng}. For the method employing the abstract Cauchy-Kovalevskaya theorem and the spectral analysis of the semigroup generated by the linearized Boltzmann equation,
we refer to  \cite{Bardos-1991,Nishida,Ukai-Asano}.  For the DiPerna-Lions \cite{Diperna-Lions} renormalized weak solutions,
we refer to  \cite{Bardos,Golse-2002,Golse-2004,Masmoudi-2003}. These methods can be applied to the Boltzmann equation with
the self-consistent electromagnetic field. In particular, we point out that Guo-Jang \cite{Guo-Jang} 
gave the first rigorous proof of fluid limit from the Vlasov-Poisson-Boltzmann system to the compressible Euler-Poisson system in the full space
of three dimensions via the Hilbert expansion and $L^2$-$L^\infty$ interplay technique initiated by Guo \cite{Guo-2008}.
This seems an important progress since two-fluid model can be first rigorously derived from a kinetic plasma model
with the self-consistent potential force. We also mention another recent progress that the same authors of this paper \cite{Duan-Yang-Yu-VMB} 
obtained the fluid limit from the  Vlasov-Maxwell-Boltzmann system to the compressible Euler-Maxwell system 
in the full space of three dimensions via an energy method in terms of the macro-micro decomposition
and a careful analysis of the Burnett function.

However, the corresponding hydrodynamic limits of the Landau equation are much harder, since the Landau collision operator exposes the velocity diffusion property so that some techniques used in Boltzmann equation with angular cutoff can not be directly applied. In the incompressible regime, 
Guo \cite{Guo-CPAM} gave the rigorous proof of the limit to the incompressible Navier-Stokes system for smooth solutions to the Landau equation near global Maxwellians; see also two recent works \cite{CaMi-1} and \cite{Rachid}.
In the compressible regime, the compressible Euler limit for smooth solutions to the Landau equation near global Maxwellians was recently studied by
Duan-Yang-Yu \cite{Duan-Yang-Yu-2} and Lei-Liu-Xiao-Zhao \cite{Lei-1} via different energy methods.
In  the setting of one space dimension, regarding solutions with wave patterns, the only results are \cite{Duan-Yang-Yu} for rarefaction waves and \cite{Yang} for contact waves.
The corresponding result for shock waves is still open as far as we know.

In the presence of the self-consistent potential force, the same authors of this paper \cite{Duan-Yang-Yu-VPL} obtained the
fluid limit from the VPL system to the compressible quasineutral  Euler system in  case of rarefaction waves
via  an intricate weighted energy method capturing the quartic dissipation. On the other hand,
a significant progress on the Landau damping of the VPL system  in the weakly collisional regime was recently obtained by
Chaturvedi-Luk-Nguyen \cite{Chaturvedi} through combining Guo's weighted energy method with the hypocoercive estimates and the commuting-vector-field techinque, see also \cite{Bedrossian} for an earlier study of the relevant toy model. 

In the following we further mention some early and recent results on the Cauchy problem of  the VPL system.
The first global classical solutions of the VPL system  near global Maxwellians in a periodic box was constructed by Guo \cite{Guo-JAMS},
and this result was extended to the whole space by Strain-Zhu \cite{Strain-Zhu}, Duan-Yang-Zhao \cite{Duan-Yang-Zhao} and Wang \cite{Wang}, respectively.
The global classical solutions of the VPL system around local Maxwellians with rarefaction waves either in the one-dimensional line with slab symmetry or in the three-dimensional infinite channel domain was obtained by Duan-Yu \cite{Duan-Yu1, DY31}. In addition to those works, great contributions also
have been done in many other kinds of topics  of the VPL system, for instance \cite{Bedrossian-1,Deng,Dong-Guo,Flynn} and the references therein.
Here we also mention an important and nice work \cite{Guo-2002} for the global classical solutions of the pure Landau equation.

It is well known that KdV equation is an interesting and important dispersive model, which describes the long waves propagating on a shallow water surface.
There are a  numerous literature have been devoted to study the long wave limits towards KdV.
We mention \cite{Chiron} for the nonlinear Schr\"{o}dinger, \cite{Ben} for the general Hyperbolic systems,
\cite{A,Bona,Craig,Schneider} for the water wave. In particular, we highlight the two independent works \cite{Guo-Pu}
and \cite{Lannes}  for the Euler-Poisson system, and a nice work \cite{Han-Kwan} for the Vlasov-Poisson system.
Though, it seems there are few results on the KdV limit in the presence of collisions.
This served as a major inspiration for our work. In the end we would like to mention
\cite{Bona-1975,Colliander,Kato,Kenig} and the references therein for the
well-posedness theory of the KdV equation with data in the $L^2$-based Sobolev spaces $H^{s}(\mathbb{R})$ with $s\in\mathbb{R}$.

Despite the fruitful progress in the hydrodynamic limits as mentioned before, none of these works consider the 
limits on the dispersive equation from the kinetic equation under the effect of the collisions.
The main reason is that  the interspecies collisions give rise to mathematical difficulties.
In this work, we develop a weighted energy method in terms of the macro-micro decomposition around local Maxwellians to derive and justify the KdV equation from the VPL system modelling the motion of ions under the Maxwell-Boltzmann relation. 
This makes us to better understand the dispersive behavior of the solutions to VPL system.
%To the best of our knowledge, this provides the first result about the dispersive equation from the kinetic equation under the effect of the collisions. 

\subsection{Ideas of the proof}\label{sec.1.6}	
We now make some comments on the analysis of this paper. Our strategy is based on an appropriately chosen scaling and the intricate
weighted energy method through the macro-micro decomposition around local Maxwellians
for designing a $\delta$-dependent energy functional and its corresponding  dissipation functional 
such that the uniform estimates can be obtained under the smallness
assumption.

The first important ingredient of the proof is based on a suitable scaling transformation of the independent variables takes the form of
\begin{equation*}
x=\delta^{\frac{1}{2}}(\bar{x}-\sqrt{\frac{8}{3}}\bar{t}), \quad \quad t=\delta^{\frac{3}{2}}\bar{t},\quad \quad v=\bar{v},
\end{equation*}
with the parameter $\delta>0$ depending on the Knudsen
number $\varepsilon>0$.
Under this scaling transformation, the original VPL system \eqref{1.1} becomes the rescaled  VPL system \eqref{1.15}.
The formal Hilbert expansion of the solutions for such system \eqref{1.15} gives rise to the KdV equation as both the fluid limit $\varepsilon\to 0$ and  the long wave limit $\delta\to 0$, see figure \ref{fig.asp}. Therefore, this scaling transformation gives a good regime for deriving
rigorously the KdV equation from the VPL system.

The second important ingredient of the proof is based on the macro-micro decomposition 
with respect to the local Maxwellian $M$ determined by the solution of the rescaled VPL system.
By using this decomposition, we can obtain the rescaled compressible Euler-Poisson equations with some source terms such that we can observe the existence of dispersive behavior, particularly the KdV limit under consideration.  As mentioned in \cite{Duan-Yang-Yu-2}, on the one hand,
the inverse of linearized operator $L^{-1}_M$ in the macroscopic equation  is very complicated  due to the velocity diffusion effect of collision operator.
To this end, we make use of the Burnett functions and velocity-decay properties to handled the terms involving $L^{-1}_M$ so that the estimates can be obtained in a clear way, see the identities \eqref{4.20} and \eqref{4.19} for details. On the other hand,
we need to consider the subtraction of $G(t,x,v)$ by $\overline{G}(t,x,v)$ as \eqref{1.32} to remove
the linear terms in the equation \eqref{2.10}, because those terms induce the energy term $\|(\partial_x\bar{\theta},\partial_x\bar{u})\|^2$ in the $L^2$ estimate
that is  out of control by $O(\delta^4)$ under the condition of \eqref{1.37}. Thus we further expand $G=\overline{G}+\sqrt{\mu}f$ to get the 
equation \eqref{2.39}. Here $f\in (\ker\mathcal{L})^{\perp}$ is purely
microscopic and it is essentially different from the one in \cite{Guo-2002,Guo-JAMS}.
This is convenient for obtaining the  $\delta$-$\varepsilon$-singularity trilinear estimate  because the macroscopic part would not longer appear.

The third important ingredient of the proof is based on a useful time-velocity weight function $w(\alpha,\beta)(t,v)$ given in \eqref{1.25}.
Since the linearized Landau operator with Coulomb interaction has no
spectral gap that results in the very weak velocity dissipation, the large velocity growth in the nonlinear electric potential terms
are hard to control. The factor	$e^{\frac{q_1}{(1+t)^{q_2}}\frac{\langle v\rangle^2}{2}}$ in \eqref{1.25} is used to
induce an extra large velocity dissipation term to overcome those difficulties. Similar mechanism was observed 
in some earlier works \cite{Duan-Yang-Zhao,Duan-Yang-Zhao-M3} as well as in some recent work \cite{Duan-Yang-Yu-VPL,Yang} in a more general way.
The other factor $\langle v\rangle^{2(l-|\alpha|-|\beta|)}$ in \eqref{1.25} is used to take care of the derivative estimates of the free streaming term and the velocity derivative of $f$. This technique was first used in \cite{Guo-JAMS}.
However, the main difference is that we need to be careful about two parameters $\varepsilon$ and $\delta$ in the estimation.
Under such weight function and the parameters $\varepsilon$ and $\delta$,
one has to deal with the complex weight estimates about the linearized operator and the nonlinear collision operator.

The last important ingredient of the proof is based on an energy functional $\mathcal{E}_{2,l,q_1}(t)$ in \eqref{1.33} including the $\delta$ and $\varepsilon$ to the highest-order derivative  $|\alpha|=2$. On the one hand, due to singularity,  the linear term $-\frac{1}{\delta^{3/2}\varepsilon}(\mathcal{L}\partial^{\alpha}f,\frac{\partial^{\alpha}M}{\sqrt{\mu}})$ in \eqref{5.3}
about the second-order space derivatives  cannot be directly estimated. The key technique for handling this term is to use the properties
of the linearized operator $\mathcal{L}$ and the relation between $M$ and $\mu$ as in \eqref{3.24} as well as 
the delicate calculation, see the estimates from \eqref{5.4} to \eqref{5.7}. Because of
the singular factor $\delta^{1/2}\varepsilon^{{-1}}$ in front of the fluid part $\|\partial^{\alpha}(\widetilde{\rho},\widetilde{u},\widetilde{\theta})\|^{2}$ in \eqref{5.7}, one has to multiply the estimate \eqref{5.7} by $\delta^{-1}\varepsilon^2$ so that the fluid term can be controlled by 
$\mathcal{D}_{2,l,q_1}(t)$  in \eqref{1.35}. But such good estimate is not longer holds true for the same term involving
the weight function, see \eqref{6.7AB} for details. On the other hand, by the orthogonal properties of $G$ and $M$, we can observe a cancellation
for the estimate \eqref{5.9A}. However, such a cancellation is not longer holds true for the same term involving
the weight function, see \eqref{6.7A} for details. Hence, one has to multiply the estimates \eqref{6.7AB} and \eqref{6.7A} by $\varepsilon^2$ so that
the energy estimates can be closed. This is why we have to design 
$\varepsilon^2$ for weight and $\delta^{-1}\varepsilon^2$ for without weight in the highest-order derivative 
for energy functional $\mathcal{E}_{2,l,q_1}(t)$ in \eqref{1.33}. In the end, we point out that it is quite nontrivial to estimate the coupling
term $\partial_x\phi$ in the momentum equation as in \eqref{2.13}. For instance, to overcome 
the difficulties caused by the linear term on electric field in \eqref{3.61}, we have to make use of
the structure of the Poisson system and some useful techniques 
to deal with them, see the estimates from \eqref{3.62} to \eqref{3.72}.

Due to the complexity of the system and two free parameters, the calculations in this paper are very complicated.
The new techniques and method introduced here can also be applied to study
the problem on the dispersive limit of other kinetic models for collisional plasma.

The outline of this paper is as follows. In Section \ref{sec.2}, we reformulate the VPL system in terms of the macro-micro decomposition
and give a formal derivation for the KdV equation. In Section \ref{sub3}, we make the {\it a priori} assumption for the solutions and derive some 
basic estimates.  Sections \ref{sec.4}-\ref{sec.6} are the main part of the proof for establishing the {\it a priori}
estimates including low order space derivative estimates, high order space derivative estimates and weighted energy
estimates, see Lemmas \ref{lem4.3}, \ref{lem5.1} and \ref{lem6.1}, respectively. In Section \ref{sec.7},
we are devoted to completing the proof of the main result Theorem \ref{thm1.1} stated in
Section \ref{sec.1}. 
\section{Reformulated system}\label{sec.2}
In this section,  we will reformulate the VPL system \eqref{1.15} in terms of the macro-micro decomposition \eqref{1.22},
and then give a formal derivation for the KdV equation \eqref{1.8}-\eqref{1.9}.
\subsection{Macroscopic-microscopic equations}\label{sub2.1}
Multiplying \eqref{1.23} by the collision invariants $\psi_{i}(v)$ $(i=0,1,2,3,4)$ and integrating the resulting equations with respect to
$v\in\mathbb{R}^{3}$, we obtain the following macroscopic system:
\begin{equation}
	\label{2.9}
	\left\{
	\begin{array}{rl}
		&\delta\partial_t\rho-A\partial_x\rho+\partial_x(\rho u_{1})=0,
		\\
		&\delta\partial_t(\rho u_{1})-A\partial_x(\rho u_{1})
		+\partial_x(\rho u_{1}^{2})+\partial_xp+\rho\partial_x\phi=-\int_{\mathbb{R}^{3}} v^{2}_{1}\partial_xG\,dv,
		\\
		&\delta\partial_t(\rho u_{i})-A\partial_x(\rho u_{i})+\partial_x(\rho u_{1}u_{i})=-\int_{\mathbb{R}^{3}} v_{1}v_{i}\partial_xG\,dv, ~~i=2,3,
		\\
		&\delta\partial_t\{\rho (\theta+\frac{|u|^{2}}{2})\}-A\partial_x\{\rho (\theta+\frac{|u|^{2}}{2})\}
		+\partial_x\{\rho u_{1}(\theta+\frac{|u|^{2}}{2})+pu_{1}\}
		\\
		&\hspace{2cm}+\rho u_{1}\partial_x\phi
		=-\int_{\mathbb{R}^{3}} \frac{1}{2}v_{1}|v|^{2}\partial_xG\,dv,
		\\
		&-\delta\partial^{2}_{x}\phi=\rho-e^{\phi},
	\end{array} \right.
\end{equation}
with the pressure $p=\frac{2}{3}\rho\theta$, where the last identity from the second equation of \eqref{1.15}.

Recall the definition of $P_{1}$ given in \eqref{1.20}, then we  take $P_{1}$ of \eqref{1.23} to obtain the following microscopic equation 
\begin{equation}
\label{2.10}
\partial_tG-A\frac{1}{\delta}\partial_xG+\frac{1}{\delta}P_{1}(v_{1}\partial_xG)+\frac{1}{\delta}P_{1}(v_{1}\partial_xM)-\frac{1}{\delta}\partial_x\phi\partial_{v_{1}}G
=\frac{1}{\delta^{3/2}\varepsilon}L_{M}G+\frac{1}{\delta^{3/2}\varepsilon}Q(G,G).
\end{equation}
Since $L_{M}$ is invertible on $\mathcal{N}_M^\perp$, thus it holds from \eqref{2.10} that
\begin{equation}
	\label{2.11}
	G=\delta^{1/2}\varepsilon L^{-1}_{M}[P_{1}(v_{1}\partial_xM)]+L^{-1}_{M}\Theta,
\end{equation}
with
\begin{equation}
	\label{2.12}
	\Theta:=\delta^{1/2}\varepsilon[\delta\partial_tG-A\partial_xG+ P_{1}(v_{1}\partial_xG)-\partial_x\phi\partial_{v_{1}}G]-Q(G,G).
\end{equation}
Substituting \eqref{2.11} into \eqref{2.9}, we obtain the following fluid-type system:
\begin{equation}
	\label{2.13}
	\left\{
	\begin{array}{rl}
		&\delta\partial_t\rho-A\partial_x\rho+\partial_x(\rho u_{1})=0,
		\\
		&	\delta\partial_t(\rho u_{1})-A\partial_x(\rho u_{1})+\partial_x(\rho u_{1}^{2})+\partial_xp+\rho\partial_x\phi
		=\frac{4}{3}\delta^{1/2}\varepsilon\partial_x(\mu(\theta)\partial_xu_{1})-\partial_x(\int_{\mathbb{R}^{3}} v^{2}_{1}L^{-1}_{M}\Theta \,dv),
		\\
		&	\delta\partial_t(\rho u_{i})-A\partial_x(\rho u_{i})+\partial_x(\rho u_{1}u_{i})=\delta^{1/2}\varepsilon\partial_x(\mu(\theta)\partial_xu_{i})
		-\partial_x(\int_{\mathbb{R}^{3}} v_{1}v_{i}L^{-1}_{M}\Theta \,dv), ~~i=2,3,
		\\
		&	\delta\partial_t\{\rho (\theta+\frac{|u|^{2}}{2})\}-A\partial_x\{\rho (\theta+\frac{|u|^{2}}{2})\}+\partial_x\{\rho u_{1}(\theta+\frac{|u|^{2}}{2})+pu_{1}\}+\rho u_{1}\partial_x\phi=\delta^{1/2}\varepsilon\partial_x(\kappa(\theta)\partial_x\theta)
		\\
		&+\frac{4}{3}\delta^{1/2}\varepsilon\partial_x(\mu(\theta)u_{1}\partial_xu_{1})
		+\delta^{1/2}\varepsilon\partial_x(\mu(\theta)u_{2}\partial_xu_{2}+\mu(\theta)u_{3}\partial_xu_{3})
		-\frac{1}{2}\partial_x(\int_{\mathbb{R}^{3}}v_{1}|v|^{2}L^{-1}_{M}\Theta \,dv),
		\\
		&-\delta\partial^{2}_{x}\phi=\rho-e^{\phi}.
	\end{array} \right.
\end{equation}
Here the viscosity coefficient $\mu(\theta)>0$ and the heat conductivity coefficient $\kappa(\theta)>0$, both are smooth functions depending only on $\theta$,
are represented by 
\begin{eqnarray*}
	\mu(\theta)=&&- K\theta\int_{\mathbb{R}^{3}}\hat{B}_{ij}(\frac{v-u}{\sqrt{K\theta}})
	B_{ij}(\frac{v-u}{\sqrt{K\theta}})dv>0,\quad i\neq j,
	\notag\\
	\kappa(\theta)=&&-K^{2}\theta\int_{\mathbb{R}^{3}}\hat{A}_{j}(\frac{v-u}{\sqrt{K\theta}})
	A_{j}(\frac{v-u}{\sqrt{K\theta}})dv>0,
\end{eqnarray*}
where  $\hat{A}_{j}(\cdot)$ and $\hat{B}_{ij}(\cdot)$ are Burnett functions (cf. \cite{Bardos,Guo-CPAM,Duan-Yang-Yu-2}) and  they are defined as
\begin{equation}
	\label{2.14}
	\hat{A}_{j}(v)=\frac{|v|^{2}-5}{2}v_{j},\quad \mbox{and} \quad \hat{B}_{ij}(v)=v_{i}v_{j}-\frac{1}{3}\delta_{ij}|v|^{2}, \quad \mbox{for} \quad i,j=1,2,3.
\end{equation}
And $A_{j}(\cdot)$ and $B_{ij}(\cdot)$ satisfy $P_{0}A_{j}(\cdot)=0$ and $P_{0}B_{ij}(\cdot)=0$, given by
\begin{equation}
	\label{2.15}
	A_{j}(v)=L^{-1}_{M}[\hat{A}_{j}(v)M]\quad
	\mbox{and} \quad B_{ij}(v)=L^{-1}_{M}[\hat{B}_{ij}(v)M].
\end{equation}
Some elementary properties of the Burnett functions are summarized in \cite[Lemma 2.1]{Duan-Yang-Yu-2}.

To make a conclusion, we have decomposed the rescaled VPL system \eqref{1.15} as the coupling
of the macroscopic equations \eqref{2.9} and the microscopic equation \eqref{2.10}.
\subsection{Formal KdV expansion}\label{sub2.2}
From \cite{Guo-Pu} and \cite{Lannes} that the KdV equation can be derived from the compressible Euler-Poisson system.
Note that system \eqref{2.9} becomes the compressible Euler-Poisson system \eqref{1.6} as $G=0$ and $u_2=u_3=0$.
Thus we make good use of \eqref{1.6} to derive formally the KdV equation \eqref{1.8}-\eqref{1.9}. 
Substituting \eqref{1.7} into \eqref{1.6}, and matching the different powers of $\delta$, we get a cascade of equations.
\\
The coefficients of $\delta$:
\begin{equation}
\left\{
\begin{array}{rl}
	\label{2.16}
&-A\partial_{x}\rho_{1}+\partial_{x}u^{(1)}_{1}=0,
\\
&\partial_{x}\rho_{1}-A\partial_{x}u^{(1)}_{1}+\partial_{x}\theta_{1}+\partial_{x}\phi_{1}=0,
		\\
&\frac{2}{3}\partial_{x}u^{(1)}_{1}-A\partial_{x}\theta_{1}=0,
		\\
&\rho_1-\phi_1=0.
\end{array} \right.	
\end{equation}
The coefficients of $\delta^2$:
\begin{equation}
	\left\{
	\begin{array}{rl}
		\label{2.19}
		&\partial_t\rho_{1}-A\partial_{x}\rho_{2}+\partial_{x}u^{(2)}_{1}+\partial_{x}(\rho_1u^{(1)}_1)=0,
		\\
		&\partial_tu^{(1)}_{1}-A\partial_{x}u^{(2)}_{1}-A\rho_1\partial_{x}u^{(1)}_{1}+u^{(1)}_{1}\partial_{x}u^{(1)}_{1}+\partial_{x}\theta_{2}
		\\
		&\hspace{1cm}+\partial_{x}\rho_{2}+\partial_{x}(\rho_1\theta_{1})+\partial_{x}\phi_{2}+\rho_1\partial_{x}\phi_{1}=0,
		\\
		&\partial_t\theta_{1}-A\partial_{x}\theta_{2}+\frac{2}{3}\partial_{x}u^{(2)}_{1}+\frac{2}{3}\theta_1\partial_{x}u^{(1)}_{1}+u^{(1)}_{1}\partial_{x}\theta_{1}=0,
		\\
		&\partial^2_{x}\phi_{1}-\phi_2-\frac{1}{2}\phi^2_1+\rho_2=0.
	\end{array} \right.
\end{equation}
The coefficients of $\delta^3$:
\begin{equation}
	\left\{
	\begin{array}{rl}
		\label{2.22}
		&\partial_{t}\rho_{2}-A\partial_{x}\rho_{3}+\partial_{x}u^{(3)}_{1}+\partial_{x}(\rho_1u^{(2)}_1)+\partial_{x}(\rho_2u^{(1)}_1)=0,
		\\
		&\partial_{t}u^{(2)}_{1}+\rho_1\partial_{t}u^{(1)}_{1}-A\partial_{x}u^{(3)}_{1}-A\rho_1\partial_{x}u^{(2)}_{1}-A\rho_2\partial_{x}u^{(1)}_{1}
		+\partial_{x}(u^{(1)}_{1}u^{(2)}_{1})
		\\
		&\hspace{1cm}   +\rho_1u^{(1)}_{1}\partial_{x}u^{(1)}_{1}+\partial_{x}\theta_{3}
		+\partial_{x}\rho_{3}+\partial_{x}(\rho_1\theta_{2})+\partial_{x}(\rho_2\theta_{1})
			\\
		&\hspace{2cm}
		+\partial_{x}\phi_{3}+\rho_2\partial_{x}\phi_{1}+\rho_1\partial_{x}\phi_{2}=0,
		\\
		&\partial_{t}\theta_{2}-A\partial_{x}\theta_{3}+\frac{2}{3}\partial_{x}u^{(3)}_{1}+\frac{2}{3}\theta_1\partial_{x}u^{(2)}_{1}
		+\frac{2}{3}\theta_2\partial_{x}u^{(1)}_{1}+u^{(1)}_{1}\partial_{x}\theta_{2}+u^{(2)}_{1}\partial_{x}\theta_{1}=0,
		\\
		&\partial^2_{x}\phi_{2}-\phi_3-\phi_1\phi_2-\frac{1}{6}\phi^3_1+\rho_3=0. 	
	\end{array} \right.
\end{equation}

For \eqref{2.16}, we write it in the matrix form
\begin{equation}
	\label{2.11A}
	\begin{pmatrix}
		-A,~~1,~~0,~~0
		\\
		1,~~-A,~~1,~~1
		\\
		0,~~\frac{2}{3},~~-A,~~0
		\\
		1,~~0,~~0,~~-1
	\end{pmatrix}
	\begin{pmatrix}
		\partial_{x}\rho_{1}
		\\
		\partial_{x}u^{(1)}_{1}
		\\
		\partial_{x}\theta_{1}
		\\
		\partial_{x}\phi_{1}
	\end{pmatrix}
	=
	\begin{pmatrix}
		0
		\\
		0
		\\
		0
		\\
		0
	\end{pmatrix}.
\end{equation}
To get a nontrivial solution $(\rho_1,u^{(1)}_{1},\theta_{1},\phi_{1})$ in \eqref{2.11A}, we  require the determinant
of the coefficient matrix to be zero so that
\begin{align}
	\label{2.17}
	A(\frac{8}{3}-A^2)=0\Leftrightarrow A=\sqrt{\frac{8}{3}}, \quad \mbox{since}~~~ A>0.
\end{align}
On the other hand, equations \eqref{2.16} enables us to assume the relation
\begin{equation}
	\label{2.18}
		u^{(1)}_{1}=A\rho_{1},
		\quad
		\theta_{1}=\frac{2}{3}\rho_{1},
		\quad
		\phi_1=\rho_1.
\end{equation}
Using $\eqref{2.19}_1\times\frac{2}{3}-\eqref{2.19}_3$ and \eqref{2.18}, we obtain
\begin{eqnarray}
	\label{2.21a}
	0&&=\frac{2}{3}(\partial_t\rho_{1}-A\partial_{x}\rho_{2}+\partial_{x}u^{(2)}_{1}+\partial_{x}(\rho_1u^{(1)}_1))
	-(\partial_t\theta_{1}-A\partial_{x}\theta_{2}+\frac{2}{3}\partial_{x}u^{(2)}_{1}+\frac{2}{3}\theta_1\partial_{x}u^{(1)}_{1}+u^{(1)}_{1}\partial_{x}\theta_{1})
	\notag\\
	&&=\frac{2}{3}(\partial_t\rho_{1}-A\partial_{x}\rho_{2}+A\partial_{x}(\rho^2_1))
	-(\frac{2}{3}\partial_t\rho_{1}-A\partial_{x}\theta_{2}+\frac{2}{3}\frac{2}{3}A\rho_{1}\partial_{x}\rho_{1}+\frac{2}{3}A\rho_1\partial_{x}\rho_{1})
	\notag\\
	&&=-\frac{2}{3}A\partial_{x}\rho_{2}+A\partial_{x}\theta_{2}+\frac{2}{9}A\rho_{1}\partial_{x}\rho_{1}.
\end{eqnarray}
Using $\eqref{2.19}_1\times A+\eqref{2.19}_2$ and \eqref{2.18}, then a direct calculation shows that
\begin{equation}
	\label{2.14a}
-\frac{5}{3}\partial_{x}\rho_{2}+\partial_{x}\theta_{2}+\partial_{x}\phi_{2}+\frac{23}{3}\rho_1\partial_{x}\rho_{1}+2A\partial_{t}\rho_{1}=0.
\end{equation}
Differentiating the last equation of $\eqref{2.19}$ and adding it to \eqref{2.14a}, then
using \eqref{2.18} and \eqref{2.21a}, we get
\begin{equation}
	\label{2.20}	
	2\sqrt{\frac{8}{3}}\partial_{t}\rho_{1}+\frac{58}{9}\rho_1\partial_{x}\rho_{1}+\partial^3_{x}\rho_{1}=0,
\end{equation}
which gives \eqref{1.9}. We should point out that $\rho_1$ satisfies the KdV equation.

To derive the equations of $(\rho_2,u_1^{(2)},\theta_2,\phi_2)$, we first use \eqref{2.19}, \eqref{2.21a} and \eqref{2.18} to write
\begin{equation}
		\left\{
	\begin{array}{rl}
	\label{2.21}
		&u^{(2)}_{1}=A\rho_{2}+h^{1}_1,\quad	\quad h^{1}_1=-\int^{x} [\partial_{t}\rho_{1}+\partial_{x}(\rho_1u^{(1)}_1)](t,\xi)\,d\xi, 
		\\
		&\theta_{2}=\frac{2}{3}\rho_{2}-\frac{1}{9}\rho^2_{1},
		\\
		&\phi_2=\rho_2+\partial^2_{x}\phi_{1}-\frac{1}{2}\phi^2_1=\rho_2+\partial^2_{x}\rho_{1}-\frac{1}{2}\rho^2_1.
\end{array} \right.
\end{equation}
By $\eqref{2.22}_2-\rho_1\times\eqref{2.19}_2+A\times\eqref{2.22}_1$ and $	u^{(1)}_{1}=A\rho_{1}$, we get
\begin{eqnarray}
	\label{2.23}
	&&\partial_{t}u^{(2)}_{1}+A\partial_{t}\rho_{2}-\frac{5}{3}\partial_{x}\rho_{3}+2A\partial_{x}(\rho_{1}u^{(2)}_{1})
	+\partial_{x}\theta_{3}
	+\partial_{x}\rho_{1}\theta_{2}+\partial_{x}(\rho_2\theta_{1})
	\notag\\
	&&+\partial_{x}\phi_{3}+\rho_2\partial_{x}\phi_{1}+\frac{5}{3}\partial_{x}\rho_{2}\rho_1
	+A\rho^2_1\partial_{x}u^{(1)}_{1}-\rho_1\partial_{x}(\rho_1\theta_{1})-\rho^2_1\partial_{x}\phi_{1}=0.
\end{eqnarray}
Differentiating the last equation of \eqref{2.22} and adding it to 
\eqref{2.23}, we have
\begin{eqnarray*}
	&&\partial_{t}u^{(2)}_{1}+A\partial_{t}\rho_{2}-\frac{2}{3}\partial_{x}\rho_{3}+2A\partial_{x}(\rho_{1}u^{(2)}_{1})
	+\partial_{x}\theta_{3}
	+\partial_{x}\rho_{1}\theta_{2}+\partial_{x}(\rho_2\theta_{1})-\partial_{x}(\phi_1\phi_2)
	\notag\\
	&&+\rho_2\partial_{x}\phi_{1}+\frac{5}{3}\partial_{x}\rho_{2}\rho_1+\partial^3_{x}\phi_2
	+A\rho^2_1\partial_{x}u^{(1)}_{1}-\rho_1\partial_{x}(\rho_1\theta_{1})-\rho^2_1\partial_{x}\phi_{1}
	-\partial_{x}(\frac{1}{6}\phi^3_1)=0.
\end{eqnarray*}
Substituting \eqref{2.21} into the above identity gives
\begin{equation}
	\label{2.24}
2A\partial_{t}\rho_{2}-\frac{2}{3}\partial_{x}\rho_{3}+\partial_{x}\theta_{3}
	+\frac{20}{3}\partial_{x}(\rho_{1}\rho_2)+\partial^3_{x}\rho_{2}=h^1_2,
\end{equation}
where we have denoted that
\begin{eqnarray*}
h^1_2=&&-A\rho^2_1\partial_{x}u^{(1)}_{1}+\rho_1\partial_{x}(\rho_1\theta_{1})+\rho^2_1\partial_{x}\phi_{1}
	+\partial_{x}(\frac{1}{6}\phi^3_1)-\partial_{t}h^{1}_{1}-2A\partial_{x}(\rho_{1}h^{1}_1)+\frac{1}{9}\rho^2_{1}\partial_{x}\rho_{1}
	\notag\\
	&&-\partial^3_{x}(\partial^2_{x}\rho_{1}-\frac{1}{2}\rho^2_1)+\partial_{x}(\rho_1\partial^2_{x}\rho_{1})-\partial_{x}(\rho_1\frac{1}{2}\rho^2_1).
\end{eqnarray*}
Using $\eqref{2.22}_1\times\frac{2}{3}-\eqref{2.22}_3$ and \eqref{2.21}, we get
\begin{equation}
	\label{2.25}
	\frac{2}{3}\partial_{x}\rho_{3}
	-\partial_{x}\theta_{3}-\frac{2}{9}\partial_{x}(\rho_1\rho_{2})
	:=-\frac{1}{A}h^1_3,
\end{equation}
where we have denoted that
\begin{equation*}
h^1_3=-\frac{2}{3}\partial_{x}(\rho_1h^{1}_1)-\frac{1}{9}\partial_{t}(\rho^2_{1})+\frac{4}{9}\rho_1\partial_{x}(h^{1}_1)-\frac{1}{9}\frac{2}{3}A\rho^2_{1}\partial_{x}\rho_{1}
	-A\rho_{1}\frac{1}{9}\partial_{x}(\rho^2_{1})
	+\frac{2}{3}h^{1}_1\partial_{x}\rho_{1}.
\end{equation*}
Adding \eqref{2.25} to \eqref{2.24}, one has
\begin{equation}
	\label{2.26}
	2\sqrt{\frac{8}{3}}\partial_{t}\rho_{2}+\frac{58}{9}\partial_{x}(\rho_{1}\rho_2)+\partial^3_{x}\rho_{2}=g(\rho_1),
\end{equation}
where $g(\rho_1)=h^1_2-\frac{1}{A}h^1_3$. By \eqref{2.18} and the definition of $g(\rho_1)$, we know that $g(\rho_1)$ depends only on $\rho_1$. 

\subsection{Perturbation equations}\label{sub2.3}
In the following, we deduce the systems for $(\widetilde{\rho},\widetilde{u},\widetilde{\theta},\widetilde{\phi})(t,x)$
and $f(t,x,v)$ given by \eqref{1.31}.
Recall $(\bar{\rho},\bar{u},\bar{\theta},\bar{\phi})$ given in \eqref{1.29},
 it holds from \eqref{2.16}, \eqref{2.19} and a detailed calculation that
\begin{equation}
	\left\{
	\begin{array}{rl}
		\label{2.29}
		&\partial_t\bar{\rho}-A\frac{1}{\delta}\partial_x\bar{\rho}+\frac{1}{\delta}\partial_x(\bar{\rho}\bar{u}_1)=\delta^2\mathcal{R}_1,		
		\\
		&\bar{\rho}\partial_t\bar{u}_{1}-A\frac{1}{\delta}\bar{\rho}\partial_x\bar{u}_{1}+\frac{1}{\delta}\bar{\rho}\bar{u}_{1}\partial_x\bar{u}_{1}
		+\frac{2}{3}\frac{1}{\delta}\partial_x(\bar{\rho}\bar{\theta})
		+\frac{1}{\delta}\bar{\rho}\partial_x\bar{\phi}
		=\delta^2\mathcal{R}_2,	
		\\
		&\partial_t\bar{\theta}-A\frac{1}{\delta}\partial_x\bar{\theta}+\frac{2}{3}\frac{1}{\delta}\bar{\theta}\partial_x\bar{u}_{1}
		+\frac{1}{\delta}\bar{u}_{1}\partial_x\bar{\theta}
		=\delta^2\mathcal{R}_3,
		\\
		&-\delta\partial^2_x\bar{\phi}=\bar{\rho}-e^{\bar{\phi}}-\delta^3\mathcal{R}_4.		
	\end{array} \right.
\end{equation}
In this expression we have used that
\begin{equation}
	\left\{
	\begin{array}{rl}
		\label{2.30}
		\mathcal{R}_1=&\partial_t\rho_{2}+\partial_x(\rho_1u^{(2)}_1)+\partial_x(\rho_2u^{(1)}_1)
		+\delta\partial_x(\rho_2u^{(2)}_1),
		\\
		\mathcal{R}_2=&\rho_1\partial_tu^{(1)}_{1}+\partial_tu^{(2)}_{1}-A\rho_2\partial_xu^{(1)}_{1}-A\rho_1\partial_xu^{(2)}_{1}+\partial_x(u^{(1)}_1u^{(2)}_{1})
		+\rho_1u^{(1)}_1\partial_xu^{(1)}_{1}+\partial_x(\rho_1\theta_{2})
		\\
		&+\partial_x(\rho_2\theta_{1})+\rho_1\partial_x\phi_{2}+\rho_2\partial_x\phi_{1}
		+\delta[\rho_2u^{(1)}_{1t}+\rho_1\partial_tu^{(2)}_{1}-A\rho_2\partial_xu^{(2)}_{1}+ u^{(2)}_1\partial_xu^{(2)}_{1}]
		\\
		&+\delta[\rho_1\partial_x(u^{(1)}_1u^{(2)}_{1})+\rho_2u^{(1)}_1\partial_xu^{(1)}_{1}+\partial_x(\rho_2\theta_{2})+\rho_2\partial_x\phi_{2}]
		\\
		&+\delta^2[\rho_2\partial_tu^{(2)}_{1}+\rho_1u^{(2)}_1\partial_xu^{(2)}_{1}+\rho_2\partial_x(u^{(1)}_1u^{(2)}_{1})]
		+\delta^3\rho_2u^{(2)}_1\partial_xu^{(2)}_{1},
		\\
		\mathcal{R}_3=&\partial_t\theta_{2}+\theta_1\partial_xu^{(2)}_{1}+\theta_2\partial_xu^{(1)}_{1}+\frac{3}{2}u^{(1)}_1\partial_x\theta_{2}
		+\frac{3}{2}u^{(2)}_1\partial_x\theta_{1}+\delta[\theta_2\partial_xu^{(2)}_{1}+\frac{3}{2}u^{(2)}_1\partial_x\theta_{2}],
		\\
		\mathcal{R}_4=&\partial^2_x\phi_{2}-\phi_1\phi_2-\frac{1}{2}\delta\phi^2_2+\frac{1}{\delta^3}\{1+(\delta\phi_1+\delta^2\phi_2)
		+\frac{1}{2!}(\delta\phi_1+\delta^2\phi_2)^2-e^{\delta\phi_1+\delta^2\phi_2}\}.
	\end{array} \right.
\end{equation}
Using \eqref{1.12}, \eqref{1.13}, \eqref{2.20} and \eqref{2.26}, for all $t\in[0,\tau]$, we have
\begin{equation}
\label{2.31}
\|(\mathcal{R}_1,\mathcal{R}_2,\mathcal{R}_3,\mathcal{R}_4)\|^2_{H^k}
+\|\partial_t(\mathcal{R}_1,\mathcal{R}_2,\mathcal{R}_3,\mathcal{R}_4)\|^2_{H^k}\leq C, \quad \mbox{for}~~k\geq 0.
\end{equation}
Subtracting \eqref{2.29} from system \eqref{2.13} and using \eqref{1.31}, after careful computations, we get
\begin{equation}
	\label{2.33}
\left\{
\begin{array}{rl}
&\partial_t\widetilde{\rho}-A\frac{1}{\delta}\partial_x\widetilde{\rho}
		+\frac{1}{\delta}\partial_x(\bar{\rho}\widetilde{u}_{1})
		+\frac{1}{\delta}\partial_x(\widetilde{\rho}u_{1})=-\delta^2\mathcal{R}_1	,
		\\
		&\partial_t\widetilde{u}_{1}-A\frac{1}{\delta}\partial_x\widetilde{u}_{1}
		+\frac{1}{\delta}u_{1}\partial_x\widetilde{u}_{1}+\frac{1}{\delta}\widetilde{u}_{1}\partial_x\bar{u}_{1}+\frac{1}{\delta}\frac{2}{3}\partial_x\widetilde{\theta}
		+\frac{1}{\delta}\frac{2}{3}(\frac{\theta}{\rho}-\frac{\bar{\theta}}{\bar{\rho}})\partial_x\rho+\frac{1}{\delta}\frac{2}{3}\frac{\bar{\theta}}{\bar{\rho}}\partial_x\widetilde{\rho}
		+\frac{1}{\delta}\partial_x\widetilde{\phi}
		\\
		&\hspace{0.5cm}=\frac{\varepsilon}{\delta^{1/2}}\frac{4}{3\rho}\partial_x(\mu(\theta)\partial_xu_{1})-\frac{1}{\delta}\frac{1}{\rho}\partial_x(\int_{\mathbb{R}^{3}} v^{2}_{1}L^{-1}_{M}\Theta\,dv)-\delta^2\frac{1}{\bar{\rho}}\mathcal{R}_2,
		\\
		&\partial_t\widetilde{u}_{i}-A\frac{1}{\delta}\partial_x\widetilde{u}_{i}+\frac{1}{\delta}u_{1}\partial_x\widetilde{u}_{i}
		=\frac{\varepsilon}{\delta^{1/2}}\frac{1}{\rho}\partial_x(\mu(\theta)\partial_xu_{i})-\frac{1}{\delta}\frac{1}{\rho}\partial_x(\int_{\mathbb{R}^3} v_{1}v_{i}L^{-1}_{M}\Theta\,dv), ~~i=2,3,
		\\
		&\partial_t\widetilde{\theta}-A\frac{1}{\delta}\partial_x\widetilde{\theta}
		+\frac{2}{3}\frac{1}{\delta}\bar{\theta}\partial_x\widetilde{u}_{1}+\frac{2}{3}\frac{1}{\delta}\widetilde{\theta}\partial_xu_{1}
		+\frac{1}{\delta}u_1\partial_x\widetilde{\theta}+\frac{1}{\delta}\widetilde{u}_{1}\partial_x\bar{\theta}
		\\
		&\hspace{0.5cm}=\frac{\varepsilon}{\delta^{1/2}}\frac{1}{\rho}\partial_x(\kappa(\theta)\partial_x\theta) 
		+D+\frac{1}{\delta}\frac{1}{\rho}u\cdot\partial_x(\int_{\mathbb{R}^3} v_{1}v L^{-1}_{M}\Theta\, dv)
		\\
		&\hspace{1cm}-\frac{1}{\delta}\frac{1}{\rho}\partial_x(\int_{\mathbb{R}^3}v_{1}\frac{|v|^{2}}{2}L^{-1}_{M}\Theta\, dv)-\delta^2\mathcal{R}_3,
		\\
		&-\delta\partial^{2}_{x}\widetilde{\phi}=\widetilde{\rho}-(e^{\phi}-e^{\bar{\phi}})+\delta^3\mathcal{R}_4,
	\end{array} \right.
\end{equation}
where $D$ is given by
\begin{equation}
	\label{2.35a}
D=\frac{\varepsilon}{\delta^{1/2}}\frac{4}{3\rho}\mu(\theta)(\partial_xu_{1})^2
+\frac{\varepsilon}{\delta^{1/2}}\frac{1}{\rho}\mu(\theta)[(\partial_xu_{2})^2+(\partial_xu_{3})^2].
\end{equation}
On the other hand, we clearly get from \eqref{2.9} and \eqref{2.29} that
\begin{equation}
	\label{2.34}
\left\{
\begin{array}{rl}
&\partial_t\widetilde{\rho}-A\frac{1}{\delta}\partial_x\widetilde{\rho}+\frac{1}{\delta}\partial_x(\widetilde{\rho}u_{1})
+\frac{1}{\delta}\partial_x(\bar{\rho}\widetilde{u}_{1})=-\delta^2\mathcal{R}_1,
		\\
&\partial_t\widetilde{u}_{1}-A\frac{1}{\delta}\partial_x\widetilde{u}_{1}
+\frac{1}{\delta}u_{1}\partial_x\widetilde{u}_{1}+\frac{1}{\delta}\widetilde{u}_{1}\partial_x\bar{u}_{1}+\frac{1}{\delta}\frac{2}{3}\partial_x\widetilde{\theta}
+\frac{1}{\delta}\frac{2}{3}(\frac{\theta}{\rho}-\frac{\bar{\theta}}{\bar{\rho}})\partial_x\rho+\frac{1}{\delta}\frac{2}{3}\frac{\bar{\theta}}{\bar{\rho}}\partial_x\widetilde{\rho}
+\frac{1}{\delta}\partial_x\widetilde{\phi}
\\
&\hspace{0.5cm} =-\frac{1}{\delta}\frac{1}{\rho}\partial_x(\int_{\mathbb{R}^{3}} v^{2}_{1}G\,dv)-\delta^2\frac{1}{\bar{\rho}}\mathcal{R}_2,
		\\
		&\partial_t\widetilde{u}_{i}-A\frac{1}{\delta}\partial_x\widetilde{u}_{i}+\frac{1}{\delta}u_{1}\partial_x\widetilde{u}_{i}
		=-\frac{1}{\delta}\frac{1}{\rho}\partial_x(\int_{\mathbb{R}^3} v_{1}v_{i}G\,dv), ~~i=2,3,
		\\
	&\partial_t\widetilde{\theta}-A\frac{1}{\delta}\partial_x\widetilde{\theta}
+\frac{2}{3}\frac{1}{\delta}\bar{\theta}\partial_x\widetilde{u}_{1}+\frac{2}{3}\frac{1}{\delta}\widetilde{\theta}\partial_xu_{1}
+\frac{1}{\delta}u_1\partial_x\widetilde{\theta}+\frac{1}{\delta}\widetilde{u}_{1}\partial_x\bar{\theta}
		\\
		&\hspace{0.5cm}=\frac{1}{\delta}\frac{1}{\rho}u\cdot\partial_x(\int_{\mathbb{R}^3} v_{1}v G\, dv)
		-\frac{1}{\delta}\frac{1}{\rho}\partial_x(\int_{\mathbb{R}^3}v_{1}\frac{|v|^{2}}{2}G\, dv)-\delta^2\mathcal{R}_3,
		\\
		&-\delta\partial^{2}_{x}\widetilde{\phi}=\widetilde{\rho}-(e^{\phi}-e^{\bar{\phi}})+\delta^3\mathcal{R}_4.
	\end{array} \right.
\end{equation}

Now  we turn to derive the microscopic perturbation equations $f$. Using
\eqref{1.18}, \eqref{1.20} and \eqref{1.32}, we have from a direct calculation that
\begin{equation*}
\frac{1}{\delta}P_{1}(v_{1}\partial_xM)=\frac{1}{\delta}P_{1}\{v_{1}M(\frac{|v-u|^{2}
\partial_x\widetilde{\theta}}{2K\theta^{2}}+\frac{(v-u)\cdot\partial_x\widetilde{u}}{K\theta})\}
+\frac{1}{\delta^{3/2}\varepsilon}L_{M}\overline{G}.
\end{equation*}
Substituting this into \eqref{2.10} leads to
\begin{multline}
\label{2.37}
\partial_tG-A\frac{1}{\delta}\partial_xG+\frac{1}{\delta}P_{1}(v_{1}\partial_xG)-\frac{1}{\delta}\partial_x\phi\partial_{v_{1}}G
-\frac{1}{\delta^{3/2}\varepsilon}L_{M}(\sqrt{\mu}{f})
\\
=\frac{1}{\delta^{3/2}\varepsilon}Q(G,G)-\frac{1}{\delta}P_{1}\{v_{1}M(\frac{|v-u|^{2}
\partial_x\widetilde{\theta}}{2K\theta^{2}}+\frac{(v-u)\cdot\partial_x\widetilde{u}}{K\theta})\}.
\end{multline}
Inspired by \cite{Guo-2002}, we denote
\begin{equation}
	\label{2.38}
	\Gamma(h,g):=\frac{1}{\sqrt{\mu}}Q(\sqrt{\mu}h,\sqrt{\mu}g),
	\quad \mathcal{L}h:=\Gamma(\sqrt{\mu},h)+\Gamma(h,\sqrt{\mu}),
\end{equation}
which together with \eqref{1.24} gives
\begin{equation*}
	\frac{1}{\sqrt{\mu}}L_{M}(\sqrt{\mu}f)
	=\mathcal{L}f+\Gamma(f,\frac{M-\mu}{\sqrt{\mu}})+
	\Gamma(\frac{M-\mu}{\sqrt{\mu}},f).
\end{equation*}
The main reason for denoting \eqref{2.38} is that many known
estimates on $\mathcal{L}$ and $\Gamma$ in \cite{Strain-Guo} can be directly employed; see, for instance Lemma \ref{lem3.3}
and Lemma \ref{lem3.4}.

Hence, from \eqref{2.38}, we further write \eqref{2.37} as
\begin{eqnarray}
\label{2.39}
&&\partial_tf-A\frac{1}{\delta}\partial_xf+\frac{1}{\delta}v_{1}\partial_xf-\frac{1}{\delta}\frac{\partial_x\phi\partial_{v_{1}}(\sqrt{\mu}f)}{\sqrt{\mu}}	-\frac{1}{\delta^{3/2}\varepsilon}\mathcal{L}f
	\notag\\
	=&&\frac{1}{\delta^{3/2}\varepsilon}\{\Gamma(f,\frac{M-\mu}{\sqrt{\mu}})+
	\Gamma(\frac{M-\mu}{\sqrt{\mu}},f)\}+\frac{1}{\delta^{3/2}\varepsilon}\Gamma(\frac{G}{\sqrt{\mu}},\frac{G}{\sqrt{\mu}})
	\notag\\
	&&+\frac{1}{\delta}\frac{P_{0}(v_{1}\sqrt{\mu}\partial_xf)}{\sqrt{\mu}}-\frac{1}{\delta}\frac{1}{\sqrt{\mu}}P_{1}\{v_{1}M(\frac{|v-u|^{2}
		\partial_x\widetilde{\theta}}{2K\theta^{2}}+\frac{(v-u)\cdot\partial_x\widetilde{u}}{K\theta})\}
	\notag\\
	&&+A\frac{1}{\delta}\frac{\partial_x\overline{G}}{\sqrt{\mu}}+\frac{1}{\delta}\frac{\partial_x\phi\partial_{v_{1}}\overline{G}}{\sqrt{\mu}}
	-\frac{1}{\delta}\frac{P_{1}(v_{1}\partial_x\overline{G})}{\sqrt{\mu}}
	-\frac{\partial_t\overline{G}}{\sqrt{\mu}}.
\end{eqnarray}
Likewise, we also write the first equation of \eqref{1.15} as
\begin{multline}
	\label{2.40}
	\frac{\partial_tF}{\sqrt{\mu}}-A\frac{1}{\delta}\frac{\partial_xF}{\sqrt{\mu}}+\frac{1}{\delta}\frac{v_{1}\partial_xF}{\sqrt{\mu}}
	-\frac{1}{\delta}\frac{\partial_x\phi\partial_{v_{1}}F}{\sqrt{\mu}}
	\\
	=\frac{1}{\delta^{3/2}\varepsilon}
	\{\mathcal{L}f+\Gamma(f,\frac{M-\mu}{\sqrt{\mu}})+
	\Gamma(\frac{M-\mu}{\sqrt{\mu}},f)+\Gamma(\frac{G}{\sqrt{\mu}},\frac{G}{\sqrt{\mu}})+\frac{L_M\overline{G}}{\sqrt{\mu}}\}.
\end{multline}

\section{Basic estimates}\label{sub3}
Starting from this section, we perform a rigorous proof of the KdV limit of the VPL system. 
We first make the {\it a priori} assumption for the solutions and then derive some necessary estimates that will be used in the later energy analysis. 
We should emphasize that in all estimates below, all constants $\widetilde{C}>1$ at different places depend on $C_0$ given in \eqref{3.1}
but do not depend on both small parameters $\varepsilon$ and $\delta$.

\subsection{A priori assumption}\label{subs3.1}
The existence and uniqueness of the short-time solution for the VPL system \eqref{1.15}  under the conditions of Theorem \ref{thm1.1}
 can be proved by the following the same strategy as in \cite{Guo-JAMS}  and details of the proof are
 omitted for brevity. In order to extend the short-time solution to any finite time $\tau$ as given in Proposition \ref{prop.1.1},
  we only  close the following {\it a priori} assumption
\begin{equation}
\label{3.1}
\sup_{0\leq t\leq T}\mathcal{E}_{2,l,q_1}(t)\leq C_0\delta^4,
\end{equation}
for arbitrary time $T\in(0,\tau]$, where $C_0>1$ is a constant independent of $\delta$ as given in \eqref{7.10}.
% which will be determined later.

For some constant $\widetilde{C}=\widetilde{C}(C_0)$, we can choose $\delta$ small enough such that 
\begin{equation}
\label{3.2a}
\widetilde{C}\delta^{a}<1, \quad \mbox{for}\quad a>0.
\end{equation}
Using the following 1D embedding inequality
\begin{equation*}
\|g\|_{L^{\infty}(\mathbb{R})}\leq \sqrt{2}\|g\|^{\frac{1}{2}}\|g'\|^{\frac{1}{2}},
\quad \mbox{for any}\quad g=g(x)\in H^{1}(\mathbb{R}),
\end{equation*}
and \eqref{1.33}, \eqref{1.34} as well as \eqref{3.1}, it holds for all $t\in(0,T]$ that
\begin{equation}
\label{3.2}
\|(\widetilde{\rho},\widetilde{u},\widetilde{\theta},\widetilde{\phi})\|_{L^{\infty}}
\leq \widetilde{C}\delta^2,
\quad 
\|\partial_x(\widetilde{\rho},\widetilde{u},\widetilde{\theta})\|_{L^{\infty}}
\leq \widetilde{C}\delta\frac{\delta^{5/4}}{\varepsilon^{1/2}},\quad \|\partial_x\widetilde{\phi}\|_{L^{\infty}}
\leq  \widetilde{C}\delta^{7/4},
\end{equation}
and
\begin{equation}
	\label{3.3}
\|\partial^2_{x}(\widetilde{\rho},\widetilde{u},\widetilde{\theta},\widetilde{\phi})\|
\leq \widetilde{C}\frac{\delta^{5/2}}{\varepsilon}, \quad \|\partial^2_{x}\widetilde{\phi}\|\leq \widetilde{C}\delta^{3/2},
\quad \|\partial^3_{x}\widetilde{\phi}\|\leq \widetilde{C}\frac{\delta^2}{\varepsilon}.
\end{equation}
On the other hand, by \eqref{1.29}, \eqref{1.12} and \eqref{1.13}, we obtain
\begin{equation}
	\label{3.4}
\|(\bar{\rho}-1,\bar{u},\bar{\theta}-\frac{3}{2},\bar{\phi})\|_{L^{\infty}}+\|\partial_x(\bar{\rho},\bar{u},\bar{\theta},\bar{\phi})\|_{H^4}
+\|\partial_t(\bar{\rho},\bar{u},\bar{\theta},\bar{\phi})\|_{H^4}\leq C\delta.
\end{equation}
Using \eqref{3.2}, \eqref{3.4}, \eqref{1.31} and $\widetilde{C}\delta<1$ in \eqref{3.2a},  it holds uniformly in all $(t,x)\in[0,T]\times\mathbb{R}$ that
$$
|\rho(t,x)-1|\leq |\rho(t,x)-\bar{\rho}(t,x)|+|\bar{\rho}(t,x)-1|
\leq C\delta.
$$
Similar estimates also hold for $u(t,x)$ and $\theta(t,x)$. Therefore, for sufficiently small $\delta>0$, we deduce that
\begin{equation}
	\label{3.5}
	|\rho(t,x)-1|+|u(t,x)|+|\theta(t,x)-\frac{3}{2}|<C\delta,\quad
	1<\theta(t,x)<3, \quad  \frac{1}{2}<\rho(t,x)< \frac{3}{2}.
\end{equation}

\subsection{Estimates on correction term $\overline{G}$}
%To perform the energy estimates for the equations \eqref{2.39} and \eqref{2.40},
We shall derive some necessary estimates on $\overline{G}$ in \eqref{1.32}.
For this, we first give the following lemma whose proof can be found in \cite[Lemma 6.1]{Duan-Yu1}.
\begin{lemma}\label{lem3.1}
Let $L_{\widehat{M}}$ defined by \eqref{1.24} for any Maxwellian $\widehat{M}=M_{[\widehat{\rho},\widehat{u},\widehat{\theta}]}(v)$  and its null space is denoted by $\ker{L_{\widehat{M}}}$. Suppose that $U(v)$ is any polynomial of $\frac{v-\hat{u}}{\sqrt{K}\hat{\theta}}$ such that
$U(v)\widehat{M}\in(\ker{L_{\widehat{M}}})^{\perp}$. Then, for any $\epsilon\in(0,1)$ and any multi-index $\beta$, there exists  $C_{\beta}>0$ such that
$$
|\partial_{\beta}L^{-1}_{\widehat{M}}(U(v)\widehat{M})|\leq C_{\beta}(\widehat{\rho},\widehat{u},\widehat{\theta})\widehat{M}^{1-\epsilon}.
$$	
Moreover, under the condition \eqref{3.5}, there exists $C_{\beta}>0$ such that
	\begin{equation}
		\label{3.6}
		|\partial_{\beta}A_{j}(\frac{v-u}{\sqrt{K\theta}})|+|\partial_{\beta}B_{ij}(\frac{v-u}{\sqrt{K\theta}})|
		\leq C_{\beta}M^{1-\epsilon},
	\end{equation}
	where $A_{j}(\cdot)$ and $B_{ij}(\cdot)$ are defined in \eqref{2.15}.
\end{lemma}
With Lemma \ref{lem3.1} and \eqref{3.1} in hand, we deduce that
\begin{lemma}
\label{lem3.2}
Let $\overline{G}$ defined in \eqref{1.32}, and $w=w(\alpha,\beta)$ in \eqref{1.25} with any small constant $q_1\in(0,1]$ and the constant $q_2>0$. 		
Assume \eqref{3.1} and \eqref{3.2a} hold, then for any $b\geq 0$, and $|\beta|\geq0$ and $|\alpha|\leq 2$, one has
\begin{equation}
	\label{3.7}
\|\langle v\rangle^{b}\partial^{\alpha}_{\beta}(\frac{\overline{G}}{\sqrt{\mu}})\|_w+
	\|\langle v\rangle^{b}\partial^{\alpha}_{\beta}(\frac{\overline{G}}{\sqrt{\mu}})\|_{\sigma,w}
\leq C\delta^{\frac{3}{2}}(\varepsilon+\delta^2).
\end{equation}
\end{lemma}
\begin{proof}
With \eqref{2.14} and \eqref{2.15} in hand, it is easy to rewrite $\overline{G}$ in \eqref{1.32} as
\begin{equation}
\label{3.8}
\overline{G}=\delta^{1/2}\varepsilon \{\frac{\sqrt{K}\partial_x\bar{\theta}}{\sqrt{\theta}}A_{1}(\frac{v-u}{\sqrt{K\theta}})
+\partial_x\bar{u}_{1}B_{11}(\frac{v-u}{\sqrt{K\theta}})\}.
\end{equation}	
Then, for $\beta_{1}=(1,0,0)$,  it is straightforward to check that
\begin{equation}
	\label{3.9}
\partial_{\beta_{1}}\overline{G}=\delta^{1/2}\varepsilon\{\frac{\sqrt{K}\partial_x\bar{\theta}}{\sqrt{\theta}}
	\partial_{v_{1}}A_{1}(\frac{v-u}{\sqrt{K\theta}})\frac{1}{\sqrt{K\theta}}
	+\partial_x\bar{u}_{1}\partial_{v_{1}}B_{11}(\frac{v-u}{\sqrt{K\theta}})\frac{1}{\sqrt{K\theta}}\},
\end{equation}	
and
\begin{eqnarray}
	\label{3.10}
	\partial_{x}\overline{G}&&=\delta^{1/2}\varepsilon\Big\{\frac{\sqrt{K}\partial^2_x\bar{\theta}}{\sqrt{\theta}}A_{1}(\frac{v-u}{\sqrt{K\theta}})
	-\frac{\sqrt{K}\partial_x\bar{\theta}\partial_x\theta}{2\sqrt{\theta^{3}}}A_{1}(\frac{v-u}{\sqrt{K\theta}})
	\notag\\
	&&\quad-\frac{\sqrt{K}\partial_x\bar{\theta}}{\sqrt{\theta}}
	\nabla_{v}A_{1}(\frac{v-u}{\sqrt{K\theta}})\cdot\frac{\partial_xu}{\sqrt{K\theta}}
	-\frac{\sqrt{K}\partial_x\bar{\theta}\partial_x\theta}{\sqrt{\theta}}
	\nabla_{v}A_{1}(\frac{v-u}{\sqrt{K\theta}})\cdot\frac{v-u}{2\sqrt{K\theta^{3}}}
	\notag\\
	&&\quad+\partial^2_x\bar{u}_{1}B_{11}(\frac{v-u}{\sqrt{K\theta}})
	-\frac{\partial_x\bar{u}_{1}\partial_xu}{\sqrt{K\theta}}\cdot\nabla_{v}B_{11}(\frac{v-u}{\sqrt{K\theta}})
	-\frac{\partial_x\bar{u}_{1}\partial_x\theta(v-u)}{2\sqrt{K\theta^{3}}}\cdot\nabla_{v}B_{11}(\frac{v-u}{\sqrt{K\theta}})
	\Big\}.
\end{eqnarray}
Let $w(\alpha,\beta)$ defined in \eqref{1.25} with sufficiently small constant $0<q_1\ll1$,
then for any $|\alpha|\geq0$, $|\beta|\geq0$, $b\geq0$ and sufficiently small  $\epsilon>0$, by \eqref{3.5} and \eqref{1.28}, we can deduce that
\begin{equation}
\label{3.11}
|\langle v\rangle^b w(\alpha,\beta)\mu^{-\frac{1}{2}}M^{1-\epsilon}|_2+|\langle v\rangle^b w(\alpha,\beta)\mu^{-\frac{1}{2}}M^{1-\epsilon}|_\sigma\leq C.
\end{equation}
By the similar expansion as \eqref{3.9}, then we use \eqref{3.6}, \eqref{3.11} and \eqref{3.4} to get
\begin{equation}
\label{3.12}
\|\langle v\rangle^{b}\partial _{\beta}(\frac{\overline{G}}{\sqrt{\mu}})\|^2_{w}+\|\langle v\rangle^{b} \partial _{\beta}(\frac{\overline{G}}{\sqrt{\mu}})\|^2_{\sigma,w}
	\leq C\delta\varepsilon^2\|\partial_x(\bar{u},\bar{\theta})\|^2\leq C\delta^3\varepsilon^2.
\end{equation}
If $|\alpha|=1$ and $|\beta|\geq0$, then by the similar calculation as \eqref{3.9}, \eqref{3.10} and \eqref{3.11}, we show that
\begin{eqnarray}
\label{3.13}
&&\|\langle v\rangle^{b}\partial^{\alpha}_{\beta}(\frac{\overline{G}}{\sqrt{\mu}})\|^2_{w}+\|\langle v\rangle^{b}\partial^{\alpha}_{\beta}(\frac{\overline{G}}{\sqrt{\mu}})\|^2_{\sigma,w}
	\notag\\
	&&\leq 
	C\delta\varepsilon^2(\|\partial^2_x(\bar{u},\bar{\theta})\|^2
	+\|\partial_x(\bar{u},\bar{\theta})\|_{L^{\infty}}^2\|\partial_x(u,\theta)\|^2)
	\notag\\
	&&\leq  C\delta\varepsilon^2(\delta^2+\delta^4)\leq  C\delta^3\varepsilon^2,
\end{eqnarray}
where in the second inequality  we have used the fact that
\begin{equation}
\label{3.14}
\|\partial_x(\rho,u,\theta)\|^2\leq 2\|\partial_x(\widetilde{\rho},\widetilde{u},\widetilde{\theta})\|^2
+2\|\partial_x(\bar{\rho},\bar{u},\bar{\theta})\|^2\leq \widetilde{C}\delta^4+C\delta^2\leq C\delta^2,
\end{equation}
due to \eqref{3.4}, \eqref{3.1} and $\widetilde{C}\delta<1$.
If $|\alpha|=2$ and $|\beta|\geq0$, then by a similar calculation as \eqref{3.13}, we obtain 
\begin{eqnarray}
		\label{3.15}
	&&\|\langle v\rangle^{b}\partial^{\alpha}_{\beta}(\frac{\overline{G}}{\sqrt{\mu}})\|^2_{w}+
	\|\langle v\rangle^{b}\partial^{\alpha}_{\beta}(\frac{\overline{G}}{\sqrt{\mu}})\|^2_{\sigma,w}
	\notag\\
	&&\leq 
	C\delta\varepsilon^2\{\|\partial^3_x(\bar{u},\bar{\theta})\|^2
	+\|\partial^2_x(\bar{u},\bar{\theta})\|_{L^{\infty}}^2\|\partial_x(u,\theta)\|^2
\notag\\	
	&&\hspace{1.5cm}+\|\partial_x(\bar{u},\bar{\theta})\|_{L^{\infty}}^2
	(\|\partial^2_x(u,\theta)\|^2+\|\partial_x(u,\theta)\|_{L^{\infty}}^2\|\partial_x(u,\theta)\|^2)\}
	\notag\\
	&&\leq  C\delta\varepsilon^2(\delta^2+\delta^4+\delta^2\|\partial_x(u,\theta)\|^2+\delta^2\|\partial^2_x(u,\theta)\|^2)
	\notag\\
	&&\leq  C\delta^3\varepsilon^2+\widetilde{C}\delta^8\leq  C\delta^3(\varepsilon+\delta^2)^2,
\end{eqnarray}
where in the third inequality we have used \eqref{3.3} and \eqref{3.4} such that
\begin{equation}
	\label{3.16}
\|\partial^2_x(\rho,u,\theta)\|^2\leq 2\|\partial^2_{x}(\widetilde{\rho},\widetilde{u},\widetilde{\theta})\|^2
+2\|\partial^2_x(\bar{\rho},\bar{u},\bar{\theta})\|^2
\leq \widetilde{C}\frac{\delta^5}{\varepsilon^2}+C\delta^2.
\end{equation}
In summary, we combine \eqref{3.12} with \eqref{3.13} and \eqref{3.15} to obtain the desired estimate \eqref{3.7}.  This ends the proof of Lemma \ref{lem3.2}.
\end{proof}
\subsection{Estimates on collision terms}\label{subs3.3}
This subsection is devoted to the estimates of the collision terms.
\subsubsection{Properties of $\mathcal{L}$ and $\Gamma$}\label{seca3.3.1}
For the linearized  Landau  operator $\mathcal{L}$ in \eqref{2.38}, one has the following standard facts \cite{Guo-2002}.
First, $\mathcal{L}$ is self-adjoint and non-positive, and its null space $(\ker\mathcal{L})$ is spanned by the functions $\{\sqrt{\mu},v\sqrt{\mu},|v|^{2}\sqrt{\mu}\}$. Moreover, there exists a constant $c_{1}>0$ such that
\begin{equation}
	\label{3.17}
	-\langle\mathcal{L}g, g \rangle\geq c_{1}|g|^{2}_{\sigma}
\end{equation}
for any $g\in (\ker\mathcal{L})^{\perp}$. Note that $f\in (\ker\mathcal{L})^{\perp}$ in \eqref{2.39} is purely
microscopic since $G$ and $\overline{G}$ are purely microscopic.

In the following, we list some lemmas on the velocity weighted estimates for $\mathcal{L}$ and $\Gamma$ in \eqref{2.38}.
\begin{lemma}\label{lem3.3}\cite[Lemma 9]{Strain-Guo}
	Let $\mathcal{L}$ defined in \eqref{2.38} and
	$w=w(\alpha,\beta)$ in \eqref{1.25} with sufficiently small $q_1\in(0,1]$ and the constant $q_2>0$. 
	For any small constant $\eta>0$, there exists $C_\eta>0$ such that
	\begin{equation}
		\label{3.18}
		-\langle\partial^\alpha_\beta\mathcal{L}g,w^2(\alpha,\beta)\partial^\alpha_\beta g\rangle\geq |w(\alpha,\beta)\partial^\alpha_\beta g|_\sigma^2-\eta\sum_{|\beta_1|=|\beta|}|w(\alpha,\beta_1)\partial^\alpha_{\beta_1} g|_\sigma^2
		-C_\eta\sum_{|\beta_1|<|\beta|}|w(\alpha,\beta_1)\partial^\alpha_{\beta_1} g|_\sigma^2.
	\end{equation}
	If $|\beta|= 0$, then one has
	\begin{equation}
		\label{3.19}
		-\langle\partial^\alpha\mathcal{L}g,w^2(\alpha,0)\partial^\alpha g\rangle\geq |w(\alpha,0)\partial^\alpha g|_\sigma^2-C_\eta|\chi_{\eta}(v)\partial^\alpha g|_2^2,
	\end{equation}
	where $\chi_\eta(v)$ is a general cutoff function depending on $\eta$.
\end{lemma}
\begin{lemma}
	\label{lem3.4} \cite[Lemmas 2.2-2.3]{Wang}
Let $\Gamma(g_1,g_2)$ defined in \eqref{2.38} and $w=w(\alpha,\beta)$ in \eqref{1.25} with sufficiently small $q_1\in(0,1]$ and  $q_2>0$. 
 For any small constant $\epsilon>0$, it holds that
	\begin{equation}
		\label{3.20}
		\langle\partial^\alpha \Gamma(g_1,g_2), g_3\rangle\leq C\sum_{\alpha_1\leq\alpha}|\mu^\epsilon\partial^{\alpha_1}g_1|_2| \partial^{\alpha-\alpha_1}g_2|_\sigma|  g_3|_\sigma,
	\end{equation}
	and
	\begin{equation}
		\label{3.21}
		\langle\partial^\alpha_\beta \Gamma(g_1,g_2), w^2(\alpha,\beta)g_3\rangle\leq
		C\sum_{\alpha_1\leq\alpha}\sum_{\bar{\beta}\leq\beta_1\leq\beta}|\mu^\epsilon\partial^{\alpha_1}_{\bar{\beta}}g_1|_2|w(\alpha,\beta)  \partial^{\alpha-\alpha_1}_{\beta-\beta_1}g_2|_{\sigma}|w(\alpha,\beta)g_3|_{\sigma}.
	\end{equation}
\end{lemma}
\subsubsection{Estimates on linear collision terms}\label{seca.4.3.2}
For later use, we now deduce the estimates of the linear collision terms. First of all, we focus on the weighted derivative estimates.
\begin{lemma}\label{lem3.5}
Assume \eqref{3.1} holds. Let $w=w(\alpha,\beta)$ defined in \eqref{1.25} with sufficiently small $q_1\in(0,1]$ and $q_2>0$.
If $|\alpha|+|\beta|\leq 2$ and $|\beta|\geq1$, one has
	\begin{equation}
		\label{3.22}
		\frac{1}{\delta^{3/2}\varepsilon}|(\partial^\alpha_\beta \Gamma(f,\frac{M-\mu}{\sqrt{\mu}}),w^2(\alpha,\beta)\partial^\alpha_\beta f)|
		+\frac{1}{\delta^{3/2}\varepsilon}|(\partial^\alpha_\beta\Gamma(\frac{M-\mu}{\sqrt{\mu}},f), w^2(\alpha,\beta)\partial^\alpha_\beta f)|
	\leq C\delta\mathcal{D}_{2,l,q_1}(t).
	\end{equation}
If $|\beta|=0$ and $|\alpha|\leq 1$, one has
\begin{equation}
	\label{3.26}
	\frac{1}{\delta^{3/2}\varepsilon}|(\partial^\alpha\Gamma(\frac{M-\mu}{\sqrt{\mu}},f)+\partial^\alpha \Gamma(f,\frac{M-\mu}{\sqrt{\mu}}),w^2(\alpha,0)\partial^\alpha f)|
	\leq C\delta\mathcal{D}_{2,l,q_1}(t).
\end{equation}
\end{lemma}
\begin{proof}
We only prove the estimate \eqref{3.22} while the estimate \eqref{3.26} can be handled in the same way. For the first term on the left-hand 
side of \eqref{3.22}, it is clear by \eqref{3.21} that
	\begin{eqnarray}
		\label{3.23}
		&&\frac{1}{\delta^{3/2}\varepsilon}|(\partial^\alpha_\beta \Gamma(f,\frac{M-\mu}{\sqrt{\mu}}),w^2(\alpha,\beta)\partial^\alpha_\beta f)|
		\notag\\
		&&\leq C\sum_{\alpha_1\leq\alpha}\sum_{\bar{\beta}\leq\beta_1\leq\beta}
		\underbrace{\frac{1}{\delta^{3/2}\varepsilon}\int_{\mathbb {R}}|\mu^\epsilon\partial^{\alpha_1}_{\bar{\beta}}f|_2| w(\alpha,\beta) \partial^{\alpha-\alpha_1}_{\beta-\beta_1}(\frac{M-\mu}{\sqrt{\mu}})|_{\sigma}|  w(\alpha,\beta)\partial^\alpha_\beta f|_{\sigma}\,dx}_{J_1}.
	\end{eqnarray}
For any $|\beta'|\geq0$ and $b\geq0$,	by \eqref{1.28} and the expression of $w(\alpha,\beta)$ given in \eqref{1.25}, one has
	\begin{eqnarray*}
		&&| \langle v\rangle^{b} w(\alpha,\beta)\partial _{\beta'}(\frac{M-\mu}{\sqrt{\mu}})|_{\sigma}^2+| \langle v\rangle^{b}w(\alpha,\beta)\partial _{\beta'}(\frac{M-\mu}{\sqrt{\mu}})|_{2}^2
		\notag\\
		&&\quad\leq C\sum_{|\beta'|\leq|\beta''|\leq|\beta'|+1}\int_{{\mathbb R}^3}\mu^{-3q_1}|\partial _{\beta''}(\frac{M-\mu}{\sqrt{\mu}})|^2\,dv.
	\end{eqnarray*}
 Let $C\delta>0$ in \eqref{3.5} be small enough
	and $q_1\in(0,1]$ be suitably small, we can find that there exists a suitably large constant  $R_1>0$ such that
	$$
	\int_{|v|\geq R_1}\mu^{-3q_1}|\partial _{\beta''}(\frac{M-\mu}{\sqrt{\mu}})|^2\,dv\leq C\delta^2,
	$$
	and
	$$
	\int_{|v|\leq R_1}\mu^{-3q_1}|\partial _{\beta''}(\frac{M-\mu}{\sqrt{\mu}})|^2\,dv\leq C(|\rho-1|+|u|+|\theta-\frac{3}{2}|)^2\leq C\delta^2.
	$$
Hence, from	the above bound, we see that for any $|\beta'|\geq0$ and $b\geq0$,
\begin{equation}
		\label{3.24}
		| \langle v\rangle^{b}w(\alpha,\beta)\partial _{\beta'}(\frac{M-\mu}{\sqrt{\mu}})|_{\sigma}+| \langle v\rangle^{b}w(\alpha,\beta)
		\partial _{\beta'}(\frac{M-\mu}{\sqrt{\mu}})|_{2}\leq C\delta.
	\end{equation}

Let's now turn to bound \eqref{3.23}. It is clear that $|\alpha|\leq 1$ since  we only consider the case $|\alpha|+|\beta|\leq 2$ and $|\beta|\geq1$.
If $|\alpha-\alpha_{1}|=0$, then by \eqref{3.24}, \eqref{1.28} and \eqref{1.35}, one has
	\begin{eqnarray*}
		J_1&&\leq C\frac{1}{\delta^{3/2}\varepsilon}\|| w(\alpha,\beta) \partial^{\alpha-\alpha_1}_{\beta-\beta_1}(\frac{M-\mu}{\sqrt{\mu}})|_{\sigma}\|_{L^{\infty}}
		\|\mu^\epsilon\partial^{\alpha_1}_{\bar{\beta}}f\|
		\|w(\alpha,\beta)\partial^{\alpha}_{\beta}f\|_{\sigma}
		\notag\\
		&&\leq C\delta\frac{1}{\delta^{3/2}\varepsilon}\|\partial^{\alpha}_{\bar{\beta}}f\|_{\sigma}\|\partial^{\alpha}_{\beta}f\|_{\sigma,w}
		\notag\\
		&&\leq C\delta\mathcal{D}_{2,l,q_1}(t).
	\end{eqnarray*}
If $|\alpha-\alpha_{1}|=1$, then $|\alpha|=1$, $|\alpha_1|=0$ and $|\alpha_1|+|\bar{\beta}|\leq|\alpha|+|\beta|-1$,  which implies that
	\begin{eqnarray*}
		J_1&&\leq C\frac{1}{\delta^{3/2}\varepsilon}\||\mu^\epsilon\partial^{\alpha_1}_{\bar{\beta}}f|_2\|_{L^{\infty}}\|| w(\alpha,\beta) \partial^{\alpha-\alpha_1}_{\beta-\beta_1}(\frac{M-\mu}{\sqrt{\mu}})|_{\sigma}\|
		\|w(\alpha,\beta)\partial^{\alpha}_{\beta}f\|_{\sigma}
		\notag\\
		&&\leq C\frac{1}{\delta^{3/2}\varepsilon}\|\partial^{\alpha-\alpha_{1}}(\rho,u,\theta)\|\|\mu^\epsilon\partial^{\alpha_1}_{\bar{\beta}}f\|^{\frac{1}{2}}
		\|\mu^\epsilon\partial^{\alpha_1}_{\bar{\beta}}\partial_xf\|^{\frac{1}{2}}\|w(\alpha,\beta)\partial^{\alpha}_{\beta}f\|_{\sigma}
		\notag\\
		&&\leq C\frac{1}{\delta^{3/2}\varepsilon}\|\partial_x(\rho,u,\theta)\|(\|\partial^{\alpha_1}_{\bar{\beta}}f\|_\sigma
		\|\partial^{\alpha_1}_{\bar{\beta}}\partial_xf\|_\sigma+\|\partial^{\alpha}_{\beta}f\|^2_{\sigma,w})
		\notag\\
		&&\leq C\delta\mathcal{D}_{2,l,q_1}(t).
	\end{eqnarray*}
	Here we have used the embedding inequality, \eqref{3.11}, \eqref{3.14} and \eqref{1.35}.
Consequently, putting the above two estimates into \eqref{3.23} gives rise to
	\begin{equation}
		\label{3.25}
		\frac{1}{\delta^{3/2}\varepsilon}|(\partial^\alpha_\beta \Gamma(f,\frac{M-\mu}{\sqrt{\mu}}),w^2(\alpha,\beta)\partial^\alpha_\beta f)|
		\leq C\delta\mathcal{D}_{2,l,q_1}(t).
	\end{equation}

For the second term on the left-hand side of \eqref{3.22}, we use \eqref{3.21} again to obtain
	\begin{eqnarray*}
		&&\frac{1}{\delta^{3/2}\varepsilon}|(\partial^\alpha_\beta \Gamma(\frac{M-\mu}{\sqrt{\mu}},f),w^2(\alpha,\beta)\partial^\alpha_\beta f)|
		\notag\\
		&&\leq C\sum_{\alpha_1\leq\alpha}\sum_{\bar{\beta}\leq\beta_1\leq\beta}\underbrace{\frac{1}{\delta^{3/2}\varepsilon}\int_{\mathbb {R}}|\mu^\epsilon\partial^{\alpha_1}_{\bar{\beta}}(\frac{M-\mu}{\sqrt{\mu}})|_2| w(\alpha,\beta) \partial^{\alpha-\alpha_1}_{\beta-\beta_1}f|_{\sigma}|  w(\alpha,\beta)\partial^\alpha_\beta f|_{\sigma}\,dx}_{J_2}.
	\end{eqnarray*}
The term $J_2$ can be handled in the similar way as $J_1$. If $|\alpha_{1}|=0$, then $w(\alpha,\beta)\leq w(\alpha-\alpha_{1},\beta-\beta_1)$, 
we thus deduce from \eqref{3.24} and \eqref{1.35} that
	\begin{eqnarray*}
		J_2&&\leq C\frac{1}{\delta^{3/2}\varepsilon}\||\mu^\epsilon\partial^{\alpha_1}_{\bar{\beta}}(\frac{M-\mu}{\sqrt{\mu}})|_2\|_{L^\infty}
		\|w(\alpha,\beta)\partial^{\alpha-\alpha_{1}}_{\beta-\beta_{1}}f\|_{\sigma}
		\|w(\alpha,\beta)\partial^\alpha_\beta f\|_{\sigma}
		\notag\\
		&&\leq C\delta\frac{1}{\delta^{3/2}\varepsilon}\|\partial^{\alpha-\alpha_{1}}_{\beta-\beta_{1}}f\|_{\sigma,w}
		\|\partial^\alpha_\beta f\|_{\sigma,w}
		\leq C\delta\mathcal{D}_{2,l,q_1}(t).
	\end{eqnarray*}
If $|\alpha_{1}|=1$, then $w(\alpha,\beta)\leq w(\alpha-\alpha_{1}+\alpha_{2},\beta-\beta_1)$ for $|\alpha_{2}|\leq1$, which implies that 
	\begin{eqnarray*}
		J_2&&\leq C\frac{1}{\delta^{3/2}\varepsilon}\|\mu^\epsilon\partial^{\alpha_1}_{\bar{\beta}}(\frac{M-\mu}{\sqrt{\mu}})\|
		\||w(\alpha,\beta)\partial^{\alpha-\alpha_1}_{\beta-\beta_{1}}f|_{\sigma}\|_{L^\infty}\|w(\alpha,\beta)\partial^{\alpha}_{\beta}f\|_{\sigma}
		\notag\\
		&&\leq C\delta\mathcal{D}_{2,l,q_1}(t).
	\end{eqnarray*}
So, from the above three estimates, we get
	\begin{equation*}
		\frac{1}{\delta^{3/2}\varepsilon}|(\partial^\alpha_\beta \Gamma(\frac{M-\mu}{\sqrt{\mu}},f),w^2(\alpha,\beta)\partial^\alpha_\beta f)|
		\leq C\delta\mathcal{D}_{2,l,q_1}(t),
	\end{equation*}
which combined with \eqref{3.25} immediately gives \eqref{3.22}. 
The estimate \eqref{3.26} can be proved by the same strategy as the proof of \eqref{3.22} and details of the proof are
omitted for brevity. Thus, Lemma \ref{lem3.5} is proved.
\end{proof}
\begin{lemma}\label{lem3.7}
Under the same conditions as in Lemma \ref{lem3.5}. For $|\alpha|=2$, there exists sufficiently small $\eta>0$ such that
\begin{eqnarray}
\label{3.27}	
&& \frac{1}{\delta^{3/2}\varepsilon}|(\partial^\alpha\Gamma(\frac{M-\mu}{\sqrt{\mu}},f)+\partial^\alpha\Gamma(f,\frac{M-\mu}{\sqrt{\mu}}),w^2(\alpha,0)\frac{\partial^\alpha F}{\sqrt{\mu}})|
\notag\\
&&\leq C\eta\frac{1}{\delta^{3/2}\varepsilon}\|\partial^{\alpha}f\|^2_{\sigma,w}+C_\eta(\frac{1}{\varepsilon}+\frac{\delta}{\varepsilon^{2}})\mathcal{D}_{2,l,q_1}(t)
+C_\eta\frac{1}{\varepsilon}\delta^{5/2}.		
\end{eqnarray}
\end{lemma}
\begin{proof}
For $|\alpha|=2$, we write
	\begin{eqnarray*}
	&&\frac{1}{\delta^{3/2}\varepsilon}(\partial^{\alpha}\Gamma(\frac{M-\mu}{\sqrt{\mu}},f),w^2(\alpha,0)\frac{\partial^{\alpha}F}{\sqrt{\mu}})=
	\frac{1}{\delta^{3/2}\varepsilon}(\Gamma(\frac{M-\mu}{\sqrt{\mu}},\partial^{\alpha}f),w^2(\alpha,0)\frac{\partial^{\alpha}F}{\sqrt{\mu}})
		\notag\\
		&&\hspace{3cm}+\sum_{1\leq\alpha_{1}\leq\alpha}C^{\alpha_{1}}_{\alpha}\frac{1}{\delta^{3/2}\varepsilon}(\Gamma(
		\frac{\partial^{\alpha_{1}}[M-\mu]}{\sqrt{\mu}},\partial^{\alpha-\alpha_{1}}f),w^2(\alpha,0)\frac{\partial^{\alpha}F}{\sqrt{\mu}}).
	\end{eqnarray*}
By a simple computation, one has the following identity
	$$
\partial_{x}M=M\big\{\frac{\partial_{x}\rho}{\rho}+\frac{(v-u)\cdot \partial_{x}u}{K\theta}
	+(\frac{|v-u|^{2}}{2K\theta}-\frac{3}{2})\frac{\partial_{x}\theta}{\theta} \big\}.
	$$
We denote $\partial^{\alpha}=\partial_{xx}$ with $|\alpha|=2$, then it holds that
	\begin{eqnarray}
		\label{3.28}
		\partial^{\alpha}M=&&M\big\{\frac{\partial^{\alpha}\rho}{\rho}
		+\frac{(v-u)\cdot\partial^{\alpha}u}{K\theta}+(\frac{|v-u|^{2}}{2K\theta}-\frac{3}{2})\frac{\partial^{\alpha}\theta}{\theta} \big\}
		\notag\\
		&&+\big\{\partial_{x}(M\frac{1}{\rho})\partial_{x}\rho+\partial_{x}(M\frac{v-u}{K\theta})\cdot \partial_{x}u
		+\partial_{x}(M\frac{|v-u|^{2}}{2K\theta^{2}}-M\frac{3}{2\theta})\partial_{x}\theta\big\}
		\notag\\
		:=&&I_1+I_2.
	\end{eqnarray}
Hence, for $|\alpha|=2$, we deduce from \eqref{3.28}, \eqref{1.28}, \eqref{3.11}, \eqref{3.4}, \eqref{3.2a} and \eqref{3.14} that
	\begin{eqnarray}
	\label{3.29}
\|\frac{\partial^{\alpha}M}{\sqrt{\mu}}\|^2_{\sigma,w}\leq&& C\|\partial^2_x(\rho,u,\theta)\|^2
+C\|\partial_x(\rho,u,\theta)\|_{L^{\infty}}^2\|\partial_x(\rho,u,\theta)\|^2
\notag\\
\leq&& C\|\partial^{\alpha}(\widetilde{\rho},\widetilde{u},\widetilde{\theta})\|^2+C\delta^2.
\end{eqnarray}
In view of the decomposition $F=M+\overline{G}+\sqrt{\mu}f$, we have from \eqref{3.7} and \eqref{3.29} that for $|\alpha|=2$
\begin{eqnarray}
	\label{3.30}
	\|\frac{\partial^{\alpha}F}{\sqrt{\mu}}\|_{\sigma,w}
	&&\leq C(\|\frac{\sqrt{\mu}\partial^{\alpha}f}{\sqrt{\mu}}\|_{\sigma,w}
	+\|\frac{\partial^{\alpha}\overline{G}}{\sqrt{\mu}}\|_{\sigma,w}+\|\frac{\partial^{\alpha}M}{\sqrt{\mu}}\|_{\sigma,w})
	\notag\\
	&&\leq C(\|\partial^{\alpha}f\|_{\sigma,w}
	+\|\partial^{\alpha}(\widetilde{\rho},\widetilde{u},\widetilde{\theta})\|+\delta).
\end{eqnarray}
So, from \eqref{3.21} and \eqref{3.24}, we have
	\begin{eqnarray}
		\label{3.31}
&&\frac{1}{\delta^{3/2}\varepsilon}|(\Gamma(\frac{M-\mu}{\sqrt{\mu}},\partial^{\alpha}f),w^2(\alpha,0)\frac{\partial^{\alpha}F}{\sqrt{\mu}})|
\notag\\
		&&\leq C\frac{1}{\delta^{3/2}\varepsilon}\|\mu^{\epsilon}\frac{M-\mu}{\sqrt{\mu}}\|_{L^\infty}
		\|\partial^{\alpha}f\|_{\sigma,w}\|\frac{\partial^{\alpha}F}{\sqrt{\mu}}\|_{\sigma,w}
		\notag\\
		&&\leq \eta\frac{1}{\delta^{3/2}\varepsilon}\|\partial^{\alpha}f\|^2_{\sigma,w}
		+C_\eta\delta^2\frac{1}{\delta^{3/2}\varepsilon}\|\frac{\partial^{\alpha}F}{\sqrt{\mu}}\|^2_{\sigma,w}
		\notag\\
		&&\leq \eta\frac{1}{\delta^{3/2}\varepsilon}\|\partial^{\alpha}f\|^2_{\sigma,w}+C_\eta\frac{\delta}{\varepsilon^{2}}\mathcal{D}_{2,l,q_1}(t)
		+C_\eta\frac{1}{\varepsilon}\delta^{5/2}.
	\end{eqnarray}
In the last inequality we have used \eqref{3.30}, \eqref{1.35} and the smallness of $\delta$ such that
	\begin{eqnarray}
	\label{3.32}
\delta^2\frac{1}{\delta^{3/2}\varepsilon}\|\frac{\partial^{\alpha}F}{\sqrt{\mu}}\|^2_{\sigma,w}
&&\leq C\delta^2\frac{1}{\delta^{3/2}\varepsilon}(\|\partial^{\alpha}f\|^2_{\sigma,w}
+\|\partial^{\alpha}(\widetilde{\rho},\widetilde{u},\widetilde{\theta})\|^2+\delta^2)
\notag\\
&&\leq C\frac{\delta^2}{\varepsilon^{2}}\frac{1}{\delta^{3/2}\varepsilon}\varepsilon^{2}\|\partial^{\alpha}f\|^2_{\sigma,w}+
C\frac{\delta}{\varepsilon^2}\frac{\varepsilon}{\delta^{1/2}}\|\partial^{\alpha}(\widetilde{\rho},\widetilde{u},\widetilde{\theta})\|^2
+C\frac{1}{\varepsilon}\delta^{5/2}
\notag\\
\notag\\
&&\leq C(\frac{\delta^2}{\varepsilon^{2}}+\frac{\delta}{\varepsilon^{2}})\mathcal{D}_{2,l,q_1}(t)+C\frac{1}{\varepsilon}\delta^{5/2}
\leq C\frac{\delta}{\varepsilon^{2}}\mathcal{D}_{2,l,q_1}(t)+C\frac{1}{\varepsilon}\delta^{5/2}.
\end{eqnarray}
Note from \eqref{3.21} that
	\begin{eqnarray*}
		&&\sum_{1\leq\alpha_{1}\leq\alpha}\frac{1}{\delta^{3/2}\varepsilon}|(\Gamma(\frac{\partial^{\alpha_{1}}[M-\mu]}{\sqrt{\mu}}
		,\partial^{\alpha-\alpha_{1}}f),w^2(\alpha,0)\frac{\partial^{\alpha}F}{\sqrt{\mu}})|
		\notag\\
		&&\leq	
		C\sum_{1\leq\alpha_{1}\leq\alpha}\underbrace{\frac{1}{\delta^{3/2}\varepsilon}\int_{\mathbb{R}}
			|\mu^\epsilon\partial^{\alpha_{1}}(\frac{M-\mu}{\sqrt{\mu}})|_{2}|w(\alpha,0)\partial^{\alpha-\alpha_{1}}f|_{\sigma}
			|w(\alpha,0)\frac{\partial^{\alpha}F}{\sqrt{\mu}}|_{\sigma}\,dx}_{J_3}.
	\end{eqnarray*}
To bound the term $J_3$, we consider the two cases. If $|\alpha_{1}|=1$,  then $w(\alpha,0)\leq w(\alpha-\alpha_{1}+\alpha_{2},0)$ for $|\alpha_{2}|\leq1$, 
hence it follows from  \eqref{3.11}, \eqref{3.14}, \eqref{3.32} and \eqref{1.35} that
	\begin{eqnarray*}	
		J_3\leq&& C\frac{1}{\delta^{3/2}\varepsilon}\|\partial^{\alpha_{1}}(\rho,u,\theta)\|
		\||w(\alpha,0)\partial^{\alpha-\alpha_{1}}f|_{\sigma}\|_{L^{\infty}}\|\frac{\partial^{\alpha}F}{\sqrt{\mu}}\|_{\sigma,w}
		\notag\\
		\leq&& C\delta\frac{1}{\delta^{3/2}\varepsilon}\|\partial^{\alpha-\alpha_{1}}f\|^{\frac{1}{2}}_{\sigma,w}
		\|\partial_x\partial^{\alpha-\alpha_{1}}f\|^{\frac{1}{2}}_{\sigma,w}
		\|\frac{\partial^{\alpha}F}{\sqrt{\mu}}\|_{\sigma,w}
		\notag\\	
		\leq&& C\frac{1}{\varepsilon}\frac{1}{\delta^{3/2}\varepsilon}(\|\partial^{\alpha-\alpha_{1}}f\|^2_{\sigma,w}+
		\varepsilon^2\|\partial_x\partial^{\alpha-\alpha_{1}}f\|^2_{\sigma,w})
		+C\delta^2\frac{1}{\delta^{3/2}\varepsilon}\|\frac{\partial^{\alpha}F}{\sqrt{\mu}}\|^2_{\sigma,w}
\notag\\
\leq&& 
C(\frac{1}{\varepsilon}+\frac{\delta}{\varepsilon^{2}})\mathcal{D}_{2,l,q_1}(t)
+C\frac{1}{\varepsilon}\delta^{5/2}.
	\end{eqnarray*}
If $|\alpha_{1}|=|\alpha|=2$, then by \eqref{3.29}, \eqref{3.3}, \eqref{3.32}, \eqref{1.35}, $\widetilde{C}\delta<1$ and $\varepsilon<1$, we obtain
\begin{eqnarray*}
		J_3
		\leq&& C\frac{1}{\delta^{3/2}\varepsilon}\|\mu^\epsilon\partial^{\alpha_{1}}(\frac{M-\mu}{\sqrt{\mu}})\|
		\||w(\alpha,0)\partial^{\alpha-\alpha_{1}}f|_{\sigma}\|_{L^{\infty}}\|\frac{\partial^{\alpha}F}{\sqrt{\mu}}\|_{\sigma,w}
		\notag\\
		\leq&& C\frac{1}{\delta^{3/2}\varepsilon}(\|\partial^{\alpha_1}(\widetilde{\rho},\widetilde{u},\widetilde{\theta})\|+\delta)
		\|\partial^{\alpha-\alpha_{1}}f\|^{\frac{1}{2}}_{\sigma,w}
		\|\partial_x\partial^{\alpha-\alpha_{1}}f\|^{\frac{1}{2}}_{\sigma,w}
		\|\frac{\partial^{\alpha}F}{\sqrt{\mu}}\|_{\sigma,w}
		\notag\\
		\leq&&C\frac{1}{\delta^2}\frac{1}{\delta^{3/2}\varepsilon}(\widetilde{C}\frac{\delta^{5}}{\varepsilon^2}+\delta^2)
		\|\partial^{\alpha-\alpha_{1}}f\|_{\sigma,w}
		\|\partial_x\partial^{\alpha-\alpha_{1}}f\|_{\sigma,w}+C\delta^2\frac{1}{\delta^{3/2}\varepsilon}
		\|\frac{\partial^{\alpha}F}{\sqrt{\mu}}\|_{\sigma,w}	
	\notag\\
	\leq&& 
	C(\frac{1}{\varepsilon}+\frac{\delta}{\varepsilon^{2}})\mathcal{D}_{2,l,q_1}(t)
	+C\frac{1}{\varepsilon}\delta^{5/2}.
	\end{eqnarray*}
From the above three estimates, we conclude  that
	\begin{equation}
		\label{3.33}
		\sum_{1\leq\alpha_{1}\leq\alpha}\frac{1}{\delta^{3/2}\varepsilon}|(\Gamma(\frac{\partial^{\alpha_{1}}[M-\mu]}{\sqrt{\mu}}
		,\partial^{\alpha-\alpha_{1}}f),w^2(\alpha,0)\frac{\partial^{\alpha}F}{\sqrt{\mu}})|
		\leq C(\frac{1}{\varepsilon}+\frac{\delta}{\varepsilon^{2}})\mathcal{D}_{2,l,q_1}(t)
		+C\frac{1}{\varepsilon}\delta^{5/2}.
	\end{equation}
In summary, the combination of \eqref{3.33} and  \eqref{3.31} directly yields that
	\begin{eqnarray}
		\label{3.34}
&&\frac{1}{\delta^{3/2}\varepsilon}|(\partial^{\alpha}\Gamma(\frac{M-\mu}{\sqrt{\mu}},f),w^2(\alpha,0)\frac{\partial^{\alpha}F}{\sqrt{\mu}})|
\notag\\
&&\leq\eta\frac{1}{\delta^{3/2}\varepsilon}\|\partial^{\alpha}f\|^2_{\sigma,w}+C_\eta(\frac{1}{\varepsilon}+\frac{\delta}{\varepsilon^{2}})\mathcal{D}_{2,l,q_1}(t)
+C_\eta\frac{1}{\varepsilon}\delta^{5/2}.
\end{eqnarray}
The other terms on the left-hand side of \eqref{3.27} has the same bound as in \eqref{3.34}. Hence, the desired estimate \eqref{3.27} follows. 
This finishes the proof of Lemma \ref{lem3.7}.
\end{proof}
In the following, we consider the linear collision terms without weighted function. 
Following the same strategy as in \eqref{3.26}, we claim that
\begin{lemma}\label{lem3.6}
For  $|\alpha|\leq 1$, it holds that
\begin{equation}
\label{3.28a}
\frac{1}{\delta^{3/2}\varepsilon}|(\partial^\alpha\Gamma(\frac{M-\mu}{\sqrt{\mu}},f)+\partial^\alpha \Gamma(f,\frac{M-\mu}{\sqrt{\mu}}),\partial^\alpha f)|
\leq C\delta\mathcal{D}_{2}(t).
\end{equation}
\end{lemma}

For the highest order space derivatives estimates without weighted function, we shall carefully deal with them since
they are most singular.
% than those the highest order derivatives estimates with weighted function.
\begin{lemma}\label{lem3.8a}
 For $|\alpha|=2$, there exists sufficiently small $\eta>0$ such that
	\begin{eqnarray}
\label{3.37a}	
&&	\frac{1}{\delta^{3/2}\varepsilon}|(\partial^\alpha\Gamma(\frac{M-\mu}{\sqrt{\mu}},f),\frac{\partial^\alpha F}{\sqrt{\mu}})|+\frac{1}{\delta^{3/2}\varepsilon}|(\partial^\alpha\Gamma(f,\frac{M-\mu}{\sqrt{\mu}}),\frac{\partial^\alpha F}{\sqrt{\mu}})|
	\notag\\
	\leq&& C(\eta+\delta)\frac{1}{\delta^{3/2}\varepsilon}\|\partial^{\alpha}f\|^2_{\sigma}
	+C_\eta\frac{\delta^{1/2}}{\varepsilon}\|\partial^{\alpha}(\widetilde{\rho},\widetilde{u},\widetilde{\theta})\|^2
	\notag\\
	&&+C_\eta(\frac{\delta^{5/4}}{\varepsilon^{2}}+\frac{1}{\delta^{1/4}})\mathcal{D}_{2}(t)+C_\eta\varepsilon\delta^{5/4}
	+C_\eta\frac{1}{\varepsilon}\delta^{9/2}.	
	\end{eqnarray}
\end{lemma}
\begin{proof}
We only compute the first term on the left hand side of \eqref{3.37a} while the second term can be handled in the same way.	
For $|\alpha|=2$, it is easy to see
\begin{eqnarray*}
	\frac{1}{\delta^{3/2}\varepsilon}(\partial^{\alpha}\Gamma(\frac{M-\mu}{\sqrt{\mu}},f),\frac{\partial^{\alpha}F}{\sqrt{\mu}})=&&	\frac{1}{\delta^{3/2}\varepsilon}(\Gamma(\frac{M-\mu}{\sqrt{\mu}},\partial^{\alpha}f),\frac{\partial^{\alpha}F}{\sqrt{\mu}})
\notag\\
	&&+\sum_{1\leq\alpha_{1}\leq\alpha}C^{\alpha_{1}}_{\alpha}\frac{1}{\delta^{3/2}\varepsilon}(\Gamma(
		\frac{\partial^{\alpha_{1}}[M-\mu]}{\sqrt{\mu}},\partial^{\alpha-\alpha_{1}}f),\frac{\partial^{\alpha}F}{\sqrt{\mu}}).
	\end{eqnarray*}
To bound the first term, we use  $F=M+\overline{G}+\sqrt{\mu}f$ to write
	\begin{equation}
		\label{3.38a}
		\frac{1}{\delta^{3/2}\varepsilon}(\Gamma(\frac{M-\mu}{\sqrt{\mu}},\partial^{\alpha}f),\frac{\partial^{\alpha}F}{\sqrt{\mu}})
		=\frac{1}{\delta^{3/2}\varepsilon}(\Gamma(\frac{M-\mu}{\sqrt{\mu}},\partial^{\alpha}f),\frac{\partial^{\alpha}M}{\sqrt{\mu}}
		+\frac{\partial^{\alpha}\overline{G}}{\sqrt{\mu}}+\partial^{\alpha}f).
	\end{equation}
Recall $\partial^{\alpha}M=I_1+I_2$ given by \eqref{3.28}, we write $I_1=I_1^1+I_1^2$ with
\begin{equation}
\label{3.39a}
	\left\{
	\begin{array}{rl}
	I_1^1=M\{\frac{\partial^{\alpha}\widetilde{\rho}}{\rho}+\frac{(v-u)\cdot\partial^{\alpha}\widetilde{u}}{K\theta}
	+(\frac{|v-u|^{2}}{2K\theta}-\frac{3}{2})\frac{\partial^{\alpha}\widetilde{\theta}}{\theta}\},
\\
	I_1^2=M\{\frac{\partial^{\alpha}\bar{\rho}}{\rho}+\frac{(v-u)\cdot\partial^{\alpha}\bar{u}}{K\theta}
	+(\frac{|v-u|^{2}}{2K\theta}-\frac{3}{2})\frac{\partial^{\alpha}\bar{\theta}}{\theta}\}.
\end{array} \right.
\end{equation}
By the expression of $I^1_1$, we get from \eqref{3.20}, \eqref{3.11} and \eqref{3.24} that
\begin{eqnarray*}
			\frac{1}{\delta^{3/2}\varepsilon}|(\Gamma(\frac{M-\mu}{\sqrt{\mu}},\partial^{\alpha}f),\frac{I^1_{1}}{\sqrt{\mu}})|
			&&\leq
			C\frac{1}{\delta^{3/2}\varepsilon}\int_{\mathbb{R}}|\mu^\epsilon(\frac{M-\mu}{\sqrt{\mu}})|_2|\partial^{\alpha}f|_{\sigma}
			|\partial^{\alpha}(\widetilde{\rho},\widetilde{u},\widetilde{\theta})|\,dx
			\notag\\
			&&\leq C\delta\frac{1}{\delta^{3/2}\varepsilon}
			\|\partial^{\alpha}f\|_{\sigma}\|\partial^{\alpha}(\widetilde{\rho},\widetilde{u},\widetilde{\theta})\|
			\notag\\
			&&\leq 
			\eta\frac{1}{\delta^{3/2}\varepsilon}\|\partial^{\alpha}f\|^2_{\sigma}+C_\eta\frac{\delta^{1/2}}{\varepsilon}\|\partial^{\alpha}(\widetilde{\rho},\widetilde{u},\widetilde{\theta})\|^2.
\end{eqnarray*}
By denoting $\partial^{\alpha}=\partial_{xx}$, we have from \eqref{3.39a} and the integration by parts that
	\begin{eqnarray*}
		&&\frac{1}{\delta^{3/2}\varepsilon}(\Gamma(\frac{M-\mu}{\sqrt{\mu}},\partial^{\alpha}f),\frac{I^2_{1}}{\sqrt{\mu}})
		\notag\\
		=&&-\frac{1}{\delta^{3/2}\varepsilon}\big(\Gamma(\frac{M-\mu}{\sqrt{\mu}},\partial_xf)
		,\partial_{x}[\frac{M}{\sqrt{\mu}}\{\frac{\partial^{\alpha}\bar{\rho}}{\rho}+\frac{(v-u)\cdot\partial^{\alpha}\bar{u}}{K\theta}
		+(\frac{|v-u|^{2}}{2K\theta}-\frac{3}{2})\frac{\partial^{\alpha}\bar{\theta}}{\theta}\}]\big)
		\notag\\
		&&-\frac{1}{\delta^{3/2}\varepsilon}\big(\Gamma(\partial_{x}[\frac{M-\mu}{\sqrt{\mu}}],\partial_xf),\frac{M}{\sqrt{\mu}}\{\frac{\partial^{\alpha}\bar{\rho}}{\rho}+\frac{(v-u)\cdot\partial^{\alpha}\bar{u}}{K\theta}
		+(\frac{|v-u|^{2}}{2K\theta}-\frac{3}{2})\frac{\partial^{\alpha}\bar{\theta}}{\theta}\}\big),
\end{eqnarray*}
which can be further bounded by
	\begin{eqnarray}
		\label{3.40a}
		&&C\delta\frac{1}{\delta^{3/2}\varepsilon}\int_{\mathbb{R}}|\partial_xf|_{\sigma}
		\{|\partial^{\alpha}\partial_{x}(\bar{\rho},\bar{u},\bar{\theta})|
		+|\partial^{\alpha}(\bar{\rho},\bar{u},\bar{\theta})||\partial_{x}(\rho,u,\theta)|\}\,dx
		\notag\\
		&&+C\frac{1}{\delta^{3/2}\varepsilon}\int_{\mathbb{R}}|\partial_{x}(\rho,u,\theta)||\partial_xf|_{\sigma}|\partial^{\alpha}(\bar{\rho},\bar{u},\bar{\theta})|\,dx
		\notag\\
		&&\leq C\frac{1}{\delta^{3/2}\varepsilon}(\delta^2\|\partial_xf\|_{\sigma}+\delta^3\|\partial_xf\|_{\sigma})
		\notag\\
		&&\leq C\frac{1}{\delta^{3/2}\varepsilon}(\frac{\delta^{5/4}}{\varepsilon^{2}}\|\partial_xf\|^{2}_{\sigma}+\varepsilon^2\delta^{11/4}),
	\end{eqnarray}
where \eqref{3.20}, \eqref{3.11}, \eqref{3.14} and \eqref{3.4} have used. We thereby show that
	\begin{eqnarray}
		\label{3.41a}
		&&\frac{1}{\delta^{3/2}\varepsilon}|(\Gamma(\frac{M-\mu}{\sqrt{\mu}},\partial^{\alpha}f),\frac{I_{1}}{\sqrt{\mu}})|
		\notag\\
		&&\leq \eta\frac{1}{\delta^{3/2}\varepsilon}\|\partial^{\alpha}f\|^2_{\sigma}+C_\eta\frac{\delta^{1/2}}{\varepsilon}\|\partial^{\alpha}(\widetilde{\rho},\widetilde{u},\widetilde{\theta})\|^2+C\frac{\delta^{5/4}}{\varepsilon^{2}}\frac{1}{\delta^{3/2}\varepsilon}\|\partial_xf\|^{2}_{\sigma}+C\varepsilon\delta^{5/4}
			\notag\\
		&&\leq \eta\frac{1}{\delta^{3/2}\varepsilon}\|\partial^{\alpha}f\|^2_{\sigma}+C_\eta\frac{\delta^{1/2}}{\varepsilon}\|\partial^{\alpha}(\widetilde{\rho},\widetilde{u},\widetilde{\theta})\|^2+C\frac{\delta^{5/4}}{\varepsilon^{2}}\mathcal{D}_{2}(t)+C\varepsilon\delta^{5/4}.
	\end{eqnarray}
By the expression of $I_2$ given in \eqref{3.28}, we deduce from \eqref{3.20}, \eqref{3.11}, \eqref{3.14}, \eqref{3.4} and \eqref{1.36} that
	\begin{eqnarray}
		\label{3.42A}
		&&\frac{1}{\delta^{3/2}\varepsilon}|(\Gamma(\frac{M-\mu}{\sqrt{\mu}},\partial^{\alpha}f),\frac{I_{2}}{\sqrt{\mu}})|
		\notag\\
		&&\leq C\delta\frac{1}{\delta^{3/2}\varepsilon}\int_{\mathbb{R}}
		|\partial^{\alpha}f|_{\sigma}|\partial_{x}(\rho,u,\theta)|
		|\partial_{x}(\rho, u,\theta)| \,dx
		\notag\\
		&&\leq \eta\frac{1}{\delta^{3/2}\varepsilon}\|\partial^{\alpha}f\|^2_{\sigma}+C_\eta\delta^2\frac{1}{\delta^{3/2}\varepsilon}\|\partial_{x}(\rho, u,\theta)\|^2_{L^\infty}\|\partial_{x}(\rho, u,\theta)\|^2
		\notag\\
		&&\leq \eta\frac{1}{\delta^{3/2}\varepsilon}\|\partial^{\alpha}f\|^2_{\sigma}
		+C_\eta\delta^4\frac{1}{\delta^{3/2}\varepsilon}(\|\partial_{x}(\widetilde{\rho}, \widetilde{u},\widetilde{\theta})\|^2_{L^\infty}+\delta^2)
		\notag\\
		&&\leq \eta\frac{1}{\delta^{3/2}\varepsilon}\|\partial^{\alpha}f\|^2_{\sigma}+C_\eta\frac{1}{\varepsilon}\delta^{9/2}
		+C_\eta\frac{\delta^3}{\varepsilon^{2}}\mathcal{D}_{2}(t).
	\end{eqnarray}
This together with \eqref{3.41a} and the fact that $\partial^{\alpha}M=I_1+I_2$, we conclude that
	\begin{eqnarray}
		\label{3.42a}
		&&\frac{1}{\delta^{3/2}\varepsilon}|(\Gamma(\frac{M-\mu}{\sqrt{\mu}},\partial^{\alpha}f),\frac{\partial^{\alpha}M}{\sqrt{\mu}})| 
		\notag\\
		&&\leq C \eta\frac{1}{\delta^{3/2}\varepsilon}\|\partial^{\alpha}f\|^2_{\sigma}+C_\eta\frac{\delta^{1/2}}{\varepsilon}\|\partial^{\alpha}(\widetilde{\rho},\widetilde{u},\widetilde{\theta})\|^2+C_\eta\frac{\delta^{5/4}}{\varepsilon^{2}}\mathcal{D}_{2}(t)+C_\eta\varepsilon\delta^{5/4}
		+C_\eta\frac{1}{\varepsilon}\delta^{9/2}.
	\end{eqnarray}
	On the other hand, by \eqref{3.20}, \eqref{3.24} \eqref{3.7}, it is direct to verify that
	\begin{eqnarray*}
		&&\frac{1}{\delta^{3/2}\varepsilon}|(\Gamma(\frac{M-\mu}{\sqrt{\mu}},\partial^{\alpha}f),[\frac{\partial^{\alpha}\overline{G}}{\sqrt{\mu}}+\partial^{\alpha}f])|
		\notag\\
		&&\leq C\delta\frac{1}{\delta^{3/2}\varepsilon}\|\partial^{\alpha}f\|_{\sigma}
		(\|\frac{\partial^{\alpha}\overline{G}}{\sqrt{\mu}}\|_{\sigma}+\|\partial^{\alpha}f\|_{\sigma})
		\notag\\
		&&\leq C\delta\frac{1}{\delta^{3/2}\varepsilon}\|\partial^{\alpha}f\|^2_{\sigma}
		+C\delta\frac{1}{\delta^{3/2}\varepsilon}\delta^3(\varepsilon^2+\delta^4).
	\end{eqnarray*}
Hence, plugging this and \eqref{3.42a} into \eqref{3.38a} leads to
	\begin{eqnarray}
		\label{3.43a}
		&&\frac{1}{\delta^{3/2}\varepsilon}|(\Gamma(\frac{M-\mu}{\sqrt{\mu}},\partial^{\alpha}f),\frac{\partial^{\alpha}F}{\sqrt{\mu}})|
		\notag\\
		&&\leq C(\eta+\delta)\frac{1}{\delta^{3/2}\varepsilon}\|\partial^{\alpha}f\|^2_{\sigma}+C_\eta\frac{\delta^{1/2}}{\varepsilon}\|\partial^{\alpha}(\widetilde{\rho},\widetilde{u},\widetilde{\theta})\|^2+C_\eta\frac{\delta^{5/4}}{\varepsilon^{2}}\mathcal{D}_{2}(t)+C_\eta\varepsilon\delta^{5/4}
		+C_\eta\frac{1}{\varepsilon}\delta^{9/2}.
	\end{eqnarray}
By \eqref{3.20}, we see that
	\begin{multline*}
		\sum_{1\leq\alpha_{1}\leq\alpha}\frac{1}{\delta^{3/2}\varepsilon}|(\Gamma(\frac{\partial^{\alpha_{1}}[M-\mu]}{\sqrt{\mu}}
		,\partial^{\alpha-\alpha_{1}}f),\frac{\partial^{\alpha}F}{\sqrt{\mu}})|
		\\
		\leq	
		C\sum_{1\leq\alpha_{1}\leq\alpha}\frac{1}{\delta^{3/2}\varepsilon}\int_{\mathbb{R}}
			|\mu^{\epsilon}\partial^{\alpha_{1}}(\frac{M-\mu}{\sqrt{\mu}})|_{2}|\partial^{\alpha-\alpha_{1}}f|_{\sigma}
			|\frac{\partial^{\alpha}F}{\sqrt{\mu}}|_{\sigma}\,dx.
	\end{multline*}
In the following, we focus on the case that $|\alpha_{1}|=|\alpha|=2$. By \eqref{3.29}, it is direct to verify that
	\begin{eqnarray*}
&&\frac{1}{\delta^{3/2}\varepsilon}\int_{\mathbb{R}}
|\mu^{\epsilon}\partial^{\alpha_{1}}(\frac{M-\mu}{\sqrt{\mu}})|_{2}|\partial^{\alpha-\alpha_{1}}f|_{\sigma}
|\frac{\partial^{\alpha}F}{\sqrt{\mu}}|_{\sigma}\,dx	
		\notag\\	
&&\leq C\frac{1}{\delta^{3/2}\varepsilon}\|\partial^{\alpha_{1}}(\frac{M-\mu}{\sqrt{\mu}})\|
\||\partial^{\alpha-\alpha_{1}}f|_{\sigma}\|_{L^{\infty}}\|\frac{\partial^{\alpha}F}{\sqrt{\mu}}\|_{\sigma}
\notag\\		
&&\leq
C\frac{\varepsilon^2}{\delta^{5/4}}\frac{1}{\delta^{3/2}\varepsilon}(\|\partial^{\alpha}(\widetilde{\rho},\widetilde{u},\widetilde{\theta})\|^2+\delta^2)
\|\frac{\partial^{\alpha}F}{\sqrt{\mu}}\|_{\sigma}^2+
		C\frac{\delta^{5/4}}{\varepsilon^{2}}\frac{1}{\delta^{3/2}\varepsilon}\||f|_{\sigma}\|^2_{L^{\infty}}
		\notag\\
&&\leq C(\frac{\delta^{5/4}}{\varepsilon^{2}}+\frac{1}{\delta^{1/4}})\mathcal{D}_{2}(t)
+C\varepsilon\delta^{5/4}.
	\end{eqnarray*}
Here we have used \eqref{3.3}, \eqref{3.2a}, \eqref{1.36} and the similar argumentas \eqref{3.30} such  that
	\begin{eqnarray*}
		\label{3.44a}
&&\frac{\varepsilon^2}{\delta^{5/4}}\frac{1}{\delta^{3/2}\varepsilon}(\|\partial^{\alpha}(\widetilde{\rho},\widetilde{u},\widetilde{\theta})\|^2+\delta^2)
\|\frac{\partial^{\alpha}F}{\sqrt{\mu}}\|_{\sigma}^2
\notag\\		
		&&\leq C\frac{\varepsilon^2}{\delta^{5/4}}\frac{1}{\delta^{3/2}\varepsilon}(\|\partial^{\alpha}(\widetilde{\rho},\widetilde{u},\widetilde{\theta})\|^2+\delta^2)\delta^2
		+C\frac{\varepsilon^2}{\delta^{5/4}}\frac{1}{\delta^{3/2}\varepsilon}(\|\partial^{\alpha}(\widetilde{\rho},\widetilde{u},\widetilde{\theta})\|^2+\delta^2)(\|\partial^{\alpha}f\|^2_{\sigma}
		+\|\partial^{\alpha}(\widetilde{\rho},\widetilde{u},\widetilde{\theta})\|^2)
		\notag\\
		&&\leq C\frac{\varepsilon}{\delta^{3/4}}\|\partial^{\alpha}(\widetilde{\rho},\widetilde{u},\widetilde{\theta})\|^2+C\varepsilon\delta^{5/4}
		+C(\widetilde{C}\frac{\delta^{5}}{\varepsilon^2}+\delta^2)
		\frac{\varepsilon^2}{\delta^{5/4}}\frac{1}{\delta^{3/2}\varepsilon}(\|\partial^{\alpha}f\|^2_{\sigma}
		+\|\partial^{\alpha}(\widetilde{\rho},\widetilde{u},\widetilde{\theta})\|^2)
		\notag\\
	&&\leq C(\frac{\delta^{5/4}}{\varepsilon^{2}}+\frac{1}{\delta^{1/4}})\mathcal{D}_{2}(t)
	+C\varepsilon\delta^{5/4}.	
	\end{eqnarray*}
The case $|\alpha_{1}|=1$ has the same bound, hence we conclude that
	\begin{equation*}
		\sum_{1\leq\alpha_{1}\leq\alpha}\frac{1}{\delta^{3/2}\varepsilon}|(\Gamma(\frac{\partial^{\alpha_{1}}[M-\mu]}{\sqrt{\mu}}
		,\partial^{\alpha-\alpha_{1}}f),\frac{\partial^{\alpha}F}{\sqrt{\mu}})|
		\leq  C(\frac{\delta^{5/4}}{\varepsilon^{2}}+\frac{1}{\delta^{1/4}})\mathcal{D}_{2}(t)
		+C\varepsilon\delta^{5/4}.	
	\end{equation*}
The combination of this and \eqref{3.43a} directly yields that
	\begin{eqnarray}
		\label{3.45a}
		&&\frac{1}{\delta^{3/2}\varepsilon}|(\partial^{\alpha}\Gamma(\frac{M-\mu}{\sqrt{\mu}},f),\frac{\partial^{\alpha}F}{\sqrt{\mu}})|
		\notag\\
		\leq&& C(\eta+\delta)\frac{1}{\delta^{3/2}\varepsilon}\|\partial^{\alpha}f\|^2_{\sigma}
		+C_\eta\frac{\delta^{1/2}}{\varepsilon}\|\partial^{\alpha}(\widetilde{\rho},\widetilde{u},\widetilde{\theta})\|^2
		\notag\\
		&&+C_\eta(\frac{\delta^{5/4}}{\varepsilon^{2}}+\frac{1}{\delta^{1/4}})\mathcal{D}_{2}(t)+C_\eta\varepsilon\delta^{5/4}
		+C_\eta\frac{1}{\varepsilon}\delta^{9/2}.
	\end{eqnarray}
	The second term on the left-hand side of \eqref{3.37a}  can be handled in the same way as
	\eqref{3.45a}, we thus obtain the desired estimate \eqref{3.37a}. This finishes
	the proof of Lemma \ref{lem3.8a}.
\end{proof}
\subsubsection{Estimates on non-linear collision terms}\label{subs3.3.3}
Now we move to deal with the estimates on nonlinear collision terms  $\Gamma(\frac{G}{\sqrt{\mu}},\frac{G}{\sqrt{\mu}})$. 
We first compute the weighted derivative estimates.
\begin{lemma}
\label{lem3.8}
Under the same conditions as in Lemma \ref{lem3.5}. For $|\alpha|+|\beta|\leq 2$ and $|\beta|\geq1$, one has
\begin{eqnarray}
\label{3.35}
&&\frac{1}{\delta^{3/2}\varepsilon}|(\partial^\alpha_\beta\Gamma(\frac{G}{\sqrt{\mu}},\frac{G}{\sqrt{\mu}}), w^2(\alpha,\beta)\partial^\alpha_\beta f)|
\notag\\
&&\leq  C\eta\frac{1}{\delta^{3/2}\varepsilon}\|\partial^\alpha_\beta f\|^{2}_{\sigma,w}
+C\delta\mathcal{D}_{2,l,q_1}(t)
+C_\eta\delta^{9/2}\varepsilon^3+C_\eta\frac{1}{\varepsilon}\delta^{17/2}\delta^4.
\end{eqnarray}
For $|\beta|=0$ and $|\alpha|\leq 1$, one has
\begin{eqnarray}
	\label{3.42}
	&&\frac{1}{\delta^{3/2}\varepsilon}|(\partial^\alpha\Gamma(\frac{G}{\sqrt{\mu}},\frac{G}{\sqrt{\mu}}), w^2(\alpha,0)\partial^\alpha f)|
	\notag\\
	&&\leq C\eta\frac{1}{\delta^{3/2}\varepsilon}\|\partial^\alpha f\|^{2}_{\sigma,w}
	+C\delta\mathcal{D}_{2,l,q_1}(t)
	+C_\eta\delta^{9/2}\varepsilon^3+C_\eta\frac{1}{\varepsilon}\delta^{17/2}\delta^4.
\end{eqnarray}
\end{lemma}
\begin{proof}
	Let $|\alpha|+|\beta|\leq 2$ and $|\beta|\geq1$, by $G=\overline{G}+\sqrt{\mu}f$, one has 
		\begin{eqnarray}
			\label{3.48A}
&&\frac{1}{\delta^{3/2}\varepsilon}(\partial^\alpha_\beta\Gamma(\frac{G}{\sqrt{\mu}},\frac{G}{\sqrt{\mu}}), w^2(\alpha,\beta)\partial^\alpha_\beta f)
\notag\\
		=&&\frac{1}{\delta^{3/2}\varepsilon}(\partial^\alpha_\beta\Gamma(\frac{\overline{G}}{\sqrt{\mu}},\frac{\overline{G}}{\sqrt{\mu}}), w^2(\alpha,\beta)\partial^\alpha_\beta f)
		+\frac{1}{\delta^{3/2}\varepsilon}(\partial^\alpha_\beta\Gamma(\frac{\overline{G}}{\sqrt{\mu}},f), w^2(\alpha,\beta)\partial^\alpha_\beta f)
		\notag\\
		&&
		+\frac{1}{\delta^{3/2}\varepsilon}(\partial^\alpha_\beta\Gamma(f,\frac{\overline{G}}{\sqrt{\mu}}), w^2(\alpha,\beta)\partial^\alpha_\beta f)
		+\frac{1}{\delta^{3/2}\varepsilon}(\partial^\alpha_\beta\Gamma(f,f), w^2(\alpha,\beta)\partial^\alpha_\beta f).
		\end{eqnarray}
For the first term on the right hand side of \eqref{3.48A}, we deduce from	\eqref{3.21} and \eqref{3.7} that
	\begin{eqnarray}
		\label{3.36}
		&&\frac{1}{\delta^{3/2}\varepsilon}|(\partial^\alpha_\beta \Gamma(\frac{\overline{G}}{\sqrt{\mu}},\frac{\overline{G}}{\sqrt{\mu}}),w^2(\alpha,\beta)\partial^\alpha_\beta f)|
		\notag\\
		&&\leq C\frac{1}{\delta^{3/2}\varepsilon}\sum_{\alpha_1\leq\alpha}\sum_{\bar{\beta}\leq\beta_1\leq\beta}\int_{{\mathbb R}}|\mu^\epsilon\partial^{\alpha_1}_{\bar{\beta}}(\frac{\overline{G}}{\sqrt{\mu}})|_2|w(\alpha,\beta) \partial^{\alpha-\alpha_1}_{\beta-\beta_1}(\frac{\overline{G}}{\sqrt{\mu}})|_{\sigma}| w(\alpha,\beta)\partial^\alpha_\beta f|_{\sigma}\,dx
		\notag\\
		&&\leq \eta\frac{1}{\delta^{3/2}\varepsilon}\|\partial^\alpha_\beta f\|^{2}_{\sigma,w}
		+C_\eta\frac{1}{\delta^{3/2}\varepsilon}\delta^{6}(\varepsilon+\delta^2)^4.
	\end{eqnarray}
Using \eqref{3.21}, \eqref{3.7} and \eqref{1.35} again, the second term on the right hand side of \eqref{3.48A} is bounded by
	\begin{eqnarray}
		\label{3.38}
		&& C\sum_{\alpha_1\leq\alpha}\sum_{\bar{\beta}\leq\beta_1\leq\beta}
		\frac{1}{\delta^{3/2}\varepsilon}\int_{{\mathbb R}}|\mu^\epsilon\partial^{\alpha_1}_{\bar{\beta}}(\frac{\overline{G}}{\sqrt{\mu}})|_2| w(\alpha,\beta) \partial^{\alpha-\alpha_1}_{\beta-\beta_1}f|_{\sigma}|w(\alpha,\beta)\partial^\alpha_\beta f|_{\sigma}\,dx
		\notag\\
		&&\leq C\delta^{\frac{3}{2}}(\varepsilon+\delta^2)
		\sum_{\alpha_1\leq\alpha}\sum_{\bar{\beta}\leq\beta_1\leq\beta}\frac{1}{\delta^{3/2}\varepsilon}\|w(\alpha,\beta)\partial^{\alpha-\alpha_{1}}_{\beta-\beta_{1}}f\|_{\sigma}
		\|w(\alpha,\beta)\partial^{\alpha}_{\beta}f\|_{\sigma}
		\notag\\	
		&&\leq C\delta^{\frac{3}{2}}(\varepsilon+\delta^2)\mathcal{D}_{2,l,q_1}(t).
	\end{eqnarray}
Note that the third term on the right hand side of \eqref{3.48A} has the same bound as \eqref{3.38}.

It now remains to handle the last term of \eqref{3.48A}.
From \eqref{3.21}, we have
	\begin{eqnarray*}
		&&\frac{1}{\delta^{3/2}\varepsilon}|(\partial^\alpha_\beta \Gamma(f,f),w^2(\alpha,\beta)\partial^\alpha_\beta f)|
		\notag\\
		&&\leq C\sum_{\alpha_1\leq\alpha}\sum_{\bar{\beta}\leq\beta_1\leq\beta}
		\underbrace{\frac{1}{\delta^{3/2}\varepsilon}\int_{{\mathbb R}}|\mu^\epsilon\partial^{\alpha_1}_{\bar{\beta}}f|_2| w(\alpha,\beta) \partial^{\alpha-\alpha_1}_{\beta-\beta_1}f|_{\sigma}|w(\alpha,\beta)\partial^\alpha_\beta f|_{\sigma}\,dx}_{J_{4}}.
	\end{eqnarray*}
In the following, we estimate $J_4$.	If $|\alpha-\alpha_{1}|+|\beta-\beta_{1}|=0$, then $1\leq|\alpha_1|+|\beta_1|=|\alpha|+|\beta|\leq 2$ 
	and $w(\alpha,\beta)\leq w(\alpha-\alpha_{1}+\alpha_{2},\beta-\beta_{1})$ for $|\alpha_2|\leq1$, thus it follows from \eqref{3.1},
	\eqref{1.33} and \eqref{1.35} that
	\begin{eqnarray}
		\label{3.40}
		J_{4}&&\leq C\frac{1}{\delta^{3/2}\varepsilon}\|\partial^{\alpha_{1}}_{\bar{\beta}}f\|
		\||w(\alpha,\beta)\partial^{\alpha-\alpha_{1}}_{\beta-\beta_{1}}f|_{\sigma}\|_{L^\infty}
	\|w(\alpha,\beta)\partial^{\alpha}_{\beta}f\|_{\sigma}
		\notag\\
		&&\leq \widetilde{C}\delta^2\frac{1}{\delta^{3/2}\varepsilon}
		\|w(\alpha,\beta)\partial^{\alpha-\alpha_{1}}_{\beta-\beta_{1}}f\|^{\frac{1}{2}}_{\sigma}\|w(\alpha,\beta)\partial^{\alpha-\alpha_{1}}_{\beta-\beta_{1}}\partial_xf\|^{\frac{1}{2}}_{\sigma}
		\|\partial^{\alpha}_{\beta}f\|_{\sigma,w}
		\notag\\
		&&\leq \widetilde{C}\delta^2\mathcal{D}_{2,l,q_1}(t).
	\end{eqnarray}
If $|\alpha-\alpha_1|+|\beta-\beta_1|=1$, then $|\alpha_1|+|\beta_1|\leq1$. If $|\alpha_1|+|\beta_1|=0$, then $|\alpha|+|\beta|=1$ and
$w(\alpha,\beta)=w(\alpha-\alpha_{1},\beta-\beta_{1})$, thus we get from \eqref{3.1}, \eqref{1.33} and \eqref{1.35} that
\begin{eqnarray}
	\label{3.41}
	J_{4}&&\leq C\frac{1}{\delta^{3/2}\varepsilon}\||\partial^{\alpha_{1}}_{\bar{\beta}}f|_{2}\|_{L^\infty}
	\|w(\alpha,\beta)\partial^{\alpha-\alpha_{1}}_{\beta-\beta_{1}}f\|_{\sigma}
	\|w(\alpha,\beta)\partial^{\alpha}_{\beta}f\|_{\sigma}
	\notag\\
&&\leq \widetilde{C}\delta^2\mathcal{D}_{2,l,q_1}(t).
\end{eqnarray}
If $|\alpha_1|+|\beta_1|=1$, then $|\alpha|+|\beta|=2$ and
$w(\alpha,\beta)\leq w(\alpha-\alpha_{1}+\alpha_2,\beta-\beta_{1})$ for $|\alpha_2|\leq1$.
In this case, if $|\alpha-\alpha_1|=0$, we use the similar arguments as \eqref{3.40} to get the same bound.
If $|\alpha-\alpha_1|=1$, the proof is similar to \eqref{3.41}.

If $|\alpha-\alpha_1|+|\beta-\beta_1|=2$, then $|\alpha_1|+|\beta_1|=0$ and $w(\alpha,\beta)=w(\alpha-\alpha_{1},\beta-\beta_{1})$,
we use the similar arguments as \eqref{3.41} to get the same bound.

As a consequence, we combine the above bounds with $\widetilde{C}\delta<1$ to obtain
	\begin{equation*}
		\frac{1}{\delta^{3/2}\varepsilon}|(\partial^\alpha_\beta \Gamma(f,f),w^2(\alpha,\beta)\partial^\alpha_\beta f)|
		\leq C\delta\mathcal{D}_{2,l,q_1}(t).
	\end{equation*}
In summary, plugging this, \eqref{3.36} and \eqref{3.38} into \eqref{3.48A} gives \eqref{3.35}. 
The proof of \eqref{3.42} follows along the same lines as the proof of \eqref{3.35},
and we omit the proof for brevity. Hence, the proof of Lemma \ref{lem3.8} is completed.
\end{proof}
\begin{lemma}
	\label{lem3.10}
	Under the same conditions as in Lemma \ref{lem3.5}. For $|\alpha|= 2$, one has
	\begin{eqnarray}
		\label{3.43}	
		&&\frac{1}{\delta^{3/2}\varepsilon}|\partial^{\alpha}(\Gamma(\frac{G}{\sqrt{\mu}},\frac{G}{\sqrt{\mu}}),w^2(\alpha,0)\frac{\partial^{\alpha}F}{\sqrt{\mu}})|
		\notag\\
		&&\leq \widetilde{C}\delta\frac{1}{\delta^{3/2}\varepsilon}\|\partial^{\alpha}f\|^2_{\sigma,w}+\widetilde{C}(1+\frac{\delta^2}{\varepsilon^{2}})\mathcal{D}_{2,l,q_1}(t)
		+C\frac{1}{\varepsilon}\delta^{7/2}+C\delta^{3/2}\varepsilon^3.
	\end{eqnarray}
\end{lemma}
\begin{proof}
	For $|\alpha|=2$, since $G=\overline{G}+\sqrt{\mu}f$, the term on the left-hand side of \eqref{3.43} is equivalent to
	\begin{equation}
		\label{3.44}
		\frac{1}{\delta^{3/2}\varepsilon}|(\partial^{\alpha}\Gamma(\frac{\overline{G}}{\sqrt{\mu}},\frac{\overline{G}}{\sqrt{\mu}})+\partial^{\alpha}\Gamma(\frac{\overline{G}}{\sqrt{\mu}},f)+\partial^{\alpha}\Gamma(f,\frac{\overline{G}}{\sqrt{\mu}})+\partial^{\alpha}\Gamma(f,f),w^2(\alpha,0)\frac{\partial^{\alpha}F}{\sqrt{\mu}})|.
	\end{equation}
Using \eqref{3.21}, \eqref{3.7} and \eqref{3.32}, we arrive at
	\begin{eqnarray*}
		&&\frac{1}{\delta^{3/2}\varepsilon}|(\partial^{\alpha}\Gamma(\frac{\overline{G}}{\sqrt{\mu}},
		\frac{\overline{G}}{\sqrt{\mu}}),w^2(\alpha,0)\frac{\partial^{\alpha}F}{\sqrt{\mu}})|
		\notag\\
		&&\leq C\frac{1}{\delta^3}\frac{1}{\delta^{3/2}\varepsilon}\sum_{\alpha_{1}\leq\alpha}\int_{\mathbb{R}}|\mu^\epsilon\partial^{\alpha_{1}}(\frac{\overline{G}}{\sqrt{\mu}})|^2_{2}
		|w(\alpha,0)\partial^{\alpha-\alpha_{1}}(\frac{\overline{G}}{\sqrt{\mu}})|_{\sigma}^2\,dx
		+C\delta^3\frac{1}{\delta^{3/2}\varepsilon}\|\frac{\partial^{\alpha}F}{\sqrt{\mu}}\|^2_{\sigma,w}
		\notag\\
		&&\leq C\frac{1}{\delta^3}\frac{1}{\delta^{3/2}\varepsilon}\delta^{6}(\varepsilon+\delta^2)^4+
		C\frac{\delta^2}{\varepsilon^{2}}\mathcal{D}_{2,l,q_1}(t)+C\frac{1}{\varepsilon}\delta^{7/2}.
	\end{eqnarray*}
Note from \eqref{3.21} that
	\begin{equation*}
		\frac{1}{\delta^{3/2}\varepsilon}|(\partial^{\alpha}\Gamma(\frac{\overline{G}}{\sqrt{\mu}},f),w^2(\alpha,0)\frac{\partial^{\alpha}F}{\sqrt{\mu}})|
		\leq C\sum_{\alpha_{1}\leq\alpha}\underbrace{\frac{1}{\delta^{3/2}\varepsilon}\int_{\mathbb{R}}|\mu^\epsilon\partial^{\alpha_{1}}(\frac{\overline{G}}{\sqrt{\mu}})|_{2}
			|w(\alpha,0)\partial^{\alpha-\alpha_{1}}f|_{\sigma}|w(\alpha,0)\frac{\partial^{\alpha}F}{\sqrt{\mu}}|_{\sigma}\,dx}_{J_5}.
	\end{equation*}
	In the following, we compute $J_5$.
If $|\alpha_{1}|\leq1$, then $w(\alpha,0)\leq w(\alpha-\alpha_{1},0)$, hence it holds  from  \eqref{3.7}, \eqref{3.32} and \eqref{1.35} that
	\begin{eqnarray*}
		J_5&&\leq C\frac{1}{\delta^{3/2}\varepsilon}\||\mu^\epsilon\partial^{\alpha_{1}}(\frac{\overline{G}}{\sqrt{\mu}})|_2\|_{L^{\infty}}
		\|w(\alpha,0)\partial^{\alpha-\alpha_{1}}f\|_{\sigma}\|w(\alpha,0)\frac{\partial^{\alpha}F}{\sqrt{\mu}}\|_{\sigma}
		\notag\\
		&&\leq C\frac{1}{\delta^3}\frac{1}{\delta^{3/2}\varepsilon}\delta^3(\varepsilon+\delta^2)^2\|\partial^{\alpha-\alpha_{1}}f\|^2_{\sigma,w}
		+C\delta^3\frac{1}{\delta^{3/2}\varepsilon}\|\frac{\partial^{\alpha}F}{\sqrt{\mu}}\|^2_{\sigma,w}
		\notag\\
	&&\leq C(1+\frac{\delta^2}{\varepsilon^{2}})\mathcal{D}_{2,l,q_1}(t)
	+C\frac{1}{\varepsilon}\delta^{7/2}.
	\end{eqnarray*}
If $|\alpha_{1}|=|\alpha|=2$, then $w(\alpha,0)\leq w(\alpha-\alpha_{1}+\alpha_{2},0)$ for $|\alpha_{1}|\leq 1$, which yields that
	\begin{eqnarray*}
		J_5&&\leq C\frac{1}{\delta^{3/2}\varepsilon}\|\mu^\epsilon\partial^{\alpha_{1}}(\frac{\overline{G}}{\sqrt{\mu}})\|\||w(\alpha,0)\partial^{\alpha-\alpha_{1}}f|_{\sigma}\|_{L^{\infty}}	\|w(\alpha,0)\frac{\partial^{\alpha}F}{\sqrt{\mu}}\|_{\sigma}
\notag\\
&&\leq C(1+\frac{\delta^2}{\varepsilon^{2}})\mathcal{D}_{2,l,q_1}(t)
+C\frac{1}{\varepsilon}\delta^{7/2}.
	\end{eqnarray*}
Collecting the above estimates, we get
	\begin{equation*}
		\frac{1}{\delta^{3/2}\varepsilon}|(\partial^{\alpha}\Gamma(\frac{\overline{G}}{\sqrt{\mu}},f),w^2(\alpha,0)\frac{\partial^{\alpha}F}{\sqrt{\mu}})|
			\leq C(1+\frac{\delta^2}{\varepsilon^{2}})\mathcal{D}_{2,l,q_1}(t)
		+C\frac{1}{\varepsilon}\delta^{7/2}.
	\end{equation*}
Similar estimates also hold for the third term on the left-hand side of \eqref{3.44}. For the last term in \eqref{3.44}, by \eqref{3.21}, one has
	\begin{equation*}
		\frac{1}{\delta^{3/2}\varepsilon}|(\partial^{\alpha}\Gamma(f,f),w^2(\alpha,0)\frac{\partial^{\alpha}F}{\sqrt{\mu}})|\leq C\sum_{\alpha_{1}\leq\alpha}\underbrace{\frac{1}{\delta^{3/2}\varepsilon}\int_{\mathbb{R}}|\mu^{\epsilon}\partial^{\alpha_{1}}f|_{2}
			|w(\alpha,0)\partial^{\alpha-\alpha_{1}}f|_{\sigma}|w(\alpha,0)\frac{\partial^{\alpha}F}{\sqrt{\mu}}|_{\sigma}\,dx}_{J_6}.
	\end{equation*}
In the following, we estimate $J_6$. If $|\alpha_{1}|=0$, then by \eqref{3.1}, \eqref{1.33} and \eqref{3.32}, one obtains
	\begin{eqnarray*}
		J_6&&\leq C\frac{1}{\delta^{3/2}\varepsilon}\|\mu^{\epsilon}\partial^{\alpha_{1}}f\|_{L^{\infty}}\|\partial^{\alpha-\alpha_{1}}f\|_{\sigma,w}\|\frac{\partial^{\alpha}F}{\sqrt{\mu}}\|_{\sigma,w}
		\notag\\
		&&\leq C\frac{1}{\delta^3}\frac{1}{\delta^{3/2}\varepsilon}\|\mu^{\epsilon}\partial^{\alpha_{1}}f\|^2_{L^{\infty}}\|\partial^{\alpha-\alpha_{1}}f\|^2_{\sigma,w}
			+C\delta^3\frac{1}{\delta^{3/2}\varepsilon}\|\frac{\partial^{\alpha}F}{\sqrt{\mu}}\|^2_{\sigma,w}
				\notag\\
			&&\leq \widetilde{C}\delta\frac{1}{\delta^{3/2}\varepsilon}\|\partial^{\alpha}f\|^2_{\sigma,w}
			+C\frac{\delta^2}{\varepsilon^{2}}\mathcal{D}_{2,l,q_1}(t)
			+C\frac{1}{\varepsilon}\delta^{7/2}.
	\end{eqnarray*}
If  $1\leq|\alpha_{1}|\leq|\alpha|$, then $w(\alpha,0)\leq w(\alpha-\alpha_{1}+\alpha_{2},0)$ for $|\alpha_{2}|\leq 1$, which implies that
	\begin{eqnarray*}
		J_6&&\leq C\frac{1}{\delta^{3/2}\varepsilon}\|\mu^{\epsilon}\partial^{\alpha_{1}}f\|
		\||\partial^{\alpha-\alpha_{1}}f|_{\sigma,w}\|_{L^{\infty}}\|\frac{\partial^{\alpha}F}{\sqrt{\mu}}\|_{\sigma,w}
	\notag\\
&&\leq C\frac{1}{\delta^3}\frac{1}{\delta^{3/2}\varepsilon}\|\mu^{\epsilon}\partial^{\alpha_{1}}f\|^2
\||\partial^{\alpha-\alpha_{1}}f|_{\sigma,w}\|^2_{L^{\infty}}
+C\delta^3\frac{1}{\delta^{3/2}\varepsilon}
\|\frac{\partial^{\alpha}F}{\sqrt{\mu}}\|^2_{\sigma,w}				
	\notag\\
	&&\leq \widetilde{C}(1+\frac{\delta^2}{\varepsilon^{2}})\mathcal{D}_{2,l,q_1}(t)
	+C\frac{1}{\varepsilon}\delta^{7/2}.	
	\end{eqnarray*}
Here we have used $|\alpha_{1}|=|\alpha|=2$, \eqref{3.1}, \eqref{1.33} and \eqref{1.34} such that
\begin{equation}
\label{3.55A}
\|\mu^{\epsilon}\partial^{\alpha_{1}}f\|^2\leq \|\partial^{\alpha}f\|^{2}\leq \widetilde{C}\frac{\delta^5}{\varepsilon^2}.
\end{equation}
From the above bound on $J_6$, we  obtain	
	\begin{equation*}
		\frac{1}{\delta^{3/2}\varepsilon}|(\partial^{\alpha}\Gamma(f,f),w^2(\alpha,0)\frac{\partial^{\alpha}F}{\sqrt{\mu}})|
	\leq \widetilde{C}\delta\frac{1}{\delta^{3/2}\varepsilon}\|\partial^{\alpha}f\|^2_{\sigma,w}+\widetilde{C}(1+\frac{\delta^2}{\varepsilon^{2}})\mathcal{D}_{2,l,q_1}(t)+C\frac{1}{\varepsilon}\delta^{7/2}.	
	\end{equation*}
Hence, plugging all the above estimates into \eqref{3.44} gives
	\begin{eqnarray}
		\label{3.55a}	
		&&\frac{1}{\delta^{3/2}\varepsilon}|\partial^{\alpha}(\Gamma(\frac{G}{\sqrt{\mu}},\frac{G}{\sqrt{\mu}}),w^2(\alpha,0)\frac{\partial^{\alpha}F}{\sqrt{\mu}})|
		\notag\\
		&&\leq \widetilde{C}\delta\frac{1}{\delta^{3/2}\varepsilon}\|\partial^{\alpha}f\|^2_{\sigma,w}+\widetilde{C}(1+\frac{\delta^2}{\varepsilon^{2}})\mathcal{D}_{2,l,q_1}(t)
		+C\frac{1}{\varepsilon}\delta^{7/2}+C\delta^{3/2}\varepsilon^3,
	\end{eqnarray}
which gives the desired estimate \eqref{3.43} and the Lemma \ref{lem3.10} follows.
\end{proof}
Next we  focus on the estimates of the non-linear collision terms without the velocity weight.
\begin{lemma}
	\label{lem3.9}
	For $|\alpha|\leq 1$, there exists sufficiently small $\eta>0$ such that
	\begin{equation}
		\label{3.56a}
\frac{1}{\delta^{3/2}\varepsilon}|(\partial^\alpha\Gamma(\frac{G}{\sqrt{\mu}},\frac{G}{\sqrt{\mu}}), \partial^\alpha f)|
\leq C\eta\frac{1}{\delta^{3/2}\varepsilon}\|\partial^\alpha f\|^{2}_{\sigma}
		+C\delta\mathcal{D}_{2}(t)
		+C_\eta\delta^{9/2}\varepsilon^3+C_\eta\frac{1}{\varepsilon}\delta^{17/2}\delta^4.
	\end{equation}
For $|\alpha|= 2$, it holds that
\begin{eqnarray}
	\label{3.57a}
	&&\frac{1}{\delta^{3/2}\varepsilon}|\partial^{\alpha}(\Gamma(\frac{G}{\sqrt{\mu}},\frac{G}{\sqrt{\mu}}),\frac{\partial^{\alpha}F}{\sqrt{\mu}})|
	\notag\\
	&&\leq \widetilde{C}\delta\frac{1}{\delta^{3/2}\varepsilon}\|\partial^{\alpha}f\|^2_{\sigma}+\widetilde{C}(1+\frac{\delta^2}{\varepsilon^{2}})\mathcal{D}_{2}(t)
	+C\frac{1}{\varepsilon}\delta^{7/2}+C\delta^{3/2}\varepsilon^3.
\end{eqnarray}	
\end{lemma}
\begin{proof}
	The proof of \eqref{3.56a} and \eqref{3.57a} are similar to that of  \eqref{3.35} and \eqref{3.55a}, respectively.
\end{proof}
\subsection{Estimates on fluid quantities with temporal derivatives}
In what follows we turn to derive the
estimates of the fluid quantities involving with the temporal derivatives, which can be bounded by the pure space
derivatives of the solution.
\begin{lemma}\label{lem3.11}
It holds that
	\begin{equation}
	\label{3.45}
	\|(\partial_t\widetilde{\rho},\partial_t\widetilde{u},\partial_t\widetilde{\theta})\|^{2}
	\leq  C\frac{1}{\delta^2}\mathcal{E}_{2}(t)+C\delta\varepsilon^2+C\delta^4,
\end{equation}
and
\begin{eqnarray}
		\label{3.46}
		\|\partial_x\partial_t(\widetilde{\rho},\widetilde{u},\widetilde{\theta})\|^{2}\leq&&
		C\frac{1}{\delta^2}(\|\partial^2_x(\widetilde{\rho},
		\widetilde{u},\widetilde{\theta},\widetilde{\phi})\|^{2}+\|\langle v\rangle^{-\frac{1}{2}}\partial^2_xf\|^{2})
		+C\delta\varepsilon^2+\widetilde{C}\delta^4.
	\end{eqnarray}
\end{lemma}
\begin{proof}
By taking  the inner product of the second equation of \eqref{2.34} with  $\partial_t\widetilde{u}_1$, then using \eqref{3.4},
 \eqref{2.31}, \eqref{3.7},
 \eqref{1.34},  \eqref{3.1} and \eqref{3.2a}, we arrive at
\begin{eqnarray}
	\label{3.61B}
	(\partial_t\widetilde{u}_{1},\partial_t\widetilde{u}_{1})
	=&&-(-A\frac{1}{\delta}\partial_x\widetilde{u}_{1}
	+\frac{1}{\delta}u_{1}\partial_x\widetilde{u}_{1}+\frac{1}{\delta}\widetilde{u}_{1}\partial_x\bar{u}_{1}+\frac{1}{\delta}\frac{2}{3}\partial_x\widetilde{\theta}
	+\frac{1}{\delta}\frac{2}{3}(\frac{\theta}{\rho}-\frac{\bar{\theta}}{\bar{\rho}})\partial_x\rho+\frac{1}{\delta}\frac{2}{3}\frac{\bar{\theta}}{\bar{\rho}}\partial_x\widetilde{\rho},\partial_t\widetilde{u}_{1})
	\notag\\
	&&-\frac{1}{\delta}(\partial_x\widetilde{\phi},\partial_t\widetilde{u}_{1})
	-\frac{1}{\delta}(\frac{1}{\rho}\int_{\mathbb{R}^3} v^2_{1}\partial_xG\,dv,\partial_t\widetilde{u}_{1})-(\delta^2\frac{1}{\bar{\rho}}\mathcal{R}_2,\partial_t\widetilde{u}_{1})
	\notag\\
	\leq&&C \eta\|\partial_t\widetilde{u}_{1}\|^{2}+C_\eta\frac{1}{\delta^2}(\|\partial_x(\widetilde{\rho},
	\widetilde{u},\widetilde{\theta},\widetilde{\phi})\|^{2}
	+\|\langle v\rangle^{-\frac{1}{2}}\partial_xf\|^{2})+C_\eta
	\|(\widetilde{\rho},\widetilde{u},\widetilde{\theta})\|^{2}
	\notag\\
	&&+C_\eta\frac{1}{\delta^2}\delta^3(\varepsilon+\delta^2)^2+C_\eta\delta^4
		\notag\\
	\leq&&C \eta\|\partial_t\widetilde{u}_{1}\|^{2}+C_\eta\frac{1}{\delta^2}\mathcal{E}_{2}(t)
	+C_\eta\delta\varepsilon^2+C_\eta\delta^4.
\end{eqnarray}	
By choosing sufficiently small $\eta>0$, we have
\begin{equation*}
	\|\partial_t\widetilde{u}_{1}\|^{2}
	\leq C\frac{1}{\delta^2}\mathcal{E}_{2}(t)
	+C\delta\varepsilon^2+C\delta^4.
\end{equation*}	
Similar estimates also hold for $\partial_t\widetilde{u}_{2}$, $\partial_t\widetilde{u}_{3}$, 
$\partial_t\widetilde{\rho}$ and $\partial_t\widetilde{\theta}$. Thus, the 	estimate \eqref{3.45} holds.

Differentiating the second equation of \eqref{2.34} with respect to $x$ and
taking the inner product of the resulting equation with $\partial_x\partial_t\widetilde{u}_1$, one obtains
\begin{eqnarray}
	\label{3.47}
		(\partial_x\partial_t\widetilde{u}_{1},\partial_x\partial_t\widetilde{u}_{1})
		=&&-(\partial_x[-A\frac{1}{\delta}\partial_x\widetilde{u}_{1}
		+\frac{1}{\delta}u_{1}\partial_x\widetilde{u}_{1}+\frac{1}{\delta}\widetilde{u}_{1}\partial_x\bar{u}_{1}+\frac{1}{\delta}\frac{2}{3}\partial_x\widetilde{\theta}
		+\frac{1}{\delta}\frac{2}{3}(\frac{\theta}{\rho}-\frac{\bar{\theta}}{\bar{\rho}})\partial_x\rho+\frac{1}{\delta}\frac{2}{3}\frac{\bar{\theta}}{\bar{\rho}}\partial_x\widetilde{\rho}],\partial_x\partial_t\widetilde{u}_{1})
		\notag\\
		&&-\frac{1}{\delta}(\partial^2_x\widetilde{\phi},\partial_x\partial_t\widetilde{u}_{1})
		-\frac{1}{\delta}(\partial_x[\frac{1}{\rho}\int_{\mathbb{R}^3} v^2_{1}\partial_xG\,dv],\partial_x\partial_t\widetilde{u}_{1})-(\delta^2	\partial_x[\frac{1}{\bar{\rho}}\mathcal{R}_2],
		\partial_x\partial_t\widetilde{u}_{1}).
	\end{eqnarray}
In view of  \eqref{3.1}, \eqref{3.4}, \eqref{3.2a} and \eqref{3.7}, it is direct to verify that
\begin{eqnarray*}
	\frac{1}{\delta}|(\partial_x[u_{1}\partial_x\widetilde{u}_{1}],\partial_x\partial_t\widetilde{u}_{1})|
	&&\leq  \eta\|\partial_x\partial_t\widetilde{u}_{1}\|^{2}+C_\eta\frac{1}{\delta^2}(\|\partial_xu_{1}\|_{L^\infty}^2\|\partial_x\widetilde{u}_{1}\|^2+
	\|u_{1}\partial^2_x\widetilde{u}_{1}\|^2)
	\notag\\
	&&\leq  \eta\|\partial_x\partial_t\widetilde{u}_{1}\|^{2}+C_\eta\widetilde{C}\delta^2\|\partial_xu_{1}\|\|\partial^2_xu_{1}\|
	+C_\eta\frac{1}{\delta^2}\|\partial^2_x\widetilde{u}_{1}\|^2
	\notag\\
	&&\leq  \eta\|\partial_x\partial_t\widetilde{u}_{1}\|^{2}+C_\eta\frac{1}{\delta^2}\|\partial^2_x\widetilde{u}_{1}\|^2
	+C_\eta\widetilde{C}\delta^4,
\end{eqnarray*}
and
\begin{equation*}
\frac{1}{\delta}|(\partial_x[\frac{1}{\rho}\int_{\mathbb{R}^3} v^2_{1}\partial_xG\,dv],\partial^{\alpha}\partial_t\widetilde{u}_{1})|
\leq  \eta\|\partial_x\partial_t\widetilde{u}_{1}\|^{2}
+C_\eta\|\langle v\rangle^{-\frac{1}{2}}\partial_xf\|^{2}
+C_\eta\delta\varepsilon^2+C_\eta\widetilde{C}\delta^4.
\end{equation*}
The other terms in \eqref{3.47} can be treated in the similar way. Hence, by choosing sufficiently small $\eta>0$, we arrive at
\begin{equation*}
\|\partial_x\partial_t\widetilde{u}_{1}\|^2
\leq C\frac{1}{\delta^2}(\|\partial^2_x(\widetilde{\rho},
\widetilde{u},\widetilde{\theta},\widetilde{\phi})\|^{2}+\|\langle v\rangle^{-\frac{1}{2}}\partial^2_xf\|^{2})+C\delta\varepsilon^2+\widetilde{C}\delta^4.
	\end{equation*}
Similar estimates also hold for $\partial_x\partial_t\widetilde{u}_{2}$, $\partial_x\partial_t\widetilde{u}_{3}$ and $\partial_x\partial_t\widetilde{\theta}$. 
Thus, the estimate \eqref{3.46} holds and then the proof of Lemma \ref{lem3.11} is completed.
\end{proof}
\begin{lemma}\label{lem3.12}
	For $|\alpha|\leq 1$, it holds that
	\begin{equation}
		\label{3.48}
	\|\partial^{\alpha}\partial_t\widetilde{\phi}\|^2+\delta\|\partial^{\alpha}\partial_x\partial_t\widetilde{\phi}\|^2
	\leq C\frac{1}{\delta^2}\|\partial^{\alpha}\partial_x(\widetilde{\rho},\widetilde{u})\|^{2}
	+\widetilde{C}\delta^4.	
	\end{equation}
\end{lemma}
\begin{proof}
Differentiating the last equation of \eqref{2.34} with respect to $t$, one has
\begin{equation}
\label{3.49}
-\delta\partial^{2}_{x}\partial_t\widetilde{\phi}
+e^{\phi}\partial_t\widetilde{\phi}=\partial_t\widetilde{\rho}
-(e^{\phi}-e^{\bar{\phi}})\partial_t\bar{\phi}+\delta^3\partial_t\mathcal{R}_{4}.
\end{equation}
By taking the inner product of \eqref{3.49} with $\partial_t\widetilde{\phi}$, we get 
\begin{eqnarray*}
	(e^{\phi}\partial_t\widetilde{\phi},\partial_t\widetilde{\phi})-\delta(\partial^{2}_{x}\partial_t\widetilde{\phi},\partial_t\widetilde{\phi})
	&&=(\partial_t\widetilde{\rho},\partial_t\widetilde{\phi})-((e^{\phi}-e^{\bar{\phi}})\partial_t\bar{\phi},\partial_t\widetilde{\phi})
	+\delta^3(\partial_t\mathcal{R}_{4},\partial_t\widetilde{\phi})
	\notag\\
	&&\leq \eta\|\partial_t\widetilde{\phi}\|^2+C_\eta\|\partial_t\widetilde{\rho}\|^2+C_\eta\|\partial_t\bar{\phi}\|^2\|\widetilde{\phi}\|^2_{L^\infty}
	+C_\eta\delta^6\|\partial_t\mathcal{R}_{4}\|^2
	\notag\\
	&&\leq \eta\|\partial_t\widetilde{\phi}\|^2+C_\eta\|\partial_t\widetilde{\rho}\|^2+C_\eta\widetilde{C}\delta^6+C_\eta\delta^6.
\end{eqnarray*}
Here we have used \eqref{2.31}, \eqref{3.1}, \eqref{3.4}, $e^\phi\approx 1$ and  $e^{\bar{\phi}}\approx 1$.

Hence, by taking $\eta>0$ small enough, we have
\begin{equation}
	\label{3.50}
	\|\partial_t\widetilde{\phi}\|^2+\delta\|\partial_x\partial_t\widetilde{\phi}\|^2
	\leq C\|\partial_t\widetilde{\rho}\|^2+\widetilde{C}\delta^6.
\end{equation}
By taking  the inner product of the first equation of \eqref{2.34} with  $\partial_t\widetilde{\rho}$, we claim that
\begin{equation}
	\label{3.64a}
	\|\partial_t\widetilde{\rho}\|^{2}\leq
	C\frac{1}{\delta^2}\|\partial_x(\widetilde{\rho},\widetilde{u})\|^{2}+\widetilde{C}\delta^4.
\end{equation}
This and \eqref{3.50} together, gives the estimate \eqref{3.48} for the case $|\alpha|=0$.

Differentiating \eqref{3.49} with respect to $x$, one obtains
\begin{equation}
	\label{3.51}
	-\delta\partial^{3}_{x}\partial_t\widetilde{\phi}
	+e^{\phi}\partial_{x}\partial_t\widetilde{\phi}+e^{\phi}\partial_{x}\phi\partial_t\widetilde{\phi}=\partial_{x}\partial_t\widetilde{\rho}
	-\partial_{x}[(e^{\phi}-e^{\bar{\phi}})\partial_t\bar{\phi}]+\delta^3\partial_{x}\partial_t\mathcal{R}_{4}.
\end{equation}
By taking the inner product of \eqref{3.51} with $\partial_{x}\partial_t\widetilde{\phi}$, we see that
\begin{eqnarray*}
	&&(e^{\phi}\partial_{x}\partial_t\widetilde{\phi},\partial_{x}\partial_t\widetilde{\phi})-\delta(\partial^{3}_{x}\partial_t\widetilde{\phi}
	,\partial_{x}\partial_t\widetilde{\phi})+
	(e^{\phi}\partial_{x}\phi\partial_t\widetilde{\phi},\partial_{x}\partial_t\widetilde{\phi})
	\notag\\
	&&=(\partial_{x}\partial_t\widetilde{\rho},\partial_{x}\partial_t\widetilde{\phi})
	-(\partial_{x}[(e^{\phi}-e^{\bar{\phi}})\partial_t\bar{\phi}],\partial_{x}\partial_t\widetilde{\phi})
	+\delta^3(\partial_{x}\partial_t\mathcal{R}_{4},\partial_{x}\partial_t\widetilde{\phi}).
\end{eqnarray*}
Thanks to \eqref{3.4}, \eqref{3.2} and \eqref{3.2a}, one has
\begin{equation}
	\label{3.57}
	\|\partial_x\phi\|_{L^{\infty}}\leq \|\partial_x\bar{\phi}\|_{L^{\infty}}+\|\partial_x\widetilde{\phi}\|_{L^{\infty}}
	\leq C\delta+\widetilde{C}\delta\delta^{3/4}\leq C\delta,
\end{equation}
which combine with $e^\phi\approx 1$ and  \eqref{3.50} implies that
\begin{eqnarray*}
|(e^{\phi}\partial_{x}\phi\partial_t\widetilde{\phi},\partial_{x}\partial_t\widetilde{\phi})|
&&\leq \eta\|\partial_{x}\partial_t\widetilde{\phi}\|^2+C_\eta\|\partial_{x}\phi\|^2_{L^\infty}\|\partial_t\widetilde{\phi}\|^2
\notag\\
&&\leq \eta\|\partial_{x}\partial_t\widetilde{\phi}\|^2+C_\eta\delta^2\|\partial_t\widetilde{\rho}\|^2+C_\eta\delta^6.
\end{eqnarray*}
Hence, by taking $\eta>0$ small enough, we claim that
\begin{eqnarray}
\label{3.52}
\|\partial_x\partial_t\widetilde{\phi}\|^2+\delta\|\partial^2_x\partial_t\widetilde{\phi}\|^2
&&\leq C\|\partial_x\partial_t\widetilde{\rho}\|^2
+C\delta^2\|\partial_t\widetilde{\rho}\|^2+\widetilde{C}\delta^6.
\end{eqnarray}
To further bound \eqref{3.52}, we apply  $\partial_{x}$ to the first equation of \eqref{2.34} and
take the inner product of the resulting equation with $\partial_x\partial_t\widetilde{\rho}$ to obtain
\begin{equation*}
	\|\partial_x\partial_t\widetilde{\rho}\|^{2}\leq
	C\frac{1}{\delta^2}\|\partial^2_x(\widetilde{\rho},\widetilde{u})\|^{2}+\widetilde{C}\delta^4.
\end{equation*}
Combining this, \eqref{3.64a} and \eqref{3.52}, we obtain \eqref{3.48} for the case $|\alpha|=1$. We then end up the proof of Lemma \ref{lem3.12}.
\end{proof}
\begin{lemma}\label{lem3.13}
For $|\alpha|\leq 1$, one has
	\begin{eqnarray}	
		\label{3.53}
&&\delta^{1/2}\varepsilon\frac{d}{dt}(\partial^{\alpha}\widetilde{u}_1,\partial_x\partial^{\alpha}\widetilde{\rho})
+c\frac{\varepsilon}{\delta^{1/2}}(\|\partial^{\alpha}\partial_x\widetilde{\rho}\|^2+\|\partial^{\alpha}\partial_x\widetilde{\phi}\|^2
+\delta\|\partial^{\alpha}\partial^2_x\widetilde{\phi}\|^2)
\notag\\
&&\leq C\frac{\varepsilon}{\delta^{1/2}}(\|\partial^{\alpha}\partial_x(\widetilde{u},\widetilde{\theta})\|^2+\|\langle v\rangle^{-\frac{1}{2}}\partial^{\alpha}\partial_xf\|^2)
+C\delta^{5/2}\varepsilon^3+C\delta^5\varepsilon.
	\end{eqnarray}
\end{lemma}
\begin{proof}	
After applying  $\partial^{\alpha}$ with  $|\alpha|\leq 1$ to the second equation of \eqref{2.34}, we take the inner product of the resulting equation with $\delta\partial^{\alpha}\partial_x\widetilde{\rho}$ to get	
\begin{eqnarray}
	\label{3.54}
	&&\delta(\partial^{\alpha}\partial_t\widetilde{u}_{1},\partial^{\alpha}\partial_x\widetilde{\rho})
	+(\frac{2\bar{\theta}}{3\bar{\rho}}\partial^{\alpha}\partial_x\widetilde{\rho},\partial^{\alpha}\partial_x\widetilde{\rho})+
	\sum_{1\leq \alpha_{1}\leq \alpha}C_{\alpha}^{\alpha_{1}}(\partial^{\alpha_1}(\frac{2\bar{\theta}}{3\bar{\rho}})\partial^{\alpha-\alpha_{1}}\partial_x\widetilde{\rho}
	,\partial^{\alpha}\partial_x\widetilde{\rho})
	\notag\\
	=&&-(\partial^{\alpha}[-A\partial_x\widetilde{u}_{1}
	+u_{1}\partial_x\widetilde{u}_{1}+\widetilde{u}_{1}\partial_x\bar{u}_{1}+\frac{2}{3}\partial_x\widetilde{\theta}
	+\frac{2}{3}(\frac{\theta}{\rho}-\frac{\bar{\theta}}{\bar{\rho}})\partial_x\rho],\partial^{\alpha}\partial_x\widetilde{\rho})
	\notag\\
	&&-(\partial^{\alpha}\partial_x\widetilde{\phi},\partial^{\alpha}\partial_x\widetilde{\rho})
	-(\partial^{\alpha}[\frac{1}{\rho}\int_{\mathbb{R}^3} v^2_{1}\partial_xG\,dv],\partial^{\alpha}\partial_x\widetilde{\rho})-(\delta^3\partial^{\alpha}[\frac{1}{\bar{\rho}}\mathcal{R}_2],\partial^{\alpha}\partial_x\widetilde{\rho}).
\end{eqnarray}
Here if $|\alpha|=0$, the last term on the Left hand side of \eqref{3.54} vanishes and if  $|\alpha|=1$, this term exists.
By integration by parts, and using \eqref{3.4}, \eqref{2.31},
\eqref{1.33} and  \eqref{3.1}, we arrive at
	\begin{eqnarray*}
	\delta(\partial^{\alpha}\partial_t\widetilde{u}_{1},\partial^{\alpha}\partial_x\widetilde{\rho})
		=&&\delta\frac{d}{dt}(\partial^{\alpha}\widetilde{u}_1,\partial_x\partial^{\alpha}\widetilde{\rho})+\delta(\partial^{\alpha}\partial_x\widetilde{u},
		\partial^{\alpha}\partial_{t}\widetilde{\rho})
		\notag\\
		=&&	\delta\frac{d}{dt}(\partial^{\alpha}\widetilde{u}_1,\partial_x\partial^{\alpha}\widetilde{\rho})
		-(\partial^{\alpha}\partial_x\widetilde{u}_1,\partial^{\alpha}[-A\partial_x\widetilde{\rho}+\partial_x(\widetilde{\rho}u_{1})+\partial_x(\bar{\rho}\widetilde{u}_{1})+\delta^3\mathcal{R}_1])
		\notag\\
		\geq&&\delta\frac{d}{dt}(\partial^{\alpha}\widetilde{u}_1,\partial_x\partial^{\alpha}\widetilde{\rho})-C\eta\|\partial^{\alpha}\partial_x\widetilde{\rho}\|^2
		-C_\eta\|\partial^{\alpha}\partial_x\widetilde{u}_1\|^2-\widetilde{C}\delta^6.
	\end{eqnarray*}
In the second equality we have used the first equation of \eqref{2.34}. Applying  $\partial^{\alpha}\partial_x$ to the last equation of \eqref{2.34}, one gets
\begin{equation}
	\label{3.71A}
-(\partial^{\alpha}\partial_x\widetilde{\phi},\partial^{\alpha}\partial_x\widetilde{\rho})
=(\partial^{\alpha}\partial_x\widetilde{\phi},\delta\partial^{\alpha}\partial^3_x\widetilde{\phi}
-\partial^{\alpha}[e^{\phi}\partial_x\widetilde{\phi}]
-\partial^{\alpha}[(e^{\phi}-e^{\bar{\phi}})\partial_x\bar{\phi}]+\delta^3\partial^{\alpha}\partial_x\mathcal{R}_{4}).
\end{equation}
In the following, we estimate \eqref{3.71A}. By \eqref{3.1}, \eqref{3.2} and \eqref{3.4}, one has
\begin{eqnarray*}
	-(\partial^{\alpha}\partial_x\widetilde{\phi},\partial^{\alpha}[e^{\phi}\partial_x\widetilde{\phi}])
	&&=-(\partial^{\alpha}\partial_x\widetilde{\phi},
	e^{\phi}\partial^{\alpha}\partial_x\widetilde{\phi})-\sum_{1\leq \alpha_{1}\leq \alpha}C_{\alpha}^{\alpha_{1}}(\partial^{\alpha}\partial_x\widetilde{\phi},
	e^{\phi}\partial^{\alpha_{1}}\phi\partial^{\alpha-\alpha_{1}}\partial_x\widetilde{\phi})
	\notag\\
	&&\leq
	-(\partial^{\alpha}\partial_x\widetilde{\phi},e^{\phi}\partial^{\alpha}\partial_x\widetilde{\phi})
	+C\eta\|\partial^{\alpha}\partial_x\widetilde{\phi}\|^2
	+C_\eta\widetilde{C}\delta^6.
\end{eqnarray*}
The other terms in \eqref{3.71A} can be estimated easily, hence we get
\begin{eqnarray*}
-(\partial^{\alpha}\partial_x\widetilde{\phi},\partial^{\alpha}\partial_x\widetilde{\rho})\leq -\delta(\partial^{\alpha}\partial^2_x\widetilde{\phi},\partial^{\alpha}\partial^2_x\widetilde{\phi})
	-(\partial^{\alpha}\partial_x\widetilde{\phi},e^{\phi}\partial^{\alpha}\partial_x\widetilde{\phi})
	+C\eta\|\partial^{\alpha}\partial_x\widetilde{\phi}\|^2
	+C_\eta\widetilde{C}\delta^6.
\end{eqnarray*}
Note that the other terms in \eqref{3.54} can be estimated by \eqref{3.4},
\eqref{2.31}, \eqref{3.7}, \eqref{1.34},  \eqref{3.1} and \eqref{3.2a}, then we claim that
	$$
C(\delta^2+\eta)\|\partial^{\alpha}\partial_x\widetilde{\rho}\|^2+
	C_\eta\{\|\partial^{\alpha}\partial_x(\widetilde{u},\widetilde{\theta})\|^2+\|\langle v\rangle^{-\frac{1}{2}}\partial^{\alpha}\partial_xf\|^2\}
	+C_\eta \delta^3\varepsilon^2+C_\eta\widetilde{C}\delta^6.
	$$
Hence, for $|\alpha|\leq1$, substituting all the above estimates into \eqref{3.54} and choosing  a small $\eta>0$, one gets
	\begin{eqnarray*}	
		&&\delta\frac{d}{dt}(\partial^{\alpha}\widetilde{u}_1,\partial_x\partial^{\alpha}\widetilde{\rho})
		+c(\|\partial^{\alpha}\partial_x\widetilde{\rho}\|^2+\|\partial^{\alpha}\partial_x\widetilde{\phi}\|^2
		+\delta\|\partial^{\alpha}\partial^2_x\widetilde{\phi}\|^2)
		\notag\\
		&&\leq C(\|\partial^{\alpha}\partial_x(\widetilde{u},\widetilde{\theta})\|^2+\|\langle v\rangle^{-\frac{1}{2}}\partial^{\alpha}\partial_xf\|^2)
		+C\delta^3\varepsilon^2+ \widetilde{C}\delta^6.
	\end{eqnarray*}
Multiplying this by  $\frac{\varepsilon}{\delta^{1/2}}$ and using \eqref{3.2a},
we can prove \eqref{3.53} holds.  Thus, Lemma \ref{lem3.13} is proved. 
\end{proof}
\subsection{Estimates on nonlinear terms with $\partial_{x}\phi$}\label{subs3.5}
We are now ready to estimate the nonlinear terms involving with the electric fields.
\subsubsection{Estimates without weight functions}\label{subs3.5.1}
We first treat such terms without the weighted function.
\begin{lemma}\label{lem3.14}
 For $|\alpha|=2$, one has
	\begin{multline}
		\label{3.55}
		\frac{1}{\delta}	(\frac{\partial^{\alpha}(\partial_x\phi\partial_{v_{1}}F)}{\sqrt{\mu}},\frac{\partial^{\alpha}F}{\sqrt{\mu}})
		\leq -\frac{1}{2}\frac{d}{dt}(\frac{1}{K\theta}e^{\phi}\partial^{\alpha}\widetilde{\phi},\partial^{\alpha}\widetilde{\phi})
		-\frac{1}{2}\delta\frac{d}{dt}(\frac{1}{K\theta}\partial^{\alpha}\partial_x\widetilde{\phi},\partial^{\alpha}\partial_x\widetilde{\phi})
	\\
	+C(\frac{1}{\varepsilon}+\frac{\delta^{3/2}}{\varepsilon^2})\mathcal{D}_{2,l,q_1}(t)+C\frac{1}{\delta^2}\mathcal{E}_{2}(t)
	+C\delta^{\frac{1}{2}}\varepsilon^2+C\delta^2+C\varepsilon^{1/2}\delta^{1/2}\|\langle v\rangle \partial^{\alpha}f\|^2.
\end{multline}	
\end{lemma}
\begin{proof}
Let $|\alpha|=2$, we write
\begin{equation}
	\label{3.56}
	\frac{1}{\delta}	(\frac{\partial^{\alpha}(\partial_x\phi\partial_{v_{1}}F)}{\sqrt{\mu}},\frac{\partial^{\alpha}F}{\sqrt{\mu}})
	=	\frac{1}{\delta}(\frac{\partial_x\phi\partial_{v_{1}}\partial^{\alpha}F}{\sqrt{\mu}},\frac{\partial^{\alpha}F}{\sqrt{\mu}})
	+	\frac{1}{\delta}\sum_{1\leq\alpha_{1}\leq \alpha}C^{\alpha_1}_\alpha(\frac{\partial^{\alpha_1}\partial_x\phi\partial_{v_{1}}\partial^{\alpha-\alpha_{1}}F}{\sqrt{\mu}},
	\frac{\partial^{\alpha}F}{\sqrt{\mu}}).
\end{equation}
By integration by parts and \eqref{1.26}, we see that
\begin{equation*}
		\frac{1}{\delta}|(\frac{\partial_x\phi\partial_{v_{1}}\partial^{\alpha}F}{\sqrt{\mu}},\frac{\partial^{\alpha}F}{\sqrt{\mu}})|
	=\frac{1}{\delta}|\int_{\mathbb R}\int_{{\mathbb R}^3}\partial_x\phi v_1\frac{(\partial^{\alpha}F)^2}{2\mu}\,dv\,dx|.
\end{equation*}
Since $F=M+\overline{G}+\sqrt{\mu}f$, we first use \eqref{3.29}, \eqref{3.7}, \eqref{3.57} and \eqref{1.36} to get
\begin{eqnarray*}
&&\frac{1}{\delta}|\int_{\mathbb R}\int_{{\mathbb R}^3}\partial_x\phi v_1\frac{|\partial^{\alpha}(M+\overline{G})|^2}{2\mu}\,dv\,dx|	
\notag\\	
&&\leq C\frac{1}{\delta}\|\partial_x\phi\|_{L^{\infty}}
(\|\langle v\rangle^{\frac{1}{2}}\frac{\partial^{\alpha}M}{\sqrt{\mu}}\|^2+\|\langle v\rangle^{\frac{1}{2}}\frac{\partial^{\alpha}\overline{G}}{\sqrt{\mu}}\|^2)
\notag\\
&&\leq C\frac{\delta^{1/2}}{\varepsilon}\mathcal{D}_{2}(t)+C\delta^2.
\end{eqnarray*}
On the other hand, we use the H\"{o}lder inequality, \eqref{3.57} and \eqref{1.36} to get
\begin{eqnarray}
	\label{3.58}
&&\frac{1}{\delta}|\int_{\mathbb R}\int_{{\mathbb R}^3}\partial_x\phi v_1(\partial^{\alpha}f)^2\,dv\,dx|
\notag\\
&&\leq C\frac{1}{\delta}\int_{\mathbb R}|\partial_x\phi|(\int_{{\mathbb R}^3}\langle v\rangle^{-1} |\partial^{\alpha}f|^2\,dv)^{1/3}
(\int_{{\mathbb R}^3}\langle v\rangle^{2} |\partial^{\alpha}f|^2\,dv)^{2/3}\,dx
\notag\\	
&&\leq C\frac{1}{\delta}\int_{\mathbb R}\{\frac{1}{\varepsilon}\int_{{\mathbb R}^3}\langle v\rangle^{-1} |\partial^{\alpha}f|^2\,dv+
\varepsilon^{1/2}|\partial_x\phi|^{3/2}\int_{{\mathbb R}^3}\langle v\rangle^{2} |\partial^{\alpha}f|^2\,dv\}\,dx
\notag\\
&&\leq C\frac{1}{\delta\varepsilon}\|\partial^{\alpha}f\|_\sigma^2
+C\varepsilon^{1/2}\frac{1}{\delta}\|\partial_x\phi\|^{3/2}_{L^{\infty}}\|\langle v\rangle \partial^{\alpha}f\|^2
\notag\\
&&\leq C\frac{\delta^{3/2}}{\varepsilon^{2}}\mathcal{D}_{2}(t)
+C\varepsilon^{1/2}\delta^{1/2}\|\langle v\rangle \partial^{\alpha}f\|^2.
\end{eqnarray}
Therefore, we have from the above three estimates that 
\begin{equation}
	\label{3.59}
	\frac{1}{\delta}|(\frac{\partial_x\phi\partial_{v_{1}}\partial^{\alpha}F}{\sqrt{\mu}},\frac{\partial^{\alpha}F}{\sqrt{\mu}})|\leq
	C(\frac{\delta^{1/2}}{\varepsilon}+\frac{\delta^{3/2}}{\varepsilon^{2}})\mathcal{D}_{2}(t)+C\varepsilon^{1/2}\delta^{1/2}\|\langle v\rangle \partial^{\alpha}f\|^2+C\delta^2.
\end{equation}
For $1\leq |\alpha_{1}|\leq |\alpha|$, we employ $F=M+\overline{G}+\sqrt{\mu}f$ again to  decompose
\begin{eqnarray}
	\label{3.60}
	&&\frac{1}{\delta}(\frac{\partial^{\alpha_1}\partial_x\phi\partial_{v_{1}}\partial^{\alpha-\alpha_{1}}F}{\sqrt{\mu}},\frac{\partial^{\alpha}F}{\sqrt{\mu}})
	\notag\\
	=&&\frac{1}{\delta}(\frac{\partial^{\alpha_1}\partial_x\phi\partial_{v_{1}}\partial^{\alpha-\alpha_{1}}M}{\sqrt{\mu}},\frac{\partial^{\alpha}M}{\sqrt{\mu}})
	+\frac{1}{\delta}(\frac{\partial^{\alpha_1}\partial_x\phi\partial_{v_{1}}\partial^{\alpha-\alpha_{1}}M}{\sqrt{\mu}},\frac{\partial^{\alpha}(\overline{G}+\sqrt{\mu}f)}{\sqrt{\mu}})
	\notag\\
	&&+\frac{1}{\delta}(\frac{\partial^{\alpha_1}\partial_x\phi\partial_{v_{1}}\partial^{\alpha-\alpha_{1}}\overline{G}}{\sqrt{\mu}},\frac{\partial^{\alpha}F}{\sqrt{\mu}})
	+\frac{1}{\delta}(\frac{\partial^{\alpha_1}\partial_x\phi\partial_{v_{1}}\partial^{\alpha-\alpha_{1}}(\sqrt{\mu}f)}{\sqrt{\mu}},\frac{\partial^{\alpha}F}{\sqrt{\mu}}).
\end{eqnarray}
In the following, we compute \eqref{3.60} term by term. The calculations for the first term of \eqref{3.60} can be divided into two cases.
One is the case  $|\alpha_{1}|=1$, and the other is the case $|\alpha_{1}|=|\alpha|=2$. For the latter one we split it into two parts as follows
\begin{eqnarray}
	\label{3.61}
	&&\frac{1}{\delta}(\frac{\partial^{\alpha_1}\partial_x\phi\partial_{v_{1}}\partial^{\alpha-\alpha_{1}}M}{\sqrt{\mu}},\frac{\partial^{\alpha}M}{\sqrt{\mu}})
	\notag\\
	&&=\frac{1}{\delta}(\partial^{\alpha}\partial_x\phi\partial_{v_{1}}M,[\frac{1}{\mu}-\frac{1}{M}]\partial^{\alpha}M)
	+\frac{1}{\delta}(\partial^{\alpha}\partial_x\phi\partial_{v_{1}}M,\frac{1}{M}\partial^{\alpha}M).
\end{eqnarray}
We first compute the second term on the right hand side of \eqref{3.61}, since the first term is more easier and 
is thus left to the end. In view of \eqref{1.17} and \eqref{1.18}, we have from a direct computation that
\begin{eqnarray}
	\label{3.62}
	\frac{1}{\delta}(\partial^{\alpha}\partial_x\phi\partial_{v_{1}}M,\frac{1}{M}\partial^{\alpha}M)&&=
	-\frac{1}{\delta}(\partial^{\alpha}\partial_x\phi\frac{v_1-u_1}{K\theta},\partial^{\alpha}M)
	\notag\\
	&&=-\frac{1}{\delta}(\frac{1}{K\theta}\partial^{\alpha}\partial_x\phi,\partial^{\alpha}(\rho u_1))
	+\frac{1}{\delta}(\frac{1}{K\theta}\partial^{\alpha}\partial_x\phi,u_1\partial^{\alpha}\rho)
	\notag\\
	&&:=I_3+I_4.
\end{eqnarray}
To bound the term $I_4$, we use the integration by parts, \eqref{3.1}, \eqref{3.4} and \eqref{1.36} to obtain
\begin{eqnarray}
	\label{3.79A}
	I_4&&=\frac{1}{\delta}(\partial^{\alpha}\partial_x\phi,\frac{1}{K\theta}u_1\partial^{\alpha}\bar{\rho})
	+\frac{1}{\delta}(\frac{1}{K\theta}\partial^{\alpha}\partial_x\phi,u_1\partial^{\alpha}\widetilde{\rho})
	\notag\\
	&&\leq -\frac{1}{\delta}(\partial^{\alpha}\phi,\partial_x[\frac{1}{K\theta}u_1\partial^{\alpha}\bar{\rho}])
	+C\frac{1}{\delta}\|u_1\|_{L^\infty}\|\partial^{\alpha}\widetilde{\rho}\|\|\partial^{\alpha}\partial_x\phi\|
	\notag\\
	&&\leq C(\|\partial_x(u_1,\theta)\|^2+\|\partial^{\alpha}\phi\|^2+\|\partial_x\partial^{\alpha}\bar{\rho}\|^2)
	+C\frac{1}{\delta^{1/2}}(
	\|\partial^{\alpha}\widetilde{\rho}\|^2+\delta\|\partial^{\alpha}\partial_x\phi\|^2)
	\notag\\
	&&\leq C(\frac{\delta^{1/2}}{\varepsilon}+\frac{1}{\varepsilon})\mathcal{D}_{2}(t)+C\delta^2.
\end{eqnarray}
For the term $I_3$, we use the integration by parts to write
\begin{eqnarray}
	\label{3.63}	
	I_3=&&\frac{1}{\delta}(\partial_x[\frac{1}{K\theta}]\partial^{\alpha}\phi,\partial^{\alpha}(\rho u_1))+
	\frac{1}{\delta}(\frac{1}{K\theta}\partial^{\alpha}\phi,\partial^{\alpha}\partial_x(\rho u_1))
	\notag\\
	=&&\frac{1}{\delta}(\partial_x[\frac{1}{K\theta}]\partial^{\alpha}\phi,\partial^{\alpha}(\rho u_1))
	+A\frac{1}{\delta}(\frac{1}{K\theta}\partial^{\alpha}\phi,\partial^{\alpha}\partial_x\rho)-(\frac{1}{K\theta}\partial^{\alpha}\phi,\partial^{\alpha}\partial_t\rho),
\end{eqnarray}
where in the last equality we have used the first equation of \eqref{2.13}.
For the second term on the right hand side of \eqref{3.63}, we use the second equation of \eqref{1.15} to write
\begin{equation}
\label{3.64}
A\frac{1}{\delta}(\frac{1}{K\theta}\partial^{\alpha}\phi,\partial^{\alpha}\partial_x\rho)=
A\frac{1}{\delta}(\frac{1}{K\theta}\partial^{\alpha}\phi,\partial^{\alpha}(e^{\phi}\partial_x\phi))
-A\frac{1}{\delta}(\frac{1}{K\theta}\partial^{\alpha}\phi,\delta\partial^{\alpha}\partial^{3}_{x}\phi).
\end{equation}
Note from \eqref{3.3}, \eqref{3.4} and \eqref{3.2a} that for $|\alpha|=2$,
\begin{equation}
	\label{3.79a}
	\|\partial^{\alpha}\phi\|^2\leq C\|\partial^{\alpha}\bar{\phi}\|^2+C\|\partial^{\alpha}\widetilde{\phi}\|^2
	\leq C\delta^2+\widetilde{C}\delta^3\leq C\delta^2,
\end{equation}
which together with the integration by parts and \eqref{1.36} gives
	\begin{eqnarray*}
		A\frac{1}{\delta}(\frac{1}{K\theta}\partial^{\alpha}\phi,e^{\phi}\partial^{\alpha}\partial_x\phi)&&=
		-\frac{1}{2}A\frac{1}{\delta}(\partial_x[\frac{1}{K\theta}e^{\phi}]\partial^{\alpha}\phi,\partial^{\alpha}\phi)
		\notag\\
		&&\leq C\frac{1}{\delta}\|\partial_x(\theta,\phi)\|_{L^\infty}\|\partial^{\alpha}\phi\|^2
		\leq C\frac{\delta^{1/2}}{\varepsilon}\mathcal{D}_{2}(t)+C\delta^2.
\end{eqnarray*}
Hence, we obtain that the first term on the right hand side of \eqref{3.64} is bounded by
\begin{eqnarray}
	\label{3.80a}	 
	A\frac{1}{\delta}(\frac{1}{K\theta}\partial^{\alpha}\phi,\partial^{\alpha}(e^{\phi}\partial_x\phi))
	&&=A\frac{1}{\delta}(\frac{1}{K\theta}\partial^{\alpha}\phi,e^{\phi}\partial^{\alpha}\partial_x\phi)
	+A\frac{1}{\delta}\sum_{1\leq\alpha_{1}\leq\alpha}
	C^{\alpha_{1}}_\alpha(\frac{1}{K\theta}\partial^{\alpha}\phi,\partial^{\alpha_{1}}[e^{\phi}]\partial^{\alpha-\alpha_{1}}\partial_x\phi)
	\notag\\
	&&\leq C\frac{\delta^{1/2}}{\varepsilon}\mathcal{D}_{2}(t)+C\delta^2.
\end{eqnarray}
For the second term on the right hand side of \eqref{3.64}, we have
\begin{eqnarray*}
-A(\frac{1}{K\theta}\partial^{\alpha}\phi,\partial^{\alpha}\partial^{3}_{x}\phi)
&&=A(\frac{1}{K\theta}\partial^{\alpha}\partial_x\phi,\partial^{\alpha}\partial^{2}_{x}\phi)
+A(\partial_x[\frac{1}{K\theta}]\partial^{\alpha}\phi,\partial^{\alpha}\partial^{2}_{x}\phi)
\notag\\
&&=-\frac{3}{2}A(\partial_{x}[\frac{1}{K\theta}]\partial^{\alpha}\partial_x\phi,\partial^{\alpha}\partial_{x}\phi)
-A(\partial^2_x[\frac{1}{K\theta}]\partial^{\alpha}\phi,\partial^{\alpha}\partial_{x}\phi).
\end{eqnarray*}
By \eqref{3.2}, \eqref{3.4}, \eqref{1.36} and $\widetilde{C}\delta^{1/4}<1$ in \eqref{3.2a}, we show that
\begin{eqnarray*}
|(\partial_{x}[\frac{1}{K\theta}]\partial^{\alpha}\partial_x\phi,\partial^{\alpha}\partial_{x}\phi)|
&&\leq C \|\partial_x\theta\|_{L^\infty}\|\partial^{\alpha}\partial_{x}\bar{\phi}\|^2
+C \|\partial_x\theta\|_{L^\infty}\|\partial^{\alpha}\partial_{x}\widetilde{\phi}\|^2
\notag\\
&&\leq C\delta^2\|\partial_x\theta\|_{L^\infty}+
(C\delta+\widetilde{C}\delta\frac{\delta^{5/4}}{\varepsilon^{1/2}})\|\partial^{\alpha}\partial_{x}\widetilde{\phi}\|^2
\notag\\
&&\leq C\delta^3+C\delta\|\partial_x\widetilde{\theta}\|^2_{L^\infty}+
C(1+\frac{\delta}{\varepsilon^{1/2}})\delta\|\partial^{\alpha}\partial_{x}\widetilde{\phi}\|^2
\notag
\\
&&\leq C(1+\frac{\delta}{\varepsilon^{1/2}})\frac{\delta^{1/2}}{\varepsilon}\mathcal{D}_{2}(t)+C\delta^3.
\end{eqnarray*}
We get clearly from the embedding inequality, \eqref{3.3} and \eqref{3.2a} that
\begin{equation}
	\label{3.65}
\|\partial^2_x\widetilde{\phi}\|_{L^\infty}\leq 2\|\partial^2_x\widetilde{\phi}\|^{1/2} \|\partial^3_x\widetilde{\phi}\|^{1/2}
\leq \widetilde{C}\delta^{3/4}\frac{\delta}{\varepsilon^{1/2}}\leq C\delta^{1/2}\frac{\delta}{\varepsilon^{1/2}},
\end{equation}
which together with \eqref{1.36} gives 
\begin{equation*}
|(\partial^2_x\widetilde{\theta}\partial^{\alpha}\widetilde{\phi},\partial^{\alpha}\partial_{x}\widetilde{\phi})|
\leq \|\partial^{\alpha}\widetilde{\phi}\|_{L^\infty}\|\partial^2_x\widetilde{\theta}\|\|\partial^{\alpha}\partial_{x}\widetilde{\phi}\|
\leq C\frac{\delta}{\varepsilon^{1/2}}\frac{\delta^{1/2}}{\varepsilon}\mathcal{D}_{2}(t).
\end{equation*}
By this, \eqref{3.2}, \eqref{3.4}, \eqref{3.1} and \eqref{1.36}, we get
\begin{eqnarray*}
&&|A(\partial^2_x[\frac{1}{K\theta}]\partial^{\alpha}\phi,\partial^{\alpha}\partial_{x}\phi)|	
\notag\\
&&\leq C\|\partial^{\alpha}\phi\|_{L^\infty}(\|\partial^2_x\theta\|\|\partial^{\alpha}\partial_{x}\phi\|
+\|\partial_x\theta\|_{L^\infty}\|\partial_x\theta\|\|\partial^{\alpha}\partial_{x}\phi\|)
\notag\\
&&\leq C(1+\frac{\delta}{\varepsilon^{1/2}})\frac{\delta^{1/2}}{\varepsilon}\mathcal{D}_{2}(t)
+C\delta^2.
\end{eqnarray*}
It follows from the above estimates that
\begin{equation*}
-A(\frac{1}{K\theta}\partial^{\alpha}\phi,\partial^{\alpha}\partial^{3}_{x}\phi)
\leq C(\frac{\delta^{1/2}}{\varepsilon}+\frac{\delta^{3/2}}{\varepsilon^{3/2}})\mathcal{D}_{2}(t)+C\delta^2.		
\end{equation*}
So, plugging this and \eqref{3.80a} into \eqref{3.64} leads to
\begin{equation}
\label{3.66}
	A\frac{1}{\delta}(\frac{1}{K\theta}\partial^{\alpha}\phi,\partial^{\alpha}\partial_x\rho)
\leq C(\frac{\delta^{1/2}}{\varepsilon}+\frac{\delta^{3/2}}{\varepsilon^{3/2}})\mathcal{D}_{2}(t)+C\delta^2.	
\end{equation}
For the first term on the right hand side of \eqref{3.63}, we use the similar calculation as \eqref{3.79A} to obtain
\begin{equation}
	\label{3.78a}
	\frac{1}{\delta}|(\partial_x[\frac{1}{K\theta}]\partial^{\alpha}\phi,\partial^{\alpha}(\rho u_1))|
	\leq C(\frac{\delta^{1/2}}{\varepsilon}+\frac{1}{\varepsilon})\mathcal{D}_{2}(t)+C\delta^2.
\end{equation}
For the last term in \eqref{3.63},  similar to \eqref{3.64},  it can be written as
\begin{equation}
	\label{3.87B}
-(\frac{1}{K\theta}\partial^{\alpha}\phi,\partial^{\alpha}\partial_t\rho)=
-(\frac{1}{K\theta}\partial^{\alpha}\phi,\partial^{\alpha}(e^{\phi}\partial_t\phi))
+(\frac{1}{K\theta}\partial^{\alpha}\phi,\delta\partial^{\alpha}\partial^{2}_{x}\partial_t\phi).
\end{equation}
To bound the first term on the right hand side of \eqref{3.87B}, we split it into four parts as follows
\begin{eqnarray}
	\label{3.87A}
	-(\frac{1}{K\theta}\partial^{\alpha}\phi,\partial^{\alpha}(e^{\phi}\partial_t\phi))
	=&&-(\frac{1}{K\theta}\partial^{\alpha}\phi,\partial^{\alpha}(e^{\phi}\partial_t\bar{\phi}))
	-(\frac{1}{K\theta}\partial^{\alpha}\bar{\phi},\partial^{\alpha}(e^{\phi}\partial_t\widetilde{\phi}))
	\notag\\
	&&-(\frac{1}{K\theta}\partial^{\alpha}\widetilde{\phi},e^{\phi}\partial^{\alpha}\partial_t\widetilde{\phi})
	-\sum_{1\leq\alpha_{1}\leq\alpha}C^{\alpha_{1}}_{\alpha}(\frac{1}{K\theta}\partial^{\alpha}\widetilde{\phi},\partial^{\alpha_1}
	(e^{\phi})\partial^{\alpha-\alpha_{1}}\partial_t\widetilde{\phi}).
\end{eqnarray}
In the following, we estimate \eqref{3.87A} term by term.
Using \eqref{3.79a}, \eqref{3.1} and \eqref{3.4}, the first term of \eqref{3.87A} can be estimated as follows
$$
|(\frac{1}{K\theta}\partial^{\alpha}\phi,\partial^{\alpha}(e^{\phi}\partial_t\bar{\phi}))|
\leq C\|\partial^{\alpha}\phi\|^2+C\|\partial^{\alpha}(e^{\phi}\partial_t\bar{\phi})\|^2
\leq C\delta^2.		
$$
For the second term of \eqref{3.87A}, by integration by parts, \eqref{3.4}, \eqref{3.48}, \eqref{1.34} and \eqref{1.36}, we get
\begin{eqnarray*}
	|(\frac{1}{K\theta}\partial^{\alpha}\bar{\phi},\partial^{\alpha}(e^{\phi}\partial_t\widetilde{\phi}))|
	&&=|(\partial^{\alpha}[\frac{1}{K\theta}\partial^{\alpha}\bar{\phi}],e^{\phi}\partial_t\widetilde{\phi})|
	\notag\\
	&&\leq C\|\partial^{\alpha}[\frac{1}{K\theta}\partial^{\alpha}\bar{\phi}]\|^2+C\|\partial_t\widetilde{\phi}\|^2
	\notag\\
	&&\leq  C\frac{\delta^{1/2}}{\varepsilon}\mathcal{D}_{2}(t)+C\frac{1}{\delta^2}\mathcal{E}_{2}(t)+C\delta^2.
\end{eqnarray*}
For the third term of \eqref{3.87A}, we use \eqref{3.45}, \eqref{3.48}, \eqref{3.4} and \eqref{1.36} to obtain
\begin{eqnarray*}
-(\frac{1}{K\theta}\partial^{\alpha}\widetilde{\phi},e^{\phi}\partial^{\alpha}\partial_t\widetilde{\phi})
&&=-\frac{1}{2}\frac{d}{dt}(\frac{1}{K\theta}e^{\phi}\partial^{\alpha}\widetilde{\phi},\partial^{\alpha}\widetilde{\phi})	
+\frac{1}{2}(\partial_t[\frac{1}{K\theta}e^{\phi}]\partial^{\alpha}\widetilde{\phi},\partial^{\alpha}\widetilde{\phi})
\notag\\	
&&
\leq -\frac{1}{2}\frac{d}{dt}(\frac{1}{K\theta}e^{\phi}\partial^{\alpha}\widetilde{\phi},\partial^{\alpha}\widetilde{\phi})
+C\delta\|\partial_t(\phi,\theta)\|^2+C\frac{1}{\delta}\|\partial^{\alpha}\widetilde{\phi}\|^2_{L^\infty}\|\partial^{\alpha}\widetilde{\phi}\|^2
\notag\\	
&&
\leq -\frac{1}{2}\frac{d}{dt}(\frac{1}{K\theta}e^{\phi}\partial^{\alpha}\widetilde{\phi},\partial^{\alpha}\widetilde{\phi})
+C\frac{1}{\delta}\mathcal{E}_{2}(t)+C\delta^2\varepsilon^2+C\delta^3
+ C\frac{\delta^{1/2}}{\varepsilon}\mathcal{D}_{2}(t).
\end{eqnarray*}
Similarly, the last term in \eqref{3.87A} can be controlled by $C\frac{\delta^{1/2}}{\varepsilon}\mathcal{D}_{2}(t)+C\frac{1}{\delta^2}\mathcal{E}_{2}(t)+C\delta^2$.
So, plugging the above estimates into \eqref{3.87A} gives
\begin{equation}
	\label{3.69}
	-(\frac{1}{K\theta}\partial^{\alpha}\phi,\partial^{\alpha}(e^{\phi}\partial_t\phi))
	\leq -\frac{1}{2}\frac{d}{dt}(\frac{1}{K\theta}e^{\phi}\partial^{\alpha}\widetilde{\phi},\partial^{\alpha}\widetilde{\phi})
	+C\frac{\delta^{1/2}}{\varepsilon}\mathcal{D}_{2}(t)
	+C\frac{1}{\delta^2}\mathcal{E}_{2}(t)+C\delta^2.
\end{equation}
To bound the second term on the right hand side of \eqref{3.87B}, we first use the integration by parts, \eqref{3.4}, \eqref{3.48}, \eqref{3.1}, \eqref{3.2a} and \eqref{1.36} to obtain
\begin{eqnarray*}
\delta(\frac{1}{K\theta}\partial^{\alpha}\bar{\phi},\partial^{\alpha}\partial^{2}_{x}\partial_t\widetilde{\phi})
&&=\delta(\partial_{x}[\partial_{x}(\frac{1}{K\theta})\partial^{\alpha}\bar{\phi}],\partial^{\alpha}\partial_t\widetilde{\phi})
-\delta(\partial^{\alpha}[\frac{1}{K\theta}\partial^{\alpha}\partial_{x}\bar{\phi}],\partial_{x}\partial_t\widetilde{\phi})
\notag\\
&&\leq C\frac{1}{\delta}\|\partial_{x}[\partial_{x}(\frac{1}{K\theta})\partial^{\alpha}\bar{\phi}]\|^2
+\|\partial^{\alpha}[\frac{1}{K\theta}\partial^{\alpha}\partial_{x}\bar{\phi}]\|^2
+C\delta^2(\|\partial_x\partial_t\widetilde{\phi}\|^2+\delta\|\partial^{\alpha}\partial_t\widetilde{\phi}\|^2)
\notag\\
&&\leq  C\frac{\delta^{1/2}}{\varepsilon}\mathcal{D}_{2}(t)+C\delta^2.
\end{eqnarray*}
Note from \eqref{3.4}, \eqref{3.45}, \eqref{3.46}, \eqref{3.48}, \eqref{3.1} and \eqref{3.2a} that
\begin{eqnarray}
	\label{3.67}
	\|\partial_t(\rho,u,\theta,\phi)\|_{L^\infty}&&\leq C\|\partial_t(\bar{\rho},\bar{u},\bar{\theta},\bar{\phi})\|_{L^\infty}
	+C\|\partial_t(\widetilde{\rho},\widetilde{u},\widetilde{\theta},\widetilde{\phi})\|_{L^\infty}
	\notag\\
	&&\leq C\delta+C\|\partial_t(\widetilde{\rho},\widetilde{u},\widetilde{\theta},\widetilde{\phi})\|^{\frac{1}{2}}
	\|\partial_x\partial_t(\widetilde{\rho},\widetilde{u},\widetilde{\theta},\widetilde{\phi})\|^{\frac{1}{2}}
	\notag\\
	&&\leq C\delta+\widetilde{C}(\delta+\delta^{1/2}\varepsilon)^{\frac{1}{2}}(\frac{1}{\varepsilon}\delta^{3/2}
	+\delta^{1/2}\varepsilon)^{\frac{1}{2}}
	\notag\\
	&&\leq \widetilde{C}(\delta+\delta^{1/2}\varepsilon+\frac{1}{\varepsilon^{1/2}}\delta^{5/4}).
\end{eqnarray}
Using \eqref{3.67}, \eqref{3.4}, \eqref{3.1}, \eqref{3.65}, \eqref{3.48} and \eqref{1.36}, we claim that
\begin{eqnarray*}
	&&\delta(\frac{1}{K\theta}\partial^{\alpha}\widetilde{\phi},\partial^{\alpha}\partial^{2}_{x}\partial_t\widetilde{\phi})
	=-\delta(\frac{1}{K\theta}\partial^{\alpha}\partial_x\widetilde{\phi},\partial^{\alpha}\partial_x\partial_t\widetilde{\phi})
	-\delta(\partial_x[\frac{1}{K\theta}]\partial^{\alpha}\widetilde{\phi},\partial^{\alpha}\partial_x\partial_t\widetilde{\phi})
	\notag\\
	&&= -\frac{1}{2}\delta\frac{d}{dt}(\frac{1}{K\theta}\partial^{\alpha}\partial_x\widetilde{\phi},\partial^{\alpha}\partial_x\widetilde{\phi})
	+\frac{1}{2}\delta(\partial_t[\frac{1}{K\theta}]\partial^{\alpha}\partial_x\widetilde{\phi},\partial^{\alpha}\partial_x\widetilde{\phi})
	+\delta(\partial_x[\partial_x(\frac{1}{K\theta})\partial^{\alpha}\widetilde{\phi}],\partial^{\alpha}\partial_t\widetilde{\phi})
	\notag\\
	&&\leq -\frac{1}{2}\delta\frac{d}{dt}(\frac{1}{K\theta}\partial^{\alpha}\partial_x\widetilde{\phi},\partial^{\alpha}\partial_x\widetilde{\phi})
	+C\delta\|\partial_t\theta\|_{L^{\infty}}\|\partial^{\alpha}\partial_x\widetilde{\phi}\|^2
	\notag\\
	&&\hspace{0.5cm}+C\frac{1}{\delta}\|\partial_x[\partial_x(\frac{1}{K\theta})\partial^{\alpha}\widetilde{\phi}]\|^2
	+C\delta^3\|\partial^{\alpha}\partial_t\widetilde{\phi}\|^2
	\notag\\
	&&\leq -\frac{1}{2}\delta\frac{d}{dt}(\frac{1}{K\theta}\partial^{\alpha}\partial_x\widetilde{\phi},\partial^{\alpha}\partial_x\widetilde{\phi})
	+C(\frac{\delta^{1/2}}{\varepsilon}+\frac{\delta^{3/2}}{\varepsilon^2})\mathcal{D}_{2}(t)+C\delta^2.
\end{eqnarray*}
Summarizing the above estimates, we arrive at
\begin{eqnarray*}
\delta(\frac{1}{K\theta}\partial^{\alpha}\phi,\partial^{\alpha}\partial^{2}_{x}\partial_t\phi)
=&&\delta(\frac{1}{K\theta}\partial^{\alpha}\bar{\phi},\partial^{\alpha}\partial^{2}_{x}\partial_t\widetilde{\phi})
+\delta(\frac{1}{K\theta}\partial^{\alpha}\widetilde{\phi},\partial^{\alpha}\partial^{2}_{x}\partial_t\widetilde{\phi})
+\delta(\frac{1}{K\theta}\partial^{\alpha}\phi,\partial^{\alpha}\partial^{2}_{x}\partial_t\bar{\phi})
\notag\\
\leq&& -\frac{1}{2}\delta\frac{d}{dt}(\frac{1}{K\theta}\partial^{\alpha}\partial_x\widetilde{\phi},\partial^{\alpha}\partial_x\widetilde{\phi})
+C(\frac{\delta^{1/2}}{\varepsilon}+\frac{\delta^{3/2}}{\varepsilon^2})\mathcal{D}_{2}(t)+C\delta^2,
\end{eqnarray*}
which  and \eqref{3.69} as well as \eqref{3.87B} together gives rise to
\begin{eqnarray*}
	\label{3.70}
	-(\frac{1}{K\theta}\partial^{\alpha}\phi,\partial^{\alpha}\partial_t\rho)
	\leq&& -\frac{1}{2}\frac{d}{dt}(\frac{1}{K\theta}e^{\phi}\partial^{\alpha}\widetilde{\phi},\partial^{\alpha}\widetilde{\phi})
	-\frac{1}{2}\delta\frac{d}{dt}(\frac{1}{K\theta}\partial^{\alpha}\partial_x\widetilde{\phi},\partial^{\alpha}\partial_x\widetilde{\phi})
	\notag\\
	&&+C(\frac{\delta^{1/2}}{\varepsilon}+\frac{\delta^{3/2}}{\varepsilon^2})\mathcal{D}_{2}(t)+C\frac{1}{\delta^2}\mathcal{E}_{2}(t)+C\delta^2.
\end{eqnarray*}
In summary, plugging this, \eqref{3.66} and \eqref{3.78a} into \eqref{3.63} leads to
\begin{eqnarray*}	
	I_3
	\leq&& -\frac{1}{2}\frac{d}{dt}(\frac{1}{K\theta}e^{\phi}\partial^{\alpha}\widetilde{\phi},\partial^{\alpha}\widetilde{\phi})
	-\frac{1}{2}\delta\frac{d}{dt}(\frac{1}{K\theta}\partial^{\alpha}\partial_x\widetilde{\phi},\partial^{\alpha}\partial_x\widetilde{\phi})
	\notag\\
&&+C(\frac{1}{\varepsilon}+\frac{\delta^{3/2}}{\varepsilon^2})\mathcal{D}_{2}(t)+C\frac{1}{\delta^2}\mathcal{E}_{2}(t)+C\delta^2.	
\end{eqnarray*}
Combining the estimates on  $I_3$ and  $I_4$ together, we get from \eqref{3.62} that
\begin{eqnarray}
\label{3.71}
\frac{1}{\delta}(\partial^{\alpha}\partial_x\phi\partial_{v_{1}}M,\frac{1}{M}\partial^{\alpha}M)
\leq&& -\frac{1}{2}\frac{d}{dt}(\frac{1}{K\theta}e^{\phi}\partial^{\alpha}\widetilde{\phi},\partial^{\alpha}\widetilde{\phi})
-\frac{1}{2}\delta\frac{d}{dt}(\frac{1}{K\theta}\partial^{\alpha}\partial_x\widetilde{\phi},\partial^{\alpha}\partial_x\widetilde{\phi})
\notag\\
&&+C(\frac{1}{\varepsilon}+\frac{\delta^{3/2}}{\varepsilon^2})\mathcal{D}_{2}(t)+C\frac{1}{\delta^2}\mathcal{E}_{2}(t)+C\delta^2.
\end{eqnarray}
This finishes the estimate on the second term of \eqref{3.61}.

In the following, we are going to compute the first term on the right hand side of \eqref{3.61}.
Recall that $\partial^{\alpha}M=I_1+I_2$ with $I_1=I_1^1+I_1^2$ given by \eqref{3.28} and \eqref{3.39a} respectively.
By the expression of $I_1^1$, and using \eqref{3.24}, \eqref{3.11}, \eqref{3.4} and \eqref{1.36}, we get
\begin{multline*}
	\frac{1}{\delta}|(\partial^{\alpha}\partial_x\phi\partial_{v_{1}}M,[\frac{1}{\mu}-\frac{1}{M}]I_1^1)|	
	\leq C\|\partial^{\alpha}\partial_x\phi\|\|\partial^{\alpha}(\widetilde{\rho},\widetilde{u},\widetilde{\theta})\|
	\\
	\leq C\delta^{\frac{1}{2}}\|\partial^{\alpha}\partial_x\phi\|^2+C\delta^{-\frac{1}{2}}\|\partial^{\alpha}(\widetilde{\rho},\widetilde{u},\widetilde{\theta})\|^2
		\leq C\frac{1}{\varepsilon}\mathcal{D}_{2}(t)+C\delta^{5/2}.	
\end{multline*}
After the integration by parts, we use the expression of $I_1^2$,  \eqref{3.24}, \eqref{3.11}, \eqref{3.4} and \eqref{1.36} to obtain
\begin{eqnarray*}
	\frac{1}{\delta}|(\partial^{\alpha}\partial_x\phi\partial_{v_{1}}M,[\frac{1}{\mu}-\frac{1}{M}]I_1^2)|	
	&&=\frac{1}{\delta}|(\partial^{\alpha}\phi,\partial_x[\partial_{v_{1}}M(\frac{1}{\mu}-\frac{1}{M})I_1^2])|
	\notag\\
	&&\leq C\|\partial^{\alpha}\phi\|(\|\partial_x(\rho,u,\theta)\|+\|\partial^{\alpha}\partial_x(\bar{\rho},\bar{u},\bar{\theta})\|)
	\notag\\
	&&\leq C\frac{1}{\varepsilon}\mathcal{D}_{2}(t)+C\delta^2.	
\end{eqnarray*}
The above two estimates and the fact that $I_1=I_1^1+I_1^2$ yield
\begin{equation*}
	\frac{1}{\delta}|(\partial^{\alpha}\partial_x\phi\partial_{v_{1}}M,[\frac{1}{\mu}-\frac{1}{M}]I_1)|	
	\leq C\frac{1}{\varepsilon}\mathcal{D}_{2}(t)+C\delta^2.	
\end{equation*}
Similar estimate also holds for $\frac{1}{\delta}|(\partial^{\alpha}\partial_x\phi\partial_{v_{1}}M,[\frac{1}{\mu}-\frac{1}{M}]I_2)|$.
Hence, the first term on the right hand side of \eqref{3.61} are bounded by
\begin{equation*}
\frac{1}{\delta}|(\partial^{\alpha}\partial_x\phi\partial_{v_{1}}M,[\frac{1}{\mu}-\frac{1}{M}]\partial^{\alpha}M)|
\leq C\frac{1}{\varepsilon}\mathcal{D}_{2}(t)+C\delta^2.	
\end{equation*}
In summary, for the case $|\alpha_{1}|=|\alpha|=2$, we plug this and \eqref{3.71} into \eqref{3.61} to get
\begin{eqnarray}
	\label{3.72}
\frac{1}{\delta}(\frac{\partial^{\alpha_1}\partial_x\phi\partial_{v_1}\partial^{\alpha-\alpha_{1}}M}{\sqrt{\mu}},\frac{\partial^{\alpha}M}{\sqrt{\mu}})
\leq&& -\frac{1}{2}\frac{d}{dt}(\frac{1}{K\theta}e^{\phi}\partial^{\alpha}\widetilde{\phi},\partial^{\alpha}\widetilde{\phi})
-\frac{1}{2}\delta\frac{d}{dt}(\frac{1}{K\theta}\partial^{\alpha}\partial_x\widetilde{\phi},\partial^{\alpha}\partial_x\widetilde{\phi})
\notag\\
&&+C(\frac{1}{\varepsilon}+\frac{\delta^{3/2}}{\varepsilon^2})\mathcal{D}_{2}(t)+C\frac{1}{\delta^2}\mathcal{E}_{2}(t)+C\delta^2.
\end{eqnarray}
Next we still need to consider the case $|\alpha_{1}|=1$. It holds from \eqref{3.29}, \eqref{3.79a}, \eqref{3.4} and \eqref{1.36} that
\begin{eqnarray*}
	\frac{1}{\delta}(\frac{\partial^{\alpha_1}\partial_x\phi\partial_{v_1}\partial^{\alpha-\alpha_{1}}M}{\sqrt{\mu}},\frac{\partial^{\alpha}M}{\sqrt{\mu}})
	&&\leq C\frac{1}{\delta^2}\|\partial^{\alpha-\alpha_{1}}(\rho,u,\theta)\|^2_{L^\infty}\|\partial^{\alpha_1}\partial_x\phi\|^2+C\|\frac{\partial^{\alpha}M}{\sqrt{\mu}}\|^2
	\notag\\
	&&\leq C\frac{\delta^{1/2}}{\varepsilon}\mathcal{D}_{2}(t)+C\delta^2.
\end{eqnarray*}
As a consequence, we combine this and \eqref{3.72} to get the estimate on the first term of  \eqref{3.60} as follows
\begin{eqnarray}
\label{3.73}
&&\frac{1}{\delta}\sum_{1\leq\alpha_{1}\leq\alpha}C^{\alpha_1}_\alpha(\frac{\partial^{\alpha_1}\partial_x\phi\partial_{v_1}\partial^{\alpha-\alpha_{1}}M}{\sqrt{\mu}},\frac{\partial^{\alpha}M}{\sqrt{\mu}})
\notag\\
\leq&& -\frac{1}{2}\frac{d}{dt}(\frac{1}{K\theta}e^{\phi}\partial^{\alpha}\widetilde{\phi},\partial^{\alpha}\widetilde{\phi})
-\frac{1}{2}\delta\frac{d}{dt}(\frac{1}{K\theta}\partial^{\alpha}\partial_x\widetilde{\phi},\partial^{\alpha}\partial_x\widetilde{\phi})
\notag\\
&&+C(\frac{1}{\varepsilon}+\frac{\delta^{3/2}}{\varepsilon^2})\mathcal{D}_{2}(t)+C\frac{1}{\delta^2}\mathcal{E}_{2}(t)+C\delta^2.
\end{eqnarray}
The second and third terms on the right hand side of \eqref{3.60} are bounded by
\begin{equation}
	\label{3.94A}
 C\frac{1}{\varepsilon}\mathcal{D}_{2}(t)+C\delta^{2}+C\varepsilon^2\delta^{1/2}.
\end{equation}
Here we have used the following estimates that for $|\alpha_1|=|\alpha|=2$,
\begin{eqnarray*}
&&\frac{1}{\delta}|(\frac{\partial^{\alpha_1}\partial_x\phi\partial_{v_{1}}\partial^{\alpha-\alpha_{1}}M}{\sqrt{\mu}},
\frac{\partial^{\alpha}(\overline{G}+\sqrt{\mu}f)}{\sqrt{\mu}})|
\notag\\
\leq&& C\delta^{\frac{1}{2}}\|\partial^{\alpha_1}\partial_x\phi\|^2+C\delta^{-\frac{5}{2}}\|\frac{\partial^{\alpha}\overline{G}}{\sqrt{\mu}}\|^2
+C\delta^{-\frac{5}{2}}\|\partial^{\alpha}f\|^2_{\sigma}
\notag\\
\leq&& C\frac{1}{\varepsilon}\frac{\varepsilon}{\delta^{1/2}}\delta\|\partial^{\alpha_1}\partial_x\widetilde{\phi}\|^2+C\delta^{5/2}+C\delta^{\frac{1}{2}}(\varepsilon+\delta^2)^2
+C\frac{1}{\varepsilon}\frac{1}{\delta^{3/2}\varepsilon}\frac{\varepsilon^2}{\delta}\|\partial^{\alpha}f\|^2_{\sigma}
\notag\\
\leq&& C\frac{1}{\varepsilon}\mathcal{D}_{2}(t)+C\delta^{2}+C\varepsilon^2\delta^{1/2}.
\end{eqnarray*}
It remains to estimate the last term of \eqref{3.60}. First of all, we write
\begin{equation}
	\label{3.74}
\frac{1}{\delta}(\frac{\partial^{\alpha_1}\partial_x\phi\partial_{v_{1}}\partial^{\alpha-\alpha_{1}}(\sqrt{\mu}f)}{\sqrt{\mu}},\frac{\partial^{\alpha}F}{\sqrt{\mu}})
	=\frac{1}{\delta}(\partial^{\alpha_1}\partial_x\phi\partial_{v_{1}}\partial^{\alpha-\alpha_{1}}f,\frac{\partial^{\alpha}F}{\sqrt{\mu}})
	-\frac{1}{\delta}(\frac{v_1}{2}\partial^{\alpha_1}\partial_{x}\phi\partial^{\alpha-\alpha_{1}}f,\frac{\partial^{\alpha}F}{\sqrt{\mu}}).
\end{equation}
To bound the first term on the right hand side of \eqref{3.74}, we consider two cases $|\alpha_{1}|=1$
and $|\alpha_{1}|=|\alpha|=2$. For the former case, we use \eqref{3.65}, \eqref{3.4}, \eqref{1.28}, \eqref{3.30} and \eqref{1.35} to get
\begin{eqnarray*}
\frac{1}{\delta}|(\partial^{\alpha_1}\partial_x\phi\partial_{v_{1}}\partial^{\alpha-\alpha_{1}}f,\frac{\partial^{\alpha}F}{\sqrt{\mu}})|
	&&\leq C\frac{1}{\delta}\|\partial^{\alpha_1}\partial_x\phi\|_{L^\infty}\|\langle v\rangle^{2}\langle v\rangle^{-\frac{3}{2}}\partial_{v_{1}}\partial^{\alpha-\alpha_{1}}f\|
	\|\langle v\rangle^{-\frac{1}{2}}\frac{\partial^{\alpha}F}{\sqrt{\mu}}\|
	\notag\\
	&&\leq C(1+\delta^{1/2}\frac{1}{\varepsilon^{1/2}})\|\langle v\rangle^{2}\partial^{\alpha-\alpha_{1}}f\|_\sigma
	\|\frac{\partial^{\alpha}F}{\sqrt{\mu}}\|_\sigma
	\notag\\
	&&\leq  C(1+\frac{\delta}{\varepsilon})\|w(\alpha-\alpha_{1},0)\partial^{\alpha-\alpha_{1}}f\|_\sigma^2+C
	\|\frac{\partial^{\alpha}F}{\sqrt{\mu}}\|^2_\sigma
	\notag\\
	&&\leq  C(\delta^{3/2}\varepsilon+\delta^{5/2})\frac{1}{\delta^{3/2}\varepsilon}\|\partial^{\alpha-\alpha_{1}}f\|_{\sigma,w}^2+
	C(\|\partial^{\alpha}f\|^2_{\sigma}
	+\|\partial^{\alpha}(\widetilde{\rho},\widetilde{u},\widetilde{\theta})\|^2+\delta^2).
	\notag\\
		&&\leq  C\frac{\delta^{1/2}}{\varepsilon}\mathcal{D}_{2,l,q_1}(t)+C\delta^2.
\end{eqnarray*}
Here we have used  $\langle v\rangle^2\leq \langle v\rangle^{2(l-|\alpha-\alpha_{1}|)}$ with $|\alpha-\alpha_{1}|=1$ and $l\geq 2$
such that $\langle v\rangle^2\leq w(\alpha-\alpha_{1},0)$. For the latter case, the similar calculation yields that
\begin{eqnarray*}
		&&\frac{1}{\delta}|(\partial^{\alpha_1}\partial_x\phi\partial_{v_{1}}\partial^{\alpha-\alpha_{1}}f,\frac{\partial^{\alpha}F}{\sqrt{\mu}})|
		\notag\\
		&&\leq C\frac{1}{\delta}\|\partial^{\alpha_1}\partial_{x}\phi\|\||\langle v\rangle^{2}\langle v\rangle^{-\frac{3}{2}}\partial_{v_{1}}\partial^{\alpha-\alpha_{1}}f|_2\|_{L^\infty}
		\|\langle v\rangle^{-\frac{1}{2}}\frac{\partial^{\alpha}F}{\sqrt{\mu}}\|
		\notag\\
		&&\leq C\frac{\delta^{1/2}}{\varepsilon}\mathcal{D}_{2,l,q_1}(t)+C\delta^2.
	\end{eqnarray*}
Combining the above two estimates, we get
\begin{equation*}
\frac{1}{\delta}\sum_{1\leq\alpha_{1}\leq\alpha}|(\partial^{\alpha_1}\partial_x\phi\partial_{v_{1}}\partial^{\alpha-\alpha_{1}}f,\frac{\partial^{\alpha}F}{\sqrt{\mu}})|\leq C\frac{\delta^{1/2}}{\varepsilon}\mathcal{D}_{2,l,q_1}(t)+C\delta^2.
\end{equation*}
Similar estimates also hold for the last term of \eqref{3.74}. We thus conclude that
\begin{equation*}
\frac{1}{\delta}\sum_{1\leq\alpha_{1}\leq\alpha}|(\frac{\partial^{\alpha_1}\partial_x\phi\partial_{v_{1}}\partial^{\alpha-\alpha_{1}}(\sqrt{\mu}f)}{\sqrt{\mu}},\frac{\partial^{\alpha}F}{\sqrt{\mu}})|		
\leq C\frac{\delta^{1/2}}{\varepsilon}\mathcal{D}_{2,l,q_1}(t)+C\delta^2.	
	\end{equation*}
Therefore, plugging this, \eqref{3.73} and \eqref{3.94A} into \eqref{3.60}, we show that
\begin{eqnarray*}
&&\sum_{1\leq \alpha_{1}\leq \alpha}C^{\alpha_1}_\alpha
\frac{1}{\delta}(\frac{\partial^{\alpha_1}\partial_x\phi\partial_{v_{1}}\partial^{\alpha-\alpha_{1}}F}{\sqrt{\mu}},\frac{\partial^{\alpha}F}{\sqrt{\mu}})
\notag\\
&&\leq -\frac{1}{2}\frac{d}{dt}(\frac{1}{K\theta}e^{\phi}\partial^{\alpha}\widetilde{\phi},\partial^{\alpha}\widetilde{\phi})
-\frac{1}{2}\delta\frac{d}{dt}(\frac{1}{K\theta}\partial^{\alpha}\partial_x\widetilde{\phi},\partial^{\alpha}\partial_x\widetilde{\phi})
\notag\\
&&\hspace{0.5cm}+C(\frac{1}{\varepsilon}+\frac{\delta^{3/2}}{\varepsilon^2})\mathcal{D}_{2,l,q_1}(t)+C\frac{1}{\delta^2}\mathcal{E}_{2}(t)
+C\delta^{\frac{1}{2}}\varepsilon^2+C\delta^2.
\end{eqnarray*}
In summary, we combine this, \eqref{3.59} and  \eqref{3.56} to get \eqref{3.55}, and complete the proof of Lemma \ref{lem3.14}.
\end{proof}
\subsubsection{Estimates with weight functions}\label{subs3.5.2}
In the following we are going to completing the estimates on the electric fields with weight functions.
We first prove the weighted mixed derivative estimates.
\begin{lemma}\label{lem3.15}
Under the same  conditions as in Lemma \ref{lem3.5}. 
For  $|\alpha|+|\beta|\leq 2$ and $|\beta|\geq 1$, one has
\begin{equation}
		\label{3.75}
		\frac{1}{\delta}|(\partial_{\beta}^{\alpha}[\frac{\partial_x\phi\partial_{v_1}(\sqrt{\mu}f)}{\sqrt{\mu}}],w^2(\alpha,\beta)\partial_{\beta}^{\alpha}f)|
		\leq  C\delta^{1/2}\mathcal{D}_{2,l,q_1}(t)
		+C\varepsilon^{1/2}\delta^{1/2}\|\langle v\rangle \partial_{\beta}^\alpha f\|_w^2.
	\end{equation}
For $|\beta|=0$ and $|\alpha|\leq 1$, one has 	
\begin{equation}
	\label{3.83}
	\frac{1}{\delta}|(\partial^{\alpha}[\frac{\partial_x\phi\partial_{v_1}(\sqrt{\mu}f)}{\sqrt{\mu}}],w^2(\alpha,0)\partial^{\alpha}f)|
	\leq  C\delta^{1/2}\mathcal{D}_{2,l,q_1}(t)
	+C\varepsilon^{1/2}\delta^{1/2}\|\langle v\rangle \partial^\alpha f\|_w^2.
\end{equation}
\end{lemma}
\begin{proof}
	Let $|\alpha|+|\beta|\leq 2$ and $|\beta|\geq 1$, it is clear that
	\begin{eqnarray}
		\label{3.76}
		&&\frac{1}{\delta}(\partial_{\beta}^{\alpha}[\frac{\partial_x\phi\partial_{v_1}(\sqrt{\mu}f)}{\sqrt{\mu}}],w^2(\alpha,\beta)\partial_{\beta}^{\alpha}f)
		\notag\\
		&&=\frac{1}{\delta}(\partial^\alpha_\beta[\partial_x\phi\partial_{v_1}f],w^2(\alpha,\beta)\partial_{\beta}^\alpha f)-\frac{1}{\delta}(\partial^\alpha_\beta[\frac{v_1}{2}\partial_x\phi f],w^2(\alpha,\beta)\partial_{\beta}^\alpha f).
	\end{eqnarray}
For the first term on the right hand side of \eqref{3.76}, we write
	\begin{eqnarray}
		\label{3.77}
		\frac{1}{\delta}(\partial^\alpha_\beta[\partial_x\phi\partial_{v_1}f],w^2(\alpha,\beta)\partial_{\beta}^\alpha f)
		=&&\frac{1}{\delta}(\partial_x\phi\partial_{v_1}\partial^\alpha_\beta f,w^2(\alpha,\beta)\partial_{\beta}^\alpha f)
		\notag\\
		&&+\frac{1}{\delta}\sum_{1\leq\alpha_{1}\leq\alpha}C^{\alpha_{1}}_\alpha(\partial^{\alpha_{1}}\partial_x\phi\partial_{v_1}\partial^{\alpha-\alpha_{1}}_\beta f,w^2(\alpha,\beta)\partial_{\beta}^\alpha f).
	\end{eqnarray}
	Here if $|\alpha|=0$, the last term of \eqref{3.77} vanishes, and if  $|\alpha|= 1$, it exists and then it can be bounded by
	\begin{eqnarray}
		\label{3.78}
		&& C\frac{1}{\delta}\|\partial^{\alpha}\partial_x\phi\|_{L^{\infty}}
		\|\langle v\rangle^{2}\langle v\rangle^{-\frac{3}{2}}w(\alpha,\beta)\partial_{v_1}\partial_\beta f\|
		\|\langle v\rangle^{-\frac{1}{2}}w(\alpha,\beta)\partial_{\beta}^\alpha f\|
		\notag\\
		&&\leq  C\frac{1}{\delta}(\delta+ \delta^{1/2}\frac{\delta}{\varepsilon^{1/2}})\|\langle v\rangle^{2}w(\alpha,\beta)\partial_\beta f\|_{\sigma}
		\|\partial_{\beta}^\alpha f\|_{\sigma,w}
		\notag\\
			\notag\\
		&&\leq C(\delta^{3/2}\varepsilon+\delta^{2}\varepsilon^{1/2})\frac{1}{\delta^{3/2}\varepsilon}
		\|\partial_\beta f\|_{\sigma,w}\|\partial_{\beta}^\alpha f\|_{\sigma,w}
		\notag\\
		&&\leq C\delta\mathcal{D}_{2,l,q_1}(t),
	\end{eqnarray}
	where we have used \eqref{1.28}, \eqref{3.65} and \eqref{3.4}, as well as $\langle v\rangle^{2}\langle v\rangle^{2(l-|\alpha|-|\beta|)}=\langle v\rangle^{2(l-|\beta|)}$ with $|\alpha|= 1$ and
	$$
	\langle v\rangle^{2}w(\alpha,\beta)= w(0,\beta).
	$$
	By the expression of $w(\alpha,\beta)$ in \eqref{1.25}, it is easy to check that
	\begin{equation}
		\label{3.79}
		|\partial_{v_1}w(\alpha,\beta)|=|2(l-|\alpha|-|\beta|)\frac{v_1}{\langle v\rangle^{2}}w(\alpha,\beta)+
		q_1(1+t)^{-q_2}v_1 w(\alpha,\beta)|\leq C\langle v\rangle w(\alpha,\beta).
	\end{equation}
Performing the similar calculation as \eqref{3.58} and using \eqref{3.79}, we get
	\begin{eqnarray*}
\frac{1}{\delta}|(\partial_x\phi\partial_{v_1}\partial^\alpha_\beta f,w^2(\alpha,\beta)\partial_{\beta}^\alpha f)|
		&&=\frac{1}{\delta}|\frac{1}{2}\int_{\mathbb R}\int_{{\mathbb R}^3}\partial_x\phi\partial_{v_1}(w^2(\alpha,\beta))
		(\partial_{\beta}^\alpha f)^2\,dv\,dx|
		\notag\\
		&&\leq C\frac{1}{\delta}\int_{\mathbb R}\int_{{\mathbb R}^3}|\partial_x\phi\langle v\rangle w^2(\alpha,\beta)
		(\partial_{\beta}^\alpha f)^2|\,dv\,dx
	\notag\\	
	&&\leq C\frac{1}{\varepsilon\delta}\| w(\alpha,\beta)\partial_{\beta}^\alpha f\|_\sigma^2
	+C\varepsilon^{1/2}\frac{1}{\delta}\|\partial_x\phi\|^{3/2}_{L^{\infty}}\|\langle v\rangle w(\alpha,\beta)\partial_{\beta}^\alpha f\|^2
	\notag\\
	&&\leq C\delta^{1/2}\mathcal{D}_{2,l,q_1}(t)
	+C\varepsilon^{1/2}\delta^{1/2}\|\langle v\rangle \partial_{\beta}^\alpha f\|_w^2,	
	\end{eqnarray*}
which together with \eqref{3.78} and \eqref{3.77}, gives
	\begin{equation}
		\label{3.80}
\frac{1}{\delta}|(\partial^\alpha_\beta(\partial_x\phi\partial_{v_1}f),w^2(\alpha,\beta)\partial_{\beta}^\alpha f)|
		\leq C\delta^{1/2}\mathcal{D}_{2,l,q_1}(t)
		+C\varepsilon^{1/2}\delta^{1/2}\|\langle v\rangle \partial_{\beta}^\alpha f\|_w^2	.
	\end{equation}
For the last term of \eqref{3.76}, we write
	\begin{eqnarray}
		\label{3.81}	
\frac{1}{\delta}(\partial^\alpha_\beta[\frac{v_1}{2}\partial_x\phi f],w^2(\alpha,\beta)\partial_{\beta}^\alpha f)
		=&&\frac{1}{\delta}(\frac{v_1}{2}\partial^\alpha[\partial_x\phi \partial_{\beta} f],w^2(\alpha,\beta)\partial_{\beta}^\alpha f)
		\notag\\
		&&+ \frac{1}{\delta}C^{\beta-e_1}_{\beta}\frac{1}{2}(\delta^{e_1}_{\beta}\partial^{\alpha}[\partial_x\phi \partial_{\beta-e_1} f],w^2(\alpha,\beta)\partial_{\beta}^\alpha f).
	\end{eqnarray}
	Here $\delta^{e_1}_{\beta}=1$ if $e_1\leq\beta$ or  $\delta^{e_1}_{\beta}=0$ otherwise.
	To bound the first term on the right hand side of \eqref{3.81}, we follow the similar computations used in \eqref{3.80} to get
	\begin{equation*}
		\frac{1}{\delta}|(\frac{v_1}{2}\partial^\alpha[\partial_x\phi \partial_{\beta} f],w^2(\alpha,\beta)\partial_{\beta}^\alpha f)|
	\leq C\delta^{1/2}\mathcal{D}_{2,l,q_1}(t)
	+C\varepsilon^{1/2}\delta^{1/2}\|\langle v\rangle \partial_{\beta}^\alpha f\|_w^2.
	\end{equation*}
	For the last term in \eqref{3.81}, we use the similar calculations as \eqref{3.78} to obtain
	\begin{eqnarray*}
		\frac{1}{\delta}(\partial^{\alpha}[\partial_x\phi \partial_{\beta-e_1} f],w^2(\alpha,\beta)\partial_{\beta}^\alpha f)
		&&=\frac{1}{\delta}(\partial_x\phi \partial^{\alpha}_{\beta-e_1}f,w^2(\alpha,\beta)\partial_{\beta}^\alpha f)
		+\frac{1}{\delta}\sum_{1\leq \alpha_{1}\leq\alpha}C^{\alpha_{1}}_\alpha(\partial^{\alpha_1}\partial_x\phi
		\partial^{\alpha-\alpha_{1}}_{\beta-e_i}f,w^2(\alpha,\beta)\partial_{\beta}^\alpha f)
		\notag\\
		&&\leq C\delta^{1/2}\mathcal{D}_{2,l,q_1}(t).
	\end{eqnarray*}
Hence,	plugging the above two estimates into \eqref{3.81}, we obtain
	\begin{equation*}
		\frac{1}{\delta}|(\partial^\alpha_\beta(\frac{v_1}{2}\partial_x\phi f),w^2(\alpha,\beta)\partial_{\beta}^\alpha f)|
		\leq C\delta^{1/2}\mathcal{D}_{2,l,q_1}(t)
		+C\varepsilon^{1/2}\delta^{1/2}\|\langle v\rangle \partial_{\beta}^\alpha f\|_w^2.
	\end{equation*}
In summary, we plug	this and \eqref{3.80} into \eqref{3.76} to get \eqref{3.75}. 
The proof of the estimate \eqref{3.83} is similar to \eqref{3.75} and details of the proof are
omitted for brevity. Hence, Lemma \ref{lem3.15} is proved.
\end{proof}
%Similar arguments as \eqref{3.83} gives
%\begin{lemma}\label{lem3.16}For $|\alpha|\leq 1$, it holds that
%\begin{equation}
	%\label{3.99}
		%\frac{1}{\delta}|(\frac{\partial^{\alpha}[\partial_x\phi\partial_{v_1}(\sqrt{\mu}f)]}{\sqrt{\mu}},\partial^{\alpha}f)|
		%\leq  C\delta^{1/2}\mathcal{D}_{2}(t)
		%+C\varepsilon^{1/2}\delta^{1/2}\|\langle v\rangle \partial^\alpha f\|^2.\end{equation}
%\end{lemma}
\begin{lemma}\label{lem3.17}
	Under the same conditions as in Lemma \ref{lem3.5}. For $|\alpha|=2$, one has
	\begin{eqnarray}
		\label{3.84}
		&&\frac{1}{\delta}|(\frac{\partial^{\alpha}(\partial_x\phi\partial_{v_1}F)}{\sqrt{\mu}},w^2(\alpha,0)\frac{\partial^{\alpha}F}{\sqrt{\mu}})|
		\notag\\
		&&\leq
	C(\frac{\delta^{1/2}}{\varepsilon^2}+\frac{1}{\delta\varepsilon})\mathcal{D}_{2,l,q_1}(t)+C\varepsilon^{1/2}\delta^{1/2}\|\langle v\rangle \partial^{\alpha}f\|_w^2
	+C\delta.
	\end{eqnarray}
\end{lemma}
\begin{proof}	
	Let $|\alpha|=2$, thanks to $F=M+\overline{G}+\sqrt{\mu}f$, we have
	\begin{eqnarray}
		\label{3.85}
\frac{1}{\delta}(\frac{\partial^{\alpha}(\partial_x\phi\partial_{v_1}F)}{\sqrt{\mu}},w^2(\alpha,0)\frac{\partial^{\alpha}F}{\sqrt{\mu}})
		=&&\frac{1}{\delta}(\frac{\partial_x\phi\partial_{v_1}\partial^{\alpha}F}{\sqrt{\mu}},w^2(\alpha,0)\frac{\partial^{\alpha}F}{\sqrt{\mu}})
		\notag\\
		&&+\frac{1}{\delta}\sum_{1\leq\alpha_{1}\leq \alpha}C^{\alpha_1}_\alpha(\frac{\partial^{\alpha_1}\partial_x\phi\partial_{v_1}\partial^{\alpha-\alpha_{1}}\overline{G}}{\sqrt{\mu}},
		w^2(\alpha,0)\frac{\partial^{\alpha}F}{\sqrt{\mu}})
		\notag\\
		&&+\frac{1}{\delta}\sum_{1\leq\alpha_{1}\leq \alpha}C^{\alpha_1}_\alpha(\frac{\partial^{\alpha_1}\partial_x\phi\partial_{v_1}\partial^{\alpha-\alpha_{1}}M}{\sqrt{\mu}},
		w^2(\alpha,0)\frac{\partial^{\alpha}F}{\sqrt{\mu}})
		\notag\\
		&&+\frac{1}{\delta}\sum_{1\leq \alpha_{1}\leq \alpha}C^{\alpha_1}_\alpha(\frac{\partial^{\alpha_1}\partial_x\phi\partial_{v_1}\partial^{\alpha-\alpha_{1}}[\sqrt{\mu}f]}{\sqrt{\mu}},
		w^2(\alpha,0)\frac{\partial^{\alpha}F}{\sqrt{\mu}}).
	\end{eqnarray}
In the following, we compute \eqref{3.85} term by term.
For the first term on the right hand side of  \eqref{3.85}, by integration by parts and \eqref{3.79}, we get
\begin{equation*}
	\frac{1}{\delta}|(\frac{\partial_x\phi\partial_{v_{1}}\partial^{\alpha}F}{\sqrt{\mu}},w^2(\alpha,0)\frac{\partial^{\alpha}F}{\sqrt{\mu}})|
	\leq C	\frac{1}{\delta}\int_{\mathbb R}\int_{{\mathbb R}^3}|\partial_x\phi \langle v\rangle w^2(\alpha,0)\frac{(\partial^{\alpha}F)^2}{\mu}|\,dv\,dx.
\end{equation*}
By the similar calculation as \eqref{3.58} and \eqref{1.35}, we get
\begin{eqnarray*}
&&\frac{1}{\delta}\int_{\mathbb R}\int_{{\mathbb R}^3}|\partial_x\phi \langle v\rangle w^2(\alpha,0)(\partial^{\alpha}f)^2|\,dv\,dx	
	\notag\\
	&&\leq C\delta^{1/2}\frac{1}{\varepsilon\delta^{3/2}}\|\partial^{\alpha}f\|_{\sigma,w}^2
	+C\varepsilon^{1/2}\delta^{1/2}\|\langle v\rangle \partial^{\alpha}f\|_w^2
	\notag\\
	&&\leq C\frac{\delta^{1/2}}{\varepsilon^2}\mathcal{D}_{2,l,q_1}(t)+C\varepsilon^{1/2}\delta^{1/2}\|\langle v\rangle \partial^{\alpha}f\|_w^2.
\end{eqnarray*}
By \eqref{3.57}, \eqref{3.29}, \eqref{3.7} and \eqref{1.35}, we get
\begin{eqnarray*}
	&&\frac{1}{\delta}\int_{\mathbb R}\int_{{\mathbb R}^3}|\partial_x\phi \langle v\rangle w^2(\alpha,0)\frac{(\partial^{\alpha}M+\partial^{\alpha}\overline{G})^2}{\mu}|\,dv\,dx
	\notag\\	
	&&\leq \frac{1}{\delta}\|\partial_x\phi\|_{L^{\infty}}\|\langle v\rangle^{\frac{1}{2}} w(\alpha,0)\frac{\partial^{\alpha}(M+\overline{G})}{\sqrt{\mu}}\|^2
	\notag\\
	&&\leq C\frac{\delta^{1/2}}{\varepsilon}\mathcal{D}_{2,l,q_1}(t)+C\delta^2.
\end{eqnarray*}
So, from the above three estimates and the fact that $F=M+\overline{G}+\sqrt{\mu}f$, one has
\begin{equation*}
\frac{1}{\delta}|(\frac{\partial_x\phi\partial_{v_1}\partial^{\alpha}F}{\sqrt{\mu}},w^2(\alpha,0)\frac{\partial^{\alpha}F}{\sqrt{\mu}})|\leq
C\frac{\delta^{1/2}}{\varepsilon^2}\mathcal{D}_{2,l,q_1}(t)+C\delta^2+C\varepsilon^{1/2}\delta^{1/2}\|\langle v\rangle \partial^{\alpha}f\|_w^2.
\end{equation*}	
For the second term on the right hand side of  \eqref{3.85}, we use \eqref{1.28}, \eqref{3.7}, \eqref{3.4}, \eqref{3.30} and \eqref{1.35} to get
\begin{eqnarray*}
	&&\frac{1}{\delta}\sum_{1\leq\alpha_{1}\leq \alpha}C^{\alpha_1}_\alpha(\frac{\partial^{\alpha_1}\partial_x\phi\partial_{v_{1}}\partial^{\alpha-\alpha_{1}}\overline{G}}{\sqrt{\mu}},
	w^2(\alpha,0)\frac{\partial^{\alpha}F}{\sqrt{\mu}})
	\notag\\
	&&\leq 	C\frac{1}{\delta^2}\sum_{1\leq\alpha_{1}\leq \alpha}\|\langle v\rangle^{\frac{1}{2}}	w^2(\alpha,0)\frac{\partial_{v_{1}}\partial^{\alpha-\alpha_{1}}\overline{G}}{\sqrt{\mu}}
\|_{L^\infty}^2\|\partial^{\alpha_1}\partial_x\phi\|^2+C\|\langle v\rangle^{-\frac{1}{2}}\frac{\partial^{\alpha}F}{\sqrt{\mu}}\|^2
	\notag\\
	&&\leq C\frac{\delta^{1/2}}{\varepsilon}\mathcal{D}_{2,l,q_1}(t)+C\delta^2.
\end{eqnarray*}
For the third term on the right hand side of  \eqref{3.85}, we consider two cases that $|\alpha_{1}|=1$ and $|\alpha_{1}|=|\alpha|=2$.
For the former case, we perform the similar calculation as the above to claim that
\begin{equation*}
	\frac{1}{\delta}(\frac{\partial^{\alpha_1}\partial_x\phi\partial_{v_{1}}\partial^{\alpha-\alpha_{1}}M}{\sqrt{\mu}},
	w^2(\alpha,0)\frac{\partial^{\alpha}F}{\sqrt{\mu}})
\leq C\frac{\delta^{1/2}}{\varepsilon}\mathcal{D}_{2,l,q_1}(t)+C\delta^2.
\end{equation*}
For the latter case, we have to split it as follows
\begin{equation}
	\label{3.106A}
\frac{1}{\delta}(\frac{\partial^{\alpha_1}\partial_x\phi\partial_{v_{1}}
	\partial^{\alpha-\alpha_{1}}M}{\sqrt{\mu}},w^2(\alpha,0)\frac{\partial^{\alpha}(\overline{G}+\sqrt{\mu}f)}{\sqrt{\mu}})	
	+\frac{1}{\delta}(\frac{\partial^{\alpha_1}\partial_x\phi\partial_{v_{1}}
		\partial^{\alpha-\alpha_{1}}M}{\sqrt{\mu}},w^2(\alpha,0)\frac{\partial^{\alpha}M}{\sqrt{\mu}}).
\end{equation}
In view of $|\alpha_{1}|=|\alpha|=2$, \eqref{3.7}, \eqref{3.4}, \eqref{3.11} and \eqref{1.35}, the first term of 
\eqref{3.106A} can be controlled by
\begin{eqnarray*}
&& C\|\partial^{\alpha_1}\partial_x\phi\|^2+
		C\frac{1}{\delta^2}\|w^2(\alpha,0)\frac{\partial_{v_{1}}
			\partial^{\alpha-\alpha_{1}}M}{\sqrt{\mu}}\frac{\partial^{\alpha}(\overline{G}+\sqrt{\mu}f)}{\sqrt{\mu}}\|^2		
	\notag\\
	&&\leq 	
	C\frac{1}{\delta^{1/2}\varepsilon}\mathcal{D}_{2,l,q_1}(t)+C\varepsilon^2\delta+C\delta^2.
\end{eqnarray*}
To bound the second term of \eqref{3.106A}, we recall $\partial^{\alpha}M=I^1_1+I^2_1+I_2$ given by \eqref{3.28} and \eqref{3.39a}.
By the expression of $I^2_1$ in \eqref{3.39a}, 
we get from the integration by parts, \eqref{3.11}, \eqref{3.3}, \eqref{3.1} and \eqref{3.4} that
\begin{eqnarray*}
\frac{1}{\delta}|(\frac{\partial^{\alpha_1}\partial_x\widetilde{\phi}\partial_{v_{1}}
		\partial^{\alpha-\alpha_{1}}M}{\sqrt{\mu}},w^2(\alpha,0)\frac{I^2_1}{\sqrt{\mu}})|
&&=\frac{1}{\delta}|(\partial^{\alpha_1}\widetilde{\phi},\partial_x[\frac{\partial_{v_{1}}
		\partial^{\alpha-\alpha_{1}}M}{\sqrt{\mu}}w^2(\alpha,0)\frac{I^2_1}{\sqrt{\mu}}])|
\notag\\
&&\leq C\frac{1}{\delta}\|\partial^{\alpha_1}\widetilde{\phi}\|(\|\partial_x(\rho,u,\theta)\|\|\partial^{\alpha}(\bar{\rho},\bar{u},\bar{\theta})\|_{L^\infty}
+\|\partial^{\alpha}\partial_{x}(\bar{\rho},\bar{u},\bar{\theta})\|)	
\notag\\
&&\leq \widetilde{C}\delta^{3/2}.
\end{eqnarray*}
On the other hand, by the expression of $I^1_1$ in \eqref{3.39a} and $I_2$ in \eqref{3.28}, we claim that
\begin{eqnarray*}
	&&\frac{1}{\delta}|(\frac{\partial^{\alpha_1}\partial_x\widetilde{\phi}\partial_{v_{1}}
		\partial^{\alpha-\alpha_{1}}M}{\sqrt{\mu}},w^2(\alpha,0)\frac{I^1_1+I_2}{\sqrt{\mu}})|
	\notag\\
	&&\leq C\frac{1}{\delta}\|\partial^{\alpha_1}\partial_x\widetilde{\phi}\|(\|\partial^{\alpha}(\widetilde{\rho},\widetilde{u},\widetilde{\theta})\|+
	\|\partial_x(\rho,u,\theta)\|_{L^\infty}\|\partial_x(\rho,u,\theta)\|)	
	\notag\\
	&&\leq C\frac{1}{\delta^{3/2}}(\delta\|\partial^{\alpha}\partial_x\widetilde{\phi}\|^2+\|\partial^{\alpha}(\widetilde{\rho},\widetilde{u},\widetilde{\theta})\|^2)
	+C\delta^{1/2}\|\partial_x(\rho,u,\theta)\|^2_{L^\infty}
	\notag\\
	&&\leq C\frac{1}{\delta\varepsilon}\mathcal{D}_{2,l,q_1}(t)+C\delta^{5/2}.	
\end{eqnarray*}
By $|\alpha_{1}|=|\alpha|=2$, \eqref{3.11} \eqref{3.4}, \eqref{3.29} and \eqref{1.35}, we get
\begin{eqnarray*}
	\frac{1}{\delta}(\frac{\partial^{\alpha_1}\partial_x\bar{\phi}\partial_{v_{1}}
		\partial^{\alpha-\alpha_{1}}M}{\sqrt{\mu}},w^2(\alpha,0)\frac{\partial^{\alpha}M}{\sqrt{\mu}})
	&&\leq C\frac{1}{\delta}\|\partial^{\alpha_1}\partial_x\bar{\phi}\|(\|\partial^{\alpha}(\widetilde{\rho},\widetilde{u},\widetilde{\theta})\|+\delta)	
	\notag\\
	&&
	\leq  C\frac{1}{\delta\varepsilon}\mathcal{D}_{2,l,q_1}(t)+C\delta.
\end{eqnarray*}
Hence, from the above three bounds, we obtain
\begin{eqnarray*}
	\frac{1}{\delta}(\frac{\partial^{\alpha_1}\partial_x\phi\partial_{v_{1}}
		\partial^{\alpha-\alpha_{1}}M}{\sqrt{\mu}},w^2(\alpha,0)\frac{\partial^{\alpha}M}{\sqrt{\mu}})
	&&\leq C\frac{1}{\delta}\|\partial^{\alpha_1}\partial_x\bar{\phi}\|(\|\partial^{\alpha}(\widetilde{\rho},\widetilde{u},\widetilde{\theta})\|+\delta)	
	\notag\\
	&&
	\leq  C\frac{1}{\delta\varepsilon}\mathcal{D}_{2,l,q_1}(t)+C\delta.
\end{eqnarray*}
Collecting the above estimates, we claim that
\begin{equation*}
	\frac{1}{\delta}\sum_{1\leq \alpha_{1}\leq \alpha}C^{\alpha_1}_\alpha|(\frac{\partial^{\alpha_1}\partial_x\phi\partial_{v_1}\partial^{\alpha-\alpha_{1}}M}{\sqrt{\mu}},w^2(\alpha,0)\frac{\partial^{\alpha}F}{\sqrt{\mu}})|
	\leq C\frac{1}{\delta\varepsilon}\mathcal{D}_{2,l,q_1}(t)+C\delta.
\end{equation*}
The last term of \eqref{3.85} can be treated in the similar way as \eqref{3.74}, then it can be controlled by $C\frac{1}{\delta\varepsilon}\mathcal{D}_{2,l,q_1}(t)+C\delta^2$.

In summary, collecting all the above estimates, we get the desired estimate \eqref{3.84}. 
This completes the proof of Lemma \ref{lem3.17}.	
\end{proof}

\section{Zeroth and first order space derivative estimates}\label{sec.4} 
In the following sections \ref{sec.4}-\ref{sec.6}, we are devote to obtaining the desired {\it a priori} estimate
\eqref{1.40} on the solution step by step in a series of lemmas in order to close the {\it a priori} assumption \eqref{3.1}.
In the current section, we shall obtain the space derivative estimates for both the fluid and non-fluid parts up to one-order.
\subsection{Estimates on the fluid part}\label{sec.4.1} 
We start from the space derivative estimates on the fluid part $(\widetilde{\rho},\widetilde{u},\widetilde{\theta},\widetilde{\phi})$.
The main result is given in Lemma \ref{lem4.1} below. 
\begin{lemma}\label{lem4.1}
Let $(F(t,x,v),\phi(t,x))$ be the smooth solution  to the Cauchy problem on the VPL system \eqref{1.15} and \eqref{1.39} for $0\leq t\leq T$. Assume
that \eqref{3.1} holds, we have the following energy estimate
\begin{eqnarray}
	\label{4.1}
	&&\frac{1}{2}\sum_{|\alpha|\leq 1}\frac{d}{dt}\int_{\mathbb{R}}( \frac{2\bar{\theta}}{3\bar{\rho}^{2}}|\partial^{\alpha}\widetilde{\rho}|^{2}+|\partial^{\alpha}\widetilde{u}|^{2}
	+\frac{1}{\bar{\theta}}|\partial^{\alpha}\widetilde{\theta}|^{2}+
	\frac{1}{\rho}e^{\phi}|\partial^{\alpha}\widetilde{\phi}|^2
	+\delta\frac{1}{\rho}|\partial^{\alpha}\partial_x\widetilde{\phi}|^2)\,dx+\frac{d}{dt}E_\alpha(t)
	\notag\\
	&&+\kappa_1\delta^{1/2}\varepsilon\sum_{|\alpha|\leq 1}\frac{d}{dt}(\partial^{\alpha}\widetilde{u}_1,\partial^{\alpha}\partial_x\widetilde{\rho})
	+c\kappa_1\sum_{|\alpha|\leq 1}\frac{\varepsilon}{\delta^{1/2}}(\|\partial^{\alpha}\partial_x(\widetilde{\rho},\widetilde{u},\widetilde{\theta})\|^2
	+\|\partial^{\alpha}\partial_x\widetilde{\phi}\|^2
	+\delta\|\partial^{\alpha}\partial^2_x\widetilde{\phi}\|^2)
	\notag\\
	\leq&& 
	C\frac{\varepsilon}{\delta^{1/2}}\sum_{|\alpha|\leq 1}\|\partial^{\alpha}\partial_xf\|_\sigma^{2}
	+C\delta^{3}\frac{1}{\delta^{3/2}\varepsilon}(\|f\|^2_{\sigma}+\|\partial_xf\|^2_{\sigma})+C\mathcal{E}_{2}(t)+C\delta^4+C\delta\varepsilon^2.
\end{eqnarray}
Here $\kappa_1>0$ is small constant and $E_\alpha(t)$ is given by
\begin{eqnarray}
	\label{4.33}
	E_\alpha(t)
	=&&\delta^{1/2}\varepsilon\sum_{|\alpha|\leq 1}\sum^{3}_{i=1}\int_{\mathbb{R}}\int_{\mathbb{R}^{3}}
	\big\{\partial^{\alpha}[\frac{1}{\rho}\partial_{x}(K\theta B_{1i}(\frac{v-u}{\sqrt{K\theta}})\frac{\sqrt{\mu}}{M}f)]
	\partial^{\alpha}\widetilde{u}_i\big\}\,dv\,dx
	\notag\\
	&&+\delta^{1/2}\varepsilon\sum_{|\alpha|\leq 1}\int_{\mathbb{R}}\int_{\mathbb{R}^{3}}
	\big\{\partial^{\alpha}[\frac{1}{\rho}\partial_{x}
	((K\theta)^{\frac{3}{2}}A_{1}(\frac{v-u}{\sqrt{K\theta}})\frac{\sqrt{\mu}}{M}f)]\frac{1}{\bar{\theta}}\partial^{\alpha}\widetilde{\theta}\big\}\,dv\,dx
	\notag\\	
	&&+\delta^{1/2}\varepsilon\sum_{|\alpha|\leq 1}\sum^{3}_{j=1}\int_{\mathbb{R}}\int_{\mathbb{R}^{3}}
	\big\{\partial^{\alpha}[\frac{1}{\rho}
	\partial_{x}u_j K\theta B_{1j}(\frac{v-u}{\sqrt{K\theta}})\frac{\sqrt{\mu}}{M}f]\frac{1}{\bar{\theta}}\partial^{\alpha}\widetilde{\theta}\big\}\,dv\,dx.
\end{eqnarray}		
\end{lemma}
\begin{proof}
	The proof of \eqref{4.1} can be divided into four steps. In the first three steps we make the direct energy estimates on 
	$(\widetilde{\rho},\widetilde{u},\widetilde{\theta},\widetilde{\phi})$
	in term of the system \eqref{2.33} and then combine
	those estimates to obtain the desired estimate \eqref{4.1} in the last step.

\medskip
\noindent{\it Step 1. Estimate on $\|\partial^{\alpha}\widetilde{\rho}(t,\cdot)\|^2$} for $|\alpha|\leq 1$. 
	
For  $|\alpha|\leq 1$, by taking $\partial^{\alpha}$ derivatives of the first equation  of \eqref{2.33} and then taking the inner product of the resulting equation with $\frac{2\bar{\theta}}{3\bar{\rho}^{2}}\partial^{\alpha}\widetilde{\rho}$, we obtain
\begin{eqnarray}
\label{4.2}
		&&\frac{1}{2}\frac{d}{dt}(\partial^{\alpha}\widetilde{\rho},\frac{2\bar{\theta}}{3\bar{\rho}^{2}}\partial^{\alpha}\widetilde{\rho})
		-\frac{1}{2}(\partial^{\alpha}\widetilde{\rho},\partial_t(\frac{2\bar{\theta}}{3\bar{\rho}^{2}})\partial^{\alpha}\widetilde{\rho})
		-A\frac{1}{\delta}(\partial^{\alpha}\partial_x\widetilde{\rho},\frac{2\bar{\theta}}{3\bar{\rho}^{2}}\partial^{\alpha}\widetilde{\rho})
		+\frac{1}{\delta}(\bar{\rho}\partial^{\alpha}\partial_x\widetilde{u}_{1},\frac{2\bar{\theta}}{3\bar{\rho}^{2}}\partial^{\alpha}\widetilde{\rho})
		\notag\\
		&&+\frac{1}{\delta}\sum_{1\leq \alpha_{1}\leq \alpha}C^{\alpha_{1}}_\alpha(\partial^{\alpha_1}\bar{\rho}\partial^{\alpha-\alpha_{1}}\partial_x\widetilde{u}_{1},
		\frac{2\bar{\theta}}{3\bar{\rho}^{2}}\partial^{\alpha}\widetilde{\rho})
		+\frac{1}{\delta}(\partial^{\alpha}(\widetilde{u}_1\partial_x\bar{\rho}),\frac{2\bar{\theta}}{3\bar{\rho}^{2}}\partial^{\alpha}\widetilde{\rho})
		\notag\\
=&&-\frac{1}{\delta}(\partial^{\alpha}(\widetilde{\rho}\partial_xu_{1}),\frac{2\bar{\theta}}{3\bar{\rho}^{2}}\partial^{\alpha}\widetilde{\rho})
-\frac{1}{\delta}(\partial^{\alpha}(u_1\partial_x\widetilde{\rho}),\frac{2\bar{\theta}}{3\bar{\rho}^{2}}\partial^{\alpha}\widetilde{\rho})
-(\delta^2\partial^{\alpha}\mathcal{R}_1,\frac{2\bar{\theta}}{3\bar{\rho}^{2}}\partial^{\alpha}\widetilde{\rho}).
\end{eqnarray}	
Here, if $|\alpha|=0$, the fifth term on the left hand side of \eqref{4.2} vanishes and if  $|\alpha|=1$, it exists.
We now estimate \eqref{4.2} term by term. Using the embedding inequality, \eqref{3.4} and \eqref{1.34}, we obtain
\begin{equation*}
|(\partial^{\alpha}\widetilde{\rho},\partial_t(\frac{2\bar{\theta}}{3\bar{\rho}^{2}})\partial^{\alpha}\widetilde{\rho})|	
\leq C\|\partial_t(\bar{\rho},\bar{\theta})\|_{L^{\infty}}\|\partial^{\alpha}\widetilde{\rho}\|^2
	\leq C\delta\mathcal{E}_{2}(t).
\end{equation*}
After the integration by parts, we use the embedding inequality, \eqref{3.4} and \eqref{1.34} to get
\begin{equation*}
-A\frac{1}{\delta}(\partial^{\alpha}\partial_x\widetilde{\rho},\frac{2\bar{\theta}}{3\bar{\rho}^{2}}\partial^{\alpha}\widetilde{\rho})
=\frac{1}{2}A\frac{1}{\delta}(\partial^{\alpha}\widetilde{\rho},\partial_x[\frac{2\bar{\theta}}{3\bar{\rho}^{2}}]\partial^{\alpha}\widetilde{\rho})
\leq C\frac{1}{\delta}\|\partial_x(\bar{\rho},\bar{\theta})\|_{L^{\infty}}\|\partial^{\alpha}\widetilde{\rho}\|^2
\leq C\mathcal{E}_{2}(t).
\end{equation*}	
Similarly, the last two terms on the left hand side of \eqref{4.2}  can be  bounded by
$C\mathcal{E}_{2}(t)$. 

For the first term on the right hand side of \eqref{4.2}, one obtains
\begin{equation}
	\label{4.3A}
	\frac{1}{\delta}(\partial^{\alpha}(\widetilde{\rho}\partial_xu_1),\frac{2\bar{\theta}}{3\bar{\rho}^{2}}\partial^{\alpha}\widetilde{\rho})
	=\frac{1}{\delta}(\widetilde{\rho}\partial_x\partial^{\alpha}u_1,\frac{2\bar{\theta}}{3\bar{\rho}^{2}}\partial^{\alpha}\widetilde{\rho})
	+\frac{1}{\delta}(\partial^{\alpha}\widetilde{\rho}\partial_xu_1,\frac{2\bar{\theta}}{3\bar{\rho}^{2}}\partial^{\alpha}\widetilde{\rho}),
\end{equation}
where if $|\alpha|=0$, the last term in \eqref{4.3A} vanishes and if  $|\alpha|=1$,
it exists. Then, we use the embedding inequality, \eqref{3.4}, \eqref{3.2}, \eqref{3.1} and \eqref{1.34} to get
\begin{eqnarray}
	\label{4.3}
\frac{1}{\delta}|(\partial^{\alpha}\widetilde{\rho}\partial_xu_1,\frac{2\bar{\theta}}{3\bar{\rho}^{2}}\partial^{\alpha}\widetilde{\rho})|&&\leq C\|\partial^{\alpha}\widetilde{\rho}\|^2
	+C\frac{1}{\delta^2}\|\partial_xu_1\|^2_{L^{\infty}}\|\partial^{\alpha}\widetilde{\rho}\|^2
	\notag\\
	&&\leq C\|\partial^{\alpha}\widetilde{\rho}\|^2
	+C\frac{1}{\delta^2}(\delta^2+\|\partial_x\widetilde{u}_1\|_{L^{\infty}}^2)\|\partial^{\alpha}\widetilde{\rho}\|^2
	\notag\\
	&&\leq  C\|\partial^{\alpha}\widetilde{\rho}\|^2
	+\widetilde{C}\frac{1}{\varepsilon}\delta^{5/2}\|\partial^{\alpha}\widetilde{\rho}\|^2
	\notag\\
	&&\leq  C\mathcal{E}_{2}(t),
\end{eqnarray}
where in the last inequality we have required that
\begin{equation}
	\label{4.5}
	\delta \leq \widetilde{C}^{-1}\varepsilon^{\frac{2}{5}}	
	\Longrightarrow \widetilde{C}\frac{1}{\varepsilon}\delta^{5/2}\leq 1, \quad \mbox{for}~~\widetilde{C}>1.
\end{equation}
For the first term in \eqref{4.3A}, we use the embedding inequality, \eqref{3.4}, \eqref{3.1}, \eqref{4.5} and \eqref{1.34} to obtain
\begin{eqnarray*}
	\frac{1}{\delta}|(\widetilde{\rho}\partial^{\alpha}\partial_xu_1,\frac{2\bar{\theta}}{3\bar{\rho}^{2}}\partial^{\alpha}\widetilde{\rho})|
	&&\leq C\frac{1}{\delta}(\|\widetilde{\rho}\|_{L^{\infty}}\|\partial^{\alpha}\partial_x\widetilde{u}_1\|\|\partial^{\alpha}\widetilde{\rho}\|
	+\|\partial^{\alpha}\partial_x\bar{u}_1\|_{L^{\infty}}\|\widetilde{\rho}\|\|\partial^{\alpha}\widetilde{\rho}\|)
	\notag\\
	&&\leq \widetilde{C}\delta\|\partial^{\alpha}\partial_x\widetilde{u}_1\|\|\partial^{\alpha}\widetilde{\rho}\|
	+C\|\widetilde{\rho}\|\|\partial^{\alpha}\widetilde{\rho}\|
	\notag\\
	&&\leq \eta\frac{\varepsilon}{\delta^{1/2}}\|\partial^{\alpha}\partial_x\widetilde{u}_1\|^2
	+C_\eta\widetilde{C}\frac{1}{\varepsilon}\delta^{1/2}\delta^2\|\partial^{\alpha}\widetilde{\rho}\|^2+C\mathcal{E}_{2}(t)
	\notag\\
	&&\leq \eta\frac{\varepsilon}{\delta^{1/2}}\|\partial^{\alpha}\partial_x\widetilde{u}_1\|^2+C_\eta\mathcal{E}_{2}(t).
\end{eqnarray*}
Plugging this and \eqref{4.3} into \eqref{4.3A} and using \eqref{4.5}, we get
\begin{equation}
	\label{4.6}
	\frac{1}{\delta}|(\partial^{\alpha}(\widetilde{\rho}\partial_xu_1),\frac{2\bar{\theta}}{3\bar{\rho}^{2}}\partial^{\alpha}\widetilde{\rho})|
	\leq C\eta\frac{\varepsilon}{\delta^{1/2}}\|\partial^{\alpha}\partial_x\widetilde{u}_1\|^2
	 +C_\eta\mathcal{E}_{2}(t).
\end{equation}
Note from \eqref{3.2}, \eqref{3.4} and \eqref{4.5} that
\begin{equation}
	\label{4.17a}
	\|\partial_x(\rho,u,\theta,\phi)\|_{L^\infty}\leq C\delta.
\end{equation}
Similar to \eqref{4.6}, the second term on the right hand side of \eqref{4.2} can be estimated as follows
\begin{eqnarray}
\label{4.4}
\frac{1}{\delta}(\partial^{\alpha}(u_1\partial_x\widetilde{\rho}),\frac{2\bar{\theta}}{3\bar{\rho}^{2}}\partial^{\alpha}\widetilde{\rho})&&=	
\frac{1}{\delta}(u_1\partial^{\alpha}\partial_x\widetilde{\rho},\frac{2\bar{\theta}}{3\bar{\rho}^{2}}\partial^{\alpha}\widetilde{\rho})
+\frac{1}{\delta}(\partial^{\alpha}u_1\partial_x\widetilde{\rho},
\frac{2\bar{\theta}}{3\bar{\rho}^{2}}\partial^{\alpha}\widetilde{\rho})
\notag\\
&&=	-\frac{1}{2}
\frac{1}{\delta}(\partial_x[\frac{2\bar{\theta}}{3\bar{\rho}^{2}}u_1]\partial^{\alpha}\widetilde{\rho},\partial^{\alpha}\widetilde{\rho})
+\frac{1}{\delta}(\partial^{\alpha}u_1\partial_x\widetilde{\rho},
\frac{2\bar{\theta}}{3\bar{\rho}^{2}}\partial^{\alpha}\widetilde{\rho})
\notag\\
&&\leq  C\eta\frac{\varepsilon}{\delta^{1/2}}\|\partial^{\alpha}\partial_x\widetilde{u}_1\|^2
+C_\eta\mathcal{E}_{2}(t).
\end{eqnarray}
For the last term of \eqref{4.2},  it is clear by \eqref{2.31} and \eqref{1.34} that 
\begin{equation*}
|(\delta^2\partial^{\alpha}\mathcal{R}_1,\frac{2\bar{\theta}}{3\bar{\rho}^{2}}\partial^{\alpha}\widetilde{\rho})|
\leq C\|\partial^{\alpha}\widetilde{\rho}\|^2+C\delta^4\leq C\mathcal{E}_{2}(t)+C\delta^4.
\end{equation*}
Consequently, collecting all the above estimates, we obtain
\begin{equation}
	\label{4.7}
	\frac{1}{2}\frac{d}{dt}(\partial^{\alpha}\widetilde{\rho},\frac{2\bar{\theta}}{3\bar{\rho}^{2}}\partial^{\alpha}\widetilde{\rho})
	+\frac{1}{\delta}(\partial^{\alpha}\partial_x\widetilde{u}_1,\frac{2\bar{\theta}}{3\bar{\rho}}\partial^{\alpha}\widetilde{\rho})
	\leq C\eta\frac{\varepsilon}{\delta^{1/2}}\|\partial^{\alpha}\partial_x\widetilde{u}_1\|^2
	+C_\eta\mathcal{E}_{2}(t)+C\delta^4.
\end{equation}	

\medskip
\noindent{\it Step 2. Estimate on $\|\partial^{\alpha}\widetilde{u}(t,\cdot)\|^2$} for $|\alpha|\leq 1$.

For  $|\alpha|\leq 1$, by taking $\partial^{\alpha}$ derivatives of the second equation  of \eqref{2.33} and then taking the inner product of the resulting equation with $\partial^{\alpha}\widetilde{u}_1$, one obtains
\begin{eqnarray}
	\label{4.8}
	&&\frac{1}{2}\frac{d}{dt}\|\partial^{\alpha}\widetilde{u}_1\|^{2}-A\frac{1}{\delta}(\partial_x\partial^{\alpha}\widetilde{u}_1,\partial^{\alpha}\widetilde{u}_1)
	+\frac{1}{\delta}
	(\partial^{\alpha}[\frac{2\bar{\theta}}{3\bar{\rho}}\partial_x\widetilde{\rho}],\partial^{\alpha}\widetilde{u}_1)
	\notag\\
	=&&-\frac{1}{\delta}(\frac{2}{3}\partial^{\alpha}\partial_x\widetilde{\theta},\partial^{\alpha}\widetilde{u}_1)
	-\frac{1}{\delta}(\partial^{\alpha}(u_1\partial_x\widetilde{u}_1),\partial^{\alpha}\widetilde{u}_1)
	-\frac{1}{\delta}(\partial^{\alpha}(\widetilde{u}_1\partial_x\bar{u}_1),\partial^{\alpha}\widetilde{u}_1)
	\notag\\
	&&-\frac{1}{\delta}(\partial^{\alpha}[\frac{2}{3}(\frac{\theta}{\rho}-\frac{\bar{\theta}}{\bar{\rho}})\partial_x\rho],\partial^{\alpha}\widetilde{u}_1)
	-\delta^2(\partial^{\alpha}[\frac{1}{\bar{\rho}}\mathcal{R}_2],\partial^{\alpha}\widetilde{u}_1)
	-\frac{1}{\delta}(\partial^{\alpha}\partial_x\widetilde{\phi},\partial^{\alpha}\widetilde{u}_1)
	\notag\\
	&&+\frac{\varepsilon}{\delta^{1/2}}(\partial^{\alpha}[\frac{4}{3\rho}\partial_x(\mu(\theta)\partial_xu_1)],\partial^{\alpha}\widetilde{u}_1)
	-\frac{1}{\delta}(\partial^{\alpha}[\frac{1}{\rho}\partial_x(\int_{\mathbb{R}^{3}} v_{1}^2L^{-1}_{M}\Theta\,dv)],\partial^{\alpha}\widetilde{u}_1).
\end{eqnarray}
In the following, we compute \eqref{4.8} term by term.
The second term on the left hand side of \eqref{4.8} vanishes after integration by parts, while the third term  on the left hand side of \eqref{4.8} has lower bound that
\begin{eqnarray*}
\frac{1}{\delta}(\frac{2\bar{\theta}}{3\bar{\rho}}\partial^{\alpha}\partial_x\widetilde{\rho},\partial^{\alpha}\widetilde{u}_1)
+\frac{1}{\delta}(\partial^{\alpha}[\frac{2\bar{\theta}}{3\bar{\rho}}]\partial_x\widetilde{\rho},\partial^{\alpha}\widetilde{u}_1)
\geq \frac{1}{\delta}(\frac{2\bar{\theta}}{3\bar{\rho}}\partial^{\alpha}\partial_x\widetilde{\rho},\partial^{\alpha}\widetilde{u}_1)
-C\mathcal{E}_{2}(t).
\end{eqnarray*}
The similar calculation as \eqref{4.4} implies that
\begin{eqnarray*}
&&\frac{1}{\delta}|(\partial^{\alpha}(u_1\partial_x\widetilde{u}_1),\partial^{\alpha}\widetilde{u}_1)|
+\frac{1}{\delta}|(\partial^{\alpha}(\widetilde{u}_1\partial_x\bar{u}_1),\partial^{\alpha}\widetilde{u}_1)|
\notag\\
&&\leq  C\mathcal{E}_{2}(t)+\eta\frac{\varepsilon}{\delta^{1/2}}\|\partial^{\alpha}\partial_x\widetilde{u}_1\|^2
+C_\eta\widetilde{C}\frac{1}{\varepsilon}\delta^{1/2}\delta^4\|\partial^{\alpha}\widetilde{u}_1\|^2
\notag\\
&&\leq \eta\frac{\varepsilon}{\delta^{1/2}}\|\partial^{\alpha}\partial_x\widetilde{u}_1\|^2
+C_\eta\mathcal{E}_{2}(t).
\end{eqnarray*}
For the fourth and fifth terms on the right hand side of \eqref{4.8}, by \eqref{2.31}, \eqref{3.4}, \eqref{3.1} and \eqref{1.34}, we show that
\begin{equation*}
\frac{1}{\delta}|(\partial^{\alpha}[\frac{2}{3}(\frac{\theta}{\rho}-\frac{\bar{\theta}}{\bar{\rho}})\partial_x\rho],\partial^{\alpha}\widetilde{u}_1)|
+\delta^2|(\partial^{\alpha}[\frac{1}{\bar{\rho}}\mathcal{R}_2],\partial^{\alpha}\widetilde{u}_1)|\leq \eta\frac{\varepsilon}{\delta^{1/2}}\|\partial^{\alpha}\partial_x\widetilde{u}_1\|^2
+C_\eta\mathcal{E}_{2}(t)+C\delta^4.
\end{equation*}
To bound the electric potential term in \eqref{4.8}, we use the integration by parts to write
\begin{multline}
\label{4.9}
-\frac{1}{\delta}(\partial^{\alpha}\partial_x\widetilde{\phi},\partial^{\alpha}\widetilde{u}_1)
=\frac{1}{\delta}(\partial^{\alpha}\widetilde{\phi},\partial^{\alpha}\partial_x\widetilde{u}_1)
\\
=\frac{1}{\delta}(\frac{1}{\rho}\partial^{\alpha}\widetilde{\phi},\partial^{\alpha}[\rho\partial_x\widetilde{u}_1])
-\frac{1}{\delta}\sum_{1\leq \alpha_{1}\leq \alpha}C^{\alpha_{1}}_\alpha
(\frac{1}{\rho}\partial^{\alpha}\widetilde{\phi},\partial^{\alpha_1}\rho\partial_x\partial^{\alpha-\alpha_{1}}\widetilde{u}_1).
\end{multline}
Here if $|\alpha|=0$, the last term of \eqref{4.9} vanishes, if $|\alpha|=1$, it is bounded by $C\mathcal{E}_{2}(t)$.
By the first equation of \eqref{2.33}, the first term on the right hand side of \eqref{4.9} can be written as
\begin{multline}
	\label{4.10}
\frac{1}{\delta}(\frac{1}{\rho}\partial^{\alpha}\widetilde{\phi},\partial^{\alpha}[\rho\partial_x\widetilde{u}_1])=-(\frac{1}{\rho}\partial^{\alpha}\widetilde{\phi},\partial^{\alpha}\partial_t\widetilde{\rho})+\frac{1}{\delta}(\frac{1}{\rho}\partial^{\alpha}\widetilde{\phi},
	A\partial^{\alpha}\partial_x\widetilde{\rho})
\\
	-\frac{1}{\delta}(\frac{1}{\rho}\partial^{\alpha}\widetilde{\phi},\partial^{\alpha}[u_{1}\partial_x\widetilde{\rho}])-\frac{1}{\delta}(\frac{1}{\rho}\partial^{\alpha}\widetilde{\phi},\partial^{\alpha}[\widetilde{u}_{1}\partial_x\bar{\rho}
+\partial_x\bar{u}_1\widetilde{\rho}+\delta^3\mathcal{R}_1]).	
\end{multline}
For the first term on the right hand of \eqref{4.10}, we use the last equation of \eqref{2.33} to write
\begin{eqnarray}
	\label{4.11}
-(\frac{1}{\rho}\partial^{\alpha}\widetilde{\phi},\partial^{\alpha}\partial_t\widetilde{\rho})=&&\delta
(\frac{1}{\rho}\partial^{\alpha}\widetilde{\phi},\partial^{\alpha}\partial^{2}_{x}\partial_t\widetilde{\phi})
-
(\frac{1}{\rho}\partial^{\alpha}\widetilde{\phi},\partial^{\alpha}[e^{\phi}\partial_t\widetilde{\phi}])
\notag\\
&&-
(\frac{1}{\rho}\partial^{\alpha}\widetilde{\phi},\partial^{\alpha}[(e^{\phi}-e^{\bar{\phi}})\partial_t\bar{\phi}])+\delta^3(\frac{1}{\rho}\partial^{\alpha}\widetilde{\phi},\partial^{\alpha}\partial_t\mathcal{R}_{4}).
\end{eqnarray}
Next we estimate each term for \eqref{4.11}.
By integration by parts, one has
\begin{eqnarray}
	\label{4.12}
	\delta(\frac{1}{\rho}\partial^{\alpha}\widetilde{\phi},\partial^{\alpha}\partial^{2}_{x}\partial_t\widetilde{\phi})
	=&&-\delta(\frac{1}{\rho}\partial^{\alpha}\partial_x\widetilde{\phi},\partial^{\alpha}\partial_x\partial_t\widetilde{\phi})
	+\delta(\frac{1}{\rho^2}\partial_x\rho\partial^{\alpha}\widetilde{\phi},\partial^{\alpha}\partial_x\partial_t\widetilde{\phi})
	\notag\\
	=&&-\frac{1}{2}\delta\frac{d}{dt}(\frac{1}{\rho}\partial^{\alpha}\partial_x\widetilde{\phi},\partial^{\alpha}\partial_x\widetilde{\phi})
	+\frac{1}{2}\delta(\partial_t[\frac{1}{\rho}]\partial^{\alpha}\partial_x\widetilde{\phi},\partial^{\alpha}\partial_x\widetilde{\phi})
	\notag\\
	&&+\delta(\frac{1}{\rho^2}\partial_x\widetilde{\rho}\partial^{\alpha}\widetilde{\phi},\partial^{\alpha}\partial_x\partial_t\widetilde{\phi})
	-\delta(\partial_x[\frac{1}{\rho^2}\partial_x\bar{\rho}\partial^{\alpha}\widetilde{\phi}],\partial^{\alpha}\partial_t\widetilde{\phi}).
\end{eqnarray}
For the second term on the right hand side of \eqref{4.12}, we use \eqref{3.67}, \eqref{4.5}, \eqref{3.2a} and \eqref{1.34} to obtain
\begin{eqnarray*}
\delta|(\partial_t[\frac{1}{\rho}]\partial^{\alpha}\partial_x\widetilde{\phi},\partial^{\alpha}\partial_x\widetilde{\phi})|
&&\leq C\delta\|\partial_t\rho\|_{L^\infty}\|\partial^{\alpha}\partial_x\widetilde{\phi}\|^2
\notag\\
&&\leq \widetilde{C}(\delta+\delta^{1/2}\varepsilon+\frac{1}{\varepsilon^{1/2}}\delta^{5/4})\delta\|\partial^{\alpha}\partial_x\widetilde{\phi}\|^2
\notag\\
&&\leq C\delta\|\partial^{\alpha}\partial_x\widetilde{\phi}\|^2
\leq C\mathcal{E}_{2}(t).
\end{eqnarray*}
For the third term on the right hand side of \eqref{4.12}, 
by \eqref{3.2}, \eqref{3.48}, \eqref{3.2a} and \eqref{1.34}, one has
\begin{eqnarray*}
	\delta|(\frac{1}{\rho^2}\partial_x\widetilde{\rho}\partial^{\alpha}\widetilde{\phi},\partial^{\alpha}\partial_x\partial_t\widetilde{\phi})|
	&&\leq C\delta\|\partial^{\alpha}\widetilde{\phi}\|_{L^\infty}\|\partial^{\alpha}\partial_x\partial_t\widetilde{\phi}\|\|\partial_x\widetilde{\rho}\|
	\notag\\
	&&\leq  \widetilde{C}\delta^2\delta^{7/2}\|\partial^{\alpha}\partial_x\partial_t\widetilde{\phi}\|^2+\|\partial_x\widetilde{\rho}\|^2
	\notag\\
	&&\leq \widetilde{C}\delta\delta^{7/2}(\frac{1}{\delta^2}\|\partial^{\alpha}\partial_x(\widetilde{\rho},\widetilde{u})\|^{2}
	+\delta^4)+\|\partial_x\widetilde{\rho}\|^2
	\notag\\
	&&\leq \widetilde{C}\delta^{5/2}\|\partial^{\alpha}\partial_x(\widetilde{\rho},\widetilde{u})\|^{2}+C\mathcal{E}_{2}(t)+C\delta^6.
\end{eqnarray*}
For the last term of \eqref{4.12}, we shall deal with it carefully. First of all, we use the last equation of \eqref{2.33} to write
\begin{equation*}
	-(\partial_x\widetilde{\rho},\partial^{3}_{x}\widetilde{\phi})=\delta(\partial^{3}_{x}\widetilde{\phi},\partial^{3}_{x}\widetilde{\phi})
	-(e^{\phi}\partial_x\widetilde{\phi},\partial^{3}_{x}\widetilde{\phi})
	-((e^{\phi}-e^{\bar{\phi}})\partial_x\bar{\phi},\partial^{3}_{x}\widetilde{\phi})
	+\delta^3(\partial_x\mathcal{R}_{4},\partial^{3}_{x}\widetilde{\phi}).
\end{equation*}
Using \eqref{3.2}, \eqref{3.2a}, \eqref{2.31} and \eqref{3.57}, we obtain the following estimate
\begin{equation}
	\label{4.13}
	\delta(\partial^{3}_{x}\widetilde{\phi},\partial^{3}_{x}\widetilde{\phi})
	+(e^{\phi}\partial^2_x\widetilde{\phi},\partial^{2}_{x}\widetilde{\phi})
	\leq C\delta^4+C\frac{1}{\delta}\|\partial_x\widetilde{\rho}\|^2.
\end{equation}
If $|\alpha|=0$, then by \eqref{3.48}, \eqref{3.4}, \eqref{3.2a} and \eqref{1.34}, we get
\begin{eqnarray*}
	\delta|(\partial_x[\frac{1}{\rho^2}\partial_x\bar{\rho}\partial^{\alpha}\widetilde{\phi}],\partial^{\alpha}\partial_t\widetilde{\phi})|
	&&\leq 	C\delta^2\|\partial^{\alpha}\partial_t\widetilde{\phi}\|^2+C\|\partial_x[\frac{1}{\rho^2}\partial_x\bar{\rho}\partial^{\alpha}\widetilde{\phi}]\|^2
		\notag\\
	&&\leq C\|\partial_x(\widetilde{\rho},\widetilde{u})\|^{2}+\widetilde{C}\delta^6+C\delta^2\|\widetilde{\phi}\|_{H^1}^2
		\notag\\
	&&\leq C\mathcal{E}_{2}(t)+C\delta^5.
\end{eqnarray*}
If $|\alpha|=1$, then by integration by parts, \eqref{3.4}, \eqref{3.48}, \eqref{3.2} and \eqref{3.1}, we show that
\begin{eqnarray*}
	&&\delta|(\partial_x[\frac{1}{\rho^2}\partial_x\bar{\rho}\partial^{\alpha}\widetilde{\phi}],\partial^{\alpha}\partial_t\widetilde{\phi})|
	=\delta|(\partial^2_x[\frac{1}{\rho^2}\partial_x\bar{\rho}\partial^{\alpha}\widetilde{\phi}],\partial_t\widetilde{\phi})|
	\notag\\
	&&\leq C\delta^2\|\partial_t\widetilde{\phi}\|^2+C\delta^2\{\|\partial^{\alpha}\widetilde{\phi}\|^2
	+\|\partial^{\alpha}\partial_x\widetilde{\phi}\|^2+\|\partial^{\alpha}\partial^2_x\widetilde{\phi}\|^2+(\|\partial^2_x\widetilde{\rho}\|^2+\delta^2)
	\|\partial^{\alpha}\widetilde{\phi}\|^2_{L^\infty}\}
	\notag\\
	&&\leq C\mathcal{E}_{2}(t)+C\delta^5+C\delta^2\|\partial^{\alpha}\partial^2_x\widetilde{\phi}\|^2+\widetilde{C}\delta^5(
	\|\partial^2_x\widetilde{\rho}\|^2+\delta^2)
		\notag\\
	&&\leq C\mathcal{E}_{2}(t)+C\|\partial_x\widetilde{\rho}\|^2+\widetilde{C}\frac{\delta^6}{\varepsilon^2}\frac{\varepsilon^2}{\delta}
	\|\partial^2_x\widetilde{\rho}\|^2+C\delta^5\leq C\mathcal{E}_{2}(t)+C\delta^5,
\end{eqnarray*}
where in the last two inequalities we have used \eqref{4.13}, \eqref{1.34} and \eqref{4.5}. With these estimates,
 the last term of \eqref{4.12} can be controlled by
\begin{equation*}
\delta|(\partial_x[\frac{1}{\rho^2}\partial_x\bar{\rho}\partial^{\alpha}\widetilde{\phi}],\partial^{\alpha}\partial_t\widetilde{\phi})|
	\leq C\mathcal{E}_{2}(t)+C\delta^5.
\end{equation*}
Hence, plugging the above estimates into \eqref{4.12} gives 
\begin{eqnarray*}
	\delta(\frac{1}{\rho}\partial^{\alpha}\widetilde{\phi},\partial^{\alpha}\partial^{2}_{x}\partial_t\widetilde{\phi})
	\leq&&-\frac{1}{2}\delta\frac{d}{dt}(\frac{1}{\rho}\partial^{\alpha}\partial_x\widetilde{\phi},\partial^{\alpha}\partial_x\widetilde{\phi})
	\notag\\
	&&+\widetilde{C}\delta^{5/2}\|\partial^{\alpha}\partial_x(\widetilde{\rho},\widetilde{u})\|^{2}+C\mathcal{E}_{2}(t)+C\delta^5.
\end{eqnarray*}
For the second term on the right hand side of \eqref{4.11}, we see easily that
\begin{equation*}
	-(\frac{1}{\rho}\partial^{\alpha}\widetilde{\phi},\partial^{\alpha}(e^{\phi}\partial_t\widetilde{\phi}))
	=-(\frac{1}{\rho}\partial^{\alpha}\widetilde{\phi},e^{\phi}\partial^{\alpha}\partial_t\widetilde{\phi})
-(\frac{1}{\rho}\partial^{\alpha}\widetilde{\phi},\partial^{\alpha}
	(e^{\phi})\partial_t\widetilde{\phi}),
\end{equation*}
where if $|\alpha|=0$, the last term  vanishes, if $|\alpha|=1$, it is bounded by
\begin{eqnarray*}
C\|\partial^{\alpha}\phi\|_{L^{\infty}}\|\partial^{\alpha}\widetilde{\phi}\|\|\partial_t\widetilde{\phi}\|
\leq C\delta^2\|\partial_t\widetilde{\phi}\|^2+C\|\partial^{\alpha}\widetilde{\phi}\|^2	\leq C\mathcal{E}_{2}(t)+C\delta^5.
\end{eqnarray*}
In the above inequality,  we have used \eqref{3.48}, \eqref{3.57} and \eqref{1.34}.
By \eqref{3.67}, \eqref{4.5} and \eqref{1.34}, we get
\begin{eqnarray*}
	-(\frac{1}{\rho}\partial^{\alpha}\widetilde{\phi},e^{\phi}\partial^{\alpha}\partial_t\widetilde{\phi})
	\leq&& -\frac{1}{2}\frac{d}{dt}(\frac{1}{\rho}e^{\phi}\partial^{\alpha}\widetilde{\phi},\partial^{\alpha}\widetilde{\phi})
	+C\|\partial_t(\phi,\rho)\|_{L^{\infty}}\|\partial^{\alpha}\widetilde{\phi}\|^2
	\notag\\
	\leq&&   -\frac{1}{2}\frac{d}{dt}(\frac{1}{\rho}e^{\phi}\partial^{\alpha}\widetilde{\phi},\partial^{\alpha}\widetilde{\phi})
	+\widetilde{C}(\delta+\delta^{1/2}\varepsilon+\frac{1}{\varepsilon^{1/2}}\delta^{5/4})\|\partial^{\alpha}\widetilde{\phi}\|^2
	\notag\\
	\leq&&   -\frac{1}{2}\frac{d}{dt}(\frac{1}{\rho}e^{\phi}\partial^{\alpha}\widetilde{\phi},\partial^{\alpha}\widetilde{\phi})
	+C\mathcal{E}_{2}(t).
\end{eqnarray*}
So, from the above bounds, we get
\begin{equation*}
	-(\frac{1}{\rho}\partial^{\alpha}\widetilde{\phi},\partial^{\alpha}(e^{\phi}\partial_t\widetilde{\phi}))
	\leq   -\frac{1}{2}\frac{d}{dt}(\frac{1}{\rho}e^{\phi}\partial^{\alpha}\widetilde{\phi},\partial^{\alpha}\widetilde{\phi})
	+C\mathcal{E}_{2}(t)+C\delta^5.
\end{equation*}
The last two terms of \eqref{4.11} can be bounded by $C\mathcal{E}_{2}(t)+C\delta^5$.
Hence, we plug the above estimate into \eqref{4.11} to get
\begin{eqnarray*}
-(\frac{1}{\rho}\partial^{\alpha}\widetilde{\phi},\partial^{\alpha}\partial_t\widetilde{\rho})
\leq&&-\frac{1}{2}\frac{d}{dt}(\frac{1}{\rho}e^{\phi}\partial^{\alpha}\widetilde{\phi},\partial^{\alpha}\widetilde{\phi})-\frac{1}{2}\delta\frac{d}{dt}(\frac{1}{\rho}\partial^{\alpha}\partial_x\widetilde{\phi},\partial^{\alpha}\partial_x\widetilde{\phi})
\notag\\
&&+\widetilde{C}\delta^{5/2}\|\partial^{\alpha}\partial_x(\widetilde{\rho},\widetilde{u})\|^{2}+C\mathcal{E}_{2}(t)+C\delta^5.
\end{eqnarray*}

To bound the second term on the right hand side of \eqref{4.10}, we use the last equation of \eqref{2.33} to write
\begin{equation}
	\label{4.14}
A\frac{1}{\delta}(\frac{1}{\rho}\partial^{\alpha}\widetilde{\phi},\partial^{\alpha}\partial_x\widetilde{\rho})
=-A\frac{1}{\delta}(\frac{1}{\rho}\partial^{\alpha}\widetilde{\phi},\partial^{\alpha}[\delta\partial^{3}_{x}\widetilde{\phi}
-e^{\phi}\partial_x\widetilde{\phi}-(e^{\phi}-e^{\bar{\phi}})\partial_x\bar{\phi}+\delta^3\partial_x\mathcal{R}_{4}]).
\end{equation}
By integration by parts, \eqref{4.17a}, \eqref{4.13} and \eqref{1.34}, one obtains
\begin{eqnarray*}
-A(\frac{1}{\rho}\partial^{\alpha}\widetilde{\phi},\partial^{\alpha}\partial^{3}_{x}\widetilde{\phi})
=&&-\frac{1}{2}A(\partial_{x}[\frac{1}{\rho}]\partial^{\alpha}\partial_{x}\widetilde{\phi},\partial^{\alpha}\partial_{x}\widetilde{\phi})
+A(\partial_{x}[\frac{1}{\rho}]\partial^{\alpha}\widetilde{\phi},\partial^{\alpha}\partial^{2}_{x}\widetilde{\phi})
\notag\\
\leq&&
C\|\partial_{x}\rho\|_{L^\infty}\|\partial^{\alpha}\partial_{x}\widetilde{\phi}\|^2
+C\|\partial_{x}\rho\|^2_{L^\infty}\|\partial^{\alpha}\partial^2_{x}\widetilde{\phi}\|^2+C\|\partial^{\alpha}\widetilde{\phi}\|^2
\notag\\
\leq&&  C\delta\|\partial^{\alpha}\partial_{x}\widetilde{\phi}\|^2+C\delta^2\|\partial^{\alpha}\partial^2_{x}\widetilde{\phi}\|^2
+C\|\partial^{\alpha}\widetilde{\phi}\|^2
\notag\\
\leq&& C\mathcal{E}_{2}(t)+C\delta^5.
\end{eqnarray*}
On the other hand, it holds that
\begin{eqnarray*}
A\frac{1}{\delta}(\frac{1}{\rho}\partial^{\alpha}\widetilde{\phi},\partial^{\alpha}
[e^{\phi}\partial_x\widetilde{\phi}]))&&=-\frac{1}{2}
A\frac{1}{\delta}(\partial_x[\frac{1}{\rho}e^{\phi}]\partial^{\alpha}\widetilde{\phi},
\partial^{\alpha}\widetilde{\phi})+A\frac{1}{\delta}(\frac{1}{\rho}\partial^{\alpha}\widetilde{\phi},\partial^{\alpha}
[e^{\phi}]\partial_x\widetilde{\phi})
\notag\\
&&\leq C\|\widetilde{\phi}\|_{H^1}^2\leq C\mathcal{E}_{2}(t),
\end{eqnarray*}
and
\begin{equation*}
A\frac{1}{\delta}|(\frac{1}{\rho}\partial^{\alpha}\widetilde{\phi},\partial^{\alpha}[(e^{\phi}-e^{\bar{\phi}})\partial_x\bar{\phi}]-\delta^3\partial_x\partial^{\alpha}\mathcal{R}_{4})|\leq C\|\widetilde{\phi}\|_{H^1}^2+C\delta^4\leq C\mathcal{E}_{2}(t)+C\delta^4.
\end{equation*}
Hence, we substitute the above estimates into \eqref{4.14}  to obtain
\begin{equation*}
A\frac{1}{\delta}|(\frac{1}{\rho}\partial^{\alpha}\widetilde{\phi},\partial^{\alpha}\partial_x\widetilde{\rho})|\leq C\mathcal{E}_{2}(t)+C\delta^4.
\end{equation*}

For the third term on the right hand side of \eqref{4.10}. We consider two cases  $|\alpha|=0$ and  $|\alpha|=1$.
For the former case, we obtain by \eqref{3.2}, \eqref{3.4} and \eqref{1.34} that
\begin{eqnarray*}
	\frac{1}{\delta}|(\frac{1}{\rho}\partial^{\alpha}\widetilde{\phi},\partial^{\alpha}[u_{1}\partial_x\widetilde{\rho}])|
	&&\leq C\frac{1}{\delta}\|u_1\|_{L^{\infty}}\|\widetilde{\phi}\|\|\partial_x\widetilde{\rho}\|\leq C\mathcal{E}_{2}(t).
\end{eqnarray*}
For the latter case, it holds that
\begin{eqnarray*}
	\frac{1}{\delta}|(\frac{1}{\rho}\partial^{\alpha}\widetilde{\phi},\partial^{\alpha}[u_{1}\partial_x\widetilde{\rho}])|
	=&&\frac{1}{\delta}|(\partial_x[\frac{1}{\rho}]\partial_x\widetilde{\phi},u_{1}\partial_x\widetilde{\rho})+
	(\frac{1}{\rho}\partial^2_x\widetilde{\phi},u_{1}\partial_x\widetilde{\rho})|
	\notag\\
	\leq&& C\|\widetilde{\phi}\|_{H^1}^2+C\|\partial_x\widetilde{\rho}\|^2
	+C\delta\|\partial^2_x\widetilde{\phi}\|^2+C\delta^5
	\notag\\
	\leq&& C\mathcal{E}_{2}(t)+C\delta^4.
\end{eqnarray*}
Here we have used \eqref{4.17a}, \eqref{2.31}, \eqref{3.4} and the last equation of \eqref{2.33} such that
	\begin{eqnarray}
		\label{4.15}
		\frac{1}{\delta}|(\frac{1}{\rho}\partial^2_x\widetilde{\phi},u_{1}\partial_x\widetilde{\rho})|
		=&&\frac{1}{\delta}|(\frac{u_{1}}{\rho}\partial^2_x\widetilde{\phi},\delta\partial^{3}_{x}\widetilde{\phi}
		-e^{\phi}\partial_x\widetilde{\phi}
		-(e^{\phi}-e^{\bar{\phi}})\partial_x\bar{\phi}+\delta^3\partial_x\mathcal{R}_{4})|	
		\notag\\
		\leq&& \frac{1}{2}|(\partial_{x}[\frac{u_{1}}{\rho}]\partial^2_x\widetilde{\phi},\partial^{2}_{x}\widetilde{\phi})|
		+\frac{1}{2}\frac{1}{\delta}|(\partial_{x}[\frac{u_{1}}{\rho}e^{\phi}]\partial_x\widetilde{\phi},\partial_x\widetilde{\phi})|
		\notag\\
		&&+\frac{1}{\delta}|(\frac{u_{1}}{\rho}\partial^2_x\widetilde{\phi},(e^{\phi}-e^{\bar{\phi}})\partial_x\bar{\phi})|
		+\delta^2|(\frac{u_{1}}{\rho}\partial^2_x\widetilde{\phi},\partial_x\mathcal{R}_{4})|
		\notag\\
		\leq&& C\|\widetilde{\phi}\|_{H^1}^2+C\delta\|\partial^2_x\widetilde{\phi}\|^2+C\delta^5.
	\end{eqnarray}
From the above estimates, we see easily that
\begin{equation*}
	\frac{1}{\delta}|(\frac{1}{\rho}\partial^{\alpha}\widetilde{\phi},\partial^{\alpha}[u_{1}\partial_x\widetilde{\rho}])|
		\leq C\mathcal{E}_{2}(t)+C\delta^4.
\end{equation*}
 The last term of \eqref{4.10} can be bounded by $C\mathcal{E}_{2}(t)+C\delta^4$. Hence, plugging these estimates into \eqref{4.10}
and using \eqref{4.9}, we conclude  that
\begin{eqnarray}
	\label{4.16}
-\frac{1}{\delta}(\partial^{\alpha}\partial_x\widetilde{\phi},\partial^{\alpha}\widetilde{u}_1)
\leq&&-\frac{1}{2}\frac{d}{dt}(\frac{1}{\rho}e^{\phi}\partial^{\alpha}\widetilde{\phi},\partial^{\alpha}\widetilde{\phi})-\frac{1}{2}\delta\frac{d}{dt}(\frac{1}{\rho}\partial^{\alpha}\partial_x\widetilde{\phi},\partial^{\alpha}\partial_x\widetilde{\phi})
\notag\\
&&+\widetilde{C}\delta^{5/2}\|\partial^{\alpha}\partial_x(\widetilde{\rho},\widetilde{u})\|^{2}+C\mathcal{E}_{2}(t)+C\delta^4.
\end{eqnarray}
This completes the estimates of the electric potential term in \eqref{4.8}.

To estimate the viscosity coefficient term in \eqref{4.8}, we first write
\begin{equation*}
\frac{\varepsilon}{\delta^{1/2}}(\partial^{\alpha}[\frac{4}{3\rho}\partial_x(\mu(\theta)\partial_xu_1)],\partial^{\alpha}\widetilde{u}_1)
=\frac{\varepsilon}{\delta^{1/2}}(\partial^{\alpha}[\frac{4}{3\rho}\partial_x(\mu(\theta)\partial_x\bar{u}_1)],\partial^{\alpha}\widetilde{u}_1)
+\frac{\varepsilon}{\delta^{1/2}}(\partial^{\alpha}[\frac{4}{3\rho}\partial_x(\mu(\theta)\partial_x\widetilde{u}_1)],\partial^{\alpha}\widetilde{u}_1).
\end{equation*}
Since both $\mu(\theta)$ and $\kappa(\theta)$ are smooth functions of $\theta$, there exists a constant $C>1$ 
such that $\mu(\theta),\kappa(\theta)\in[C^{-1},C]$. Moreover, derivatives of $\mu(\theta)$ and $\kappa(\theta)$ are also bounded.
Using these facts and \eqref{4.17a} as well as \eqref{1.34}, we show that
\begin{eqnarray*}
&&\frac{\varepsilon}{\delta^{1/2}}(\partial^{\alpha}[\frac{4}{3\rho}\partial_x(\mu(\theta)\partial_x\bar{u}_1)],\partial^{\alpha}\widetilde{u}_1)
\notag\\
&&\leq C\|\partial^{\alpha}\widetilde{u}_1\|^2+C\frac{\varepsilon^2}{\delta}(\|\partial^{\alpha}\partial^2_x\bar{u}_1\|^2+
\|\partial^{\alpha}(\rho,\theta)\|_{L^\infty}^2\|\partial^2_x\bar{u}_1\|^2)
\notag\\
&&\hspace{0.5cm}+C\frac{\varepsilon^2}{\delta}(\|\partial_x\theta\|_{L^\infty}^2\|\partial^{\alpha}\partial_x\bar{u}_1\|^2
+\|\partial^{\alpha}\partial_x\theta\|^2\|\partial_x\bar{u}_1\|_{L^\infty}^2
+\|\partial^{\alpha}\rho\|_{L^\infty}^2\|\partial_x\theta\|_{L^\infty}^2\|\partial_x\bar{u}_1\|^2)
\notag\\
&&\leq C\|\partial^{\alpha}\widetilde{u}_1\|^2+C\varepsilon^2\delta\|\partial^{\alpha}\partial_x\widetilde{\theta}\|^2+C\varepsilon^2\delta
\notag\\
&&\leq C\mathcal{E}_{2}(t)+C\varepsilon^2\delta.
\end{eqnarray*}
On the other hand, after integration by parts, we claim that
\begin{eqnarray*}
	\label{4.17}
&&\frac{\varepsilon}{\delta^{1/2}}(\partial^{\alpha}[\frac{4}{3\rho}\partial_x(\mu(\theta)\partial_x\widetilde{u}_1)],\partial^{\alpha}\widetilde{u}_1)
\notag\\
&&=-\frac{\varepsilon}{\delta^{1/2}}(\partial^{\alpha}[\frac{4}{3\rho}\mu(\theta)\partial_x\widetilde{u}_1],\partial^{\alpha}\partial_x\widetilde{u}_1)
-\frac{\varepsilon}{\delta^{1/2}}(\partial^{\alpha}[\partial_x(\frac{4}{3\rho})\mu(\theta)\partial_x\widetilde{u}_1],\partial^{\alpha}\widetilde{u}_1)
	\notag\\
&&\leq -\frac{\varepsilon}{\delta^{1/2}}(\frac{4}{3\rho}\mu(\theta)\partial^{\alpha}\partial_x\widetilde{u}_1,\partial^{\alpha}\partial_x\widetilde{u}_1)
+C\mathcal{E}_{2}(t).
\end{eqnarray*}
Hence, we deduce from the above three estimates that
\begin{equation}
	\label{4.18}
	\frac{\varepsilon}{\delta^{1/2}}(\partial^{\alpha}[\frac{4}{3\rho}\partial_x(\mu(\theta)\partial_xu_1)],\partial^{\alpha}\widetilde{u}_1)
	\leq-\frac{\varepsilon}{\delta^{1/2}}(\frac{4}{3\rho}\mu(\theta)\partial^{\alpha}\partial_x\widetilde{u}_1,\partial^{\alpha}\partial_x\widetilde{u}_1)
+C\mathcal{E}_{2}(t)+C\varepsilon^2\delta.	
\end{equation}

Finally, it remains to compute the last term of \eqref{4.8}. First of all, we 
use the self-adjoint property of $L^{-1}_{M}$ and \eqref{2.15} to get the following two identities that
\begin{eqnarray}
	\label{4.20}
	\int_{\mathbb{R}^{3}}v_{1}v_{j}L^{-1}_{M}\Theta \,dv=&&{\int_{\mathbb{R}^{3}}\frac{v_{1}v_{j}ML^{-1}_{M}\Theta}{M} \,dv}=
	\int_{\mathbb{R}^{3}} \frac{L^{-1}_{M}\{P_{1}(v_{1}v_{j}M)\}\Theta}{M}\,dv
	\notag\\
	=&&\int_{\mathbb{R}^{3}} L^{-1}_{M}\{K\theta\hat{B}_{1j}(\frac{v-u}{\sqrt{K\theta}})M\}\frac{\Theta}{M}\,dv
	=K\theta\int_{\mathbb{R}^{3}}B_{1j}(\frac{v-u}{\sqrt{K\theta}})\frac{\Theta}{M}\,dv
\end{eqnarray}
and 
\begin{eqnarray}
	\label{4.19}
	&&\int_{\mathbb{R}^{3}}(\frac{1}{2}v_{i}|v|^{2}-v_{i}u\cdot v)L^{-1}_{M}\Theta \,dv=
	\int_{\mathbb{R}^{3}} L^{-1}_{M}\{P_{1}(\frac{1}{2}v_{i}|v|^{2}-v_{i}u\cdot v)M\}\frac{\Theta}{M}\,dv
	\notag\\
	&&=\int_{\mathbb{R}^{3}} L^{-1}_{M}\{(K\theta)^{\frac{3}{2}}\hat{A}_{i}(\frac{v-u}{\sqrt{K\theta}})M\}\frac{\Theta}{M}\,dv
	=(K\theta)^{\frac{3}{2}}\int_{\mathbb{R}^{3}}A_{i}(\frac{v-u}{\sqrt{K\theta}})\frac{\Theta}{M}\,dv,
\end{eqnarray}
for $i,j=1,2,3$. Using \eqref{4.20} and the integration by parts, we denote that
\begin{eqnarray}
	\label{4.21}
	&&-\frac{1}{\delta}(\partial^{\alpha}[\frac{1}{\rho}\partial_x(\int_{\mathbb{R}^{3}} v_{1}^2L^{-1}_{M}\Theta\,dv)],\partial^{\alpha}\widetilde{u}_1)
	\notag\\
	=&&-\frac{1}{\delta}(\partial^{\alpha}[\frac{1}{\rho}\partial_{x}(\int_{\mathbb{R}^{3}}K\theta B_{11}(\frac{v-u}{\sqrt{K\theta}})\frac{\Theta}{M}\, dv)],\partial^{\alpha}\widetilde{u}_1)
	\notag\\
	=&&\frac{1}{\delta}(\partial^{\alpha}[\frac{1}{\rho}\int_{\mathbb{R}^{3}}K\theta B_{11}(\frac{v-u}{\sqrt{K\theta}})\frac{\Theta}{M}\, dv],\partial^{\alpha}\partial_{x}\widetilde{u}_1)
	\notag\\
	&&+\frac{1}{\delta}(\partial^{\alpha}[\partial_{x}(\frac{1}{\rho})\int_{\mathbb{R}^{3}}K\theta B_{11}(\frac{v-u}{\sqrt{K\theta}})\frac{\Theta}{M}\, dv],\partial^{\alpha}\widetilde{u}_1)
	\notag\\
	:=&&I_{5}+I_{6}.
\end{eqnarray}
To bound \eqref{4.21}, we first use  \eqref{3.6} and \eqref{3.11} to obtain the following desired estimate
\begin{equation}
	\label{4.22}
	\int_{\mathbb{R}^{3}}\frac{|\langle v\rangle^{b}\sqrt{\mu}
		\partial_{\beta}A_{1}(\frac{v-u}{\sqrt{K\theta}})|^{2}}{M^{2}}\,dv
	+\int_{\mathbb{R}^{3}}\frac{|\langle v\rangle^{b}
		\sqrt{\mu}\partial_{\beta}B_{1j}(\frac{v-u}{\sqrt{K\theta}})|^{2}}{M^{2}}\,dv\leq C,
\end{equation}
for any multi-index $\beta$ and $b\geq 0$. Recall $\Theta$ in \eqref{2.12} given by
\begin{equation*}
	\Theta:=\delta^{3/2}\varepsilon \partial_tG-A\delta^{1/2}\varepsilon \partial_xG+\delta^{1/2}\varepsilon P_{1}(v_{1}\partial_xG)-\delta^{1/2}\varepsilon\partial_x\phi\partial_{v_{1}}G-Q(G,G).
\end{equation*}
Then we easily see that
\begin{eqnarray}
	\label{4.23}
	I_{5}=&&\delta^{1/2}\varepsilon\int_{\mathbb{R}}\int_{\mathbb{R}^{3}}\partial^{\alpha}[\frac{1}{\rho}K\theta B_{11}(\frac{v-u}{\sqrt{K\theta}})\frac{\partial_tG}{M}]\,\partial^{\alpha}\partial_{x}\widetilde{u}_1\,dv\,dx
	\notag\\
	&&-\frac{\varepsilon}{\delta^{1/2}}\int_{\mathbb{R}}\int_{\mathbb{R}^{3}}\partial^{\alpha}[\frac{1}{\rho}K\theta B_{11}(\frac{v-u}{\sqrt{K\theta}})\frac{A \partial_xG}{M}]\,\partial^{\alpha}\partial_{x}\widetilde{u}_1\,dv\,dx
	\notag\\
	&&+\frac{\varepsilon}{\delta^{1/2}}\int_{\mathbb{R}}\int_{\mathbb{R}^{3}}\partial^{\alpha}[\frac{1}{\rho}K\theta B_{11}(\frac{v-u}{\sqrt{K\theta}})\frac{ P_{1}(v_1\partial_xG)}{M}]\,\partial^{\alpha}\partial_{x}\widetilde{u}_1\,dv\,dx
	\notag\\
	&&-\frac{\varepsilon}{\delta^{1/2}}\int_{\mathbb{R}}\int_{\mathbb{R}^{3}}\partial^{\alpha}[\frac{1}{\rho}K\theta B_{11}(\frac{v-u}{\sqrt{K\theta}})
	\frac{\partial_x\phi\partial_{v_{1}}G}{M}]\,\partial^{\alpha}\partial_{x}\widetilde{u}_1\,dv\,dx
	\notag\\
	&&-\frac{1}{\delta}\int_{\mathbb{R}}\int_{\mathbb{R}^{3}}\partial^{\alpha}[\frac{1}{\rho}K\theta B_{11}(\frac{v-u}{\sqrt{K\theta}})\frac{Q(G,G)}{M}]\,\partial^{\alpha}\partial_{x}\widetilde{u}_1\,dv\,dx
	\notag\\
	:=&&I^{1}_{5}+I^{2}_{5}+I^{3}_{5}+I^{4}_{5}+I^5_5.
\end{eqnarray}
In the following we estimate each $I^{i}_{5}$ $(i=1,2,3,4,5)$ in \eqref{4.23}. By $G=\overline{G}+\sqrt{\mu}f$, we write
\begin{eqnarray*}
	I^{1}_{5}=&&\delta^{1/2}\varepsilon\int_{\mathbb{R}}\int_{\mathbb{R}^{3}}\partial^{\alpha}[\frac{1}{\rho}K\theta B_{11}(\frac{v-u}{\sqrt{K\theta}})\frac{\partial_t\overline{G}}{M}]\,\partial^{\alpha}\partial_{x}\widetilde{u}_1\,dv\,dx
	\notag\\
	&&+\delta^{1/2}\varepsilon\int_{\mathbb{R}}\int_{\mathbb{R}^{3}}\partial^{\alpha}[\frac{1}{\rho}K\theta B_{11}(\frac{v-u}{\sqrt{K\theta}})\frac{\sqrt{\mu}\partial_tf}{M}]\,\partial^{\alpha}\partial_{x}\widetilde{u}_1 \,dv\,dx
	\notag\\
	:=&&I^{11}_{5}+I^{12}_{5}.
\end{eqnarray*}
Note that $\partial_t\overline{G}$ has the similar expression as \eqref{3.10}, for $|\alpha|\leq 1$ and any $b\geq0$,
by the similar calculations as \eqref{3.13} and \eqref{3.15}, we deduce from \eqref{3.4}, \eqref{3.45}, \eqref{3.46}, \eqref{3.1}, \eqref{3.2a} and \eqref{1.33} that
\begin{eqnarray}
	\label{4.24}
	\|\langle v\rangle^{b}\frac{\partial^{\alpha}\partial_t\overline{G}}{\sqrt{\mu}}\|^2_{2,w}
		&&\leq  C\delta\varepsilon^2(\delta^2+\delta^4+\delta^2\|\partial_t(u,\theta)\|^2
		+\delta^2\|\partial^{\alpha}\partial_t(u,\theta)\|^2)
		\notag\\
		&&\leq  C\delta^3\varepsilon^2+\widetilde{C}\delta^6\leq C\delta^3(\varepsilon^2+\delta^2),
\end{eqnarray}
which immediately gives 
\begin{eqnarray*}
	|I^{11}_5|&&\leq \eta\frac{\varepsilon}{\delta^{1/2}}\int_{\mathbb{R}}|\partial^{\alpha}\partial_{x}\widetilde{u}_1|^2\,dx+C_\eta\varepsilon\delta^{3/2}
	\int_{\mathbb{R}}|
	\int_{\mathbb{R}^{3}}\partial^{\alpha}[\frac{1}{\rho}K\theta B_{11}(\frac{v-u}{\sqrt{K\theta}})\frac{\partial_t\overline{G}}{M}]\, dv|^2\,dx
	\notag\\
	&&\leq \eta\frac{\varepsilon}{\delta^{1/2}}\|\partial^{\alpha}\partial_{x}\widetilde{u}_1\|^{2}+C_\eta\varepsilon\delta^{9/2}
	 (\varepsilon^2+\delta^2).
\end{eqnarray*}
As for the term $I^{12}_{5}$, we see easily that
\begin{eqnarray}
	\label{4.25}
	I^{12}_{5}=&&\delta^{1/2}\varepsilon\frac{d}{dt}\int_{\mathbb{R}}\int_{\mathbb{R}^{3}}\partial^{\alpha}
	\big\{[\frac{1}{\rho}K\theta B_{11}(\frac{v-u}{\sqrt{K\theta}})\frac{\sqrt{\mu}}{M}]f\big\}\partial^{\alpha}\partial_{x}\widetilde{u}_1\,dv\,dx
	\nonumber\\
	&&-\delta^{1/2}\varepsilon\int_{\mathbb{R}}\int_{\mathbb{R}^{3}}\partial^{\alpha}	\big\{\partial_t[\frac{1}{\rho}K\theta B_{11}(\frac{v-u}{\sqrt{K\theta}})\frac{\sqrt{\mu}}{M}]f\big\}\partial^{\alpha}\partial_{x}\widetilde{u}_1\,dv\,dx
	\nonumber\\
	&&-\delta^{1/2}\varepsilon\int_{\mathbb{R}}\int_{\mathbb{R}^{3}}\partial^{\alpha}
	\big\{[\frac{1}{\rho}K\theta B_{11}(\frac{v-u}{\sqrt{K\theta}})\frac{\sqrt{\mu}}{M}]f\big\}\partial^{\alpha}\partial_{x}\partial_t\widetilde{u}_1\,dv\,dx.
\end{eqnarray}
By integration by parts, \eqref{4.22}, \eqref{4.17a} and \eqref{3.1}, the last term of \eqref{4.25} can be bounded by
\begin{eqnarray*}
	&&\delta^{1/2}\varepsilon\int_{\mathbb{R}}\int_{\mathbb{R}^{3}}\partial^{\alpha}\partial_{x}
	\big\{[\frac{1}{\rho}K\theta B_{11}(\frac{v-u}{\sqrt{K\theta}})\frac{\sqrt{\mu}}{M}]f\big\}\partial^{\alpha}\partial_t\widetilde{u}_1\,dv\,dx
	\notag\\
	&&\leq \eta\frac{\varepsilon}{\delta^{1/2}}\delta^2\|\partial^{\alpha}\partial_{t}\widetilde{u}_1\|^{2}+C_\eta\frac{\varepsilon}{\delta^{1/2}}
	\int_{\mathbb{R}}|\int_{\mathbb{R}^{3}}\partial^{\alpha}\partial_{x}
	\big\{[\frac{1}{\rho}K\theta B_{11}(\frac{v-u}{\sqrt{K\theta}})\frac{\sqrt{\mu}}{M}]f\big\}\,dv|^{2}\,dx
	\notag\\
	&&\leq C\eta\varepsilon\delta^{3/2}\|\partial^{\alpha}\partial_{t}\widetilde{u}_1\|^{2}
	+C_\eta\frac{\varepsilon}{\delta^{1/2}}\|\langle v\rangle^{-\frac{1}{2}}\partial^{\alpha}\partial_xf\|^{2}
	\notag\\
	&&\hspace{0.5cm}+C_\eta\frac{\varepsilon}{\delta^{1/2}}\|\partial_x(\rho,u,\theta)\|^2_{L^\infty}\|\partial^{\alpha}f\|^{2}
	+C_\eta\frac{\varepsilon}{\delta^{1/2}}(\|\partial^{\alpha}\partial_x(\rho,u,\theta)\|^2
	+\|\partial^{\alpha}(\rho,u,\theta)\|^2_{L^\infty}\|\partial_x(\rho,u,\theta)\|^2)\||f|_2\|_{L^\infty}^{2}
\notag\\
&&\leq C\eta\frac{\varepsilon}{\delta^{1/2}}\|\partial^{\alpha}\partial_x(\widetilde{\rho},
\widetilde{u},\widetilde{\theta},\widetilde{\phi})\|^{2}
+C_\eta\frac{\varepsilon}{\delta^{1/2}}\|\partial^{\alpha}\partial_xf\|_\sigma^{2}	
	\notag\\
&&\hspace{0.5cm}
+C_\eta\widetilde{C}\varepsilon\delta^{7/2}\|\partial^{\alpha}\partial_x(\widetilde{\rho},
\widetilde{u},\widetilde{\theta})\|^{2}	+C_\eta\varepsilon\delta^{3/2}
	(\delta\varepsilon^2+\widetilde{C}\delta^4),
\end{eqnarray*}
where in the last inequality we have used \eqref{3.46} and the first inequality in \eqref{3.61B}.
The second term on the right hand side of \eqref{4.25}  can be treated in an almost similar way, and it is bounded by
\begin{eqnarray*}
	&&\eta\frac{\varepsilon}{\delta^{1/2}}\|\partial^{\alpha}\partial_{x}\widetilde{u}_1\|^{2}
	+C_\eta\delta^{3/2}\varepsilon\int_{\mathbb{R}}|\int_{\mathbb{R}^{3}}\partial^{\alpha}	\big\{\partial_t[\frac{1}{\rho}K\theta B_{11}(\frac{v-u}{\sqrt{K\theta}})\frac{\sqrt{\mu}}{M}]f\big\}|\,dv|^2\,dx
		\notag\\
\leq&&
\eta\frac{\varepsilon}{\delta^{1/2}}\|\partial^{\alpha}\partial_{x}\widetilde{u}_1\|^{2}
	+C_\eta\widetilde{C}\varepsilon\delta^{3/2}\{\|\partial^{\alpha}\partial_x(\widetilde{\rho},
	\widetilde{u},\widetilde{\theta},\widetilde{\phi})\|^{2}
	+\|\partial^{\alpha}\partial_xf\|_\sigma^{2}\}
	+C_\eta\widetilde{C}\varepsilon\delta^{3/2}(\delta\varepsilon^2+\delta^4).
\end{eqnarray*}
Hence, substituting the above two estimates into \eqref{4.25} and using \eqref{3.2a}, one gets
\begin{eqnarray*}
I^{12}_{5}\leq&&\delta^{1/2}\varepsilon\frac{d}{dt}\int_{\mathbb{R}}\int_{\mathbb{R}^{3}}\partial^{\alpha}
\big\{[\frac{1}{\rho}K\theta B_{11}(\frac{v-u}{\sqrt{K\theta}})\frac{\sqrt{\mu}}{M}]f\big\}\partial^{\alpha}\partial_{x}\widetilde{u}_1\,dv\,dx	
\notag\\
&&+(C\eta+C_\eta\delta^{3/2})\frac{\varepsilon}{\delta^{1/2}}\|\partial^{\alpha}\partial_x(\widetilde{\rho},
\widetilde{u},\widetilde{\theta},\widetilde{\phi})\|^{2}
+C_\eta\frac{\varepsilon}{\delta^{1/2}}\|\partial^{\alpha}\partial_xf\|_\sigma^{2}+C_\eta\varepsilon\delta
(\delta\varepsilon^2+\delta^4).
\end{eqnarray*}
We combine the estimates $I^{11}_{5}$ with $I^{12}_{5}$ to obtain 
\begin{eqnarray*}
	I^1_{5}\leq&&\delta^{1/2}\varepsilon\frac{d}{dt}\int_{\mathbb{R}}\int_{\mathbb{R}^{3}}\partial^{\alpha}
	\big\{[\frac{1}{\rho}K\theta B_{11}(\frac{v-u}{\sqrt{K\theta}})\frac{\sqrt{\mu}}{M}]f\big\}\partial^{\alpha}\partial_{x}\widetilde{u}_1\,dv\,dx	
	\notag\\
	&&+(C\eta+C_\eta\delta^{3/2})\frac{\varepsilon}{\delta^{1/2}}\|\partial^{\alpha}\partial_x(\widetilde{\rho},
	\widetilde{u},\widetilde{\theta},\widetilde{\phi})\|^{2}
	+C_\eta\frac{\varepsilon}{\delta^{1/2}}\|\partial^{\alpha}\partial_xf\|_\sigma^{2}+C_\eta\varepsilon\delta
	(\delta\varepsilon^2+\delta^4).
\end{eqnarray*}
On the other hand, we can bound $I^{2}_{5}$, $I^{3}_{5}$ and  $I^{4}_{5}$ by
\begin{equation*}
|I^{2}_{5}|+|I^{3}_{5}|+|I^{4}_{5}|\leq C\eta\frac{\varepsilon}{\delta^{1/2}}\|\partial^{\alpha}\partial_{x}\widetilde{u}_1\|^{2}+C_\eta\frac{\varepsilon}{\delta^{1/2}}\|\partial^{\alpha}\partial_xf\|_\sigma^{2}
+C_\eta\varepsilon\delta(\delta\varepsilon^2+\delta^4).
\end{equation*}
For the term $I^{5}_{5}$ in \eqref{4.23}, we have from \eqref{2.38},\eqref{4.22} and \eqref{3.20} that
\begin{eqnarray*}
I^{5}_{5}=&&-\frac{1}{\delta}\int_{\mathbb{R}}\int_{\mathbb{R}^{3}}\partial^{\alpha}[\frac{1}{\rho}K\theta B_{11}(\frac{v-u}{\sqrt{K\theta}})\frac{\sqrt{\mu}}{M}\Gamma(\frac{G}{\sqrt{\mu}},\frac{G}{\sqrt{\mu}})]\,\partial^{\alpha}\partial_{x}\widetilde{u}_1\,dv\,dx
\notag\\
\leq&& C\frac{1}{\delta}\|\partial^{\alpha}(\rho,u,\theta)\|_{L^\infty}
\int_{\mathbb{R}}|\mu^\epsilon\frac{G}{\sqrt{\mu}}|_{2}|\frac{G}{\sqrt{\mu}}|_{\sigma}|\partial^{\alpha}\partial_{x}\widetilde{u}_1|\,dx
\notag\\
&&+C\frac{1}{\delta}
\int_{\mathbb{R}}|\mu^\epsilon\frac{\partial^{\alpha}G}{\sqrt{\mu}}|_{2}|\frac{G}{\sqrt{\mu}}|_{\sigma}|\partial^{\alpha}\partial_{x}\widetilde{u}_1|\,dx
+C\frac{1}{\delta}
\int_{\mathbb{R}}|\mu^\epsilon\frac{G}{\sqrt{\mu}}|_{2}|\frac{\partial^{\alpha}G}{\sqrt{\mu}}|_{\sigma}|\partial^{\alpha}\partial_{x}\widetilde{u}_1|\,dx,
\end{eqnarray*}
which can be further bounded by
\begin{eqnarray*}
	&&  C\eta\frac{\varepsilon}{\delta^{1/2}}\|\partial^{\alpha}\partial_{x}\widetilde{u}_1\|^2+
	C_\eta\frac{1}{\delta^{3/2}\varepsilon}\||\mu^\epsilon\frac{G}{\sqrt{\mu}}|^2_{2}\|_{L^\infty}\|\frac{G}{\sqrt{\mu}}\|^2_{\sigma}
	\notag\\
	&&+
	C_\eta\frac{1}{\delta^{3/2}\varepsilon}\|\mu^\epsilon\frac{\partial^{\alpha}G}{\sqrt{\mu}}\|^2_{2}\||\frac{G}{\sqrt{\mu}}|^2_{\sigma}\|_{L^\infty}
	+C_\eta\frac{1}{\delta^{3/2}\varepsilon}\||\mu^\epsilon\frac{G}{\sqrt{\mu}}|^2_{2}\|_{L^\infty}\|\frac{\partial^{\alpha}G}{\sqrt{\mu}}\|^2_{\sigma}
	\notag\\
	\notag\\
	\leq&&  C\eta\frac{\varepsilon}{\delta^{1/2}}\|\partial^{\alpha}\partial_{x}\widetilde{u}_1\|^2
	+C_\eta[\delta^{3}(\varepsilon+\delta^2)^2+\widetilde{C}\delta^4]\frac{1}{\delta^{3/2}\varepsilon}[\|f\|^2_{\sigma}+\|\partial_xf\|^2_{\sigma}
	+\delta^{3}(\varepsilon+\delta^2)^2]
	\notag\\
	\leq&&  C\eta\frac{\varepsilon}{\delta^{1/2}}\|\partial^{\alpha}\partial_{x}\widetilde{u}_1\|^2
	+C_\eta\delta^{3}\frac{1}{\delta^{3/2}\varepsilon}(\|f\|^2_{\sigma}+\|\partial_xf\|^2_{\sigma})
	+C_\eta\delta^{6}\frac{1}{\delta^{3/2}\varepsilon}(\varepsilon+\delta^2)^2.
\end{eqnarray*}
In the last two inequalities we have used \eqref{3.7}, \eqref{1.35}, \eqref{3.1} and \eqref{3.2a}.
Therefore,  we substitute the estimates on $I^{1}_{5}$, $I^{2}_{5}$, $I^{3}_{5}$, $I^{4}_{5}$  and $I^{5}_{5}$ into \eqref{4.23} to obtain
\begin{eqnarray}
	\label{4.26}
	I_{5}\leq&&\delta^{1/2}\varepsilon\frac{d}{dt}\int_{\mathbb{R}}\int_{\mathbb{R}^{3}}\partial^{\alpha}
	\big\{[\frac{1}{\rho}K\theta B_{11}(\frac{v-u}{\sqrt{K\theta}})\frac{\sqrt{\mu}}{M}]f\big\}\partial^{\alpha}\partial_{x}\widetilde{u}_1\,dv\,dx	
	\notag\\
	&&+(C\eta+C_\eta\delta^{3/2})\frac{\varepsilon}{\delta^{1/2}}\|\partial^{\alpha}\partial_x(\widetilde{\rho},
	\widetilde{u},\widetilde{\theta},\widetilde{\phi})\|^{2}
	+C_\eta\frac{\varepsilon}{\delta^{1/2}}\|\partial^{\alpha}\partial_xf\|_\sigma^{2}
	\notag\\
	&&+C_\eta\delta^{3}\frac{1}{\delta^{3/2}\varepsilon}(\|f\|^2_{\sigma}+\|\partial_xf\|^2_{\sigma}
	)+C_\eta\delta^{6}\frac{1}{\delta^{3/2}\varepsilon}(\varepsilon^2+\delta^4)+C_\eta\delta^2\varepsilon^3.
\end{eqnarray}
The term $I_{6}$ in \eqref{4.21} has the similar structure as the term $I_5$, and it can be done in an  almost similar
way as $I_5$, so it holds that
\begin{eqnarray*}
I_{6}\leq&&\delta^{1/2}\varepsilon\frac{d}{dt}\int_{\mathbb{R}}\int_{\mathbb{R}^{3}}\partial^{\alpha}
	\big\{\partial_{x}[\frac{1}{\rho}]K\theta B_{11}(\frac{v-u}{\sqrt{K\theta}})\frac{\sqrt{\mu}}{M}f\big\}\partial^{\alpha}\widetilde{u}_1\,dv\,dx	
	\notag\\
	&&+(C\eta+C_\eta\delta^{3/2})\frac{\varepsilon}{\delta^{1/2}}\|\partial^{\alpha}\partial_x(\widetilde{\rho},
	\widetilde{u},\widetilde{\theta},\widetilde{\phi})\|^{2}
	+C_\eta\frac{\varepsilon}{\delta^{1/2}}\|\partial^{\alpha}\partial_xf\|_\sigma^{2}
	\notag\\
	&&+C_\eta\delta^{3}\frac{1}{\delta^{3/2}\varepsilon}(\|f\|^2_{\sigma}+\|\partial_xf\|^2_{\sigma}
	)+C_\eta\delta^{6}\frac{1}{\delta^{3/2}\varepsilon}(\varepsilon^2+\delta^4)+C_\eta\delta^2\varepsilon^3.
\end{eqnarray*}
Hence, substituting the estimates of $I_{5}$ and $I_{6}$ into \eqref{4.21} yields that
\begin{eqnarray}
	\label{4.27}	
	&&-\frac{1}{\delta}(\partial^{\alpha}[\frac{1}{\rho}\partial_x(\int_{\mathbb{R}^{3}} v_{1}^2L^{-1}_{M}\Theta\,dv)],\partial^{\alpha}\widetilde{u}_1).
	\notag\\
	\leq&&-\delta^{1/2}\varepsilon\frac{d}{dt}\int_{\mathbb{R}}\int_{\mathbb{R}^{3}}
	\big\{\partial^{\alpha}[\frac{1}{\rho}\partial_{x}(K\theta B_{11}(\frac{v-u}{\sqrt{K\theta}})\frac{\sqrt{\mu}}{M}f)]
	\partial^{\alpha}\widetilde{u}_1\big\}\,dv\,dx
	\notag\\
	&&+(C\eta+C_\eta\delta^{3/2})\frac{\varepsilon}{\delta^{1/2}}\|\partial^{\alpha}\partial_x(\widetilde{\rho},
	\widetilde{u},\widetilde{\theta},\widetilde{\phi})\|^{2}
	+C_\eta\frac{\varepsilon}{\delta^{1/2}}\|\partial^{\alpha}\partial_xf\|_\sigma^{2}
	\notag\\
&&+C_\eta\delta^{3}\frac{1}{\delta^{3/2}\varepsilon}(\|f\|^2_{\sigma}+\|\partial_xf\|^2_{\sigma}
)+C_\eta\delta^{6}\frac{1}{\delta^{3/2}\varepsilon}(\varepsilon^2+\delta^4)+C_\eta\delta^2\varepsilon^3.
\end{eqnarray}
We have required that $\delta\leq \widetilde{C}^{-1}\varepsilon^{2/5}$ in \eqref{4.5}, then it is easy to see that
\begin{equation}
	\label{4.30A}
\delta^{6}\frac{1}{\delta^{3/2}\varepsilon}(\varepsilon^2+\delta^4)\leq \delta^4,\quad 
\widetilde{C}\delta^{5/2}\leq 	\delta^{1/2}\frac{\varepsilon}{\delta^{1/2}}.
\end{equation}
Using \eqref{4.30A} and collecting the above estimates, we have from \eqref{4.8} that for $|\alpha|\leq 1$ and  any small $\eta>0$,
\begin{eqnarray}
	\label{4.28}
	&&\frac{1}{2}\frac{d}{dt}\|\partial^{\alpha}\widetilde{u}_1\|^{2}
	+\delta^{1/2}\varepsilon\frac{d}{dt}\int_{\mathbb{R}}\int_{\mathbb{R}^{3}}
	\big\{\partial^{\alpha}[\frac{1}{\rho}\partial_{x}(K\theta B_{11}(\frac{v-u}{\sqrt{K\theta}})\frac{\sqrt{\mu}}{M}f)]
	\partial^{\alpha}\widetilde{u}_1\big\}\,dv\,dx
	\notag\\
	&&+\frac{1}{2}\frac{d}{dt}(\frac{1}{\rho}e^{\phi}\partial^{\alpha}\widetilde{\phi},\partial^{\alpha}\widetilde{\phi})+\frac{1}{2}\delta\frac{d}{dt}(\frac{1}{\rho}\partial^{\alpha}\partial_x\widetilde{\phi},\partial^{\alpha}\partial_x\widetilde{\phi})
	\notag\\
	&&+\frac{1}{\delta}
	(\frac{2\bar{\theta}}{3\bar{\rho}}\partial^{\alpha}\partial_x\widetilde{\rho},\partial^{\alpha}\widetilde{u}_1)
	+\frac{1}{\delta}(\frac{2}{3}\partial^{\alpha}\partial_x\widetilde{\theta},\partial^{\alpha}\widetilde{u}_1)
	+c\frac{\varepsilon}{\delta^{1/2}}\|\partial^{\alpha}\partial_x\widetilde{u}_1\|^{2}
	\notag\\
	\leq&& (C\eta+C_\eta\delta^{1/2})\frac{\varepsilon}{\delta^{1/2}}\|\partial^{\alpha}\partial_x(\widetilde{\rho},
	\widetilde{u},\widetilde{\theta},\widetilde{\phi})\|^{2}
	+C_\eta\frac{\varepsilon}{\delta^{1/2}}\|\partial^{\alpha}\partial_xf\|_\sigma^{2}
	\notag\\
	&&+C_\eta\delta^{3}\frac{1}{\delta^{3/2}\varepsilon}(\|f\|^2_{\sigma}+\|\partial_xf\|^2_{\sigma})
	+C_\eta\mathcal{E}_{2}(t)+C_\eta\delta^4+C_\eta\delta\varepsilon^2.
\end{eqnarray}
The energy analysis in $\partial^{\alpha}\widetilde{u}_2$ and $\partial^{\alpha}\widetilde{u}_3$
can be done in an almost similar way as $\partial^{\alpha}\widetilde{u}_1$ in \eqref{4.28}, and the details 
for estimating these terms will be omitted for brevity. For this, we
apply $\partial^{\alpha}$ with $|\alpha|\leq 1$ to the third equation of \eqref{2.33} and take the inner product of the resulting equation with 
$\partial^{\alpha}\widetilde{u}_i$ $(i=2,3)$ to get the similar argument as \eqref{4.28} that
\begin{eqnarray}
	\label{4.29}
	&&\frac{1}{2}\frac{d}{dt}\|\partial^{\alpha}\widetilde{u}_i\|^{2}
	+\delta^{1/2}\varepsilon\frac{d}{dt}\int_{\mathbb{R}}\int_{\mathbb{R}^{3}}
	\big\{\partial^{\alpha}[\frac{1}{\rho}\partial_{x}(K\theta B_{1i}(\frac{v-u}{\sqrt{K\theta}})\frac{\sqrt{\mu}}{M}f)]
	\partial^{\alpha}\widetilde{u}_i\big\}\,dv\,dx
	+c\frac{\varepsilon}{\delta^{1/2}}\|\partial^{\alpha}\partial_x\widetilde{u}_i\|^{2}
	\notag\\
&&\leq (C\eta+C_\eta\delta^{1/2})\frac{\varepsilon}{\delta^{1/2}}\|\partial^{\alpha}\partial_x(\widetilde{\rho},
\widetilde{u},\widetilde{\theta},\widetilde{\phi})\|^{2}
+C_\eta\frac{\varepsilon}{\delta^{1/2}}\|\partial^{\alpha}\partial_xf\|_\sigma^{2}
\notag\\
&&\hspace{1cm}+C_\eta\delta^{3}\frac{1}{\delta^{3/2}\varepsilon}(\|f\|^2_{\sigma}+\|\partial_xf\|^2_{\sigma})
+C_\eta\mathcal{E}_{2}(t)+C_\eta\delta^4+C_\eta\delta\varepsilon^2.	
\end{eqnarray}
\medskip
\noindent{\it Step 3. Estimate on $\|\partial^{\alpha}\widetilde{\theta}(t,\cdot)\|^2$} for $|\alpha|\leq 1$. 

The estimate of $\partial^{\alpha}\widetilde{\theta}$ is similar to  $\partial^{\alpha}\widetilde{u}$.  
Applying $\partial^{\alpha}$ with $|\alpha|\leq1$
to the fourth equation of \eqref{2.33} and taking the inner product of the resulting equation with 
$\frac{1}{\bar{\theta}}\partial^{\alpha}\widetilde{\theta}$, one gets
\begin{eqnarray}
	\label{4.30}
	&&\frac{1}{2}\frac{d}{dt}(\partial^{\alpha}\widetilde{\theta},\frac{1}{\bar{\theta}}\partial^{\alpha}\widetilde{\theta})-\frac{1}{2}(\partial^{\alpha}\widetilde{\theta},\partial_{t}(\frac{1}{\bar{\theta}})\partial^{\alpha}\widetilde{\theta})
	-A\frac{1}{\delta}(\partial^{\alpha}\partial_x\widetilde{\theta},\frac{1}{\bar{\theta}}\partial^{\alpha}\widetilde{\theta})
	\notag\\
	=&&-\frac{2}{3}\frac{1}{\delta}(\partial^{\alpha}[\bar{\theta}\partial_x\widetilde{u}_{1}],\frac{1}{\bar{\theta}}\partial^{\alpha}\widetilde{\theta})
	-\frac{1}{\delta}\frac{2}{3}(\partial^{\alpha}(\widetilde{\theta}\partial_xu_1),\frac{1}{\bar{\theta}}\partial^{\alpha}\widetilde{\theta})
	\notag\\
&&-\frac{1}{\delta}(\partial^{\alpha}(u_1\partial_x\widetilde{\theta}),\frac{1}{\bar{\theta}}\partial^{\alpha}\widetilde{\theta})
	-\frac{1}{\delta}(\partial^{\alpha}(\widetilde{u}_1\partial_{x}\bar{\theta}),\frac{1}{\bar{\theta}}\partial^{\alpha}\widetilde{\theta})
	+\frac{\varepsilon}{\delta^{1/2}}(\partial^{\alpha}[\frac{1}{\rho}\partial_x(\kappa(\theta)\partial_x\theta)],\frac{1}{\bar{\theta}}\partial^{\alpha}\widetilde{\theta})
	\notag\\
	&&+(\partial^{\alpha}D,\frac{1}{\bar{\theta}}\partial^{\alpha}\widetilde{\theta})
	+\frac{1}{\delta}(\partial^{\alpha}[\frac{1}{\rho}u\cdot\partial_x(\int_{\mathbb{R}^3} v_{1}v L^{-1}_{M}\Theta\, dv)]
	,\frac{1}{\bar{\theta}}\partial^{\alpha}\widetilde{\theta})
	\notag\\
	&&-\frac{1}{\delta}(\partial^{\alpha}[\frac{1}{\rho}\partial_x(\int_{\mathbb{R}^3}v_{1}\frac{|v|^{2}}{2}L^{-1}_{M}\Theta\, dv)],\frac{1}{\bar{\theta}}\partial^{\alpha}\widetilde{\theta})
	-\delta^2(\partial^{\alpha}\mathcal{R}_3,\frac{1}{\bar{\theta}}\partial^{\alpha}\widetilde{\theta}).
\end{eqnarray}
In the following we estimate each term for \eqref{4.30}. The second term in \eqref{4.30} is bounded by $C\mathcal{E}_{2}(t)$,
while the third term in \eqref{4.30} can be estimated as
\begin{equation*}
\frac{1}{\delta}|(\partial^{\alpha}\partial_x\widetilde{\theta},\frac{1}{\bar{\theta}}\partial^{\alpha}\widetilde{\theta})|
\leq C\frac{1}{\delta}\|\partial_x\bar{\theta}\|_{L^\infty}\|\partial^{\alpha}\widetilde{\theta}\|^2
\leq C\mathcal{E}_{2}(t).
\end{equation*}
Here we have used the integration by parts, \eqref{3.4} and \eqref{1.34}.
The first term on the right hand side of \eqref{4.30} are bounded by
\begin{equation*}
-\frac{2}{3}\frac{1}{\delta}(\partial^{\alpha}\partial_x\widetilde{u}_{1},\partial^{\alpha}\widetilde{\theta})+C\mathcal{E}_{2}(t).
\end{equation*}
Similar to \eqref{4.4} and \eqref{4.6}, we have
\begin{eqnarray*}
&&	\frac{1}{\delta}|(\frac{2}{3}\partial^{\alpha}(\widetilde{\theta}\partial_xu_1),\frac{1}{\bar{\theta}}\partial^{\alpha}\widetilde{\theta})|
	+\frac{1}{\delta}|(\partial^{\alpha}(u_1\partial_x\widetilde{\theta}),\frac{1}{\bar{\theta}}\partial^{\alpha}\widetilde{\theta})|
	+\frac{1}{\delta}|(\partial^{\alpha}(\widetilde{u}_1\partial_{x}\bar{\theta}),\frac{1}{\bar{\theta}}\partial^{\alpha}\widetilde{\theta})|
	\notag\\	
	&&\leq C\eta\frac{\varepsilon}{\delta^{1/2}}\|\partial^{\alpha}\partial_x\widetilde{u}_1\|^2
	+C_\eta\mathcal{E}_{2}(t).
\end{eqnarray*}
For the fifth term on the right hand side of \eqref{4.30}, we perform the similar arguments as \eqref{4.18} to get
\begin{eqnarray*}
	&&\frac{\varepsilon}{\delta^{1/2}}(\partial^{\alpha}[\frac{1}{\rho}\partial_x(\kappa(\theta)\partial_x\theta)],\frac{1}{\bar{\theta}}\partial^{\alpha}\widetilde{\theta})
	\notag\\
	&&\leq-\frac{\varepsilon}{\delta^{1/2}}(\frac{1}{\rho}\kappa(\theta)\partial^{\alpha}\partial_{x}\widetilde{\theta},
	\frac{1}{\bar{\theta}}\partial^{\alpha}\partial_{x}\widetilde{\theta})
	+C\mathcal{E}_{2}(t)+C\delta\varepsilon^2.
\end{eqnarray*}
Recall the definition $D$ in \eqref{2.35a}, we get from \eqref{4.17a}, \eqref{3.4} and \eqref{1.34} that
\begin{equation*}
(\partial^{\alpha}D,\frac{1}{\bar{\theta}}\partial^{\alpha}\widetilde{\theta})
\leq C\|\partial^{\alpha}\widetilde{\theta}\|^2+C\varepsilon^2\delta^3
	+C\varepsilon^2\delta\|\partial^{\alpha}\partial_x\widetilde{u}\|^2\leq C\mathcal{E}_{2}(t)+C\delta^3\varepsilon^2.
\end{equation*}
For the seventh and eighth terms on the right hand side of \eqref{4.30}, we use \eqref{4.19} and \eqref{4.20} to write
\begin{eqnarray*}
&&\frac{1}{\delta}(\partial^{\alpha}[\frac{1}{\rho}u\cdot\partial_x(\int_{\mathbb{R}^3} v_{1}v L^{-1}_{M}\Theta\, dv)]
,\frac{1}{\bar{\theta}}\partial^{\alpha}\widetilde{\theta})
-\frac{1}{\delta}(\partial^{\alpha}[\frac{1}{\rho}\partial_x(\int_{\mathbb{R}^3}v_{1}\frac{|v|^{2}}{2}L^{-1}_{M}\Theta\, dv)],\frac{1}{\bar{\theta}}\partial^{\alpha}\widetilde{\theta})	
\notag\\	
	=&&-\frac{1}{\delta}(\partial^{\alpha}[\frac{1}{\rho}\partial_{x}(\int_{\mathbb{R}^{3}}
	(\frac{1}{2}|v|^{2}v_1-u\cdot vv_{1})L^{-1}_{M}\Theta\,dv)],\frac{1}{\bar{\theta}}\partial^{\alpha}\widetilde{\theta})
	\notag\\
	&&\quad-\frac{1}{\delta}
	(\partial^{\alpha}[\frac{1}{\rho}\int_{\mathbb{R}^{3}}
	\partial_{x}u\cdot vv_{1}L^{-1}_{M}\Theta\,dv],\frac{1}{\bar{\theta}}\partial^{\alpha}\widetilde{\theta})
	\notag\\
	=&&-\int_{\mathbb{R}}\int_{\mathbb{R}^{3}}\partial^{\alpha}[\frac{1}{\rho}\partial_{x}((K\theta)^{\frac{3}{2}}A_{1}(\frac{v-u}{\sqrt{K\theta}})\frac{\Theta}{M})]\,\frac{1}{\bar{\theta}}\partial^{\alpha}\widetilde{\theta}\,dv\,dx
	\notag\\
	&&\quad-\sum^{3}_{j=1}\int_{\mathbb{R}}\int_{\mathbb{R}^{3}}\partial^{\alpha}[\frac{1}{\rho}
	\partial_{x}u_j K\theta B_{1j}(\frac{v-u}{\sqrt{K\theta}})\frac{\Theta}{M}]\,\frac{1}{\bar{\theta}}\partial^{\alpha}\widetilde{\theta}\,dv\,dx,
\end{eqnarray*}
which has almost the same structure  as $I_5$ and $I_{6}$ in \eqref{4.21}, so we  follow
the similar calculations as \eqref{4.27} and use \eqref{4.30A} to bound them as
\begin{eqnarray*}
	&&-\delta^{1/2}\varepsilon\frac{d}{dt}\int_{\mathbb{R}}\int_{\mathbb{R}^{3}}
	\big\{\partial^{\alpha}[\frac{1}{\rho}\partial_{x}
	((K\theta)^{\frac{3}{2}}A_{1}(\frac{v-u}{\sqrt{K\theta}})\frac{\sqrt{\mu}}{M}f)]\frac{1}{\bar{\theta}}\partial^{\alpha}\widetilde{\theta}\big\}\,dv\,dx
	\notag\\	
	&&-\delta^{1/2}\varepsilon\sum^{3}_{j=1}\frac{d}{dt}\int_{\mathbb{R}}\int_{\mathbb{R}^{3}}
	\big\{\partial^{\alpha}[\frac{1}{\rho}
	\partial_{x}u_j K\theta B_{1j}(\frac{v-u}{\sqrt{K\theta}})\frac{\sqrt{\mu}}{M}f]\frac{1}{\bar{\theta}}\partial^{\alpha}\widetilde{\theta}\big\}\,dv\,dx
	\notag\\
	&&+(C\eta+C_\eta\delta^{1/2})\frac{\varepsilon}{\delta^{1/2}}\|\partial^{\alpha}\partial_x(\widetilde{\rho},
	\widetilde{u},\widetilde{\theta},\widetilde{\phi})\|^{2}
	+C_\eta\frac{\varepsilon}{\delta^{1/2}}\|\partial^{\alpha}\partial_xf\|_\sigma^{2}
	\notag\\
	&&+C_\eta\delta^{3}\frac{1}{\delta^{3/2}\varepsilon}(\|f\|^2_{\sigma}+\|\partial_xf\|^2_{\sigma})
	+C_\eta\mathcal{E}_{2}(t)+C_\eta\delta^4+C_\eta\varepsilon^2\delta.
\end{eqnarray*}
The last term of \eqref{4.30} is bounded by $C\mathcal{E}_{2}(t)+C\delta^4$ in terms of \eqref{2.31} and \eqref{1.34}.
In summary, for any small $\eta>0$, we substitute the above estimates into \eqref{4.30} to get
\begin{eqnarray}
	\label{4.31}
	&&\frac{1}{2}\frac{d}{dt}(\partial^{\alpha}\widetilde{\theta},\frac{1}{\bar{\theta}}\partial^{\alpha}\widetilde{\theta})+\frac{2}{3}\frac{1}{\delta}(\partial_x \partial^{\alpha}\widetilde{u}_1,\partial^{\alpha}\widetilde{\theta})
	+c\frac{\varepsilon}{\delta^{1/2}}\|\partial^{\alpha}\partial_{x}\widetilde{\theta}\|^2
	\notag\\
	&&+\delta^{1/2}\varepsilon\frac{d}{dt}\int_{\mathbb{R}}\int_{\mathbb{R}^{3}}
	\big\{\partial^{\alpha}[\frac{1}{\rho}\partial_{x}
	((K\theta)^{\frac{3}{2}}A_{1}(\frac{v-u}{\sqrt{K\theta}})\frac{\sqrt{\mu}}{M}f)]\frac{1}{\bar{\theta}}\partial^{\alpha}\widetilde{\theta}\big\}\,dv\,dx
	\notag\\	
	&&+\delta^{1/2}\varepsilon\sum^{3}_{j=1}\frac{d}{dt}\int_{\mathbb{R}}\int_{\mathbb{R}^{3}}
	\big\{\partial^{\alpha}[\frac{1}{\rho}
	\partial_{x}u_j K\theta B_{1j}(\frac{v-u}{\sqrt{K\theta}})\frac{\sqrt{\mu}}{M}f]\frac{1}{\bar{\theta}}\partial^{\alpha}\widetilde{\theta}\big\}\,dv\,dx
	\notag\\
\leq&& (C\eta+C_\eta\delta^{1/2})\frac{\varepsilon}{\delta^{1/2}}\|\partial^{\alpha}\partial_x(\widetilde{\rho},
\widetilde{u},\widetilde{\theta},\widetilde{\phi})\|^{2}
+C_\eta\frac{\varepsilon}{\delta^{1/2}}\|\partial^{\alpha}\partial_xf\|_\sigma^{2}
\notag\\
&&+C_\eta\delta^{3}\frac{1}{\delta^{3/2}\varepsilon}(\|f\|^2_{\sigma}+\|\partial_xf\|^2_{\sigma})
+C_\eta\mathcal{E}_{2}(t)+C_\eta\delta^4+C_\eta\delta\varepsilon^2.
\end{eqnarray}
\medskip
\noindent{\it Step 4. Estimate on $\|\partial^{\alpha}(\widetilde{\rho},\widetilde{u},\widetilde{\theta})(t,\cdot)\|^2$} for $|\alpha|\leq 1$.

By integration by parts, \eqref{3.4} and \eqref{1.34}, one has that for $|\alpha|\leq 1$,
\begin{eqnarray*}
	&&\frac{1}{\delta}|(\partial_x\partial^{\alpha}\widetilde{u}_1,\frac{2\bar{\theta}}{3\bar{\rho}}\partial^{\alpha}\widetilde{\rho})+	(\frac{2\bar{\theta}}{3\bar{\rho}}\partial^{\alpha}\partial_{x}\widetilde{\rho},\partial^{\alpha}\widetilde{u}_1)+
	\frac{2}{3}(\partial^{\alpha}\partial_{x}\widetilde{\theta},\partial^{\alpha}\widetilde{u}_1)
	+\frac{2}{3}(\partial_{x} \partial^{\alpha}\widetilde{u}_1,\partial^{\alpha}\widetilde{\theta})|
	\notag\\
	&&=\frac{1}{\delta}|(\partial^{\alpha}\widetilde{u}_1,\partial_{x}(\frac{2\bar{\theta}}{3\bar{\rho}})\partial^{\alpha}\widetilde{\rho})|\leq C\frac{1}{\delta}\|\partial_{x}(\bar{\rho},\bar{\theta})\|_{L^{\infty}}
	\|\partial^{\alpha}\widetilde{u}_1\|\|\partial^{\alpha}\widetilde{\rho}\|\leq C\mathcal{E}_{2}(t).
\end{eqnarray*}
Adding \eqref{4.7}, \eqref{4.28}, \eqref{4.29} and \eqref{4.31} together, and using the aforementioned estimate,
then the summation of the resulting equation over $|\alpha|$ with $|\alpha|\leq 1$ through a suitable linear combination gives
\begin{eqnarray}
	\label{4.32}
	&&\frac{1}{2}\sum_{|\alpha|\leq 1}\frac{d}{dt}\int_{\mathbb{R}}( \frac{2\bar{\theta}}{3\bar{\rho}^{2}}|\partial^{\alpha}\widetilde{\rho}|^{2}+|\partial^{\alpha}\widetilde{u}|^{2}
	+\frac{1}{\bar{\theta}}|\partial^{\alpha}\widetilde{\theta}|^{2}+
	\frac{1}{\rho}e^{\phi}|\partial^{\alpha}\widetilde{\phi}|^2
	+\delta\frac{1}{\rho}|\partial^{\alpha}\partial_x\widetilde{\phi}|^2)\,dx
	\notag\\
	&&+\frac{d}{dt}E_\alpha(t)+c\frac{\varepsilon}{\delta^{1/2}}\sum_{|\alpha|\leq 1}\|\partial^{\alpha}\partial_{x}(\widetilde{u},\widetilde{\theta})\|^{2}
	\notag\\
\leq&& (C\eta+C_\eta\delta^{1/2})\frac{\varepsilon}{\delta^{1/2}}\sum_{|\alpha|\leq 1}\|\partial^{\alpha}\partial_x(\widetilde{\rho},
\widetilde{u},\widetilde{\theta},\widetilde{\phi})\|^{2}
+C_\eta\frac{\varepsilon}{\delta^{1/2}}\sum_{|\alpha|\leq 1}\|\partial^{\alpha}\partial_xf\|_\sigma^{2}
\notag\\
&&+C_\eta\delta^{3}\frac{1}{\delta^{3/2}\varepsilon}(\|f\|^2_{\sigma}+\|\partial_xf\|^2_{\sigma})
+C_\eta\mathcal{E}_{2}(t)+C_\eta\delta^4+C_\eta\delta\varepsilon^2.		
\end{eqnarray}
Here we have used the expression of $E_\alpha(t)$ given by \eqref{4.33}.

On the other hand, it holds from \eqref{3.53} that 
	\begin{eqnarray}	
	\label{4.34}
	&&\delta^{1/2}\varepsilon\sum_{|\alpha|\leq 1}\frac{d}{dt}(\partial^{\alpha}\widetilde{u}_1,\partial^{\alpha}\partial_x\widetilde{\rho})
	+c\sum_{|\alpha|\leq 1}\frac{\varepsilon}{\delta^{1/2}}(\|\partial^{\alpha}\partial_x\widetilde{\rho}\|^2+\|\partial^{\alpha}\partial_x\widetilde{\phi}\|^2
	+\delta\|\partial^{\alpha}\partial^2_x\widetilde{\phi}\|^2)
	\notag\\
	&&\leq C\sum_{|\alpha|\leq 1}\frac{\varepsilon}{\delta^{1/2}}(\|\partial^{\alpha}\partial_x(\widetilde{u},\widetilde{\theta})\|^2
	+\|\partial^{\alpha}\partial_xf\|_\sigma^2)
	+C\delta^{5/2}\varepsilon^3+C\delta^5\varepsilon.
\end{eqnarray}	
In summary, adding $\eqref{4.34}\times \kappa_1$ to \eqref{4.32},  and then firstly choosing $\kappa_1>0$ sufficiently small such that
$C\kappa_1<\frac{1}{2}c$, then taking $\eta>0$ and $\delta>0$ small enough  such that $C\eta+C_\eta\delta^{1/2}<\frac{1}{2}\kappa_1c$, 
we can prove \eqref{4.1} holds. This completes the proof of Lemma \ref{lem4.1}.
\end{proof}
\subsection{Estimates on the non-fluid part}\label{sec.4.2} 
In the following we make use of the microscopic equation \eqref{2.39} to derive the 
space derivative estimates  up to one-order for the non-fluid part $f$. 
The main result is given in Lemma \ref{lem4.2} below.
\begin{lemma}\label{lem4.2}
Under the same conditions as in Lemma \ref{lem4.1}.	It holds that
\begin{eqnarray}
	\label{4.36}
	&&\sum_{|\alpha|\leq 1}(\frac{d}{dt}\|\partial^{\alpha}f\|^{2}+\frac{c_1}{2}\frac{1}{\delta^{3/2}\varepsilon}\|\partial^{\alpha}f\|_{\sigma}^{2})
	\notag\\
	\leq&& C\frac{\varepsilon}{\delta^{1/2}}\sum_{|\alpha|\leq 1}(\|\partial^{\alpha}\partial_{x}(\widetilde{u},\widetilde{\theta})\|^{2}+\|\partial^{\alpha}\partial_{x}f\|_\sigma^{2})
	+C\delta^{1/2}\mathcal{D}_{2,l,q_1}(t)
	\notag\\
	&&+C\delta^{5/2}\varepsilon^3+C\delta^{9/2}+C\varepsilon^{1/2}\delta^{1/2}\sum_{|\alpha|\leq 1}\|\langle v\rangle \partial^\alpha f\|^2.
\end{eqnarray}		
\end{lemma}
\begin{proof}
	The proof the estimate \eqref{4.36} is based on the microscopic equation \eqref{2.39}. Let $|\alpha|\leq 1$, 
	by taking $\partial^{\alpha}$ derivatives of \eqref{2.39} and taking the inner product of the resulting identity with $\partial^{\alpha}f$, we show that
	\begin{eqnarray}
		\label{4.38}
		&&\frac{1}{2}\frac{d}{dt}\|\partial^{\alpha}f\|^{2}
		-\frac{1}{\delta}(\partial^{\alpha}[\frac{\partial_x\phi\partial_{v_{1}}(\sqrt{\mu}f)}{\sqrt{\mu}}],\partial^{\alpha}f)
		-\frac{1}{\delta^{3/2}\varepsilon}(\mathcal{L}\partial^{\alpha}f,\partial^{\alpha}f)
		\notag\\
		=&&\frac{1}{\delta^{3/2}\varepsilon}(\partial^{\alpha}\Gamma(\frac{M-\mu}{\sqrt{\mu}},f)
		+\partial^{\alpha}\Gamma(f,\frac{M-\mu}{\sqrt{\mu}}),\partial^{\alpha}f)
		+\frac{1}{\delta^{3/2}\varepsilon}(\partial^{\alpha}\Gamma(\frac{G}{\sqrt{\mu}},\frac{G}{\sqrt{\mu}}),\partial^{\alpha}f)
		\notag\\
		&&+\frac{1}{\delta}(\frac{\partial^{\alpha}P_{0}(v_1\sqrt{\mu}\partial_{x}f)}{\sqrt{\mu}},\partial^{\alpha}f)
	-\frac{1}{\delta}(\frac{1}{\sqrt{\mu}}\partial^{\alpha}P_{1}\big\{v_{1}M(\frac{|v-u|^{2}\partial_x\widetilde{\theta}}{2K\theta^{2}}+\frac{(v-u)\cdot\partial_x\widetilde{u}}{K\theta})\big\},\partial^{\alpha}f)
		\notag\\
		&&+A\frac{1}{\delta}(\frac{\partial^{\alpha}\partial_x\overline{G}}{\sqrt{\mu}},\partial^{\alpha}f)
		+\frac{1}{\delta}(\frac{\partial^{\alpha}[\partial_x\phi\partial_{v_{1}}\overline{G}]}{\sqrt{\mu}},\partial^{\alpha}f)
		-\frac{1}{\delta}(\frac{\partial^{\alpha}P_{1}(v_1\partial_x\overline{G})}{\sqrt{\mu}},\partial^{\alpha}f)
		-(\frac{\partial^{\alpha}\partial_{t}\overline{G}}{\sqrt{\mu}},\partial^{\alpha}f).
	\end{eqnarray}
Here we have used the fact that
$$
-A\frac{1}{\delta}(\partial_x\partial^{\alpha}f,\partial^{\alpha}f)+\frac{1}{\delta}(v_{1}\partial_x\partial^{\alpha}f,\partial^{\alpha}f)=0,
$$
due to the integration by parts. Next we compute each term for \eqref{4.38}. Similar arguments as \eqref{3.83} imply that
\begin{equation*}
	\frac{1}{\delta}|(\frac{\partial^{\alpha}[\partial_x\phi\partial_{v_1}(\sqrt{\mu}f)]}{\sqrt{\mu}},\partial^{\alpha}f)|
	\leq  C\delta^{1/2}\mathcal{D}_{2,l,q_1}(t)
	+C\varepsilon^{1/2}\delta^{1/2}\|\langle v\rangle \partial^\alpha f\|^2.
\end{equation*}	
For the third term on the left hand side of \eqref{4.38}, we see easily by \eqref{3.17} that
\begin{equation*}
-\frac{1}{\delta^{3/2}\varepsilon}(\mathcal{L}\partial^{\alpha}f,\partial^{\alpha}f)\geq c_1\frac{1}{\delta^{3/2}\varepsilon}\|\partial^{\alpha}f\|^{2}_{\sigma}.
\end{equation*}
The first term on the right hand side of \eqref{4.38} is bounded by $ C\delta\mathcal{D}_{2}(t)$ in terms of \eqref{3.28a}.
Using \eqref{3.56a} and \eqref{4.5}, we get
\begin{equation*}
\frac{1}{\delta^{3/2}\varepsilon}|(\partial^\alpha\Gamma(\frac{G}{\sqrt{\mu}},\frac{G}{\sqrt{\mu}}), \partial^\alpha f)|
\leq C\eta\frac{1}{\delta^{3/2}\varepsilon}\|\partial^\alpha f\|^{2}_{\sigma}
+C\delta\mathcal{D}_{2}(t)
+C_\eta\delta^{9/2}.
\end{equation*}
For the third term on the right hand side of \eqref{4.38}, we use \eqref{1.20}, \eqref{1.19},
\eqref{1.28}, \eqref{3.11} and \eqref{4.17a} to show that
\begin{eqnarray}
	\label{4.38A}
\frac{1}{\delta}|(\frac{\partial^{\alpha}P_{0}(v_1\sqrt{\mu}\partial_{x}f)}{\sqrt{\mu}},\partial^{\alpha}f)|
	&&=\frac{1}{\delta}|(\frac{1}{\sqrt{\mu}}\sum_{i=0}^{4}\partial^{\alpha}[\langle v_1\sqrt{\mu}\partial_{x}f,\frac{\chi_{i}}{M}\rangle\chi_{i}],\partial^{\alpha}f)|
	\notag\\	
	&&\leq C\frac{1}{\delta}\|\langle v\rangle^{-1/2}\partial^{\alpha}f\|\|\langle v\rangle^{1/2}\frac{1}{\sqrt{\mu}}\sum_{i=0}^{4}\partial^{\alpha}[\langle v_1\sqrt{\mu}\partial_{x}f,\frac{\chi_{i}}{M}\rangle\chi_{i}]\|
	\notag\\
	&&\leq C\frac{1}{\delta}\|\partial^{\alpha}f\|_{\sigma}(\|\partial^{\alpha}\partial_{x}f\|_\sigma+
		\|\partial^{\alpha}(\rho,u,\theta)\|_{L^\infty}\|\partial_{x}f\|_\sigma)
	\notag\\
	&&\leq  \eta\frac{1}{\delta^{3/2}\varepsilon}\|\partial^{\alpha}f\|_{\sigma}^{2}
	+C_{\eta}\frac{\varepsilon}{\delta^{1/2}}\|\partial^{\alpha}\partial_{x}f\|_\sigma^{2}+C_{\eta}\delta^2\frac{\varepsilon}{\delta^{1/2}}\|\partial_{x}f\|_\sigma^{2}.
\end{eqnarray}
Similar to \eqref{4.38A}, the fourth term on the right hand side of \eqref{4.38} can be controlled by
\begin{eqnarray*}	
C \eta\frac{1}{\delta^{3/2}\varepsilon}\|\partial^{\alpha}f\|_{\sigma}^{2}
+ C_{\eta}\frac{\varepsilon}{\delta^{1/2}}\|\partial^{\alpha}\partial_{x}(\widetilde{u},\widetilde{\theta})\|^{2}
+C_{\eta}\delta^2\frac{\varepsilon}{\delta^{1/2}}\|\partial_{x}(\widetilde{u},\widetilde{\theta})\|^{2}.
\end{eqnarray*}
Using \eqref{3.7} and \eqref{1.28}, we know that the fifth term on the right hand side of \eqref{4.38} can be estimated as
\begin{eqnarray}
	\label{4.40A}
	A\frac{1}{\delta}|(\frac{\partial^{\alpha}\partial_x\overline{G}}{\sqrt{\mu}},\partial^{\alpha}f)|
	&&\leq \eta\frac{1}{\delta^{3/2}\varepsilon}\|\langle v\rangle^{-1/2}\partial^{\alpha}f\|^{2}
	+C_{\eta}\frac{\varepsilon}{\delta^{1/2}}\|\langle v\rangle^{1/2}\frac{\partial^{\alpha}\partial_x\overline{G}}{\sqrt{\mu}}\|^{2}
	\notag\\
	&&\leq C\eta\frac{1}{\delta^{3/2}\varepsilon}\|\partial^{\alpha}f\|_{\sigma}^{2}
	+C_{\eta}\varepsilon\delta^{5/2}(\varepsilon+\delta^2)^2.
\end{eqnarray}
The last three terms in \eqref{4.38} can be done in an almost similar way as \eqref{4.40A}. 

Consequently, substituting all the above estimates into \eqref{4.38} and choosing $\eta>0$ small enough, then the summation of the resulting equation over $|\alpha|$ with $|\alpha|\leq 1$, we can prove \eqref{4.36} holds. We thus
end the proof of Lemma \ref{lem4.2}.
\end{proof}
Finally, combining Lemma \ref{lem4.1} and Lemma \ref{lem4.2}, we obtain the following main result in this section. 
\begin{lemma}\label{lem4.3}
There exist some constants $\kappa_1$ and $\kappa_2$ with $\kappa_1\gg \kappa_2>0$ such that
	\begin{eqnarray}
		\label{4.38a}
		&&\frac{d}{dt}\widetilde{E}(t)+\frac{1}{2}c\kappa_1\sum_{|\alpha|\leq 1}\frac{\varepsilon}{\delta^{1/2}}(\|\partial^{\alpha}\partial_x(\widetilde{\rho},\widetilde{u},\widetilde{\theta})\|^2
		+\|\partial^{\alpha}\partial_x\widetilde{\phi}\|^2
		+\delta\|\partial^{\alpha}\partial^2_x\widetilde{\phi}\|^2)+\frac{c_1}{4}\kappa_2\sum_{|\alpha|\leq 1}\frac{1}{\delta^{3/2}\varepsilon}\|\partial^{\alpha}f\|_{\sigma}^{2}
		\notag\\
		&&\leq 
		C\frac{\varepsilon}{\delta^{1/2}}\sum_{|\alpha|=2}\|\partial^{\alpha}f\|_\sigma^{2}+C\delta^{1/2}\mathcal{D}_{2,l,q_1}(t)
		+C\mathcal{E}_{2}(t)+C\delta^4+C\delta\varepsilon^2+C\varepsilon^{1/2}\delta^{1/2}\sum_{|\alpha|\leq 1}\|\langle v\rangle \partial^\alpha f\|^2.
	\end{eqnarray}		
Here $\widetilde{E}(t)$ is given by
	\begin{eqnarray}
	\label{4.39}
\widetilde{E}(t)=&&\frac{1}{2}\sum_{|\alpha|\leq 1}\int_{\mathbb{R}}( \frac{2\bar{\theta}}{3\bar{\rho}^{2}}|\partial^{\alpha}\widetilde{\rho}|^{2}+|\partial^{\alpha}\widetilde{u}|^{2}
	+\frac{1}{\bar{\theta}}|\partial^{\alpha}\widetilde{\theta}|^{2}+
	\frac{1}{\rho}e^{\phi}|\partial^{\alpha}\widetilde{\phi}|^2
	+\delta\frac{1}{\rho}|\partial^{\alpha}\partial_x\widetilde{\phi}|^2)\,dx+E_\alpha(t)
	\notag\\
	&&+\kappa_1\delta^{1/2}\varepsilon\sum_{|\alpha|\leq 1}(\partial^{\alpha}\widetilde{u}_1,\partial^{\alpha}\partial_x\widetilde{\rho})
	+\kappa_2\sum_{|\alpha|\leq 1}\|\partial^{\alpha}f\|^{2},
\end{eqnarray}	
with $E_\alpha(t)$ given in \eqref{4.33}.
\end{lemma}
\begin{proof}
Adding $\eqref{4.36}\times \kappa_2$ to \eqref{4.1}, and choosing $\kappa_2>0$ sufficiently small such that
 $C\kappa_2<\frac{1}{2}c\kappa_1$ and $C\delta^{3}<\frac{c_1}{4}\kappa_2$, the estimate \eqref{4.38a} follows.
This completes the proof of Lemma \ref{lem4.3}.
\end{proof}
\section{Second order space derivative estimates}
In this section we derive the highest second-order space
derivatives to complete the estimate on space derivatives of all orders.
The main result is given in Lemma \ref{lem5.1} below. 
\begin{lemma}\label{lem5.1}
Under the same conditions as in Lemma \ref{lem4.1}.	It holds that
\begin{eqnarray}
	\label{5.1}
	&&\frac{1}{2}\frac{\varepsilon^2}{\delta}\frac{d}{dt}\sum_{|\alpha|=2}\{\|\frac{\partial^{\alpha}F}{\sqrt{\mu}}\|^{2}
	+(\frac{1}{K\theta}e^{\phi}\partial^{\alpha}\widetilde{\phi},\partial^{\alpha}\widetilde{\phi})
	+\delta(\frac{1}{K\theta}\partial^{\alpha}\partial_x\widetilde{\phi},\partial^{\alpha}\partial_x\widetilde{\phi})\}
	+c\frac{1}{\delta^{3/2}\varepsilon}\frac{\varepsilon^2}{\delta}\sum_{|\alpha|=2}\|\partial^{\alpha}f\|_{\sigma}^{2}
	\notag\\
	&&\leq C\frac{\varepsilon}{\delta^{1/2}}\sum_{|\alpha|=2}\|\partial^{\alpha}(\widetilde{\rho},\widetilde{u},\widetilde{\theta})\|^2
	+C\varepsilon^{1/2}\delta^{1/2}\frac{\varepsilon^2}{\delta}\sum_{|\alpha|=2}\|\langle v\rangle \partial^{\alpha}f\|^2	
	+C(\frac{\varepsilon}{\delta}+\delta^{1/4}+\frac{\varepsilon^{2}}{\delta^{5/4}})\mathcal{D}_{2,l,q_1}(t)
	\notag\\
	&&\hspace{0.5cm}+C\frac{\varepsilon^2}{\delta^3}\mathcal{E}_{2}(t)+C\delta^{1/4}\varepsilon^3+C\delta^{5/2}\varepsilon+C\frac{\varepsilon^4}{\delta^{1/2}}+C\delta\varepsilon^2.	
\end{eqnarray}		
\end{lemma}
\begin{proof}
Let	$|\alpha|=2$, we take $\partial^{\alpha}=\partial_{xx}$ derivatives of \eqref{2.40} and take the inner product of the resulting equation with $\frac{\partial^{\alpha}F}{\sqrt{\mu}}$  to obtain
\begin{eqnarray}
\label{5.2}
&&\frac{1}{2}\frac{d}{dt}\|\frac{\partial^{\alpha}F}{\sqrt{\mu}}\|^{2}
-\frac{1}{\delta}(\frac{\partial^{\alpha}(\partial_x\phi\partial_{v_{1}}F)}{\sqrt{\mu}},\frac{\partial^{\alpha}F}{\sqrt{\mu}})
-\frac{1}{\delta^{3/2}\varepsilon}(\mathcal{L}\partial^{\alpha}f,\frac{\partial^{\alpha}F}{\sqrt{\mu}})
-\frac{1}{\delta^{3/2}\varepsilon}(\frac{\partial^{\alpha}L_{M}\overline{G}}{\sqrt{\mu}},\frac{\partial^{\alpha}F}{\sqrt{\mu}})
\notag\\
&&=\frac{1}{\delta^{3/2}\varepsilon}(\partial^{\alpha}\Gamma(\frac{M-\mu}{\sqrt{\mu}},f)+\partial^{\alpha}\Gamma(f,\frac{M-\mu}{\sqrt{\mu}}),\frac{\partial^{\alpha}F}{\sqrt{\mu}})
+\frac{1}{\delta^{3/2}\varepsilon}(\partial^{\alpha}\Gamma(\frac{G}{\sqrt{\mu}},\frac{G}{\sqrt{\mu}}),\frac{\partial^{\alpha}F}{\sqrt{\mu}}).
\end{eqnarray}
Here we have used the integration by parts such that following identity 
$$
-A\frac{1}{\delta}(\frac{\partial_x\partial^{\alpha}F}{\sqrt{\mu}},\frac{\partial^{\alpha}F}{\sqrt{\mu}})+\frac{1}{\delta}
(\frac{v_{1}\partial_x\partial^{\alpha}F}{\sqrt{\mu}},\frac{\partial^{\alpha}F}{\sqrt{\mu}})=0.
$$
In the following we compute each term for \eqref{5.2}. From \eqref{3.55}, it obviously holds that
	\begin{multline*}
	\frac{1}{\delta}	(\frac{\partial^{\alpha}(\partial_x\phi\partial_{v_{1}}F)}{\sqrt{\mu}},\frac{\partial^{\alpha}F}{\sqrt{\mu}})
	\leq -\frac{1}{2}\frac{d}{dt}(\frac{1}{K\theta}e^{\phi}\partial^{\alpha}\widetilde{\phi},\partial^{\alpha}\widetilde{\phi})
	-\frac{1}{2}\delta\frac{d}{dt}(\frac{1}{K\theta}\partial^{\alpha}\partial_x\widetilde{\phi},\partial^{\alpha}\partial_x\widetilde{\phi})
	\\
	+C(\frac{1}{\varepsilon}+\frac{\delta^{3/2}}{\varepsilon^2})\mathcal{D}_{2,l,q_1}(t)+C\frac{1}{\delta^2}\mathcal{E}_{2}(t)
	+C\delta^{\frac{1}{2}}\varepsilon^2+C\delta^2+C\varepsilon^{1/2}\delta^{1/2}\|\langle v\rangle \partial^{\alpha}f\|^2.
\end{multline*}	
To bound the third term on the left hand side of \eqref{5.2}, we use the decomposition $F=M+\overline{G}+\sqrt{\mu}f$ to write
\begin{equation}
	\label{5.3}
	-\frac{1}{\delta^{3/2}\varepsilon}(\mathcal{L}\partial^{\alpha}f,\frac{\partial^{\alpha}F}{\sqrt{\mu}})=
	-\frac{1}{\delta^{3/2}\varepsilon}(\mathcal{L}\partial^{\alpha}f,\frac{\partial^{\alpha}M}{\sqrt{\mu}})
	-\frac{1}{\delta^{3/2}\varepsilon}(\mathcal{L}\partial^{\alpha}f,\partial^{\alpha}f)
	-\frac{1}{\delta^{3/2}\varepsilon}(\mathcal{L}\partial^{\alpha}f,\frac{\partial^{\alpha}\overline{G}}{\sqrt{\mu}}).
\end{equation}
Due to singularity, we have to develop delicate estimates for the first term on the right hand side of \eqref{5.3}. Note that
$\partial^{\alpha}M=I_{1}+I_{2}$ in terms of \eqref{3.28}, we further decompose $I_{1}$ as
\begin{eqnarray}
	\label{5.4}
	I_{1}=&&\mu\big(\frac{\partial^{\alpha}\rho}{\rho}+\frac{(v-u)\cdot\partial^{\alpha}u}{K\theta}+(\frac{|v-u|^{2}}{2K\theta}-\frac{3}{2})\frac{\partial^{\alpha}\theta}{\theta}\big)
	\notag\\
	&&+(M-\mu)\big(\frac{\partial^{\alpha}\widetilde{\rho}}{\rho}
	+\frac{(v-u)\cdot\partial^{\alpha}\widetilde{u}}{K\theta}+(\frac{|v-u|^{2}}{2K\theta}
	-\frac{3}{2})\frac{\partial^{\alpha}\widetilde{\theta}}{\theta}\big)
	\notag\\
	&&+(M-\mu)\big(\frac{\partial^{\alpha}\bar{\rho}}{\rho}
	+\frac{(v-u)\cdot\partial^{\alpha}\bar{u}}{K\theta}+(\frac{|v-u|^{2}}{2K\theta}
	-\frac{3}{2})\frac{\partial^{\alpha}\bar{\theta}}{\theta}\big)
	\notag\\
	:=&&\widetilde{I}^{1}+\widetilde{I}^{2}+\widetilde{I}^{3}.
\end{eqnarray}
Since $\frac{\widetilde{I}^{1}}{\sqrt{\mu}}\in\ker{\mathcal{L}}$ in terms of
the properties of $\mathcal{L}$ stated in subsection \ref{seca3.3.1}, it follows that
$(\mathcal{L}f,\frac{\widetilde{I}^{1}}{\sqrt{\mu}})=0$. 
Recall $\mathcal{L}f=\Gamma(\sqrt{\mu},f)+\Gamma(f,\sqrt{\mu})$ given in \eqref{2.38}, we use \eqref{3.20}, \eqref{3.24} and \eqref{3.11} to get
\begin{eqnarray}
	\label{5.5}
\frac{1}{\delta^{3/2}\varepsilon}|(\mathcal{L}\partial^{\alpha}f,\frac{\widetilde{I}^{2}}{\sqrt{\mu}})|&&\leq C\delta\frac{1}{\delta^{3/2}\varepsilon}(\|\partial^{\alpha}f\|^{2}_{\sigma}+\|\partial^{\alpha}(\widetilde{\rho},\widetilde{u},\widetilde{\theta})\|^{2})
	\notag\\
	&&\leq \eta\frac{1}{\delta^{3/2}\varepsilon}\|\partial^{\alpha}f\|^{2}_{\sigma}
	+C_\eta\frac{1}{\delta^{3/2}\varepsilon}\delta^2\|\partial^{\alpha}(\widetilde{\rho},\widetilde{u},\widetilde{\theta})\|^{2}.
\end{eqnarray}
Performing the similar calculations as \eqref{3.40a} and using \eqref{1.35}, one obtains
\begin{equation*}
\frac{1}{\delta^{3/2}\varepsilon}|(\mathcal{L}\partial^{\alpha}f,\frac{\widetilde{I}^{3}}{\sqrt{\mu}})|
\leq C\frac{1}{\delta^{3/2}\varepsilon}(\frac{\delta^{5/4}}{\varepsilon^{2}}\|\partial_xf\|^{2}_{\sigma}+\varepsilon^2\delta^{11/4})
\leq C\frac{\delta^{5/4}}{\varepsilon^{2}}\mathcal{D}_{2}(t)+C\varepsilon\delta^{5/4}.
\end{equation*}
It follows from the  above estimates that
\begin{eqnarray}
	\label{5.6}
\frac{1}{\delta^{3/2}\varepsilon}|(\mathcal{L}\partial^{\alpha}f,\frac{I_{1}}{\sqrt{\mu}})|
	\leq&& \eta\frac{1}{\delta^{3/2}\varepsilon}\|\partial^{\alpha}f\|^{2}_{\sigma}
	+C_\eta\frac{\delta^{1/2}}{\varepsilon}\|\partial^{\alpha}(\widetilde{\rho},\widetilde{u},\widetilde{\theta})\|^{2}
	\notag\\
	&&+C\frac{\delta^{5/4}}{\varepsilon^{2}}\mathcal{D}_{2}(t)+C\varepsilon\delta^{5/4}.
\end{eqnarray}
Recall $I_2$ given in \eqref{3.28} and use the integration by parts, we arrive at
\begin{eqnarray*}
\frac{1}{\delta^{3/2}\varepsilon}|(\mathcal{L}\partial^{\alpha}f,\frac{I_{2}}{\sqrt{\mu}})|
&&=\frac{1}{\delta^{3/2}\varepsilon}|(\mathcal{L}\partial_xf,\partial_x[\frac{I_{2}}{\sqrt{\mu}}])|
\notag\\
&& \leq C\frac{1}{\delta^{3/2}\varepsilon}\|\partial_xf\|_{\sigma}(\|\partial_x(\rho,u,\theta)\|_{L^{\infty}}\|\partial^2_x(\rho,u,\theta)\|
+\|\partial_x(\rho,u,\theta)\|^2_{L^{\infty}}\|\partial_x(\rho,u,\theta)\|)
\notag\\
&& \leq C\frac{1}{\delta^{3/2}\varepsilon}\|\partial_xf\|_{\sigma}(\delta\|\partial^2_x(\widetilde{\rho},\widetilde{u},\widetilde{\theta})\|
+\delta^2+\delta^3)
\notag\\
&&
\leq C\frac{\delta^{1/2}}{\varepsilon}\frac{1}{\delta^{3/2}\varepsilon}\|\partial_xf\|^2_{\sigma}
+C\|\partial^2_x(\widetilde{\rho},\widetilde{u},\widetilde{\theta})\|^2
+C\frac{\delta^{5/4}}{\varepsilon^{2}}\frac{1}{\delta^{3/2}\varepsilon}\|\partial_xf\|^2_{\sigma}+C\varepsilon\delta^{5/4}
\notag\\
&&
\leq C(\frac{\delta^{5/4}}{\varepsilon^{2}}+\frac{\delta^{1/2}}{\varepsilon})\mathcal{D}_{2}(t)+C\varepsilon\delta^{5/4},
\end{eqnarray*}
where we have used \eqref{3.20}, \eqref{3.11}, \eqref{4.17a} and \eqref{1.35}.
This together with \eqref{5.6} and the fact $\partial^{\alpha}M=I_{1}+I_{2}$, one obtains
\begin{eqnarray*}
\frac{1}{\delta^{3/2}\varepsilon}|(\mathcal{L}\partial^{\alpha}f,\frac{\partial^{\alpha}M}{\sqrt{\mu}})|
\leq&& \eta\frac{1}{\delta^{3/2}\varepsilon}\|\partial^{\alpha}f\|^{2}_{\sigma}
+C_\eta\frac{\delta^{1/2}}{\varepsilon}\|\partial^{\alpha}(\widetilde{\rho},\widetilde{u},\widetilde{\theta})\|^{2}
\notag\\
&&+C(\frac{\delta^{5/4}}{\varepsilon^{2}}+\frac{\delta^{1/2}}{\varepsilon})\mathcal{D}_{2}(t)+C\varepsilon\delta^{5/4}.	
\end{eqnarray*}
By \eqref{3.17}, we see easily that
$$
-\frac{1}{\delta^{3/2}\varepsilon}(\mathcal{L}\partial^{\alpha}f,\partial^{\alpha}f)\geq c_1\frac{1}{\delta^{3/2}\varepsilon}\|\partial^{\alpha}f\|_{\sigma}^{2}.
$$
For the last term of \eqref{5.3}, using $\mathcal{L}f=\Gamma(\sqrt{\mu},f)+\Gamma(f,\sqrt{\mu})$, \eqref{3.20} and \eqref{3.7},
one gets
\begin{equation*}
\frac{1}{\delta^{3/2}\varepsilon}|(\mathcal{L}\partial^{\alpha}f,\frac{\partial^{\alpha}\overline{G}}{\sqrt{\mu}})|
	\leq C\frac{1}{\delta^{3/2}\varepsilon}\|\partial^{\alpha}f\|_{\sigma}\|\frac{\partial^{\alpha}\overline{G}}{\sqrt{\mu}}\|_{\sigma}\leq \eta\frac{1}{\delta^{3/2}\varepsilon}\|\partial^{\alpha}f\|^{2}_{\sigma}+C_\eta\frac{1}{\delta^{3/2}\varepsilon}\delta^3(\varepsilon^2+\delta^4).
\end{equation*}
Hence, putting the aforementioned  three estimates into \eqref{5.3}, we obtain 
\begin{eqnarray}
	\label{5.7}
	-\frac{1}{\delta^{3/2}\varepsilon}(\mathcal{L}\partial^{\alpha}f,\frac{\partial^{\alpha}F}{\sqrt{\mu}})
	\geq&& \frac{c_1}{2}\frac{1}{\delta^{3/2}\varepsilon}\|\partial^{\alpha}f\|_{\sigma}^{2}-
	C\frac{\delta^{1/2}}{\varepsilon}\|\partial^{\alpha}(\widetilde{\rho},\widetilde{u},\widetilde{\theta})\|^{2}
	\notag\\
	&&-C(\frac{\delta^{5/4}}{\varepsilon^{2}}+\frac{\delta^{1/2}}{\varepsilon})\mathcal{D}_{2}(t)-C\delta^{5/4}\varepsilon-C\frac{1}{\varepsilon}\delta^{9/2}.
\end{eqnarray}
To bound the last term on the left hand side of \eqref{5.2}, we use the expression of $\overline{G}$ in \eqref{1.32} to write
\begin{equation*}
\partial^{\alpha}L_{M}\overline{G}=\delta^{1/2}\varepsilon\partial^{\alpha}\{ P_{1}v_{1}M(\frac{|v-u|^{2}
		\partial_x\bar{\theta}}{2K\theta^{2}}+\frac{(v-u)\cdot\partial_x\bar{u}}{K\theta})\}.
\end{equation*}
By the expression of $I_{2}$ in \eqref{3.28}, we get from \eqref{1.20}, \eqref{1.19}, \eqref{3.11} and \eqref{4.17a}  that
\begin{eqnarray}
	\label{5.8a}
	\frac{1}{\delta^{3/2}\varepsilon}|(\frac{\partial^{\alpha}L_{M}\overline{G}}{\sqrt{\mu}},\frac{I_2}{\sqrt{\mu}})|
	&&\leq C\frac{1}{\delta}\|\partial_{x}(\rho,u,\theta)\|_{L^{\infty}}\|\partial_{x}(\rho,u,\theta)\|\|\frac{\partial^{\alpha}\{ P_{1}v_{1}M(\frac{|v-u|^{2}
			\partial_x\bar{\theta}}{2K\theta^{2}}+\frac{(v-u)\cdot\partial_x\bar{u}}{K\theta})\}}{\sqrt{\mu}}\|
	\notag\\
	&&\leq C\delta\|\frac{\partial^{\alpha}\{ P_{1}v_{1}M(\frac{|v-u|^{2}
			\partial_x\bar{\theta}}{2K\theta^{2}}+\frac{(v-u)\cdot\partial_x\bar{u}}{K\theta})\}}{\sqrt{\mu}}\|
\notag\\
&&		\leq  C\|\partial^{\alpha}(\widetilde{\rho},\widetilde{u},\widetilde{\theta})\|^2+C\delta^2.
\end{eqnarray}
By the expression of $\widetilde{I}^{2}$ and $\widetilde{I}^{3}$ in \eqref{5.4},
 \eqref{3.24} and \eqref{3.11}, one has
\begin{equation*}
	\frac{1}{\delta^{3/2}\varepsilon}|(\frac{\partial^{\alpha}L_{M}\overline{G}}{\sqrt{\mu}},\frac{\widetilde{I}^{2}+\widetilde{I}^{3}}{\sqrt{\mu}})|
	\leq  C\|\partial^{\alpha}(\widetilde{\rho},\widetilde{u},\widetilde{\theta})\|^2+C\delta^2.
\end{equation*}
Recall $\widetilde{I}^{1}$ given by \eqref{5.4}, from \eqref{1.20} and \eqref{1.21}, we find that
\begin{eqnarray}
	\label{5.9A}
&&\frac{1}{\delta^{3/2}\varepsilon}(\frac{\partial^{\alpha}L_{M}\overline{G}}{\sqrt{\mu}},\frac{\widetilde{I}^{1}}{\sqrt{\mu}})
\notag\\
&&=\frac{1}{\delta}(\partial^{\alpha}P_{1}v_{1}M(\frac{|v-u|^{2}
\partial_x\bar{\theta}}{2K\theta^{2}}+\frac{(v-u)\cdot\partial_x\bar{u}}{K\theta}),
\frac{\partial^{\alpha}\rho}{\rho}+\frac{(v-u)\cdot\partial^{\alpha}u}{K\theta}+(\frac{|v-u|^{2}}{2K\theta}-\frac{3}{2})\frac{\partial^{\alpha}\theta}{\theta})=0.
\end{eqnarray}
With the above two estimates and $I_1=\widetilde{I}^{1}+\widetilde{I}^{2}+\widetilde{I}^{3}$ in hand, we have
\begin{equation*}
	\frac{1}{\delta^{3/2}\varepsilon}|(\frac{\partial^{\alpha}L_{M}\overline{G}}{\sqrt{\mu}},\frac{I_1}{\sqrt{\mu}})|
	\leq  C\|\partial^{\alpha}(\widetilde{\rho},\widetilde{u},\widetilde{\theta})\|^2+C\delta^2.
\end{equation*}
This and \eqref{5.8a} as well as $\partial^{\alpha}M=I_{1}+I_{2}$ together gives
\begin{equation*}
	\frac{1}{\delta^{3/2}\varepsilon}|(\frac{\partial^{\alpha}L_{M}\overline{G}}{\sqrt{\mu}},\frac{\partial^{\alpha}M}{\sqrt{\mu}})|
	\leq  C\|\partial^{\alpha}(\widetilde{\rho},\widetilde{u},\widetilde{\theta})\|^2+C\delta^2.
\end{equation*}
From the Cauchy-Schwarz inequality and \eqref{3.7}, we get
\begin{equation*}
\frac{1}{\delta^{3/2}\varepsilon}|(\frac{\partial^{\alpha}L_{M}\overline{G}}{\sqrt{\mu}},\frac{\partial^{\alpha}f+\partial^{\alpha}\overline{G}}{\sqrt{\mu}})|
\leq \eta\frac{1}{\delta^{3/2}\varepsilon}\|\partial^{\alpha}f\|^2_{\sigma}
+C_\eta\frac{1}{\delta^{3/2}\varepsilon}\delta^3(\varepsilon+\delta^2)^2.
\end{equation*}
Hence, we deduce from the above two estimates that
\begin{equation}
	\label{5.10A}
	\frac{1}{\delta^{3/2}\varepsilon}|(\frac{\partial^{\alpha}L_{M}\overline{G}}{\sqrt{\mu}},\frac{\partial^{\alpha}F}{\sqrt{\mu}})|
	\leq  \eta\frac{1}{\delta^{3/2}\varepsilon}\|\partial^{\alpha}f\|^2_{\sigma}
	+C\|\partial^{\alpha}(\widetilde{\rho},\widetilde{u},\widetilde{\theta})\|^2
	+C\delta^2+C_\eta\frac{1}{\varepsilon}\delta^{3/2}(\varepsilon+\delta^2)^2.
\end{equation}
For the first term on the right hand side of \eqref{5.2}, by \eqref{3.37a}, one gets 
	\begin{eqnarray*}	
	&&	\frac{1}{\delta^{3/2}\varepsilon}|(\partial^\alpha\Gamma(\frac{M-\mu}{\sqrt{\mu}},f)+\partial^\alpha\Gamma(f,\frac{M-\mu}{\sqrt{\mu}}),\frac{\partial^\alpha F}{\sqrt{\mu}})|
	\notag\\
\leq&& C(\eta+\delta)\frac{1}{\delta^{3/2}\varepsilon}\|\partial^{\alpha}f\|^2_{\sigma}
+C_\eta\frac{\delta^{1/2}}{\varepsilon}\|\partial^{\alpha}(\widetilde{\rho},\widetilde{u},\widetilde{\theta})\|^2
\notag\\
&&+C_\eta(\frac{\delta^{5/4}}{\varepsilon^{2}}+\frac{1}{\delta^{1/4}})\mathcal{D}_{2}(t)+C_\eta\varepsilon\delta^{5/4}
+C_\eta\frac{1}{\varepsilon}\delta^{9/2}.	
\end{eqnarray*}
From \eqref{3.57a}, the last term of \eqref{5.2} can be bounded by
\begin{equation*}
	\frac{1}{\delta^{3/2}\varepsilon}|\partial^{\alpha}(\Gamma(\frac{G}{\sqrt{\mu}},\frac{G}{\sqrt{\mu}}),\frac{\partial^{\alpha}F}{\sqrt{\mu}})|
	\leq \widetilde{C}\delta\frac{1}{\delta^{3/2}\varepsilon}\|\partial^{\alpha}f\|^2_{\sigma}+\widetilde{C}(1+\frac{\delta^2}{\varepsilon^{2}})\mathcal{D}_{2}(t)
	+C\frac{1}{\varepsilon}\delta^{7/2}+C\delta^{3/2}\varepsilon^3.
\end{equation*}
In summary, collecting all the above estimates and taking a small $\eta>0$, we have from \eqref{5.2} that
\begin{eqnarray}
	\label{5.8}
	&&\frac{1}{2}\frac{d}{dt}\sum_{|\alpha|=2}\{\|\frac{\partial^{\alpha}F}{\sqrt{\mu}}\|^{2}
	+(\frac{1}{K\theta}e^{\phi}\partial^{\alpha}\widetilde{\phi},\partial^{\alpha}\widetilde{\phi})
	+\delta(\frac{1}{K\theta}\partial^{\alpha}\partial_x\widetilde{\phi},\partial^{\alpha}\partial_x\widetilde{\phi})\}
	+c\frac{1}{\delta^{3/2}\varepsilon}\sum_{|\alpha|=2}\|\partial^{\alpha}f\|_{\sigma}^{2}
	\notag\\
&&\leq C\frac{\delta^{1/2}}{\varepsilon}\sum_{|\alpha|=2}\|\partial^{\alpha}(\widetilde{\rho},\widetilde{u},\widetilde{\theta})\|^2
+C\varepsilon^{1/2}\delta^{1/2}\sum_{|\alpha|=2}\|\langle v\rangle \partial^{\alpha}f\|^2	
+C(\frac{1}{\varepsilon}+\frac{\delta^{5/4}}{\varepsilon^2}+\frac{1}{\delta^{1/4}})\mathcal{D}_{2,l,q_1}(t)
\notag\\
&&\hspace{0.5cm}+C\frac{1}{\delta^2}\mathcal{E}_{2}(t)+C\delta^{5/4}\varepsilon+C\frac{1}{\varepsilon}\delta^{7/2}
+C\delta^{\frac{1}{2}}\varepsilon^2+C\delta^2.	
\end{eqnarray}
Here we have used \eqref{1.36} and the smallness of $\delta$ and $\varepsilon$.

Multiplying \eqref{5.8} by $\frac{\varepsilon^{2}}{\delta}$, 
we obtain \eqref{5.1} and then complete the proof of Lemma \ref{lem5.1}.
\end{proof}
Finally, based on Lemma \ref{lem5.1} and Lemma \ref{lem4.3}, we complete the energy estimates without weight functions.
\begin{lemma}\label{lem5.2}
There exists a constant $\kappa_3>0$ with $\kappa_2\gg \kappa_3$ such that
	\begin{eqnarray}
		\label{5.9}
		&&\frac{d}{dt}\widetilde{E}(t)+\frac{1}{2}\kappa_3\frac{\varepsilon^2}{\delta}\frac{d}{dt}\sum_{|\alpha|=2}\{\|\frac{\partial^{\alpha}F}{\sqrt{\mu}}\|^{2}
		+(\frac{1}{K\theta}e^{\phi}\partial^{\alpha}\widetilde{\phi},\partial^{\alpha}\widetilde{\phi})
		+\delta(\frac{1}{K\theta}\partial^{\alpha}\partial_x\widetilde{\phi},\partial^{\alpha}\partial_x\widetilde{\phi})\}+c\mathcal{D}_{2}(t)
		\notag\\
		\leq&& C\delta^{1/4}\mathcal{D}_{2,l,q_1}(t)
		+C\mathcal{E}_{2}(t)+C\delta^4+C\varepsilon^{1/2}\delta^{1/2}
		\{\sum_{|\alpha|\leq 1}\|\langle v\rangle \partial^\alpha f\|^2+\frac{\varepsilon^2}{\delta}\sum_{|\alpha|=2}\|\langle v\rangle \partial^{\alpha}f\|^2\},
	\end{eqnarray}	
where $\widetilde{E}(t)$ is given in \eqref{4.39} and $\kappa_2$ stated as in \eqref{4.38a}.
\end{lemma}
\begin{proof}
In order to close the energy estimates, from \eqref{5.1}, we need to require that
\begin{equation}
	\label{5.10}
C\delta^{1/4}\varepsilon^3+C\delta^{5/2}\varepsilon+C\frac{\varepsilon^4}{\delta^{1/2}}+C\delta\varepsilon^2
\leq C\delta^4, \quad \mbox{and} \quad \frac{\varepsilon^2}{\delta^3}\mathcal{E}_{2}(t)\leq\mathcal{E}_{2}(t),
\quad \Leftrightarrow\quad  \varepsilon^{2/3}\leq \delta.
\end{equation}
Under the conditions of \eqref{5.10}, we see easily that
$$
\frac{\varepsilon}{\delta}\leq \delta^{1/2}, \quad \frac{\varepsilon^{2}}{\delta^{5/4}}\leq \delta^{7/4}.
$$
Adding \eqref{5.1}$\times\kappa_3$ to \eqref{4.38a}, and firstly choosing $\kappa_3>0$ sufficiently small 
such that $C\kappa_3<\frac{1}{4}c\kappa_1$, then taking $\delta$ small enough such that 
$\frac{1}{2}c\kappa_3\frac{1}{\delta^{3/2}\varepsilon}\frac{\varepsilon^2}{\delta}>C\frac{\varepsilon}{\delta^{1/2}}$,
we can prove the desired estimate \eqref{5.9} holds by using \eqref{5.10}.
This completes the proof of Lemma \ref{lem5.2}.
 \end{proof}
\section{Weighted energy estimates}\label{sec.6}
Since the linearized Landau operator with Coulomb interaction has no
spectral gap that results in the very weak velocity dissipation, the 
large-velocity growth in the nonlinear electric potential terms and
the free streaming term are hard to control. Therefore, in this section we adapt techniques in \cite{Guo-JAMS}
 basing on the velocity weight function $w(\alpha,\beta)$ in \eqref{1.25} to overcome these difficulties.
The weight function will be acted on  the microscopic component $f$ for the equation \eqref{2.39}.
\begin{lemma}\label{lem6.1}
There exists a constant $0<\kappa_4<1$ with $\kappa_3\gg \kappa_4$ such that
\begin{eqnarray}
	\label{6.1}
	&&\frac{d}{dt}\big\{\sum_{|\alpha|\leq 1}\|\partial^\alpha f\|_{w}^2+\varepsilon^2\sum_{|\alpha|=2}\|\frac{\partial^\alpha F}{\sqrt{\mu}}\|_{w}^2+\kappa_4\sum_{|\alpha|+|\beta|\leq 2,|\beta|\geq 1}\|\partial_{\beta}^\alpha f\|_{w}^2\big\}
	\notag\\
		&&+c\frac{1}{\delta^{3/2}\varepsilon}\big\{\sum_{|\alpha|\leq 1}\|\partial^\alpha f\|^2_{\sigma,w}+\varepsilon^2\sum_{|\alpha|=2}\| \partial^\alpha f\|_{\sigma,w}^2+\sum_{|\alpha|+|\beta|\leq 2,|\beta|\geq 1}\|\partial_{\beta}^\alpha f\|_{\sigma,w}^2\big\}
		\notag\\
	&&+\kappa_4q_{1}q_{2}(1+t)^{-(1+q_2)}\mathcal{H}_{2,l,q_1}(t)
	\notag\\
	&&\leq C\mathcal{D}_{2}(t)+C\delta^{1/4}\mathcal{D}_{2,l,q_1}(t)+C\delta^4
	+C\varepsilon^{1/2}\delta^{1/2}\mathcal{H}_{2,l,q_1}(t).
\end{eqnarray}
Here $\kappa_3$ stated as in \eqref{5.9} and $\mathcal{H}_{2,l,q_1}(t)$ is defined by
\begin{eqnarray}
	\label{6.2a}
	\mathcal{H}_{2,l,q_1}(t):=&&\sum_{|\alpha|\leq1}\|\langle v\rangle \partial^{\alpha}f(t)\|_{w}^{2}+
	\varepsilon^2\sum_{|\alpha|=2}\|\langle v\rangle \partial^{\alpha}f(t)\|_{w}^{2}
	\notag\\
	&&\hspace{1cm}+\sum_{|\alpha|+|\beta|\leq 2,|\beta|\geq1}\|\langle v\rangle \partial^{\alpha}_{\beta}f(t)\|_{w}^{2}.
\end{eqnarray}	
\end{lemma}
\begin{proof}
The proof is divided into the following four steps.	
	
\noindent{\it \Red{Step 1.}}	
We starting from weighted estimates on zeroth and first order. Let $|\alpha|\leq 1$,
by taking $\partial^{\alpha}$ derivatives of \eqref{2.39}  and  taking
the inner product of the resulting equation with $w^2(\alpha,0)\partial^\alpha f$ over $\mathbb{R}\times{\mathbb R}^3$, one has
	\begin{eqnarray}
	\label{6.2}
	&&(\partial_t\partial^{\alpha}f,w^2(\alpha,0)\partial^{\alpha}f)
-\frac{1}{\delta}(\partial^{\alpha}[\frac{\partial_x\phi\partial_{v_{1}}(\sqrt{\mu}f)}{\sqrt{\mu}}],w^2(\alpha,0)\partial^{\alpha}f)
-\frac{1}{\delta^{3/2}\varepsilon}(\mathcal{L}\partial^{\alpha}f,w^2(\alpha,0)\partial^{\alpha}f)
	\notag\\
	=&&\frac{1}{\delta^{3/2}\varepsilon}(\partial^{\alpha}\Gamma(\frac{M-\mu}{\sqrt{\mu}},f)
	+\partial^{\alpha}\Gamma(f,\frac{M-\mu}{\sqrt{\mu}})+\partial^{\alpha}\Gamma(\frac{G}{\sqrt{\mu}},\frac{G}{\sqrt{\mu}}),w^2(\alpha,0)\partial^{\alpha}f)
	\nonumber\\
	&&+\frac{1}{\delta}(\frac{\partial^{\alpha}P_{0}(v_1\sqrt{\mu}\partial_{x}f)}{\sqrt{\mu}}
	-\frac{1}{\sqrt{\mu}}\partial^{\alpha}P_{1}\big\{v_{1}M(\frac{|v-u|^{2}\partial_x\widetilde{\theta}}{2K\theta^{2}}+\frac{(v-u)\cdot\partial_x\widetilde{u}}{K\theta})\big\},w^2(\alpha,0)\partial^{\alpha}f)
	\notag\\
	&&+(A\frac{1}{\delta}\frac{\partial^{\alpha}\partial_x\overline{G}}{\sqrt{\mu}}+\frac{1}{\delta}\frac{\partial^{\alpha}[\partial_x\phi\partial_{v_{1}}\overline{G}]}{\sqrt{\mu}}-\frac{1}{\delta}\frac{\partial^{\alpha}P_{1}(v_1\partial_x\overline{G})}{\sqrt{\mu}}
	-\frac{\partial^{\alpha}\partial_{t}\overline{G}}{\sqrt{\mu}},w^2(\alpha,0)\partial^{\alpha}f).
\end{eqnarray}
Here we have used the weight function $w(\alpha,0)$ in \eqref{1.25} and the integration by parts such that
$$
-A\frac{1}{\delta}(\partial_x\partial^{\alpha}f,w^2(\alpha,0)\partial^{\alpha}f)+\frac{1}{\delta}(v_{1}\partial_x\partial^{\alpha}f,w^2(\alpha,0)\partial^{\alpha}f)=0.
$$

In the following we shall compute \eqref{6.2} term by term. 
Note from the definition of $w(\alpha,\beta)$ in \eqref{1.25} that
\begin{equation}
\label{6.3}
\partial_{t}[w^{2}(\alpha,\beta)] =-q_{1}q_{2}(1+t)^{-(1+q_2)}\langle v\rangle^2 w^{2}(\alpha,\beta).
\end{equation}
This immediately gives
\begin{eqnarray*}
(\partial_{t}\partial^{\alpha}f,w^2(\alpha,0)\partial^{\alpha}f)=&&\frac{1}{2}\frac{d}{dt}(\partial^\alpha f,w^2(\alpha,0)\partial^\alpha f)
-\frac{1}{2}(\partial^\alpha f,\partial_{t}[w^2(\alpha,0)]\partial^\alpha f)
\notag\\
=&&\frac{1}{2}\frac{d}{dt}\|\partial^\alpha f\|_{w}^2+\frac{1}{2}q_{1}q_{2}(1+t)^{-(1+q_2)}\|\langle v\rangle \partial^\alpha f\|^2_{w}.
\end{eqnarray*}
Note from \eqref{3.83} that the second term on the left hand side of \eqref{6.2} is bounded by
\begin{equation*}
 C\delta^{1/2}\mathcal{D}_{2,l,q_1}(t)
	+C\varepsilon^{1/2}\delta^{1/2}\|\langle v\rangle \partial^\alpha f\|_w^2.
\end{equation*}
For the last term on the left hand side of \eqref{6.2}, we get clearly by \eqref{3.19} that
$$
-\frac{1}{\delta^{3/2}\varepsilon}(\mathcal{L}\partial^\alpha f,w^2(\alpha,0)\partial^\alpha f)
\geq c\frac{1}{\delta^{3/2}\varepsilon}\|\partial^\alpha f\|^2_{\sigma,w}-C\frac{1}{\delta^{3/2}\varepsilon}\|\partial^\alpha f\|^2_{\sigma}.
$$
From \eqref{3.26}, \eqref{3.42} and \eqref{4.5}, 
the first line on the right hand side of \eqref{6.2} can be controlled by
\begin{equation*}
 C\eta\frac{1}{\delta^{3/2}\varepsilon}\|\partial^\alpha f\|^{2}_{\sigma,w}
	+C\delta\mathcal{D}_{2,l,q_1}(t)+C_\eta\delta^{1/2}\delta^4.
\end{equation*}
Performing the similar calculations as \eqref{4.38A} and using \eqref{1.35},
the second line on the right hand side of \eqref{6.2} can be controlled by
\begin{equation*}
C\eta\frac{1}{\delta^{3/2}\varepsilon}\|\partial^\alpha f\|^{2}_{\sigma,w}
+C_{\eta}\frac{\varepsilon}{\delta^{1/2}}(\|\partial^{\alpha}\partial_{x}f\|_\sigma^{2}+\|\partial^{\alpha}\partial_{x}(\widetilde{u},\widetilde{\theta})\|^{2})+C_{\eta}\delta^{1/2}\mathcal{D}_{2,l,q_1}(t).
\end{equation*}
The last line of \eqref{6.2} can be estimated by \eqref{3.7} and \eqref{4.24} that
\begin{eqnarray*}
&& \eta\frac{1}{\delta^{3/2}\varepsilon}\|\langle v\rangle^{-\frac{1}{2}}w(\alpha,0)\partial^{\alpha}f\|^{2}
+C_{\eta}\frac{\varepsilon}{\delta^{1/2}}
\|\langle v\rangle^{\frac{1}{2}}w(\alpha,0)\frac{\partial^{\alpha}[\partial_x\phi\partial_{v_{1}}\overline{G}]}{\sqrt{\mu}}\|^{2}
\notag\\
&&\hspace{0.3cm}+C_{\eta}\frac{\varepsilon}{\delta^{1/2}}
\|\langle v\rangle^{\frac{1}{2}}w(\alpha,0)(\frac{\partial^{\alpha}\partial_x\overline{G}}{\sqrt{\mu}}-\frac{\partial^{\alpha}P_{1}(v_1\partial_x\overline{G})}{\sqrt{\mu}}
-\delta^2\frac{\partial^{\alpha}\partial_{t}\overline{G}}{\sqrt{\mu}})\|^{2}
\notag\\
&&\leq C\eta\frac{1}{\delta^{3/2}\varepsilon}\|\partial^\alpha f\|^{2}_{\sigma,w}+C_\eta\frac{\varepsilon}{\delta^{1/2}}\delta^3(\varepsilon+\delta^2)^2.
\end{eqnarray*}
Here we have used the fact that $\|\partial^{\alpha}\partial_x\phi\|\leq C$.

Hence, for $|\alpha|\leq1$, substituting the above estimates into \eqref{6.2}
and taking a small $\eta>0$, then the summation of the resulting equations over $|\alpha|$
through a suitable linear combination gives
\begin{eqnarray}
	\label{6.4}
	&&\sum_{|\alpha|\leq 1}\big\{\frac{d}{dt}\|\partial^\alpha f\|_{w}^2+
	q_{1}q_{2}(1+t)^{-(1+q_2)}\|\langle v\rangle \partial^\alpha f\|^2_{w}
	+c\frac{1}{\delta^{3/2}\varepsilon}\|\partial^\alpha f\|^2_{\sigma,w}\big\}
	\notag\\
	&&\leq C\frac{1}{\delta^{3/2}\varepsilon}\sum_{|\alpha|\leq 1}\|\partial^\alpha f\|^2_{\sigma}
	+C\frac{\varepsilon}{\delta^{1/2}}\sum_{|\alpha|\leq 1}(\|\partial^\alpha\partial_xf\|^2_{\sigma}
	+\|\partial^\alpha\partial_x(\widetilde{u},\widetilde{\theta})\|^2)
	\notag\\
	&&\hspace{0.5cm}+C\delta^{1/2}\mathcal{D}_{2,l,q_1}(t)+C\delta^{1/2}\delta^4+C\frac{\varepsilon}{\delta^{1/2}}\delta^3\varepsilon^2
	+C\varepsilon^{1/2}\delta^{1/2}\sum_{|\alpha|\leq 1}\|\langle v\rangle \partial^\alpha f\|_w^2
		\notag\\
	&&\leq C\mathcal{D}_{2}(t)+C\delta^{1/2}\mathcal{D}_{2,l,q_1}(t)+C\delta^{1/2}\delta^4
	+C\varepsilon^{1/2}\delta^{1/2}\sum_{|\alpha|\leq 1}\|\langle v\rangle \partial^\alpha f\|_w^2,
\end{eqnarray}
where \eqref{5.10} and \eqref{1.36} have used in the last inequality. 

\noindent{\it \Red{Step 2.}}
As before,  we take $\partial^{\alpha}=\partial_{xx}$ derivatives of \eqref{2.40} and take the inner product of the resulting equation with $w^2(\alpha,0)\frac{\partial^\alpha F}{\sqrt{\mu}}$  over $\mathbb{R}^{3}\times\mathbb{R}^{3}$ to obtain
	\begin{eqnarray}
		\label{6.5}
		&&(\partial_t(\frac{\partial^\alpha F}{\sqrt{\mu}}),w^2(\alpha,0)\frac{\partial^\alpha F}{\sqrt{\mu}})
	-\frac{1}{\delta}(\frac{\partial^\alpha(\partial_x\phi\partial_{v_{1}}F)}{\sqrt{\mu}},w^2(\alpha,0)\frac{\partial^\alpha F}{\sqrt{\mu}})
		\notag\\
		=&&\frac{1}{\delta^{3/2}\varepsilon}(\mathcal{L} \partial^\alpha f,w^2(\alpha,0)\frac{\partial^\alpha F}{\sqrt{\mu}})
		+\frac{1}{\delta^{3/2}\varepsilon}(\partial^\alpha\Gamma(f,\frac{M-\mu}{\sqrt{\mu}})
		+\partial^\alpha\Gamma(\frac{M-\mu}{\sqrt{\mu}},f),w^2(\alpha,0)\frac{\partial^\alpha F}{\sqrt{\mu}})
		\notag\\
		&&+\frac{1}{\delta^{3/2}\varepsilon}(\partial^\alpha\Gamma(\frac{G}{\sqrt{\mu}},\frac{G}{\sqrt{\mu}}),w^2(\alpha,0)\frac{\partial^\alpha F}{\sqrt{\mu}})
		+\frac{1}{\delta^{3/2}\varepsilon}(\frac{\partial^{\alpha}L_{M}\overline{G}}{\sqrt{\mu}},w^2(\alpha,0)\frac{\partial^\alpha F}{\sqrt{\mu}}).
	\end{eqnarray}
Here we have used the integration by parts such that
$$
-A\frac{1}{\delta}(\frac{\partial_x\partial^{\alpha}F}{\sqrt{\mu}},w^2(\alpha,0)\frac{\partial^{\alpha}F}{\sqrt{\mu}})+\frac{1}{\delta}
(\frac{v_{1}\partial_x\partial^{\alpha}F}{\sqrt{\mu}},w^2(\alpha,0)\frac{\partial^{\alpha}F}{\sqrt{\mu}})=0.
$$

In the following we estimate \eqref{6.5} term by term. 
By $F=M+\overline{G}+\sqrt{\mu}f$, \eqref{3.29} and \eqref{3.7}, one obtains
\begin{eqnarray*}
	\|\langle v\rangle \frac{\partial^\alpha F}{\sqrt{\mu}}\|_{w}^2
	\geq&& \frac{1}{2}\|\langle v\rangle \partial^\alpha f\|_{w}^2-C\|\langle v\rangle \frac{\partial^\alpha (M +\overline{G})}{\sqrt{\mu}}\|_{w}^2
	\notag\\
	\geq&& \frac{1}{2}\|\langle v\rangle \partial^\alpha f\|_{w}^2
	-C(\|\partial^\alpha(\widetilde{\rho},\widetilde{u},\widetilde{\theta})\|^2+\delta^2).
\end{eqnarray*}
It follows from this and \eqref{6.3} that
\begin{eqnarray*}
	(\partial_t(\frac{\partial^\alpha F}{\sqrt{\mu}}),w^2(\alpha,0)\frac{\partial^\alpha F}{\sqrt{\mu}})
	=&&\frac{1}{2}\frac{d}{dt}\|\frac{\partial^\alpha F}{\sqrt{\mu}}\|_{w}^2
	+\frac{1}{2}q_{1}q_{2}(1+t)^{-(1+q_2)}\|\langle v\rangle \frac{\partial^\alpha F}{\sqrt{\mu}}\|_{w}^2
	\notag\\
	\geq&& \frac{1}{2}\frac{d}{dt}\|\frac{\partial^\alpha F}{\sqrt{\mu}}\|_{w}^2
	+\frac{1}{4}q_{1}q_{2}(1+t)^{-(1+q_2)}\|\langle v\rangle \partial^\alpha f\|_{w}^2
	\notag\\
	&&-Cq_{1}q_{2}(1+t)^{-(1+q_2)}(\|\partial^\alpha(\widetilde{\rho},\widetilde{u},\widetilde{\theta})\|^2+\delta^2).
\end{eqnarray*}
Note from \eqref{3.84} that the second term on the left hand side of \eqref{6.5} can be bounded by
\begin{equation*}
	C(\frac{\delta^{1/2}}{\varepsilon^2}+\frac{1}{\delta\varepsilon})\mathcal{D}_{2,l,q_1}(t)+C\varepsilon^{1/2}\delta^{1/2}\|\langle v\rangle \partial^{\alpha}f\|_w^2
	+C\delta.
\end{equation*}
To bound the first term on the right hand side of \eqref{6.5}, we first have from \eqref{3.19}  that
$$
-\frac{1}{\delta^{3/2}\varepsilon}(\mathcal{L} \partial^\alpha f,w^2(\alpha,0)\partial^\alpha f)
\geq c\frac{1}{\delta^{3/2}\varepsilon}\|\partial^\alpha f\|_{\sigma,w}^2-C\frac{1}{\delta^{3/2}\varepsilon}\|\partial^\alpha f\|_{\sigma }^2.
$$
By $\mathcal{L}f=\Gamma(f,\sqrt{\mu})+\Gamma(\sqrt{\mu},f)$, \eqref{3.21} and \eqref{3.7}, we have
\begin{eqnarray*}
	\frac{1}{\delta^{3/2}\varepsilon}|(\mathcal{L} \partial^\alpha f,w^2(\alpha,0)\frac{\partial^\alpha \overline{G}}{\sqrt{\mu}})|&&\leq
	C\frac{1}{\delta^{3/2}\varepsilon}\|\partial^\alpha f\|_{\sigma}\|w^2(\alpha,0)\frac{\partial^\alpha \overline{G}}{\sqrt{\mu}}\|_{\sigma}
	\notag\\
	&&\leq C\frac{1}{\delta^{3/2}\varepsilon}\|\partial^\alpha f\|^2_{\sigma}+C\frac{1}{\varepsilon}\delta^{3/2}(\varepsilon+\delta^2)^2.
\end{eqnarray*}
Note that $\mathcal{L}f=\Gamma(f,\sqrt{\mu})+\Gamma(\sqrt{\mu},f)$ and $|M-\mu|_2\leq C\delta$, we
perform the similar calculations as \eqref{3.42a} to get
\begin{eqnarray*}
&&\frac{1}{\delta^{3/2}\varepsilon}|(\mathcal{L} \partial^\alpha f,w^2(\alpha,0)\frac{\partial^\alpha M}{\sqrt{\mu}})
\notag\\
&&\leq C\frac{1}{\delta}(\frac{1}{\delta^{3/2}\varepsilon}\|\partial^{\alpha}f\|^2_{\sigma}+\frac{\delta^{1/2}}{\varepsilon}\|\partial^{\alpha}(\widetilde{\rho},\widetilde{u},\widetilde{\theta})\|^2+\frac{\delta^{5/4}}{\varepsilon^{2}}\mathcal{D}_{2,l,q_1}(t)+\varepsilon\delta^{5/4}
+\frac{1}{\varepsilon}\delta^{9/2}).	
\end{eqnarray*}
Hence, we get from the above three estimates and  $F=M+\overline{G}+\sqrt{\mu}f$ that
\begin{eqnarray}
	\label{6.7AB}
	\frac{1}{\delta^{3/2}\varepsilon}(\mathcal{L} \partial^\alpha f,w^2(\alpha,0)\frac{\partial^\alpha F}{\sqrt{\mu}})\leq&&
	-c\frac{1}{\delta^{3/2}\varepsilon}\|\partial^\alpha f\|_{\sigma,w}^2
	+C\frac{1}{\delta}\frac{1}{\delta^{3/2}\varepsilon}\|\partial^{\alpha}f\|^2_{\sigma}
	\notag\\
	&&+C\delta\frac{1}{\delta^{3/2}\varepsilon}\|\partial^{\alpha}(\widetilde{\rho},\widetilde{u},\widetilde{\theta})\|^2+C\delta^{1/4}\frac{1}{\varepsilon^{2}}\mathcal{D}_{2,l,q_1}(t)+C\varepsilon\delta^{1/4}+C\frac{1}{\varepsilon}\delta^{7/2}.		
\end{eqnarray}
From \eqref{3.27}, we know that the second term on the right hand side of \eqref{6.5} can be controlled by
	\begin{equation*}	
 C\eta\frac{1}{\delta^{3/2}\varepsilon}\|\partial^{\alpha}f\|^2_{\sigma,w}+C_\eta(\frac{1}{\varepsilon}+\frac{\delta}{\varepsilon^{2}})\mathcal{D}_{2,l,q_1}(t)
	+C_\eta\frac{1}{\varepsilon}\delta^{5/2}.
\end{equation*}
Using  \eqref{3.43}, one gets
\begin{eqnarray*}	
&&\frac{1}{\delta^{3/2}\varepsilon}|\partial^{\alpha}(\Gamma(\frac{G}{\sqrt{\mu}},\frac{G}{\sqrt{\mu}}),w^2(\alpha,0)\frac{\partial^{\alpha}F}{\sqrt{\mu}})|
\notag\\
&&\leq  \widetilde{C}\delta\frac{1}{\delta^{3/2}\varepsilon}\|\partial^{\alpha}f\|^2_{\sigma,w}+\widetilde{C}(1+\frac{\delta^2}{\varepsilon^{2}})\mathcal{D}_{2,l,q_1}(t)
+C\frac{1}{\varepsilon}\delta^{7/2}+C\delta^{3/2}\varepsilon^3.
\end{eqnarray*}
Note that some cancellation property in \eqref{5.9A} is again not available for
the last term of \eqref{6.5} due to the weight function. It is hard to obtain a good estimate as \eqref{5.10A}.
We  only obtain  the following estimate
\begin{eqnarray}
	\label{6.7A}
\frac{1}{\delta^{3/2}\varepsilon}|(\frac{\partial^{\alpha}L_{M}\overline{G}}{\sqrt{\mu}},w^2(\alpha,0)\frac{\partial^\alpha F}{\sqrt{\mu}})|
&&\leq C\frac{1}{\delta}\|\langle v\rangle^{-1/2}\frac{\partial^\alpha F}{\sqrt{\mu}}\|^2
+C\frac{1}{\delta^3\varepsilon^2}\delta\||\langle v\rangle^{1/2}\frac{\partial^{\alpha}L_{M}\overline{G}}{\sqrt{\mu}}w^2(\alpha,0)\|^2
\notag\\
	&&\leq 
C\frac{1}{\delta}(\|\partial^{\alpha}f\|^2_{\sigma}+\|\partial^{\alpha}(\widetilde{\rho},\widetilde{u},\widetilde{\theta})\|^2)
+C\frac{1}{\varepsilon^2}\delta^{5}+C\delta,
\end{eqnarray}
where we have used \eqref{3.7} and the similar arguments as \eqref{3.30}.

In summary, substituting all the above estimates into \eqref{6.5} and
multiplying the resulting equations by $\varepsilon^{2}$, then choosing $\eta>0$ small enough and using \eqref{3.2a},
we show that
\begin{eqnarray}
	\label{6.6}
	&&\varepsilon^{2}
	\sum_{|\alpha|=2}\Big\{\frac{d}{dt}\|\frac{\partial^\alpha F}{\sqrt{\mu}}\|_{w}^2+\frac{1}{2}q_{1}q_{2}(1+t)^{-(1+q_2)}\|\langle v\rangle \partial^\alpha f\|_{w}^2+c\frac{1}{\delta^{3/2}\varepsilon}\| \partial^\alpha f\|_{\sigma,w}^2\Big\}
	\notag\\
	\leq&& C\frac{\varepsilon^2}{\delta}\frac{1}{\delta^{3/2}\varepsilon}\sum_{|\alpha|=2}\| \partial^\alpha f\|_{\sigma}^2+C(\frac{\varepsilon}{\delta^{1/2}}+\frac{\varepsilon^2}{\delta})\sum_{|\alpha|=2}\|\partial^\alpha(\widetilde{\rho},\widetilde{u},\widetilde{\theta})\|^2
	+C\varepsilon^{5/2}\delta^{1/2}\sum_{|\alpha|=2}\|\langle v\rangle \partial^{\alpha}f\|_{w}^2
	\notag\\
	&& +C(\delta^{1/4}+\varepsilon^{1/2}+\frac{\varepsilon}{\delta})\mathcal{D}_{2,l,q_1}(t)+C\varepsilon^2\delta+C\delta^5
	+C\varepsilon\delta^{5/2}+\varepsilon^3\delta^{1/4}
		\notag\\
	\leq&& C\mathcal{D}_{2}(t)+C\varepsilon^{5/2}\delta^{1/2}\sum_{|\alpha|=2}\|\langle v\rangle \partial^{\alpha}f\|_{w}^2
	+C\delta^{1/4}\mathcal{D}_{2,l,q_1}(t)+C\delta^4,
\end{eqnarray}
where in the last inequality we have used \eqref{5.10} and $\varepsilon\leq \delta^{3/2}$ in \eqref{1.36}.

\noindent{\it \Red{Step 3.}}
Let $|\alpha|+|\beta|\leq 2$ and $|\beta|\geq1$, we take $\partial^\alpha_\beta$
 derivatives of \eqref{2.39} and take the inner product of the resulting equation with $w^2(\alpha,\beta)\partial_{\beta}^\alpha f$ over $\mathbb{R}^3\times{\mathbb R}$ to get
	\begin{eqnarray}
		\label{6.7}
		&&(\partial^\alpha_\beta \partial_tf,w^2(\alpha,\beta)\partial_{\beta}^\alpha f)-A\frac{1}{\delta}(\partial_x\partial^\alpha_\beta f,w^2(\alpha,\beta)\partial_{\beta}^\alpha f)
		+\frac{1}{\delta}(v_1\partial^\alpha_\beta\partial_xf,w^2(\alpha,\beta)\partial_{\beta}^\alpha f)
		\notag\\
		&&+C_\beta^{\beta-e_1}\frac{1}{\delta}(\delta^{e_1}_{\beta}\partial^{\alpha+e_1}_{\beta-e_1}f,w^2(\alpha,\beta)\partial_{\beta}^\alpha f)
		-\frac{1}{\delta}(\partial_{\beta}^{\alpha}[\frac{\partial_x\phi\partial_{v_{1}}(\sqrt{\mu}f)}{\sqrt{\mu}}],w^2(\alpha,\beta)\partial_{\beta}^{\alpha}f)
		\notag\\
		=&&\frac{1}{\delta^{3/2}\varepsilon}(\partial^\alpha_\beta\mathcal{L}f,w^2(\alpha,\beta)\partial_{\beta}^\alpha f)
		+\frac{1}{\delta^{3/2}\varepsilon}(\partial^\alpha_\beta\Gamma(f,\frac{M-\mu}{\sqrt{\mu}})+
		\partial^\alpha_\beta\Gamma(\frac{M-\mu}{\sqrt{\mu}},f),w^2(\alpha,\beta)\partial_{\beta}^\alpha f)
		\notag\\
		&&+\frac{1}{\delta^{3/2}\varepsilon}(\partial^\alpha_\beta\Gamma(\frac{G}{\sqrt{\mu}},\frac{G}{\sqrt{\mu}}),w^2(\alpha,\beta)\partial_{\beta}^\alpha f)
		+\frac{1}{\delta}(\partial^\alpha_\beta[\frac{P_{0}(v_{1}\sqrt{\mu}\partial_xf)}{\sqrt{\mu}}],w^2(\alpha,\beta)\partial_{\beta}^\alpha f)
		\notag\\
		&&-\frac{1}{\delta}(\partial^\alpha_\beta\{\frac{1}{\sqrt{\mu}}P_{1}v_{1}M(\frac{|v-u|^{2}
			\partial_x\widetilde{\theta}}{2K\theta^{2}}+\frac{(v-u)\cdot\partial_x\widetilde{u}}{K\theta})\},w^2(\alpha,\beta)\partial_{\beta}^\alpha f)
			\notag\\
		&&+A\frac{1}{\delta}(\partial^\alpha_\beta[\frac{\partial_x\overline{G}}{\sqrt{\mu}}]+
		\partial^\alpha_\beta[\frac{\partial_x\phi\partial_{v_{1}}\overline{G}}{\sqrt{\mu}}]
		-\partial^\alpha_\beta[\frac{P_{1}(v_{1}\partial_x\overline{G})}{\sqrt{\mu}}]
		-\delta\partial^\alpha_\beta[\frac{\partial_t\overline{G}}{\sqrt{\mu}}],w^2(\alpha,\beta)\partial_{\beta}^\alpha f).
	\end{eqnarray}
	Here $\delta^{e_1}_{\beta}=1$ if $e_1\leq\beta$ or  $\delta^{e_1}_{\beta}=0$ otherwise.

In the following, we estimate \eqref{6.7} term by term. First note that $|\alpha|\leq 1$ since
we only consider the cases $|\alpha|+|\beta|\leq 2$ and $|\beta|\geq 1$.
For the first term of \eqref{6.7}, one gets from \eqref{6.3}  that
\begin{eqnarray*}
	(\partial^\alpha_\beta \partial_tf,w^2(\alpha,\beta)\partial^\alpha_\beta f)
	=&&\frac{1}{2}\frac{d}{dt}\|\partial^\alpha_\beta f\|_{w}^2
	-\frac{1}{2}(\partial^\alpha_\beta f,\partial_t[w^2(\alpha,\beta)]\partial^\alpha_\beta f)
	\\
	=&&\frac{1}{2}\frac{d}{dt}\|\partial^\alpha_\beta f\|_{w}^2
	+\frac{1}{2}q_{1}q_{2}(1+t)^{-(1+q_2)}\|\langle v\rangle \partial^\alpha_\beta f\|^2_{w}.
\end{eqnarray*}
The second and third terms on the left hand of \eqref{6.7} vanish after integration by parts. The
fourth term on the left hand side of \eqref{6.7} is bounded by
\begin{eqnarray*}
	&&\frac{1}{\delta}|(\delta^{e_1}_{\beta}\partial^{\alpha+e_1}_{\beta-e_1}f,w^2(\alpha,\beta)\partial_{\beta}^\alpha f)|
	\notag\\
	&&\leq C\frac{\varepsilon}{\delta^{1/2}}\|\langle v\rangle^{-\frac{1}{2}}w(\alpha,\beta)\partial^{\alpha+e_1}_{\beta-e_1}f\|^2
	+C\frac{1}{\delta^{3/2}\varepsilon}\|\langle v\rangle^2\langle v\rangle^{-\frac{3}{2}}w(\alpha,\beta) \partial_{e_1}\partial^\alpha_{\beta-e_1}f\|^2
	\notag\\
	&&\leq C\frac{\varepsilon}{\delta^{1/2}}\|w(\alpha+e_1,\beta-e_1)\partial^{\alpha+e_1}_{\beta-e_1}f\|^2_\sigma
	+C\frac{1}{\delta^{3/2}\varepsilon}\|\langle v\rangle^2w(\alpha,\beta) \partial^\alpha_{\beta-e_1}f\|^2_\sigma
	\notag\\
	&&\leq C\frac{\varepsilon}{\delta^{1/2}}\|\partial^{\alpha+e_1}_{\beta-e_1}f\|^2_{\sigma,w(\alpha+e_1,\beta-e_1)}
	+C\frac{1}{\delta^{3/2}\varepsilon}\| \partial^\alpha_{\beta-e_1}f\|^2_{\sigma,w(\alpha,\beta-e_1)}.
\end{eqnarray*}
Here we have used the facts that \eqref{1.28}, $\partial^\alpha_\beta=\partial_{e_1}\partial^\alpha_{\beta-e_1}$ and $\langle v\rangle^2\langle v\rangle^{2(l-|\alpha|-|\beta|)}=\langle v\rangle^{2(l-|\alpha|-|\beta-e_1|)}$.

For the last term on the left hand side of  \eqref{6.7}, we see by \eqref{3.75} that
\begin{equation*}
	\frac{1}{\delta}|(\partial_{\beta}^{\alpha}[\frac{\partial_x\phi\partial_{v_1}(\sqrt{\mu}f)}{\sqrt{\mu}}],w^2(\alpha,\beta)\partial_{\beta}^{\alpha}f)|
	\leq  C\delta^{1/2}\mathcal{D}_{2,l,q_1}(t)
	+C\varepsilon^{1/2}\delta^{1/2}\|\langle v\rangle \partial_{\beta}^\alpha f\|_w^2.
\end{equation*}
For the first term on the right hand side of \eqref{6.7}, we deduce from \eqref{3.18} that
$$
-\frac{1}{\delta^{3/2}\varepsilon}(\partial^\alpha_\beta\mathcal{L} f,w^2(\alpha,\beta)\partial^\alpha_\beta f)
\geq \frac{1}{\delta^{3/2}\varepsilon}\big(c\|\partial^\alpha_\beta f\|^2_{\sigma,w}-\eta\sum_{|\beta_1|=|\beta|}\|\partial^\alpha_{\beta_1} f\|_{\sigma,w}^2
-C_\eta\sum_{|\beta_1|<|\beta|}\|\partial^\alpha_{\beta_1}f\|_{\sigma,w}^2\big).
$$
Using \eqref{3.22} and \eqref{3.35}, the second and third terms on the right hand side of \eqref{6.7} can be controlled by
\begin{equation*}
C\eta\frac{1}{\delta^{3/2}\varepsilon}\|\partial^\alpha_\beta f\|^{2}_{\sigma,w}
	+C\delta\mathcal{D}_{2,l,q_1}(t)
	+C_\eta\delta^{9/2}\varepsilon^3+C_\eta\frac{1}{\varepsilon}\delta^{17/2}\delta^4. 
\end{equation*}
Similar to \eqref{4.38A}, the fourth and fifth terms on the right hand side of \eqref{6.7} can be controlled by
\begin{equation*}
	C\eta\frac{1}{\delta^{3/2}\varepsilon}\|\partial_{\beta}^\alpha f\|_{\sigma,w}^{2}
	+C_{\eta}\frac{\varepsilon}{\delta^{1/2}}(\|\partial^{\alpha}\partial_{x}f\|_\sigma^{2}+\|\partial^{\alpha}\partial_{x}(\widetilde{u},\widetilde{\theta})\|^{2})+C_\eta\delta\mathcal{D}_{2,l,q_1}(t).
\end{equation*}
For the remaining terms on the right hand side of \eqref{6.7}, we use \eqref{3.7} and \eqref{1.28} to bound them as follow
\begin{eqnarray*}
 &&C\eta\frac{1}{\delta^{3/2}\varepsilon}\|\partial_{\beta}^\alpha f\|_{\sigma,w}^{2}
 +C_\eta\varepsilon\delta^{5/2}(\varepsilon+\delta^2)^2.
\end{eqnarray*}
Hence, substituting the above estimates into \eqref{6.7}, we have for any small $\eta>0$ that
\begin{eqnarray}
	\label{6.8}
	&&\frac{d}{dt}\|\partial_{\beta}^\alpha f\|_{w}^2+
	q_{1}q_{2}(1+t)^{-(1+q_2)}\|\langle v\rangle \partial_{\beta}^\alpha f\|^2_{w}
	+c\frac{1}{\delta^{3/2}\varepsilon}\|\partial_{\beta}^\alpha f\|^2_{\sigma,w}
	\nonumber\\
	\leq&&  \frac{1}{\delta^{3/2}\varepsilon}\big(\eta\sum_{|\beta_1|=|\beta|}\|\partial^\alpha_{\beta_1} f\|_{\sigma,w}^2
	+C_\eta\sum_{|\beta_1|<|\beta|}\|\partial^\alpha_{\beta_1}f\|_{\sigma,w}^2\big)
	+C\frac{\varepsilon}{\delta^{1/2}}\|\partial^{\alpha+e_1}_{\beta-e_1}f\|^2_{\sigma,w(\alpha+e_1,\beta-e_1)}
	\notag\\
	&&+C\frac{1}{\delta^{3/2}\varepsilon}\| \partial^\alpha_{\beta-e_1}f\|^2_{\sigma,w(\alpha,\beta-e_1)}
	+C_\eta\frac{\varepsilon}{\delta^{1/2}}(\|\partial^{\alpha}\partial_{x}f\|_\sigma^{2}+\|\partial^{\alpha}\partial_{x}(\widetilde{u},\widetilde{\theta})\|^{2})
	\nonumber\\
	&&\hspace{0.5cm}+C_\eta\delta^{1/2}\mathcal{D}_{2,l,q_1}(t)+C_\eta\varepsilon\delta^{5/2}(\varepsilon+\delta^2)^2
	+C_\eta\frac{1}{\varepsilon}\delta^{17/2}\delta^4
	+C_\eta\varepsilon^{1/2}\delta^{1/2}\|\langle v\rangle \partial_{\beta}^\alpha f\|_w^2.
\end{eqnarray}
Consequently, by choosing $\eta>0$ small enough, then the summation of \eqref{6.8} over $|\alpha|+|\beta|\leq 2$ and $|\beta|\geq 1$ through a suitable linear combination gives
\begin{eqnarray}
	\label{6.9}
	&&\sum_{|\alpha|+|\beta|\leq 2,|\beta|\geq 1}\big\{\frac{d}{dt}\|\partial_{\beta}^\alpha f\|_{w}^2+
	q_{1}q_{2}(1+t)^{-(1+q_2)}\|\langle v\rangle \partial_{\beta}^\alpha f\|^2_{w}
	+c\frac{1}{\delta^{3/2}\varepsilon}\|\partial_{\beta}^\alpha f\|^2_{\sigma,w}\big\}
	\nonumber\\
	&&\leq C\frac{1}{\delta^{3/2}\varepsilon}\sum_{|\alpha|\leq 1}\|\partial^\alpha f\|^2_{\sigma,w}
	+C\frac{\varepsilon}{\delta^{1/2}}\sum_{1\leq|\alpha|\leq2}\|\partial^\alpha f\|^2_{\sigma,w}
	+C\frac{\varepsilon}{\delta^{1/2}}\sum_{|\alpha|\leq 1}(\|\partial^\alpha\partial_xf\|^2_{\sigma}
	+\|(\partial^\alpha\partial_x(\widetilde{u},\widetilde{\theta})\|^2)
	\nonumber\\
	&&\hspace{0.5cm}+C\delta^{1/2}\mathcal{D}_{2,l,q_1}(t)+C\delta^5
	+C\varepsilon^{1/2}\delta^{1/2}\sum_{|\alpha|+|\beta|\leq 2,|\beta|\geq 1}\|\langle v\rangle \partial_{\beta}^\alpha f\|_w^2.
\end{eqnarray}
Here we have used \eqref{5.10} and \eqref{4.5}.

\noindent{\it \Red{Step 4.}}	
In summary, combining \eqref{6.4}, \eqref{6.6} and \eqref{6.9}$\times \kappa_4$, and choosing $\kappa_4>0$ sufficiently small 
such that $C\kappa_4<\frac{1}{2}c$, we thus obtain the desired estimate \eqref{6.1}.
This ends the proof of Lemma \ref{lem6.1}.
\end{proof}

\section{Global existence and KdV limit}\label{sec.7}
Based on the {\it a priori} estimates obtained in previous Sections \ref{sec.4}-\ref{sec.6},
in this section, we are now in a position to complete the proof of the main result Theorem \ref{thm1.1} stated in
Section \ref{sec.1}.

\subsection{ Global existence}
By taking a constant $C_2>1$ large enough, then adding \eqref{5.9}$\times C_2$ to \eqref{6.1} and using \eqref{1.35}, 
there exist some constant $C_3>1$ and $0<c_2<1$ such that
	\begin{eqnarray}
	\label{7.1}
	&&\frac{d}{dt}\widetilde{\mathcal{E}}_{2,l,q_1}(t)+2c_2\mathcal{D}_{2,l,q_1}(t)+2c_2q_{1}q_{2}(1+t)^{-(1+q_2)}\mathcal{H}_{2,l,q_1}(t)
	\notag\\
	\leq&& C_3\delta^{1/4}\mathcal{D}_{2,l,q_1}(t)
	+C_3\mathcal{E}_{2}(t)+C_3\delta^4+C_3\varepsilon^{1/2}\delta^{1/2}\mathcal{H}_{2,l,q_1}(t)
	\notag\\
	&&+C_3\varepsilon^{1/2}\delta^{1/2}
	\{\sum_{|\alpha|\leq 1}\|\langle v\rangle \partial^\alpha f\|^2+\frac{\varepsilon^2}{\delta}\sum_{|\alpha|=2}\|\langle v\rangle \partial^{\alpha}f\|^2\}.
\end{eqnarray}	
Here we used $\widetilde{\mathcal{E}}_{2,l,q_1}(t)$ to denote that
\begin{eqnarray}
	\label{7.2}
\widetilde{\mathcal{E}}_{2,l,q_1}(t)=&&C_2\widetilde{E}(t)
+\frac{1}{2}C_2\kappa_3\frac{\varepsilon^2}{\delta}\sum_{|\alpha|=2}
\{\|\frac{\partial^{\alpha}F}{\sqrt{\mu}}\|^{2}
+(\frac{1}{K\theta}e^{\phi}\partial^{\alpha}\widetilde{\phi},\partial^{\alpha}\widetilde{\phi})
+\delta(\frac{1}{K\theta}\partial^{\alpha}\partial_x\widetilde{\phi},\partial^{\alpha}\partial_x\widetilde{\phi})\}
\notag\\
&&+\sum_{|\alpha|\leq 1}\|\partial^\alpha f\|_{w}^2+\varepsilon^2\sum_{|\alpha|=2}\|\frac{\partial^\alpha F}{\sqrt{\mu}}\|_{w}^2
+\kappa_4\sum_{|\alpha|+|\beta|\leq 2,|\beta|\geq 1}\|\partial_{\beta}^\alpha f\|_{w}^2.
\end{eqnarray}
To further bound \eqref{7.1}, we use \eqref{6.2a} to see that
$$
\sum_{|\alpha|\leq 1}\|\langle v\rangle \partial^\alpha f\|^2\leq \mathcal{H}_{2,l,q_1}(t), \quad
\frac{\varepsilon^2}{\delta}\sum_{|\alpha|=2}\|\langle v\rangle \partial^{\alpha}f\|^2
\leq \frac{1}{\delta}\mathcal{H}_{2,l,q_1}(t).
$$
We further choose $\delta>0$ small enough such that $C_3\delta^{1/4}<c_2$. Then we have from \eqref{7.1} that
\begin{eqnarray}
	\label{7.3}
	&&\frac{d}{dt}\widetilde{\mathcal{E}}_{2,l,q_1}(t)+c_2\mathcal{D}_{2,l,q_1}(t)+2c_2q_{1}q_{2}(1+t)^{-(1+q_2)}\mathcal{H}_{2,l,q_1}(t)
	\notag\\
	&&\leq C_3\mathcal{E}_{2}(t)+C_3\delta^4+3C_3\frac{\varepsilon^{1/2}}{\delta^{1/2}}\mathcal{H}_{2,l,q_1}(t).
\end{eqnarray}
For any given $\tau>0$ as Theorem \ref{thm1.1}, by choosing $\varepsilon>0$ small enough,  we can ensure that
\begin{equation}
	\label{7.5}
	\tau\leq (\frac{2c_2q_{1}q_{2}}{3C_3}\frac{\varepsilon^{1/3}}{\varepsilon^{1/2}})^{\frac{1}{1+q_2}}
	-1\leq (\frac{2c_2 q_{1}q_{2}}{3C_3}\frac{\delta^{1/2}}{\varepsilon^{1/2}})^{\frac{1}{1+q_2}}
	-1,
\end{equation}
where in the second inequality we have used \eqref{5.10}. Estimate \eqref{7.5} together with $t\leq \tau$ gives
\begin{equation}
\label{7.4}
t\leq (\frac{2c_2q_{1}q_{2}}{3C_3}\frac{\delta^{1/2}}{\varepsilon^{1/2}})^{\frac{1}{1+q_2}}
-1\Rightarrow 
3C_3\frac{\varepsilon^{1/2}}{\delta^{1/2}}\leq 2c_2q_{1}q_{2}(1+t)^{-(1+q_2)}.
\end{equation}
It follows from this and \eqref{7.3} that
\begin{equation}
\label{7.6}
\frac{d}{dt}\widetilde{\mathcal{E}}_{2,l,q_1}(t)+c_2\mathcal{D}_{2,l,q_1}(t)\leq C_3\mathcal{E}_{2}(t)+C_3\delta^4.
\end{equation}
Following the method used as in \cite[lemma 5.10]{Duan-Yang-Yu-2} and using  \eqref{3.1},  \eqref{3.4} and \eqref{1.37}, we claim that
\begin{equation*}
\frac{1}{2}\kappa_3C_2\frac{\varepsilon^2}{\delta}\sum_{|\alpha|=2}\|\frac{\partial^{\alpha}F}{\sqrt{\mu}}\|^{2}
\geq c\frac{\varepsilon^2}{\delta}\sum_{|\alpha|=2}(\|\partial^\alpha f\|^2+\|\partial^{\alpha}(\widetilde{\rho},\widetilde{u},\widetilde{\theta})\|^{2})
-C\delta^4.
\end{equation*}
On the other hand, by  \eqref{3.30} and \eqref{1.37}, one obtains
\begin{equation*}
	\frac{1}{2}\kappa_3C_2\frac{\varepsilon^2}{\delta}\sum_{|\alpha|=2}\|\frac{\partial^{\alpha}F}{\sqrt{\mu}}\|^{2}
	\leq C\frac{\varepsilon^2}{\delta}\sum_{|\alpha|=2}(\|\partial^\alpha f\|^2+\|\partial^{\alpha}(\widetilde{\rho},\widetilde{u},\widetilde{\theta})\|^{2})
	+C\delta^4.
\end{equation*}
Similarly, by \eqref{3.29}, \eqref{3.30}, \eqref{3.7}, \eqref{3.1}, \eqref{3.4} and \eqref{1.37}, it holds that
\begin{eqnarray*}
c\varepsilon^2\sum_{|\alpha|=2}\|\partial^\alpha f\|_{w}^2-C\delta^4\leq\varepsilon^2\sum_{|\alpha|=2}\|\frac{\partial^\alpha F}{\sqrt{\mu}}\|_{w}^2\leq 
	C\varepsilon^2\sum_{|\alpha|=2}\|\partial^\alpha f\|_{w}^2
	+C\delta^4.
\end{eqnarray*}
By the definition of $\widetilde{E}(t)$ in \eqref{4.39}, we claim that
\begin{eqnarray*}
\widetilde{E}(t)\geq c\sum_{|\alpha|\leq1}\{\|\partial^{\alpha}(\widetilde{\rho},\widetilde{u},\widetilde{\theta})(t)\|^{2}
+\|\partial^{\alpha}f(t)\|^{2}+\|\partial^{\alpha}\widetilde{\phi}(t)\|^{2}
+\delta\|\partial^{\alpha}\partial_x\widetilde{\phi}(t)\|^{2}\}-C\delta^4,
\end{eqnarray*}
and 
\begin{eqnarray*}
	\widetilde{E}(t)\leq C\sum_{|\alpha|\leq1}\{\|\partial^{\alpha}(\widetilde{\rho},\widetilde{u},\widetilde{\theta})(t)\|^{2}
	+\|\partial^{\alpha}f(t)\|^{2}+\|\partial^{\alpha}\widetilde{\phi}(t)\|^{2}
	+\delta\|\partial^{\alpha}\partial_x\widetilde{\phi}(t)\|^{2}\}+C\delta^4.
\end{eqnarray*}
With the above estimates and  \eqref{1.33} in hand, for $0\leq t\leq T$ with $T>0$,
there exists a constant  $C_4>1$ such that
\begin{equation}
\label{7.7}
 C_4^{-1}\mathcal{E}_{2,l,q_1}(t)-C_4\delta^4\leq	\widetilde{\mathcal{E}}_{2,l,q_1}(t)\leq C_4\mathcal{E}_{2,l,q_1}(t)+C_4\delta^4.
\end{equation}
Therefore, we get from \eqref{7.7}, \eqref{7.6} and \eqref{1.33} that
\begin{equation}
	\label{7.8}
	\frac{d}{dt}(\widetilde{\mathcal{E}}_{2,l,q_1}(t)+C_4\delta^4)
	\leq C_3C_4(\widetilde{\mathcal{E}}_{2,l,q_1}(t)+C_4\delta^4).
\end{equation}
By the Gronwall inequality, we get from \eqref{7.8} that
\begin{equation}
	\label{7.9}
\sup_{0\leq t\leq T}\widetilde{\mathcal{E}}_{2,l,q_1}(t)+C_4\delta^4\leq (\widetilde{\mathcal{E}}_{2,l,q_1}(0)+C_4\delta^4)e^{C_3C_4T}.
\end{equation}
In summary,  letting $C_0=8C_4C_4e^{C_3C_4\tau}$ and $T\leq\tau$, then using \eqref{7.9}, \eqref{7.7} and \eqref{1.39}, we have
\begin{eqnarray}
	\label{7.10}
\sup_{0\leq t\leq T}\mathcal{E}_{2,l,q_1}(t)&&\leq C_4(\sup_{0\leq t\leq T}\widetilde{\mathcal{E}}_{2,l,q_1}(t)+C_4\delta^4)
\notag\\
&&
\leq C_4(\widetilde{\mathcal{E}}_{2,l,q_1}(0)+C_4\delta^4)e^{C_3C_4T}
\notag\\
&&\leq C_4e^{C_3C_4T}C_4(\mathcal{E}_{2,l,q_1}(0)+\delta^4)+(C_4)^2e^{C_3C_4T}\delta^4
\notag\\
&&\leq 4C_4C_4e^{C_3C_4T}\delta^4\leq \frac{1}{2}C_0\delta^4.
\end{eqnarray}
Thus the a priori estimate \eqref{3.1} can be closed by \eqref{7.10}.
By the uniform a priori estimates \eqref{7.10} and the local existence of the solution,
the standard continuity argument gives the existence and uniqueness of
long-time solutions to the VPL system \eqref{1.15} and the initial data \eqref{1.39}.

\subsection{KdV limit}
To complete the proof of Theorem \ref{thm1.1}, we still need to prove the uniform convergence rate in
$\delta$ as in \eqref{1.41} and \eqref{1.42}. Note from \eqref{3.4} and \eqref{3.5} that $(\rho,u,\theta)$ and $(\bar{\rho},\bar{u},\bar{\theta})$
are close enough to the state $(1,0,3/2)$ due to the smallness of $\delta$, we deduce from  these facts, \eqref{3.11} and \eqref{7.10} that
\begin{eqnarray*}
	&&\|\frac{(M_{[\rho,u,\theta]}-M_{[\bar{\rho},\bar{u},\bar{\theta}]})(t)}{\sqrt{\mu}}\|_{L_{x}^{2}L_{v}^{2}}
	+\|\frac{(M_{[\rho,u,\theta]}-M_{[\bar{\rho},\bar{u},\bar{\theta}]})(t)}{\sqrt{\mu}}\|_{L_{x}^{\infty}L_{v}^{2}}
\notag\\
	&&\leq C(\|(\widetilde{\rho},\widetilde{u},\widetilde{\theta})(t)\|+\|(\widetilde{\rho},\widetilde{u},\widetilde{\theta})(t)\|_{L_{x}^{\infty}})\leq C\delta^2,
\end{eqnarray*}
for any $t\in[0,\tau]$. Similarly, it holds from \eqref{7.10} and \eqref{3.7} that
\begin{equation*}
	\sup_{t\in[0,\tau]}(\|f(t)\|_{L^{2}_{x}L^{2}_{v}}+\|f(t)\|_{L^{\infty}_{x}L^{2}_{v}})\leq C\delta^2,
\end{equation*}
and
\begin{equation*}
	\sup_{t\in[0,\tau]}(
	\|\frac{\overline{G}(t)}{\sqrt{\mu}}\|_{L_{x}^{2}L_{v}^{2}}
	+\|\frac{\overline{G}(t)}{\sqrt{\mu}}\|_{L_{x}^{\infty}L_{v}^{2}})
	\leq C\delta^3.
\end{equation*}
With these three estimates in hand, we have from $F=M+\overline{G}+\sqrt{\mu}f$ that
\begin{equation}
	\label{7.11}
	\|\frac{(F-M_{[\bar{\rho},\bar{u},\bar{\theta}]})(t)}{\sqrt{\mu}}\|_{L_{x}^{2}L_{v}^{2}}
	+\|\frac{(F-M_{[\bar{\rho},\bar{u},\bar{\theta}]})(t)}{\sqrt{\mu}}\|_{L_{x}^{\infty}L_{v}^{2}}
	\leq C\delta^2,
\end{equation}
for any $t\in[0,\tau]$. Using \eqref{1.29}, \eqref{3.11}, \eqref{1.13} and $\hat{u}=(u_1^{(1)},0,0)$, we get
\begin{eqnarray*}
	&&\|\frac{(M_{[\bar{\rho},\bar{u},\bar{\theta}]}- M_{[1+\delta\rho_1,\delta \hat{u},\frac{3}{2}+\frac{3}{2}\delta\theta_1]})(t)}{\sqrt{\mu}}\|_{L_{x}^{2}L_{v}^{2}}
	\notag\\
	&&\leq C(\|\bar{\rho}-1-\delta\rho_1\|+\|\bar{u}-\delta \hat{u}\|+\|\bar{\theta}-\frac{3}{2}-\frac{3}{2}\delta\theta_1\|)
	\leq C\delta^2.
\end{eqnarray*}
Likewise, it holds that
\begin{equation*}
\|\frac{(M_{[\bar{\rho},\bar{u},\bar{\theta}]}- M_{[1+\delta\rho_1,\delta \hat{u},\frac{3}{2}+\frac{3}{2}\delta\theta_1]})(t)}{\sqrt{\mu}}\|_{L_{x}^{\infty}L_{v}^{2}}
	\leq C\delta^2.
\end{equation*}
With the above two estimates and \eqref{7.11} in hand, for any $t\in[0,\tau]$, we get
\begin{equation}
	\label{7.13}
	\|\frac{(F-M_{[1+\delta\rho_1,\delta \hat{u},\frac{3}{2}+\frac{3}{2}\delta\theta_1]})(t)}{\sqrt{\mu}}\|_{L_{x}^{2}L_{v}^{2}}
	+\|\frac{(F-M_{[1+\delta\rho_1,\delta \hat{u},\frac{3}{2}+\frac{3}{2}\delta\theta_1]})(t)}{\sqrt{\mu}}\|_{L_{x}^{\infty}L_{v}^{2}}
	\leq C\delta^2.
\end{equation}
On the other hand,  by \eqref{1.29} and \eqref{3.11}, we obtain easily that
\begin{equation}
\label{7.14}
\|(\phi-\delta\phi_{1})(t)\|+\|(\phi-\delta\phi_{1})(t)\|_{L_{x}^{\infty}}\leq C\delta^2.
\end{equation}
This combined with \eqref{7.13} gives \eqref{1.41} and \eqref{1.42}, and hence ends the proof of Theorem \ref{thm1.1}. \qed

\medskip
\noindent {\bf Acknowledgment:}\,
The research of Renjun Duan was partially supported by the General Research Fund (Project No.~14301719) from RGC of Hong Kong and a Direct Grant from CUHK. The research of Hongjun Yu was supported by the GDUPS 2017 and the NNSFC Grant 11371151. Dongcheng Yang would like to thank Department of Mathematics, CUHK  for hosting his visit in the period 2020-2023.

\medskip

\noindent{\bf Conflict of Interest:} The authors declare that they have no conflict of interest.

\end{document}